%% file: main.tex
\documentclass{amsart}

\input{preamble}
\input{macros}

\title{Moment Polyptychs and the Equivariant Quantisation of Hypertoric Varieties}

\author{Benjamin C. W. Brown}
\address{The University of Edinburgh}

\begin{document}

\begin{abstract}
	We study the equivariant geometric quantisation of a non-compact hypertoric variety of complex dimension $2n$ using symplectic cutting.
	When the defining arrangement is cooriented so that its base region is non-empty and bounded, fibre rotation has proper moment map and produces compact K\"ahler orbifold cuts carrying the residual torus action.
	We introduce the moment polyptych of such a cut, a polytopal complex assembled from truncated arrangement regions, and show that its vertices and local lattice data determine the fixed points, tangent weights, and finite orbifold groups required for equivariant orbifold Riemann--Roch localisation.
	Even for a smooth hypertoric manifold, a reducible core forces every cut in the outer cut chamber to have non-trivial cyclic isotropy, while the cuts are smooth if and only if the manifold is the cotangent bundle of its distinguished compact toric subvariety, which is then a product of projective spaces.
	For a positive regular integral level $d$ in the outer cut chamber, algebraic cutting identifies the cut's quantisation with a finite filtration piece of the circle-graded quantisation of the original hypertoric variety, whose successive quotients recover the individual circle-weight spaces.
	We prove that the higher cohomology of the pre-quantum line orbibundle vanishes at every such level, so that the localised index is the character of the actual space of sections.
	Finally, a residual weight $x$ occurs at level $d$ if and only if $\mcS(x) \leq d$, where $\mcS$ is the arrangement slack function, with multiplicity $\binom{d - \mcS(x) + n}{n}$.
	Thus, each finite stage admits a weighted lattice-point description whose multiplicity function is supported on the lattice points in the support of the level-$d$ moment polyptych.
\end{abstract}

\maketitle

\input{sections/1-introduction}

\input{sections/2-hypertoric-varieties}

\input{sections/3-hyperplanes-subvarieties}

\input{sections/4-cut-spaces}

\input{sections/5-localisation-formulae-isotropy-data}

\input{sections/6-hypertoric-equivariant-quantisation}

\input{sections/7-algebraic-cutting-combinatorial-quantisation}

\input{sections/8-conclusion}

\printbibliography

\end{document}

%% file: preamble.tex
\usepackage[T1]{fontenc}
\usepackage[utf8]{inputenc}
\usepackage[lining]{ebgaramond}
\usepackage{microtype}
\usepackage{dsfont}

\usepackage{amsmath,amssymb,amsthm}
\usepackage{stmaryrd} 
\usepackage{braket}   

\usepackage{geometry}
\usepackage[skip=0pt,indent=20pt]{parskip}

\usepackage[shortlabels]{enumitem}

\usepackage{graphicx}
\usepackage{float}
\usepackage{tikz}
\usepackage{tikz-cd}
\usetikzlibrary{calc,arrows.meta}

\usepackage{subcaption}

\usepackage[
	backend=biber,
	style=alphabetic,
	sorting=nyt,
	maxalphanames=4,
	minalphanames=4,
	maxbibnames=99
]{biblatex}

\usepackage{hyperref}

\hypersetup{
	hidelinks
}

\tikzset{
	polyboundary/.style={line width=0.9pt},
	polyedge/.style={line width=0.45pt},
	vertex/.style={circle,fill,inner sep=1.35pt},
	region/.style={font=\scriptsize,inner sep=1pt},
}

\theoremstyle{plain}
\newtheorem{theorem}{Theorem}[section]
\newtheorem{proposition}[theorem]{Proposition}
\newtheorem{corollary}[theorem]{Corollary}
\newtheorem{lemma}[theorem]{Lemma}

\theoremstyle{definition}
\newtheorem{definition}[theorem]{Definition}
\newtheorem{example}[theorem]{Example}

\theoremstyle{remark}
\newtheorem{remark}[theorem]{Remark}
\newtheorem{notation}[theorem]{Notation}

%% file: macros.tex
\newcommand{\ra}{\rightarrow}
\newcommand{\lra}{\longrightarrow}

\newcommand{\w}{\omega}
\newcommand{\W}{\Omega}

\newcommand{\PP}{\mathbb{P}}
\newcommand{\RR}{\mathbb{R}}
\newcommand{\CC}{\mathbb{C}}
\newcommand{\QQ}{\mathbb{Q}}
\newcommand{\FF}{\mathbb{F}}
\newcommand{\ZZ}{\mathbb{Z}}
\newcommand{\HH}{\mathbb{H}}
\newcommand{\one}{\mathds{1}}

\newcommand{\mcA}{\mathcal{A}}
\newcommand{\mcB}{\mathcal{B}}
\newcommand{\mcC}{\mathcal{C}}
\newcommand{\mcE}{\mathcal{E}}
\newcommand{\mcF}{\mathcal{F}}
\newcommand{\mcFc}{\mathcal{F}^{c}}
\newcommand{\mcG}{\mathcal{G}}
\newcommand{\mcI}{\mathcal{I}}

\newcommand{\mcJ}{\mathcal{J}}
\newcommand{\mcL}{\mathcal{L}}

\newcommand{\mcN}{\mathcal{N}}
\newcommand{\mcO}{\mathcal{O}}
\newcommand{\mcP}{\mathcal{P}}
\newcommand{\mcR}{\mathcal{R}}
\newcommand{\mcS}{\mathcal{S}}
\newcommand{\mcV}{\mathcal{V}}

\newcommand{\mfk}{\mathfrak{k}}
\newcommand{\mft}{\mathfrak{t}}

\makeatletter
\newcount\wb@n
\newdimen\wb@delta
\newcommand{\wb@xht}[1]{%
	\ifx#1\displaystyle\fontdimen5\textfont2 \else
	\ifx#1\textstyle\fontdimen5\textfont2 \else
	\ifx#1\scriptstyle\fontdimen5\scriptfont2 \else
	\fontdimen5\scriptscriptfont2 \fi\fi\fi}
\newsavebox{\wb@nuc}
\newsavebox{\wb@acc}
\def\wb@rep{%
	\ifnum\wb@n>\z@
	\copy\wb@acc \kern-\wd\wb@acc \kern\wb@delta
	\advance\wb@n\m@ne
	\expandafter\wb@rep
	\fi}
\newcommand{\wb@draw}[2]{%
	\sbox\wb@nuc{$\m@th#1#2$}%
	\sbox\wb@acc{$\m@th#1\mathchar"0016$}%
	\@tempdimb=\ht\wb@acc
	\dimen@=\ht\wb@nuc \advance\dimen@-\wb@xht{#1}\relax
	\ifdim\dimen@>\z@ \advance\@tempdimb\dimen@ \fi
	\advance\@tempdimb-\ht\wb@acc
	\ifdim\wd\wb@nuc>\wd\wb@acc
	\wb@delta=.25\wd\wb@acc                
	\wb@n=\@ne \@tempdimc=\wd\wb@acc
	\@whiledim\@tempdimc<\wd\wb@nuc\do{%
		\advance\@tempdimc\wb@delta \advance\wb@n\@ne}%
	\wb@delta=\dimexpr(\wd\wb@nuc-\wd\wb@acc)/\numexpr\wb@n-1\relax\relax
	\advance\wb@n\m@ne                     
	\else
	\wb@n=\z@ \wb@delta=\z@
	\fi
	\mathord{\rlap{\copy\wb@nuc}%
		\raise\@tempdimb\hbox to\wd\wb@nuc{\hfil\wb@rep\copy\wb@acc\hfil}}}
\newcommand{\widebar}[1]{\mathpalette\wb@draw{#1}}
\makeatother

\newcommand{\sssslash}{\mathbin{/\!\!/\!\!/\!\!/}}

\newcommand{\Cout}{\mathfrak{C}_{\mathrm{out}}}
\newcommand{\orb}{\mathrm{orb}}
\newcommand{\diff}{\mathrm{diff}}
\newcommand{\aux}{\mathrm{aux}}
\renewcommand{\div}{\mathrm{div}}
\newcommand{\coarse}{\mathrm{coarse}}
\newcommand{\alg}{\mathrm{alg}}
\newcommand{\bd}{\mathrm{bd}}
\newcommand{\eff}{\mathrm{eff}}
\newcommand{\std}{\mathrm{std}}
\renewcommand{\aux}{\mathrm{aux}}

\newcommand{\aff}{\mathrm{aff}}
\newcommand{\tang}{\mathrm{tan}}
\newcommand{\fib}{\mathrm{fib}}

\DeclareMathOperator{\CritVal}{CritVal}

\DeclareMathOperator{\sign}{sgn}

\DeclareMathOperator{\Tot}{Tot}
\DeclareMathOperator{\supp}{supp}
\DeclareMathOperator{\mult}{mult}
\DeclareMathOperator{\Conv}{Conv}
\DeclareMathOperator{\Sym}{Sym}
\DeclareMathOperator{\Verts}{Vert}

\DeclareMathOperator{\Identity}{Id}

\DeclareMathOperator{\Hom}{Hom}

\DeclareMathOperator{\Image}{im}
\DeclareMathOperator{\Span}{Span}

\DeclareMathOperator{\Td}{Td}
\DeclareMathOperator{\Ch}{Ch}

\DeclareMathOperator{\rk}{rk}
\DeclareMathOperator{\crk}{crk}
\DeclareMathOperator{\codim}{codim}

\DeclareMathOperator{\HK}{HK}
\DeclareMathOperator{\cut}{cut}
\DeclareMathOperator{\Proj}{Proj}
\DeclareMathOperator{\Spec}{Spec}
\DeclareMathOperator{\sstable}{ss}
\DeclareMathOperator{\ustable}{us}

\DeclareMathOperator{\diag}{diag}
\DeclareMathOperator{\wt}{wt}
\DeclareMathOperator{\Faces}{Faces}

\DeclareMathOperator{\SU}{SU}
\DeclareMathOperator{\SL}{SL}

%% file: sections/1-introduction.tex
\section{Introduction}
\label{section:section-1-introduction}

For a compact K\"ahler manifold $M$ equipped with a pre-quantum line bundle $\mcL$, the cohomology groups $H^{q}(M ; \mcL)$ are finite-dimensional, and Riemann--Roch computes the holomorphic Euler characteristic of $\mcL$.
On a non-compact space, however, the section space may be infinite-dimensional, so its dimension alone is no longer an adequate description of the quantisation.
If a compatible circle action exists, then this suggests an equivariant approach: when the resulting weight spaces are finite-dimensional, their dimensions can be assembled into a formal character.
In this paper, we study the equivariant geometric quantisation of hypertoric varieties through compact cuts, relating the fibre-rotation grading to their section spaces and comparing the orbifold localisation with a weighted lattice-point count.

For this approach, one motivation comes from the equivariant Verlinde formula for moduli spaces of Higgs bundles, proposed for $\SL(n,\CC)$ by Gukov and Pei \cite{GukovPei2017}.
Andersen, Gukov, and Pei generalise and prove the formula for simple and simply connected reductive groups \cite{AndersenGukovPei2016}, while Halpern-Leistner establishes it independently in an algebro-geometric setting \cite{HalpernLeistner2016}.
Here, the circle action scaling the Higgs field provides the grading through which the non-compact quantisation is studied.

A related motivation is the connection between real polarisations, toric geometry, and the ordinary Verlinde formula.
For the moduli space of flat $\SU(2)$-connections on a compact Riemann surface, Weitsman constructs a map whose generic fibres are Lagrangian, giving a real polarisation away from the singular locus \cite{Weitsman1992}.
Jeffrey and Weitsman identify the corresponding Bohr--Sommerfeld fibres and show that their number is given by the Verlinde formula \cite{Verlinde1988,JeffreyWeitsman1992}; a separate argument identifies that formula with the dimension of the K\"ahler quantisation \cite[\S1]{JeffreyWeitsman1992}.
Their later work describes the fibres as orbits of commuting Hamiltonian circle actions and gives a dense open subset of the moduli space the structure, as a symplectic manifold, of a non-compact toric variety \cite{JeffreyWeitsman1994}.
They construct a natural toric compactification, but its complex structure does not agree with the one inherited from the moduli space, so the lattice-point count does not directly yield an algebraic comparison of section spaces \cite[p.~153]{JeffreyWeitsman1994}.

These constructions suggest a broader question: can the circle action scaling the Higgs field be incorporated into an analogous geometric description of equivariant quantisation?
Here, we work with hypertoric varieties~---~hyperk\"ahler quotients by tori~---~where the fibre-rotation action, its moment map, and the combinatorics of the hyperplane arrangement can be analysed explicitly.
Our section-space comparison concerns compact cuts of these varieties; we do not address any extension to Higgs-bundle moduli.

Martens's calculation of equivariant integrals over non-compact symplectic and hyperk\"ahler quotients is a more direct antecedent for our method \cite{Martens2008}.
To compute equivariant volumes, Martens passes to compact symplectic cuts, applies the Jeffrey--Kirwan residue formula in the compact setting, and studies the limit in which the original non-compact problem is recovered.
The cuts introduce additional fixed-point components at infinity, whose contributions are essential to the calculation.
For equivariant geometric quantisation, this suggests retaining the individual cuts as finite stages of the non-compact problem.
We make this idea precise through algebraic cutting: at positive regular integral levels in the outer cut chamber, the section spaces of the cuts realise finite filtration pieces, without requiring an asymptotic limit of the fixed-point expressions.

Hausel's compactification of Higgs-bundle moduli provides a second antecedent.
For stable rank-two Higgs bundles with fixed determinant of odd degree over a compact Riemann surface of genus at least two, the moduli space is hyperk\"ahler and the circle action that scales the Higgs field has a proper moment map \cite[Theorem~6.7 and Proposition~7.1]{Hitchin1987}.
Hausel applies Lerman's symplectic-cut construction \cite{Lerman1995} to obtain a compact K\"ahler orbifold which compactifies this moduli space \cite{Hausel1998}.
He observes that the construction applies more generally and suggests studying it for the toric hyperk\"ahler manifolds of Bielawski and Dancer \cite[\S1]{Hausel1998}.
We carry out that suggestion for hypertoric varieties \cite{BielawskiDancer2000} and use the resulting cuts to study their equivariant geometric quantisation.

Fix a subtorus $K\subset T^{N}$ as in Section~\ref{section:section-2-hypertoric-varieties}, and let $M_{\nu}$ be the hyperk\"ahler reduction of $T^{\ast}\CC^{N}$ by $K$ at a generic parameter $(\nu , 0)$ with a non-empty quotient, and write $n = N - \dim K$.
Then $M_{\nu}$ has real dimension $4n$ and carries a residual action of $T^{n} =T^{N} / K$; finite orbifold stabilisers are allowed unless smoothness is explicitly assumed.
The primitive lattice vectors $u_{i}$ defining $K$, together with a fixed lift $\lambda$ of $\nu$, determine the affine hyperplane arrangement
\[
	\mcA = \{ H_{1}, \ldots, H_{N} \}, \qquad
	H_{i} = \{ x \in (\mft^{n})^{\ast} ~|~ \ell_{i}(x) = 0 \},
	\qquad
	\ell_{i}(x) = \langle x ,\, u_{i} \rangle + \lambda_{i}.
\]
For the quantisation statements, we assume that the lift $\lambda$ is integral and use the corresponding $T^{n}$-equivariant pre-quantum line orbibundle $\mcL_{\nu}\to M_{\nu}$.

For $A \subseteq \{ 1 , \ldots , N \}$, let $\Delta_{A}$ be the closed region defined by $\ell_{i} \leq 0$ for $i\in A$ and $\ell_{i} \geq 0$ for $i \not\in A$.
We choose the coorientation so that the distinguished base region $\Delta_{\emptyset}$ is non-empty and bounded.
The non-empty regions are the real moment images of the toric orbifolds $\mcE_{A}$ forming the extended core $\mcE_{\nu} = \mu_{\CC}^{-1}(0)$.
Those components whose regions are bounded form the compact core $\mcC$, and the distinguished component is $X_{\nu} = \mcE_{\emptyset}$.

The fibre-rotation circle acts by $\tau\cdot[z , w]=[z , \tau w]$ for $\tau \in U_{1}$ and is Hamiltonian for the distinguished K\"ahler form $\w_{\RR}$ with moment map $\Phi[z , w] = \tfrac{1}{2}\|w\|^{2}$.
Boundedness of $\Delta_{\emptyset}$ ensures that $\Phi$ is proper \cite[Proposition~1.3 and the following remark]{HaradaProudfoot2004}.
We write $\delta > 0$ for a regular cut level, using $d$ when the level is also integral.
Reducing $M_{\nu} \times \CC$ by the diagonal cutting circle at level $\delta$ gives the compact K\"ahler orbifold $M_{\nu}^{\leq \delta}$.
Its cut interior $M_{\nu}^{< \delta}$ is $T^{n}$-equivariantly symplectomorphic to $\{\Phi < \delta\}$, while the cut divisor is $Z_{\nu}^{\delta} \cong \Phi^{-1}(\delta) / U_{1}$.
The cut inherits a K\"ahler structure from the product reduction, but is not generally hyperk\"ahler.
At a positive regular integral level $d$, the pre-quantum construction also gives the descended line orbibundle $\mcL_{\nu}^{\leq d} \ra M_{\nu}^{\leq d}$.

The geometry of the truncation is encoded by the piecewise-linear \textbf{slack function}
\[
	\mcS(x) = \sum\nolimits_{i = 1}^{N} \max\{ 0 ,\, -\ell_{i}(x) \}.
\]
On each extended-core component $\mcE_{A}$, we have $\Phi = \mcS \circ \mu_{\RR}$, so Proposition~\ref{prop:cut-component-image} identifies the moment image of its cut component $\mcE_{A}^{\leq \delta}$ with $\Delta_{A}^{\leq \delta} = \Delta_{A} \cap \{ \mcS \leq \delta \}$.
These truncated polytopes agree on common faces and form the \textbf{moment polyptych} $\Delta_{\nu}^{\leq \delta}$, and its support is the residual moment image of the full cut: $|\Delta_{\nu}^{\leq \delta}| = \mu_{\RR}(M_{\nu}^{\leq \delta}) = \{ \mcS \leq \delta \}$ by Proposition~\ref{prop:slack-support}.
The distinction between the polyptych and its support is important: the subdivision, together with the arrangement and cutting lattice data, retains information needed for the local orbifold calculations.

Proposition~\ref{prop:critical-values-slack} identifies the critical values of $\Phi$ with the finitely many values $\mcS(v)$ at arrangement vertices.
The cut chambers are the connected components of $(0,\infty)\setminus\CritVal(\Phi)$, and the \textbf{outer cut chamber} is $\Cout=(c_{\max},\infty)$, where $c_{\max}:=\max\CritVal(\Phi)$, as in \eqref{eqn:outer-cut-chamber}.
For $\delta \in \Cout$, the cut divisor is disjoint from the compact core, and $M_{\nu}^{<\delta}$ is biholomorphic to all of $M_{\nu}$, as explained in Subsection~\ref{subsection:subsection-4.2-cut-spaces}.

The arrangement description leads to two structural results used in the quantisation calculations.
First, Theorem~\ref{thm:hypertoric-subvariety} associates a closed hypertoric subvariety $M_{\mcF}\subseteq M_{\nu}$ to each arrangement flat $H_{\mcF}$.
For local calculations, we retain the local groups of the ambient quotient, even when some elements act trivially on the restricted charts; removing this ineffective isotropy gives the effective orbifold structure.
Proposition~\ref{prop:cut-subspace-subpolyptych} makes the construction compatible with cutting and restriction of the moment polyptych, allowing local questions to be reduced to lower-dimensional hypertoric geometry.
The construction for a single arrangement hyperplane in the smooth setting is given by Konno \cite[\S7, Claim~7.1 and its proof]{Konno2000}.

Second, the orbifold setting is necessary even when $M_{\nu}$ itself is smooth.
Theorem~\ref{thm:reducible-core-orbifold} shows that, for a smooth hypertoric manifold with reducible core, every outer-chamber cut contains a boundary point with non-trivial cyclic isotropy and a singular coarse local quotient.
For smooth $M_{\nu}$, Corollary~\ref{corollary:smooth-cut-space-characterisation} shows that an outer-chamber cut is smooth if and only if the distinguished open subset identified with $T^{\ast}X_{\nu}$ is all of $M_{\nu}$.
In that case, $X_{\nu}$ is a product of projective spaces, and every positive level is regular and gives a smooth cut \cite[Theorem~7.2]{BielawskiDancer2000}.

These geometric calculations lead to one of the main results of the paper: Theorem~\ref{theorem:character-hypertoric-orbifold} gives an explicit Kawasaki localisation formula for the equivariant holomorphic Euler characteristic of
$\mcL_{\nu}^{\leq d}$ on $M_{\nu}^{\leq d}$.
By Lemma~\ref{lemma:cut-space-fixed-points}, the residual $T^{n}$-fixed points correspond to vertices of the moment polyptych.
The edges and lattice data at each vertex determine the corresponding finite orbifold group and tangent representation; together with the pre-quantum fibre data, these specify the numerator and denominator factors of the summands averaged over that group to obtain the fixed-point contribution.
At a boundary fixed point, the cutting weight occurs twice in the full tangent representation, giving a squared denominator factor even when the local group is trivial.
The formula computes the virtual class $\sum_{q\geq0}(-1)^{q} [H^{q}(M_{\nu}^{\leq d} ; \mcL_{\nu}^{\leq d})]$ in $R(T^{n})$, and we later identify it with $[H^{0}(M_{\nu}^{\leq d} ; \mcL_{\nu}^{\leq d})]$ by proving higher-cohomology vanishing in the outer cut chamber in Corollary~\ref{cor:outer-chamber-vanishing}.

To relate these compact section spaces to the original non-compact quantisation, we use an algebraic description of cutting following Edidin and Graham \cite{EdidinGraham1998}.
Here $H^{0}(M_{\nu} ; \mcL_{\nu})$ denotes the space of algebraic sections in the complex structure $J_{1}$, with the pullback action fixed in \eqref{eqn:pullback-action-on-sections}.
The fibre-rotation $\CC^{\ast}$-action gives a weight decomposition and hence an increasing exhaustive filtration
\[
	F_{d}H^{0}(M_{\nu} ; \mcL_{\nu}) = \bigoplus\nolimits_{i \leq d} H^{0}(M_{\nu} ; \mcL_{\nu})_{i}, \qquad d \in \ZZ_{\geq0}.
\]
Theorem~\ref{theorem:combinatorial-multiplicity} shows that the grading is non-negative and every filtration piece is finite-dimensional.

For every positive regular integral level $d$ in the outer cut chamber, Theorem~\ref{theorem:algebraic-cut-weight-space} gives the natural $T^{n}$-equivariant isomorphism
\[
	\Psi_{d}:F_{d}H^{0}(M_{\nu} ; \mcL_{\nu})
	\overset{\sim}{\longrightarrow}
	H^{0}(M_{\nu}^{\leq d} ; \mcL_{\nu}^{\leq d})
\]
from \eqref{eqn:filtration-cut-sections}, hence the compact cut space realises a finite stage of the fibre-rotation filtration.
For every $d \geq 1$, each successive quotient yields
\[
	\frac{F_{d} H^{0}(M_{\nu} ; \mcL_{\nu})}{F_{d - 1} H^{0}(M_{\nu} ; \mcL_{\nu})} \cong H^{0}(M_{\nu} ; \mcL_{\nu})_{d}.
\]
When both $d$ and $d-1$ are positive regular integral levels in the outer cut chamber, this becomes a quotient of the spaces of sections on the two cut spaces.
Under the identification of the two cuts with their common complex orbifold, the injection from the level-$(d-1)$ section space to the level-$d$ section space is multiplication by the canonical section of the cut divisor, as constructed in the proof of Lemma~\ref{lemma:variation-polarisation-lemma}.

The vanishing argument in Proposition~\ref{prop:Nadel-vanishing} compares the cut polarisations at levels $d$ and $d + n + 1$ on their common complex orbifold and establishes the higher-cohomology vanishing whenever $d > 0$ is a regular integral level and both levels lie in the same cut chamber.
Note that this is a sufficient chamber condition, not a characterisation of all vanishing levels~--~however, it is automatic in the outer cut chamber, giving Corollary~\ref{cor:outer-chamber-vanishing}.
Thus, at every positive regular integral level in that chamber, the localisation formula computes the character of $H^{0}(M_{\nu}^{\leq d} ; \mcL_{\nu}^{\leq d})$.
The argument takes place in the orbifold category and does not require the pre-quantum line orbibundle to descend to a line bundle on the coarse space.

The same filtration has a direct combinatorial description.
For a residual weight $x\in(\mft_{\ZZ}^{n})^{\ast}$, the exponent differences determine a minimal monomial $p_{x}$ of fibre-rotation weight $\mcS(x)$.
After imposing the complex moment-map equations, the subspace of $H^{0}( M_{\nu} ; \mcL_{\nu} )$ of residual $T^{n}$-weight $x$ is generated by the section represented by $p_{x}$ over the polynomial ring $\CC[q_{1}, \ldots, q_{n}]$, whose generators have cut degree one.
Together with the auxiliary coordinate $y$, these give $n + 1$ variables among which the remaining cut degree $d - \mcS(x)$ is distributed.
Theorem~\ref{theorem:combinatorial-multiplicity} therefore shows, for every integer $d \geq 0$, that $x$ occurs in $F_{d}H^{0}(M_{\nu} ; \mcL_{\nu})$ if and only if $\mcS(x)\leq d$, with multiplicity $\mult_{d}(x) = \tbinom{d-\mcS(x) + n}{n}$.

At every positive regular integral level in the outer cut chamber, Corollary~\ref{cor:weighted-lattice-point-character} gives the weighted lattice-point character of \eqref{eqn:weighted-lattice-point-character}:
\[
	\operatorname{char}_{T^{n}} H^{0}(M_{\nu}^{\leq d} ; \mcL_{\nu}^{\leq d})(e^{\xi}) = \sum_{x \in |\Delta_{\nu}^{\leq d}| \cap (\mft_{\ZZ}^{n})^{\ast}} \binom{d - \mcS(x) + n}{n} e^{\langle x ,\, \xi \rangle}.
\]
Moreover, the formula shows that the support $|\Delta_{\nu}^{\leq d}|$ determines which residual weights occur, while the slack function $\mcS$ determines their multiplicities.
This formula gives a hypertoric counterpart to the familiar toric description in which each integral point of the moment polytope contributes once with multiplicity one \cite{Hamilton2008}.

The localisation formula, algebraic cutting, and the weighted lattice-point count thus describe the same finite-dimensional quantisation at outer-chamber levels.
We compare these descriptions in the recurring examples of $T^{\ast}\CC\PP^{1}$, the $A_{r}$-resolution family, $T^{\ast}\CC\PP^{2}$, and $T^{\ast}(\CC\PP^{1} \times \CC\PP^{1})$.
Together, they illustrate both the smooth cotangent-bundle case and the finite-isotropy contributions that are always present at some boundary vertices for outer-chamber cuts of a smooth hypertoric manifold with a reducible core.

\subsection*{Acknowledgements}
\label{subsection:subsection-1.1-acknowledgements}

This article grew out of the author's doctoral research at the University of Edinburgh.
The author is especially grateful to Johan Martens for suggesting the problem that became the starting point for this work, and for his continued guidance throughout its development.
The author also thanks Andrew Dancer for many helpful conversations and for hosting a visit to the Mathematical Institute at the University of Oxford, whose hospitality and support are gratefully acknowledged.

\subsection*{Organisation}
\label{subsection:subsection-1.2-organisation}

Section~\ref{section:section-2-hypertoric-varieties} fixes the differential-geometric and algebro-geometric quotient descriptions, together with the moment-map, lattice, pre-quantum, and representation-theoretic conventions.
Section~\ref{section:section-3-hyperplanes-subvarieties} develops the arrangement geometry and the hypertoric subvarieties associated with its flats.

Section~\ref{section:section-4-cutting-hypertoric-variety} constructs the compact cuts and their moment polyptychs, introduces the recurring examples, and establishes the smoothness and orbifold-singularity results.
Section~\ref{section:section-5-localisation-formulae-isotropy-data} extracts the fixed-point, tangent, lattice, and fibre data needed for localisation.
Section~\ref{section:section-6-equivariant-quantisation-hypertoric-algebraic-cutting} assembles the explicit Kawasaki localisation formula and evaluates it in the examples.

Section~\ref{section:section-7-algebraic-cutting-combinatorial-quantisation} identifies the outer-chamber cut section spaces with filtration pieces, proves the required higher-cohomology vanishing, and derives the weighted lattice-point character.
It then compares the algebraic and localisation calculations in the same examples.
Section~\ref{section:section-8-conclusion} discusses the resulting picture and questions for further study.

%% file: sections/2-hypertoric-varieties.tex
\section{Hypertoric Varieties}
\label{section:section-2-hypertoric-varieties}

This section fixes the differential-geometric and algebro-geometric descriptions of hypertoric varieties used throughout the paper, together with the moment-map, pre-quantum, and representation-theoretic conventions needed later for cutting and quantisation.
The hyperk\"ahler quotient construction makes the residual torus action and its moment maps explicit, while the GIT description provides the holomorphic line orbibundles and the section rings required later for algebraic cutting and quantisation.

\subsection{Differential-Geometric Hypertoric Varieties}
\label{subsection:subsection-2.1-differential-hypertoric}

The differential-geometric point of view uses the hyperk\"ahler quotient construction of Hitchin, Karlhede, Lindstr\"om, and Ro\v{c}ek \cite{HitchinKarlhedeLindstromRocek1987}.
For every compact torus $T = \mft / \mft_{\ZZ}$, we adopt the period-one lattice convention: an integral weight $\alpha \in \mft_{\ZZ}^{\ast}$ determines the character $\chi_{\alpha}([\xi]) = \exp \, ( 2\pi\sqrt{-1}\langle \alpha ,\, \xi \rangle )$.

Write $(J_{1}, J_{2}, J_{3})$ for the standard hyperk\"ahler structure on $T^{\ast}\CC^{N} \cong \HH^{N}$, rescaled so that $\w_{\RR}$ is the standard flat K\"ahler form divided by $2 \pi$.
With respect to the complex structure $J_{1}$ and the standard complex coordinates $( z, w ) \in T^{\ast}\CC^{N}$, set
\begin{equation}
	\label{eqn:standard-hyperkahler-forms}
	\w_{\RR} := \frac{\sqrt{-1}}{4 \pi} \sum_{i = 1}^{N}
	\left(
		dz_{i} \wedge d\overline{z}_{i} +
		dw_{i} \wedge d\overline{w}_{i}
	\right), \qquad
	\w_{\CC} := \frac{1}{2 \pi} \sum_{i = 1}^{N} dz_{i} \wedge dw_{i}.
\end{equation}
This fixes $\w_{\RR} = \w_{1}$ as the K\"ahler form of $J_{1}$ and $\w_{\CC} = \w_{2} + \sqrt{-1}\w_{3}$ as the corresponding holomorphic-symplectic form.
Throughout the paper, we use the moment-map convention
\begin{equation}
	\label{eqn:moment-map-sign-convention}
	d\langle \mu ,\, \xi \rangle = \w(\, - \,, X_{\xi}),
\end{equation}
where $X_{\xi}$ is the fundamental vector field of $\xi$, normalised according to the period-one convention above.

The real torus $T^{N}$ acts linearly on $T^{\ast}\CC^{N}$ by
\begin{equation}
	\label{eqn:real-TN-action}
	( t_{1} , \ldots , t_{N} ) \cdot ( z_{1} , \ldots , z_{N} , w_{1} , \ldots , w_{N} ) =
	( t_{1} z_{1} , \ldots , t_{N} z_{N} , t_{1}^{-1} w_{1} , \ldots , t_{N}^{-1} w_{N} ),
\end{equation}
which preserves the hyperk\"ahler structure and is tri-Hamiltonian.
The subtorus $K \subset T^{N}$ used in the hyperk\"ahler quotient is defined as follows.
Let $n < N$, choose a full-rank lattice $\mft_{\ZZ}^{n} \subset \mft^{n}$, and set $T^{n} := \mft^{n} / \mft_{\ZZ}^{n}$ and $\mft_{\QQ}^{n} := \mft_{\ZZ}^{n} \otimes_{\ZZ} \QQ$.
Choose a set of primitive non-zero vectors $u_{1} , \ldots , u_{N} \in \mft_{\ZZ}^{n}$ which generate $\mft_{\ZZ}^{n}$ as a $\ZZ$-module.

With the standard basis $e_{1} , \ldots , e_{N}$ of $\mft_{\ZZ}^{N} \cong \ZZ^{N}$, define the surjective lattice map $\pi_{\ZZ} : \mft_{\ZZ}^{N} \ra \mft_{\ZZ}^{n}$ by $\pi_{\ZZ}(e_{i}) = u_{i}$, and set $\mfk_{\ZZ} := \ker( \pi_{\ZZ} )$, $\mfk := \mfk_{\ZZ} \otimes_{\ZZ} \RR$ and $\mfk_{\ZZ}^{\ast} := \Hom_{\ZZ}( \mfk_{\ZZ} ; \ZZ )$.
Let $\imath_{\ZZ} : \mfk_{\ZZ} \hookrightarrow \mft_{\ZZ}^{N}$ denote the inclusion and write $\epsilon_{1} , \ldots , \epsilon_{N}$ for the dual basis of $( \mft_{\ZZ}^{N} )^{\ast}$.
All of these define a cocharacter-lattice sequence and its scalar extension to $\RR$, which are exact:
\begin{equation}
	\label{eqn:lie-algebra-seq}
	\begin{tikzcd}[ampersand replacement=\&, row sep=small]
		0 \ar[r] \& \mfk_{\ZZ} \ar[r, hook, "\imath_{\ZZ}"] \& \mft_{\ZZ}^{N} \ar[r, "\pi_{\ZZ}"] \& \mft_{\ZZ}^{n} \ar[r] \& 0,
		\\
		0 \ar[r] \& \mfk \ar[r, hook, "\imath"] \& \mft^{N} \ar[r, "\pi"] \& \mft^{n} \ar[r] \& 0.
	\end{tikzcd}
\end{equation}
Now set $K := \mfk / \mfk_{\ZZ}$; since $\mft_{\ZZ}^{n}$ is torsion-free, $\mfk_{\ZZ}$ is saturated in $\mft_{\ZZ}^{N}$, so $\imath$ embeds the connected torus $K$ as a subtorus of $T^{N}$, and the surjectivity of $\pi_{\ZZ}$ identifies this subtorus with the full kernel of $\pi : T^{N} \ra T^{n}$.
The induced sequence of compact tori is exact:
\begin{equation}
	\label{eqn:torus-seq}
	\begin{tikzcd}[ampersand replacement=\&]
		1 \ar[r] \& K \ar[r, hook, "\imath"] \& T^{N} \ar[r, "\pi"] \& T^{n} \ar[r] \& 1.
	\end{tikzcd}
\end{equation}
Note that we use the same symbols $\imath$ and $\pi$ for the induced maps on lattices, Lie algebras, and tori whenever no confusion can arise.
Dualising the Lie-algebra sequence \eqref{eqn:lie-algebra-seq} gives
\begin{equation}
	\label{eqn:dual-lie-algebra-seq}
	\begin{tikzcd}[ampersand replacement=\&]
		0 \ar[r] \& ( \mft^{n} )^{\ast} \ar[r, "\pi^{\ast}"] \& ( \mft^{N} )^{\ast} \ar[r, "\imath^{\ast}"] \& \mfk^{\ast} \ar[r] \& 0.
	\end{tikzcd}
\end{equation}
Throughout the paper, we identify $\ker( \imath^{\ast} )$ with
$( \mft^{n} )^{\ast}$ via $\pi^{\ast}$, and write $\imath_{\CC}^{\ast}$ and $\pi_{\CC}^{\ast}$ for the complexifications of $\imath^{\ast}$ and $\pi^{\ast}$.

The induced $K$-action on $T^{\ast}\CC^{N}$ has weight $+\alpha_{i}$ on $z_{i}$ and $-\alpha_{i}$ on $w_{i}$, where $\alpha_{i} := \imath_{\ZZ}^{\ast}(\epsilon_{i}) \in \mfk_{\ZZ}^{\ast}$.
With the convention \eqref{eqn:moment-map-sign-convention}, the real moment map is the map $\phi_{\RR}$ displayed below.
For the holomorphic component $\phi_{\CC}$, we instead retain the usual hypertoric normalisation:
\[
	\phi_{\RR}(z , w) = \frac{1}{2}\sum\nolimits_{i = 1}^{N}(|z_{i}|^{2} - |w_{i}|^{2})\alpha_{i},
	\qquad
	\phi_{\CC}(z , w) = \sum\nolimits_{i = 1}^{N}(z_{i}w_{i})\alpha_{i}.
\]
With the period-one normalisation of $\w_{\CC}$ in \eqref{eqn:standard-hyperkahler-forms}, the displayed map $\phi_{\CC}$ differs by a fixed non-zero scalar from the holomorphic moment map obtained by applying \eqref{eqn:moment-map-sign-convention} directly.
We write $\mfk_{\CC} := \mfk \otimes_{\RR} \CC$ and $\phi_{\HK} := (\phi_{\RR} , \phi_{\CC} ) : T^{\ast}\CC^{N} \ra \mfk^{\ast} \oplus \mfk_{\CC}^{\ast}$ for the normalised hyperkähler moment map.

Following \cite[Proposition~2.1]{Konno2000}, we call $\nu \in \mfk^{\ast}$ \textbf{generic} if
\begin{equation}
	\label{eqn:genericity-span-condition}
	\nu \not\in \Span_{\RR}\{ \alpha_{j} ~|~ j \in J \} \qquad
	\text{for every } J \subseteq \{ 1 , \ldots , N \} \text{ with } |J| < N - n.
\end{equation}
Konno's criterion says that $\nu$ is generic if and only if $( \nu , 0 )$ is a regular value of $\phi_{\HK}$.
When $\nu$ is integral, the same regularity condition is also equivalent to the equality of the semistable and stable loci for the corresponding GIT problem, so there are no strictly semistable points \cite[\S2.1]{Proudfoot2004}.
In the language of hyperplane arrangements which will be introduced in Section~\ref{section:section-3-hyperplanes-subvarieties}, this condition is equivalent to the associated arrangement being simple \cite[Theorem~3.3]{BielawskiDancer2000}.
We use the words \emph{generic}, \emph{regular}, and \emph{simple} for these equivalent conditions in their respective settings, and we assume throughout that $\nu$ is generic and that the resulting quotient is non-empty.

\begin{definition}
	\label{def:hypertoric-variety}
	The \textbf{(differential-geometric) hypertoric variety} associated with the value $\nu \in \mfk^{\ast}$ is the hyperk\"ahler quotient
	\begin{equation}
		\label{eqn:hypertoric-definition}
		M_{\nu} := T^{\ast}\CC^{N} \sssslash_{( \nu , 0 )} K :=
		\phi_{\HK}^{-1}( \nu , 0 ) / K.
	\end{equation}
\end{definition}

Since $(\nu , 0)$ is a regular value of $\phi_{\HK}$, the group $K$ acts locally freely on $\phi_{\HK}^{-1}(\nu , 0)$, and so $M_{\nu}$ is a hyperk\"ahler orbifold of real dimension $4N - 4\dim_{\RR}K = 4n$.
The quotient is a manifold if and only if the $K$-action is free, or equivalently when the associated arrangement is smooth in the sense recalled in Section~\ref{section:section-3-hyperplanes-subvarieties}; see \cite[Theorems~3.2--3.3]{BielawskiDancer2000}.
Note that we allow finite orbifold stabilisers unless smoothness is explicitly assumed.

Via \eqref{eqn:torus-seq}, the quotient torus $T^{N} / K$ is identified with $T^{n}$ and acts hyperhamiltonianly on $M_{\nu}$.
Fix, once and for all, a lift
\begin{equation}
	\label{eqn:fixed-lambda-lift}
	\lambda = \sum\nolimits_{i = 1}^{N} \lambda_{i}\epsilon_{i} \in (\mft^{N})^{\ast},
	\qquad
	\imath^{\ast}(\lambda) = \nu,
\end{equation}
and, whenever $\nu$ is integral and a GIT or pre-quantum construction is invoked, the fixed lift $\lambda$ will also be integral.
The existence of such an integral lift follows from \eqref{eqn:character-lattice-sequence} below.

Write $[z , w]$ for the $K$-orbit of $(z , w) \in \phi_{\HK}^{-1}(\nu , 0)$.
The residual hyperk\"ahler moment map $\mu_{\HK} = \mu_{\RR} \oplus \mu_{\CC} : M_{\nu} \ra (\mft^{n})^{\ast} \oplus (\mft_{\CC}^{n})^{\ast}$ is specified by
\[
	(\pi^{\ast} \circ \mu_{\RR})[z , w] = \frac{1}{2} \sum\nolimits_{i = 1}^{N} (|z_{i}|^{2} - |w_{i}|^{2}) \epsilon_{i} - \lambda, \qquad
	(\pi_{\CC}^{\ast} \circ \mu_{\CC})[z , w] = \sum\nolimits_{i = 1}^{N} (z_{i}w_{i}) \epsilon_{i}.
\]
The two right-hand sides lie respectively in $\ker(\imath^{\ast})$ and $\ker(\imath_{\CC}^{\ast})$, so the injective maps $\pi^{\ast}$ and $\pi_{\CC}^{\ast}$ determine $\mu_{\RR}$ and $\mu_{\CC}$ uniquely.

\begin{lemma}
	\label{lemma:effective-Tn-action}
	The residual $T^{n}$-action is effective.
\end{lemma}

\begin{proof}
	In anticipation of the notation used in Section~\ref{section:section-3-hyperplanes-subvarieties}, write $\ell_{i}(y) := \langle y ,\, u_{i} \rangle + \lambda_{i}$ for the affine-linear functionals on $(\mft^{n})^{\ast}$ that define each hyperplane.
	Since each $u_{i}$ is non-zero, each zero set $\{ y \in (\mft^{n})^{\ast} ~|~ \ell_{i}(y) = 0 \}$ is an affine hyperplane.
	Choose $x \in (\mft^{n})^{\ast}$ outside the union of these hyperplanes, so that $\ell_{i}(x) \neq 0$ for each $i$.
	For each $i$, define
	\[
		(z_{i},w_{i}) =
		\begin{cases}
			( \sqrt{2\ell_{i}(x)} , 0 ),  & \ell_{i}(x) > 0, \\
			( 0 , \sqrt{-2\ell_{i}(x)} ), & \ell_{i}(x) < 0.
		\end{cases}
	\]
	Then $(z_{i}, w_{i}) \neq (0, 0)$ for each $i$, so $\phi_{\CC}(z , w)=0$, and
	\[
		\frac{1}{2} \sum\nolimits_{i = 1}^{N} (|z_{i}|^{2} - |w_{i}|^{2}) \epsilon_{i} = \lambda + \pi^{\ast}(x),
	\]
	thus $(z , w) \in \phi_{\HK}^{-1}(\nu , 0)$.
	If $t \in T^{n}$ fixes $[z , w]$ and $\widetilde{t} \in T^{N}$ is a lift of $t$, then
	$\widetilde{t} \cdot (z , w) = k \cdot (z , w)$ for some $k \in K$.
	Since one entry of every pair $(z_{i} , w_{i})$ is non-zero, this forces $\widetilde{t}_{i} = k_{i}$ for every $i$, hence $\widetilde{t} = k$ and $t = 1$.
\end{proof}

The choice of lift $\lambda$ fixes the normalisation of the residual real moment map and the position of the associated hyperplane arrangement, introduced in Section~\ref{section:section-3-hyperplanes-subvarieties}.
If $\lambda$ is replaced by $\lambda + \pi^{\ast}(y)$, with $y \in (\mft^{n})^{\ast}$, then $\mu_{\RR}$ is replaced by $\mu_{\RR} - y$ while the arrangement and its regions are translated by $-y$.
We will henceforth continue to use the single choice of lift \eqref{eqn:fixed-lambda-lift} throughout the paper.


\subsection{Algebro-Geometric Hypertoric Varieties}
\label{subsection:subsection-2.2-algebraic-hypertoric}

For the algebro-geometric and pre-quantum constructions of this subsection, assume that $\nu \in \mfk_{\ZZ}^{\ast}$ is integral.
Applying $\Hom_{\ZZ}(- ; \ZZ)$ to the cocharacter-lattice sequence in \eqref{eqn:lie-algebra-seq} gives the exact sequence of character lattices
\begin{equation}
	\label{eqn:character-lattice-sequence}
	\begin{tikzcd}[ampersand replacement=\&]
		0 \ar[r] \& ( \mft_{\ZZ}^{n} )^{\ast} \ar[r, "\pi_{\ZZ}^{\ast}"] \& ( \mft_{\ZZ}^{N} )^{\ast} \ar[r, "\imath_{\ZZ}^{\ast}"] \& \mfk_{\ZZ}^{\ast} \ar[r] \& 0.
	\end{tikzcd}
\end{equation}
Since $\imath_{\ZZ}^{\ast}$ is surjective, every integral $\nu$ admits an integral lift (which justifies our choice of integral lift $\lambda$ following \eqref{eqn:fixed-lambda-lift}).
Applying the contravariant functor $\Hom_{\ZZ}( - ; \CC^{\ast})$ to this character-lattice sequence gives the exact sequence of complex algebraic tori:
\begin{equation}
	\label{eqn:complex-tori-sequence}
	\begin{tikzcd}[ampersand replacement=\&]
		1 \ar[r] \& K_{\CC} \ar[r, hook, "\imath_{\CC}"] \& T_{\CC}^{N} \ar[r, "\pi_{\CC}"] \& T_{\CC}^{n} \ar[r] \& 1,
	\end{tikzcd}
\end{equation}
where $T_{\CC}^{N} = \Hom_{\ZZ}( (\mft_{\ZZ}^{N})^{\ast} ; \CC^{\ast} ) \cong (\CC^{\ast})^{N}$ and $K_{\CC} = \Hom_{\ZZ}( \mfk_{\ZZ}^{\ast} ; \CC^{\ast} )$.
Note that the functor is exact because $\CC^{\ast}$ is divisible and hence injective as a $\ZZ$-module.
The complex torus $T_{\CC}^{N}$ acts on $T^{\ast}\CC^{N}$ by the complexification of \eqref{eqn:real-TN-action}, and the $K_{\CC}$-action preserves the zero fibre $\mcN := \phi_{\CC}^{-1}(0)$ of the complex moment map.

Let $\chi_{\nu} : K_{\CC} \ra \CC^{\ast}$ and $\chi_{\lambda} : T_{\CC}^{N} \ra \CC^{\ast}$ be the characters determined respectively by $\nu$ and $\lambda$, and write $L_{\nu} := \mcN \times \CC \ra \mcN$ for the trivial holomorphic line bundle equipped with the geometric $K_{\CC}$-linearisation
\begin{equation}
	\label{eqn:positive-character-linearisation}
	k \cdot ( x , v ) := ( k \cdot x , \chi_{\nu}(k) v ),
	\qquad
	k \in K_{\CC}.
\end{equation}
It carries a compatible $T_{\CC}^{N}$-equivariant extension determined by $\chi_{\lambda}$, which records the residual $T^{n}$-equivariant structure after descent.

\begin{notation}[Representation Convention]
	\label{notation:representation-convention}
	If a torus $T$ acts geometrically on a line bundle $L$, with the fibre map associated with $t \in T$ denoted by $t_{L}$, then its induced action on sections is the \emph{pullback action}
	\begin{equation}
		\label{eqn:pullback-action-on-sections}
		(t \cdot s)(x) := t_{L}^{-1} \big( s ( t \cdot x) \big).
	\end{equation}
	For the linearisation \eqref{eqn:positive-character-linearisation} and an $m \in \ZZ_{\geq 0}$, a section $s_{f}(x) = ( x , f(x) )$ of $L_{\nu}^{\otimes m}$ is $K_{\CC}$-invariant if and only if $f( g \cdot x ) = \chi_{\nu}(g)^{m} f(x)$ for every $g \in K_{\CC}$ and $x \in \mcN$.

	Henceforth, we write $\CC[\mcN]_{m\nu} := \{ f \in \CC[\mcN] ~|~ f( g \cdot x ) = \chi_{\nu}(g)^{m}f(x) \}$ for the positive $m\nu$-semi-invariant piece.
	For $m = 1$, if $f = z^{a}w^{b}$ satisfies $\imath_{\ZZ}^{\ast}( a - b ) = \nu$, then there is a unique $x \in ( \mft_{\ZZ}^{n} )^{\ast}$ such that $a - b = \lambda + \pi_{\ZZ}^{\ast}(x)$.
	Thus, under the pullback action \eqref{eqn:pullback-action-on-sections}, the corresponding descended section has residual $T^{n}$-weight $+x$.
	This convention is used in the monomial calculation of Section~\ref{section:section-7-algebraic-cutting-combinatorial-quantisation}.
\end{notation}

As noted after \eqref{eqn:genericity-span-condition}, the semistable and stable loci coincide for a generic integral weight $\nu$; we denote the common locus by $\mcN^{\sstable}(\nu)$, and an explicit description of it relative to the hypertoric coordinates is given in \cite[Lemma~3.4]{Konno2008}.

\begin{definition}
	\label{def:GIT-hypertoric-variety}
	The \textbf{(algebro-geometric) hypertoric variety} associated with the generic integral character $\nu$ is
	\begin{equation}
		\label{eqn:GIT-hypertoric-definition}
		M_{\nu}^{\alg} := \mcN^{\sstable}(\nu) \sslash K_{\CC} \cong \Proj R_{\nu},
		\qquad \text{ where } \qquad
		R_{\nu} := \bigoplus\nolimits_{m \geq 0} \CC[\mcN]_{m\nu}.
	\end{equation}
\end{definition}

As the degree-zero piece is $R_{\nu , 0} = \CC[\mcN]^{K_{\CC}}$, the natural morphism
\begin{equation}
	\label{eqn:hypertoric-affinisation}
	M_{\nu}^{\alg} = \Proj R_{\nu} \longrightarrow \Spec \CC[\mcN]^{K_{\CC}}
\end{equation}
is projective; in the terminology of \cite[\S6]{HauselSturmfels2002}, this is the hypertoric affinisation morphism.

With respect to the complex structure $J_{1}$, the inclusion of the hyperk\"ahler level set into the semistable locus induces the hypertoric Kempf--Ness biholomorphism of complex orbifolds
\begin{equation}
	\label{eqn:hypertoric-kempf-ness}
	( \phi_{\HK}^{-1}( \nu , 0 ) / K , J_{1} )
	\overset{\sim}{\longrightarrow}
	\mcN^{\sstable}(\nu) / K_{\CC}.
\end{equation}
The geometric quotient on the right has the underlying scheme $M_{\nu}^{\alg}$; see \cite[Theorem~6.3]{HauselSturmfels2002} for the hypertoric Kempf--Ness biholomorphism and \cite[Theorem~2.3]{Proudfoot2004} for the general Kempf--Ness and regularity--stability statements.
Since the compact $T^{N}$-action commutes with $K_{\CC}$ and preserves both the hyperk\"ahler level set and the semistable locus, the inclusion is $T^{N}$-equivariant, and the induced identification is equivariant for the residual torus $T^{n} = T^{N} / K$.
We henceforth denote the common complex orbifold given by Definitions~\ref{def:hypertoric-variety} and \ref{def:GIT-hypertoric-variety} simply by $M_{\nu}$.

Let us now show that a hypertoric variety $M_{\nu}$ is \textbf{semiprojective} in the sense of \cite[\S2]{HauselSturmfels2002}: it is projective over its affinisation and possesses a fixed point of the residual algebraic torus.
The projectivity is supplied by the morphism \eqref{eqn:hypertoric-affinisation} and, since the vectors $u_{1}, \ldots, u_{N}$ span $\mft^{n}$, the associated hyperplane arrangement introduced in Section~\ref{section:section-3-hyperplanes-subvarieties} has a vertex.
By the stabiliser description in \cite[Theorem~3.1(ii)]{BielawskiDancer2000}, a point of the hyperk\"ahler quotient lying over this vertex is fixed by the residual torus $T^{n}$ hence, under the biholomorphism \eqref{eqn:hypertoric-kempf-ness}, it therefore corresponds to a fixed point of the residual algebraic torus and thus $M_{\nu}$ is semiprojective.

Since $u_{i} \neq 0$ for every $i$, the complex moment-map fibre $\mcN = \phi_{\CC}^{-1}(0)$ is a reduced irreducible complete intersection \cite[Lemma~4.7]{BellamyKuwabara2012}.
As its semistable locus is a non-empty open subset of $\mcN$, it is irreducible, so its GIT quotient $M_{\nu}$ is also irreducible.
Furthermore, for the generic parameter $\nu$, the quotient has only finite quotient singularities \cite[Proposition~6.2]{HauselSturmfels2002}, hence is normal.
Together, these facts show that $M_{\nu}$ is connected, irreducible, normal, and semiprojective as a $J_{1}$-complex variety.

Finally, let us fix the pre-quantum normalisation used in Sections~\ref{section:section-5-localisation-formulae-isotropy-data}--\ref{section:section-7-algebraic-cutting-combinatorial-quantisation}.
Equip $L_{\nu} = \mcN \times \CC$ with the Gaussian Hermitian metric whose squared norm in the standard holomorphic frame is $\|((z , w) , v)\|^{2} := |v|^{2} \exp(-\tfrac{1}{2}( \|z\|^{2} + \|w\|^{2} ) )$.
For a Hermitian holomorphic line bundle, let $h$ denote the squared norm of a local holomorphic frame, and let $F_{\nabla} = \overline{\partial} \partial \log h$ denote the (Chern) curvature.
With the normalisation \eqref{eqn:standard-hyperkahler-forms}, the Chern connection $\nabla$ of $L_{\nu} $ then has the curvature
\begin{equation}
	\label{eqn:prequantum-curvature-convention}
	F_{\nabla} = (2\pi / \sqrt{-1}) \, \w_{\RR}|_{\mcN}, \qquad
	c_{1}(L_{\nu}) = [ \w_{\RR}|_{\mcN} ].
\end{equation}
This is the standard Gaussian construction, rewritten for the period-one lattice convention fixed above; compare \cite[\S3.1]{Konno2008}.
The compact torus $T^{N}$, and hence the subtorus $K$, preserves the Hermitian metric, the Chern connection, and the holomorphic structure.
Restricting to the hyperk\"ahler level set and reducing by $K$ gives the residual $T^{n}$-equivariant $J_{1}$-holomorphic pre-quantum line orbibundle
\begin{equation}
	\label{eqn:prequantum-orbibundle}
	\mcL_{\nu} := \left. L_{\nu} \right|_{\phi_{\HK}^{-1}(\nu , 0)} / K
	\longrightarrow M_{\nu},
\end{equation}
whose descended Chern connection has curvature $(2\pi / \sqrt{-1}) \w_{\RR}$.
Under the Kempf--Ness identification \eqref{eqn:hypertoric-kempf-ness}, $\mcL_{\nu}$ agrees with the line orbibundle descended from the linearisation \eqref{eqn:positive-character-linearisation} on the semistable locus.
When finite stabilisers are present, we do not assume that $\mcL_{\nu}$ is an ordinary line bundle on the coarse quotient.

\begin{notation}[Weight Conventions]
	\label{notation:weight-conventions}
	We continue to use the $+\chi_{\lambda}$ geometric lift of \eqref{eqn:positive-character-linearisation} and we will always compute characters of section spaces using the pullback action \eqref{eqn:pullback-action-on-sections}, which means that the geometric fibre weights and section weights will have opposite signs.
	At a residual $T^{n}$-fixed point $p$, the Kostant formula \cite[p.~241, (7)]{Meinrenken1998}, with our conventions, identifies the geometric fibre weight with $-\mu_{\RR}(p)$ and the corresponding weight in the representation on sections with $+\mu_{\RR}(p)$.
	This is the convention underlying the numerators and positive-exponent denominator factors in the localisation formulae of Sections~\ref{section:section-5-localisation-formulae-isotropy-data}--\ref{section:section-6-equivariant-quantisation-hypertoric-algebraic-cutting}.
\end{notation}

%% file: sections/3-hyperplanes-subvarieties.tex
\section{Hyperplane Arrangements and Hypertoric Subvarieties}
\label{section:section-3-hyperplanes-subvarieties}

This section recalls how the geometry of $M_{\nu}$ is translated into the combinatorics of its associated hyperplane arrangement.
The regions of the arrangement describe the extended-core components, while its flats determine lower-dimensional hypertoric subvarieties.
These subvarieties will later allow us to answer local questions about the cut spaces by reducing them to ones on one- and two-dimensional arrangement configurations.
Throughout this section, we retain the data, hypotheses, and conventions fixed in Section~\ref{section:section-2-hypertoric-varieties}.


\subsection{Hyperplane Arrangements}
\label{subsection:subsection-3.1-hyperplane-arrangements}

Recall from Subsection~\ref{subsection:subsection-2.1-differential-hypertoric} the primitive lattice vectors $u_{1} , \ldots , u_{N} \in \mft_{\ZZ}^{n}$ and the chosen lift $\lambda \in ( \mft^{N} )^{\ast}$ with $\imath^{\ast}(\lambda) = \nu$.
They define the affine \textbf{hyperplane arrangement} $\mcA = \{ H_{1} , \ldots , H_{N} \}$ in $( \mft^{n} )^{\ast}$ by
\begin{equation}
	\label{eqn:hyperplane}
	H_{i} := \{ x \in ( \mft^{n} )^{\ast} ~|~ \ell_{i}(x) = 0 \}, \qquad
	\ell_{i}(x) := \langle x ,\, u_{i} \rangle + \lambda_{i},
\end{equation}
where $u_{i}$ is the chosen cooriented normal and $\lambda_{i}$ is the affine offset.
An arrangement is \textbf{simple} if, whenever $i_{1} , \ldots , i_{k}$ are distinct and $H_{i_{1}} \cap \cdots \cap H_{i_{k}} \neq \varnothing$, the normals $u_{i_{1}} , \ldots , u_{i_{k}}$ are linearly independent over $\RR$; it is \textbf{smooth} if it is simple and, whenever $n$ distinct hyperplanes have a common point, their corresponding normals form a $\ZZ$-basis of $\mft_{\ZZ}^{n}$.
As mentioned in Subsection~\ref{subsection:subsection-2.1-differential-hypertoric}, the genericity condition \eqref{eqn:genericity-span-condition} on $\nu$ is equivalent to the simplicity of $\mcA$.
The stronger condition that $\mcA$ is smooth corresponds to the $K$-action on $\phi_{\HK}^{-1}(\nu , 0)$ being free.

The \textbf{regions}, or chambers, of $\mcA$ are the connected components of $(\mft^{n})^{\ast} \setminus \bigcup_{i = 1}^{N} H_{i}$.
Each $H_{i}$ determines the two closed half-spaces $H_{i}^{+} := \{ x ~|~ \ell_{i}(x) \geq 0 \}$ and $H_{i}^{-} := \{ x ~|~ \ell_{i}(x) \leq 0 \}$.
Given a subset $A \subseteq \{ 1 , \ldots , N \}$, define the closed polyhedron
\[
	\Delta_{A} := \big( \bigcap\nolimits_{i \not\in A} H_{i}^{+} \big) \cap \big( \bigcap\nolimits_{i \in A} H_{i}^{-} \big).
\]
Thus, $A$ records the hyperplanes for which the negative side has been chosen for each region.
As shown below, simplicity implies that every non-empty $\Delta_{A}$ is full-dimensional; its interior is therefore a region of $\mcA$, and $\Delta_{A}$ is the closure of that region.
We identify regions with their closures when using the notation $\Delta_{A}$, and we let $\mcR(\mcA) := \{ \Delta_{A} ~|~ A \subseteq \{ 1 , \ldots , N \} ,\ \Delta_{A} \neq \varnothing \}$ denote the set of non-empty regions of $\mcA$.

We assume that $\mcA$ has at least one bounded region, and we coorient $\mcA$ so that this region is $\Delta_{\emptyset} = \bigcap_{i = 1}^{N} H_{i}^{+}$; the non-empty bounded region $\Delta_{\emptyset}$ is said to be the \textbf{distinguished base region} of $\mcA$.
Later in Section~\ref{section:section-4-cutting-hypertoric-variety}, the boundedness of $\Delta_{\emptyset}$ will be used to guarantee that the fibre-rotation moment map is proper.
For a point $x \in ( \mft^{n} )^{\ast}$, define
\begin{equation}
	\label{eqn:active-index-set}
	\mcI_{x} := \{ i ~|~ \ell_{i}(x) = 0 \} = \{ i ~|~ x \in H_{i} \}
\end{equation}
to be the index set of the hyperplanes that pass through $x$, and set
\[
	H_{\mcI_{x}} := \bigcap\nolimits_{i \in \mcI_{x}} H_{i}, \qquad
	\rk(\mcI_{x}) := \codim_{(\mft^{n})^{\ast}} H_{\mcI_{x}},
\]
with the convention that $H_{\emptyset} = (\mft^{n})^{\ast}$ and $\rk(\varnothing) = 0$.
A point $v$ is said to be an \textbf{(arrangement) vertex} of $\mcA$ if $\rk(\mcI_{v}) = n$, which is equivalent to $|\mcI_{v}| = n$ since $\mcA$ is simple.
Let $\Verts(\mcA)$ denote the set of arrangement vertices.

\begin{lemma}
	\label{lemma:regions-pointed}
	Every non-empty region $\Delta_{A}$ is a full-dimensional pointed polyhedron.
	Furthermore, every non-empty face of $\Delta_{A}$ contains a vertex and, in particular, every vertex of $\Delta_{A}$ is a vertex of the arrangement $\mcA$.
\end{lemma}

\begin{proof}
	We first show that $\Delta_{A}$ is full-dimensional.
	Fix a point $x \in \Delta_{A}$.
	If $\mcI_{x} = \varnothing$, then every defining $\ell_{i}$-inequality is strict at $x$, and $x$ already lies in the interior of $\Delta_{A}$.
	So suppose instead that $\mcI_{x} \neq \varnothing$.
	Since $\mcA$ is simple, the normals $\{ u_{i} \}_{i \in \mcI_{x}}$ are linearly independent.
	Hence, the linear map $y \mapsto (\langle y ,\, u_{i} \rangle )_{i \in \mcI_{x}}$ from $(\mft^{n})^{\ast}$ to $\RR^{\mcI_{x}}$ is surjective, and we may therefore choose $y$ such that $\langle y ,\, u_{i} \rangle < 0$ for $i \in \mcI_{x} \cap A$ and $\langle y ,\, u_{i} \rangle > 0$ for $i \in \mcI_{x} \setminus A$.

	For each $i \in \mcI_{x}$, the identity $\ell_{i}( x + \epsilon y ) = \epsilon \langle y ,\, u_{i} \rangle$ shows that the respective defining $\ell_{i}$-inequality becomes strict when $\epsilon > 0$.
	The remaining inequalities are already strict at $x$, and, because there are only finitely many of them, they remain strict for all sufficiently small $\epsilon > 0$.
	Whence, $x + \epsilon y$ lies in the interior of $\Delta_{A}$, proving that $\Delta_{A}$ is full-dimensional.

	The lineality space of a polyhedron defined by linear inequalities is the kernel of its coefficient matrix \cite[\S8.2]{Schrijver86}.
	For $\Delta_{A}$, the coefficient vectors are $u_{i}$ or $-u_{i}$, so this kernel consists of vectors $p \in (\mft^{n})^{\ast}$ satisfying $\langle p ,\, u_{i} \rangle = 0$ for each $i$.
	Since $u_{1} , \ldots , u_{N}$ span $\mft^{n}$, this kernel is trivial and $\Delta_{A}$ is pointed, hence every non-empty face of $\Delta_{A}$ therefore contains a vertex \cite[\S\S8.3,~8.5]{Schrijver86}.

	Finally, let $v$ be a vertex of $\Delta_{A}$.
	If the normals indexed by $\mcI_{v}$ did not span $\mft^{n}$, there would be a non-zero point $p \in (\mft^{n})^{\ast}$ satisfying $\langle p ,\, u_{i} \rangle = 0$ for every $i \in \mcI_{v}$.
	The active equalities would remain unchanged at $v \pm \epsilon p$, while the remaining defining inequalities would remain strict for sufficiently small $\epsilon > 0$, so both points would therefore belong to $\Delta_{A}$.
	However, $v$ is their midpoint, contradicting the fact that it is a vertex.
	Hence, $\rk(\mcI_{v}) = n$, and thus $v \in \Verts(\mcA)$.
\end{proof}

Every non-empty region $\Delta_{A}$ is full-dimensional, so its interior lies entirely in one of the two open half-spaces bounded by each $H_{i}$.
The chosen coorientation of $H_{i}$ distinguishes these as the positive and negative sides, and the resulting half-space choices determine $A$ uniquely.
In particular, the partial order in the following remark is well-defined.

\begin{remark}
	\label{remark:poset-regions}
	Following \cite[\S1]{Edelman1984}, the \textbf{poset of regions} is denoted by $(\mcR(\mcA) , \preceq)$, where the partial order is defined as: $\Delta_{A} \preceq \Delta_{B}$ if and only if $A \subseteq B$.
	Relative to the distinguished base region $\Delta_{\emptyset}$, the set $A$ indexes the hyperplanes separating the interiors of $\Delta_{A}$ and $\Delta_{\emptyset}$, and thus $\Delta_{\emptyset}$ is the unique minimal element and $|A|$ is the number of separating hyperplanes.
	The order records the separating sets themselves~---~and not merely their cardinalities~---~since it depends only on the chosen distinguished base region, or equivalently on the coorientation of $\mcA$.

	More generally, for two non-empty regions $\Delta_{A}$ and $\Delta_{B}$, the symmetric difference $A \triangle B := (A \setminus B) \cup (B \setminus A)$ indexes the hyperplanes separating their interiors so that, for each $i \in A \triangle B$, one interior lies in $\{ \ell_{i} > 0 \}$ whereas the other lies in $\{ \ell_{i} < 0 \}$.
	Geometrically, these are the hyperplanes that one must cross to pass from the interior of $\Delta_{A}$ to that of $\Delta_{B}$: for each $i \in A \triangle B$, the strict sign condition $\ell_{i} > 0$ on one interior is replaced by $\ell_{i} < 0$ on the other, or vice versa.
	Since the two regions lie in opposite closed half-spaces bounded by each such $H_{i}$, any point of $\Delta_{A} \cap \Delta_{B}$ must satisfy $\ell_{i} = 0$ for every $i \in A \triangle B$.
	In particular, replacing $A$ by $A \triangle \{ i \}$ switches only the chosen side of $H_{i}$.
\end{remark}


\subsection{The Core and Extended Core}
\label{subsection:subsection-3.2-core-extended-core}

Following \cite[\S2]{HaradaProudfoot2004}, the \textbf{extended core} of $M_{\nu}$ is defined to be the zero set of the residual complex moment map,
\[
	\mcE_{\nu} := \mu_{\CC}^{-1}(0) = \{ [z , w] \in M_{\nu} ~|~ z_{i}w_{i} = 0 \text{ for every } i \} \subset M_{\nu}.
\]
For each $A \subseteq \{1, \ldots, N\}$ with $\Delta_{A} \neq \varnothing$, the preimage $\mcE_{A} := \mu_{\RR}^{-1}(\Delta_{A}) \cap \mcE_{\nu}$ is said to be an \textbf{extended-core component} and thus are the irreducible components of $\mcE_{\nu}$ satisfying $\mcE_{\nu} = \bigcup_{\Delta_{A} \neq \varnothing} \mcE_{A}$; see \cite[\S2]{HaradaProudfoot2004}.
Since $\Delta_{\emptyset}$ is bounded and the regions cover $( \mft^{n} )^{\ast}$, there is at least one additional unbounded region; thus, $|\mcR(\mcA)| \geq 2$ and $\mcE_{\nu}$ has at least two components.

\begin{lemma}
	\label{lemma:half-space-coordinates}
	Let $[z , w] \in \mcE_{\nu}$.
	If $w_{i} = 0$, then $\mu_{\RR}[z , w] \in H_{i}^{+}$; if $z_{i} = 0$, then $\mu_{\RR}[z , w] \in H_{i}^{-}$.
	Hence,
	\[
		\mcE_{A} = \{ [z , w] \in M_{\nu} ~|~ w_{i} = 0 ~\text{ if }~ i \not\in A ,~ z_{i} = 0 ~\text{ if }~ i \in A \}.
	\]
	Moreover, for every $A \subseteq \{1, \ldots, N\}$ with $\Delta_{A} \neq \varnothing$, the restriction $\left. \mu_{\RR} \right|_{\mcE_{A}} : \mcE_{A} \ra \Delta_{A}$ is a proper surjection.
\end{lemma}

\begin{proof}
	By the residual moment-map formula in Subsection~\ref{subsection:subsection-2.1-differential-hypertoric}, $\ell_{i}(\mu_{\RR}[z , w]) = (|z_{i}|^{2} - |w_{i}|^{2}) / 2$.
	Thus, $w_{i} = 0$ implies $\mu_{\RR}[z , w] \in H_{i}^{+}$, whereas $z_{i} = 0$ implies $\mu_{\RR}[z , w] \in H_{i}^{-}$.
	On the extended core, we also have $z_{i}w_{i} = 0$ for every $i$, which gives the claimed coordinate description of $\mcE_{A}$.
	If $[z , w] \in \mcE_{A}$ satisfies $z_{i} = w_{i} = 0$, then $\mu_{\RR}[z, w] \in H_{i}$ and the point also belongs to $\mcE_{A \triangle \{ i \}}$, which describes the intersection of the two components over that hyperplane.

	To prove the surjectivity of $\mu_{\RR}|_{\mcE_{A}}$ onto $\Delta_{A}$, fix $x \in \Delta_{A}$, and choose coordinates $(z , w) \in T^{\ast}\CC^{N}$ satisfying
	\[
		w_{i} = 0, \quad |z_{i}|^{2} = 2\ell_{i}(x)
		\quad \text{ if } i \not\in A,
		\qquad \text{ and } \qquad
		z_{i} = 0, \quad |w_{i}|^{2} = -2\ell_{i}(x)
		\quad \text{ if } i \in A.
	\]
	The sign conditions defining $\Delta_{A}$ ensure that the prescribed squared norms are non-negative, and also that the coordinates satisfy $z_{i}w_{i} = 0$ for each $i$, so $\phi_{\CC}(z , w) = 0$.
	Moreover,
	\[
		\tfrac{1}{2} \sum\nolimits_{i} ( |z_{i}|^{2} - |w_{i}|^{2} ) \epsilon_{i} =
		\sum\nolimits_{i} \ell_{i}(x) \epsilon_{i} =
		\pi^{\ast}(x) + \lambda,
	\]
	and applying $\imath^{\ast}$ to this yields $\phi_{\RR}(z , w) = \nu$, since $\imath^{\ast}(\lambda) = \nu$ and $\imath^{\ast} \circ \pi^{\ast} = 0$.
	Hence, $[z , w] \in \mcE_{A}$, and the residual moment-map formula from Subsection~\ref{subsection:subsection-2.1-differential-hypertoric} gives $\mu_{\RR}[z , w] = x$, proving surjectivity.

	For properness, let $C \subseteq \Delta_{A}$ be a compact subset.
	The set of representatives $(z , w)$ in the hyperk\"ahler level set satisfying $[z , w] \in \mcE_{A}$ and $\mu_{\RR}[z , w] \in C$ is closed in $T^{\ast}\CC^{N}$.
	It is also bounded, since the coordinate description gives $\tfrac{1}{2}( \|z\|^{2} + \|w\|^{2} ) = \sum_{i} |\ell_{i}( \mu_{\RR}[z , w] )|$, and the right-hand side is bounded as the moment-map image ranges over $C$.
	This set is therefore compact, and its image under the quotient map is $(\left. \mu_{\RR} \right|_{\mcE_{A}})^{-1}(C)$, proving properness.
\end{proof}

Each extended-core component $\mcE_{A}$ is $\w_{\CC}$-Lagrangian, has real dimension $2n$, and is a \emph{symplectic toric orbifold} for the residual $T^{n}$-action; see \cite[Theorem~6.5(iii)]{BielawskiDancer2000} and \cite[\S2]{HaradaProudfoot2004}.
By Lemma~\ref{lemma:half-space-coordinates}, its moment-map image is $\Delta_{A}$, and $\mcE_{A}$ is compact if and only if $\Delta_{A}$ is bounded.
Furthermore, when the arrangement is smooth, these are ordinary symplectic toric manifolds.
For the base region, $X_{\nu} := \mcE_{\emptyset}$ is the compact toric orbifold corresponding to $\Delta_{\emptyset}$ and is the K\"ahler quotient of the zero section $\CC^{N} \subset T^{\ast}\CC^{N}$ by $K$; compare \cite{BielawskiDancer2000}.

Set $I_{\bd} := \{ A ~|~ \Delta_{A} \text{ is bounded} \}$.
The union $\mcC := \bigcup_{A \in I_{\bd}} \mcE_{A}$ is said to be the \textbf{core} of $M_{\nu}$ and its irreducible components make up the bounded regions of $\mcA$; in particular, $\mcC$ is irreducible if and only if $\Delta_{\emptyset}$ is the unique bounded region of the hyperplane arrangement.

\begin{remark}
	\label{remark:hyperkahler-analogues}
	Changing the coorientation of the arrangement $\mcA$ results in a change of the distinguished base region, the poset of regions, and the distinguished toric subvariety $X_{\nu} = \mcE_{\emptyset}$.
	The underlying hypertoric variety is unchanged up to $T^{n}$-equivariant diffeomorphism \cite[Lemma~2.2 and the discussion following Example~2.3]{HaradaProudfoot2004}, though this diffeomorphism need not intertwine the fibre-rotation actions and therefore need not be $T^{n} \times U_{1}$-equivariant.
	The hypertoric variety together with its distinguished circle action can therefore depend on the coorientation \cite[Lemma~3.5 and the discussion following Example~3.6]{Proudfoot2004}.
\end{remark}


\subsection{Flats of an Arrangement}
\label{subsection:subsection-3.3-flats-arrangement}

For a subset $\mcF \subseteq \{ 1 , \ldots , N \}$, write $H_{\mcF} := \bigcap_{i \in \mcF} H_{i}$ and $\mcFc := \{ 1 , \ldots , N \} \setminus \mcF$, with the convention that $H_{\emptyset} = ( \mft^{n} )^{\ast}$.
If $H_{\mcF} \neq \varnothing$, then we call it a \textbf{flat} of $\mcA$ and write $L(\mcA)$ for the set of flats, partially ordered by reverse inclusion.
We set $\rk(\mcF) := \codim H_{\mcF}$ and $\crk(\mcF) := \dim H_{\mcF} = n - \rk(\mcF)$ for the rank and corank of $\mcF$, so that $\rk(\mcF) = |\mcF|$ for a simple arrangement.

Observe that the index set $\mcF$ is determined uniquely by the flat $H_{\mcF}$.
Indeed, no additional hyperplane $H_{i}$, with $i \not\in \mcF$, can contain $H_{\mcF}$: otherwise, $H_{\mcF} = H_{\mcF \cup \{ i \}}$ would be a non-empty intersection of $|\mcF| + 1$ hyperplanes with codimension only $|\mcF|$, contradicting simplicity.
The \textbf{restriction arrangement} on the affine space $H_{\mcF}$ is defined by
\[
	\mcA_{\mcF} := \{ H_{i} \cap H_{\mcF} ~|~ i \in \mcFc ,\ H_{i} \cap H_{\mcF} \neq \varnothing \}.
\]
Hyperplanes that are disjoint from the flat are omitted rather than treated as hyperplanes of the restriction.

Introduce the normal space of $H_{\mcF}$ and its associated lattices by
\begin{equation}
	\label{eqn:normal-space-associated-lattices}
	\mft_{\mcF} := \Span_{\RR}\{ u_{i} ~|~ i \in \mcF \}, \quad
	\mft^{\crk \mcF} := \mft^{n} / \mft_{\mcF}, \quad \text{ and } \quad
	U_{\mcF} := \Span_{\ZZ}\{ u_{i} ~|~ i \in \mcF \}, \quad
	V_{\mcF} := \mft_{\mcF} \cap \mft_{\ZZ}^{n}.
\end{equation}
The quotient map $\mft^{n} \ra \mft^{\crk \mcF}$ dualises to give an identification of $( \mft^{\crk \mcF} )^{\ast}$ with the annihilator of $\mft_{\mcF}$, and hence with the space of directions of $H_{\mcF}$.
Both $U_{\mcF}$ and $V_{\mcF}$ have rank $\rk(\mcF)$; the lattice $V_{\mcF}$ is saturated in $\mft_{\ZZ}^{n}$ whereas $U_{\mcF}$ need not be. So the quotient $\Gamma_{\mcF} := V_{\mcF} / U_{\mcF}$ is finite, and it is furthermore trivial if and only if $U_{\mcF}$ is saturated.
We set
\[
	T_{U}^{\rk \mcF} := \mft_{\mcF} / U_{\mcF}, \quad
	T_{V}^{\rk \mcF} := \mft_{\mcF} / V_{\mcF}, \quad
	\text{ and define } \quad
	T_{V}^{\crk \mcF} := T^{n} / T_{V}^{\rk \mcF}.
\]
The integral lattice of $T_{V}^{\crk \mcF}$ is $\mft_{\ZZ}^{\crk \mcF} := \mft_{\ZZ}^{n} / V_{\mcF}$, which is torsion-free because $V_{\mcF}$ is saturated.

\begin{proposition}
	\label{prop:torus-decomposition}
	The inclusion $U_{\mcF} \subseteq V_{\mcF}$ induces a finite covering $T_{U}^{\rk \mcF} \ra T_{V}^{\rk \mcF}$ with kernel $\Gamma_{\mcF}$, yielding an exact sequence $1 \ra \Gamma_{\mcF} \ra T_{U}^{\rk \mcF} \ra T_{V}^{\rk \mcF} \ra 1$.
	Moreover, after choosing a complementary lattice to $V_{\mcF}$ in $\mft_{\ZZ}^{n}$, there exists a non-canonical splitting
	\begin{equation}
		\label{eqn:unsat-decomp}
		T^{n} \cong T_{V}^{\rk \mcF} \times T_{V}^{\crk \mcF} \cong \bigl( T_{U}^{\rk \mcF} / \Gamma_{\mcF} \bigr) \times T_{V}^{\crk \mcF}.
	\end{equation}
\end{proposition}

\begin{proof}
	The first assertion follows immediately from $\ker( \mft_{\mcF} / U_{\mcF} \ra \mft_{\mcF} / V_{\mcF} ) = V_{\mcF} / U_{\mcF} = \Gamma_{\mcF}$.
	Since $V_{\mcF}$ is saturated, $\mft_{\ZZ}^{n} / V_{\mcF}$ is free abelian, so there exists a complementary sublattice $W_{\mcF}$ with the non-canonical splitting $\mft_{\ZZ}^{n} \cong V_{\mcF} \oplus W_{\mcF}$ (compare \cite[Exercise~1.3.5]{CoxLittleSchenck2011}); tensoring with $\RR$ and quotienting by the respective lattices gives \eqref{eqn:unsat-decomp}.
\end{proof}

Note, however, that even though the restriction of a simple arrangement remains simple by Lemma~\ref{lemma:restricted-arrangement-simple}, its natural integral normals need not be primitive.
Let $\widebar{\pi}_{\mcF} : \mft^{n} \ra \mft^{\crk \mcF}$ be the quotient map and set $u_{i}^{\mcFc} := \widebar{\pi}_{\mcF}(u_{i}) \in \mft_{\ZZ}^{\crk \mcF}$ for $i \in \mcFc$.

\begin{notation}[Active and Inactive Indices]
	\label{notation:active-inactive-indices}
	For a flat $H_{\mcF}$, put
	\begin{equation}
		\label{eqn:active-index-set-flat}
		\mcF^{\dagger} := \{ i \in \mcFc ~|~ H_{i} \cap H_{\mcF} \neq \varnothing \} .
	\end{equation}
	We say an index $i \in \mcFc$ is \textbf{active} for $\mcF$ if $i \in \mcF^{\dagger}$, and \textbf{inactive} if $i \in \mcFc \setminus \mcF^{\dagger}$.
	That is to say, an index $i$ is active if and only if $u_{i}^{\mcFc} \neq 0$.
	Indeed, the projected vector $u_{i}^{\mcFc}$ represents the linear part of $\ell_{i}|_{H_{\mcF}}$.
	If it vanishes, the restriction is constant; that constant is non-zero because no additional hyperplane contains $H_{\mcF}$, and thus $H_{i} \cap H_{\mcF} = \varnothing$.

	Conversely, a non-constant affine-linear functional on $H_{\mcF}$ takes the value zero somewhere, so the intersection is non-empty.
	The active indices are the ones that contribute a hyperplane to the restriction arrangement, whereas an inactive index contributes only a non-zero constant of fixed sign.
	We shall also use the terms active and inactive for the corresponding indexed coordinates, normals, and affine functionals.
\end{notation}

\begin{lemma}
	\label{lemma:restricted-arrangement-simple}
	The restriction arrangement $\mcA_{\mcF}$ is simple: if $G \subseteq \mcF^{\dagger}$ and $H_{\mcF} \cap (\bigcap_{i \in G} H_{i}) \neq \varnothing$, then the vectors $\{ u_{i}^{\mcFc} \}_{i \in G}$ are linearly independent in $\mft^{\crk \mcF}$.
\end{lemma}

\begin{proof}
	The intersection in the statement is the non-empty intersection of the hyperplanes indexed by $\mcF \cup G$, so simplicity of $\mcA$ makes $\{ u_{j} \}_{j \in \mcF \cup G}$ linearly independent.
	A non-trivial real linear relation among the images $u_{i}^{\mcFc}$ would therefore lift to a linear combination of the $u_{i}$, with $i \in G$, lying in $\Span_{\RR}\{ u_{j} ~|~ j \in \mcF \}$, contradicting that independence.
\end{proof}

Even though the ambient normals $u_{i}$ are primitive, a non-zero projected vector $u_{i}^{\mcFc}$ need not be primitive in $\mft_{\ZZ}^{\crk \mcF}$: write $u_{i}^{\mcFc} = k_{i}  \widebar{u}_{i}^{\mcFc}$ with $k_{i} \in \ZZ_{> 0}$ and $\widebar{u}_{i}^{\mcFc}$ primitive.
Then $\widebar{u}_{i}^{\mcFc}$ determines the normal direction, and $k_{i}$ is the \emph{label} of that hyperplane, in the sense of labelled toric-orbifold data \cite[\S6]{LermanTolman1997}.
More generally, whenever a codimension-one face is defined by a non-zero integral normal $u = k\widebar{u}$ with $\widebar{u}$ primitive and $k \in \ZZ_{> 0}$, then $u$ is called the \emph{labelled normal} and $k$ is the label of that face.
In particular, the restriction $\mcA_{\mcF}$, with the labels $k_{i}$, is a labelled simple arrangement.
The same convention will later be applied to a cut facet whose normal is a possibly non-primitive vector $u_{A}$.

By the associated \emph{labelled hypertoric orbifold}, we mean the hyperk\"ahler quotient obtained by the construction of Subsection~\ref{subsection:subsection-2.1-differential-hypertoric}, with non-primitive projected normals allowed and their full multiplicities retained.
Since $u_{1} , \ldots , u_{N}$ generate $\mft^{n}_{\ZZ}$, the projected vectors $\{ u_{i}^{\mcFc} ~|~ i \in \mcFc \}$ generate $\mft^{n}_{\ZZ} / V_{\mcF} = \mft^{\crk \mcF}_{\ZZ}$ and, as the vectors indexed by $\mcFc \setminus \mcF^{\dagger}$ vanish, the non-zero projected vectors indexed by $\mcF^{\dagger}$ still generate the quotient lattice; compare \cite[\S1]{ProudfootWebster2007}.
Figure~\ref{fig:restriction-arrangement} presents an example that illustrates the active and inactive indices of a flat and the induced primitive normal directions.

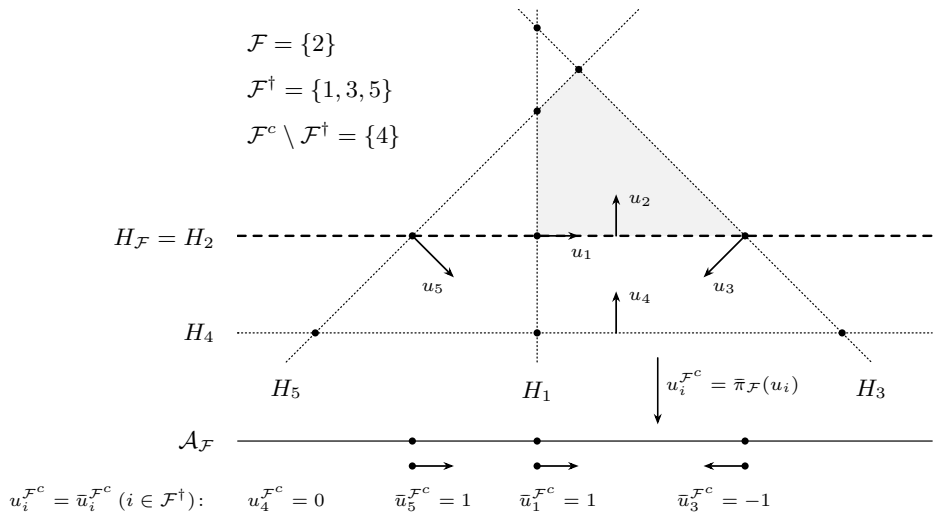
\begin{figure}[!hbp]
	\centering
	\begingroup
	\providecommand{\mcF}{\mathcal{F}}
	\providecommand{\mcFc}{\mcF^{c}}
	\providecommand{\mcA}{\mathcal{A}}
	\def\arrangementunit{0.55cm}
	\def\arrangementtop{5.45}
	\pgfmathsetmacro{\Hfourlevel}{-3.5 * 2 / 3}
	\pgfmathsetmacro{\arrangementbottom}{\Hfourlevel - 0.7}
	\def\arrangementleft{-7.2}
	\def\arrangementright{9.5}
	\def\arrangementcentre{1.15}
	\def\Hfourleft{-7.2}
	\pgfmathsetmacro{\restrictionlevel}{\Hfourlevel - 2.6}
	\def\normalstep{1.0}
	\def\restrictednormallevel{-0.6}

	\scalebox{1.00}{
	\begin{tikzpicture}[
		x = \arrangementunit,
		y = \arrangementunit,
		line cap = round,
		line join = round,
		font = \small,
		draw = black,
		text = black,
		fill = black,
		hyperplane/.style = {line width = 0.45pt, densely dotted},
		chosenflat/.style = {line width = 0.9pt, dashed},
		restrictedaxis/.style = {line width = 0.45pt},
		restrictionarrow/.style = {
			-{Stealth[length = 4pt, width = 3pt]},
			line width = 0.6pt
		},
		normalarrow/.style = {
			-{Stealth[length = 4pt, width = 3pt]},
			line width = 0.65pt
		},
		normallabel/.style = {font = \scriptsize, inner sep = 1.5pt}
		]

		\begin{scope}
			\clip (\arrangementleft, \arrangementbottom)
			rectangle (\arrangementright, \arrangementtop);
			\fill[gray!10] (0 , 0) -- (0 , 3) -- (1 , 4) -- (5 , 0) -- cycle;

			\draw[hyperplane] (0, \arrangementbottom) -- (0, \arrangementtop);
			\draw[hyperplane] ({5 - \arrangementtop}, \arrangementtop)
			-- ({5-\arrangementbottom}, \arrangementbottom);
			\draw[hyperplane] ({\arrangementbottom-3}, \arrangementbottom)
			-- ({\arrangementtop-3}, \arrangementtop);
			\draw[hyperplane] (\arrangementleft, \Hfourlevel) -- (\arrangementright, \Hfourlevel);
			\draw[chosenflat] (\arrangementleft, 0) -- (\arrangementright, 0);

			\foreach \x/\y in {-3/0 , 0/0 , 5/0 , 0/3 , 1/4 , 0/5} {
				\fill (\x, \y) circle (1.35pt);
			}

			\foreach \x in {\Hfourlevel - 3, 0, 5 - \Hfourlevel} {
				\fill ({\x}, \Hfourlevel) circle (1.35pt);
			}
		\end{scope}

		\draw[hyperplane] (\Hfourleft, \Hfourlevel) -- (\arrangementleft, \Hfourlevel);
		\node[below = 4pt] at (0, \arrangementbottom) {$H_{1}$};
		\node[below = 3pt] at ({5-\arrangementbottom}, \arrangementbottom) {$H_{3}$};
		\node[left = 5pt] at (\Hfourleft, \Hfourlevel) {$H_{4}$};
		\node[below = 3pt] at ({\arrangementbottom-3}, \arrangementbottom) {$H_{5}$};

		\node[left = 5pt, yshift = -1pt] at (\arrangementleft, 0)
		{$H_{\mcF} = H_{2}$};

		\node[anchor = north west, align = left] at (\Hfourleft, 5.1)
		{$\mcF = \{ 2 \}$ \\[6pt]
			$\mcF^{\dagger} = \{ 1 , 3 , 5 \}$ \\[6pt]
			$\mcFc \setminus \mcF^{\dagger} = \{ 4 \}$};

		\draw[normalarrow] (0, 0) -- ++(\normalstep, 0)
		node[midway, below right = 3pt, normallabel] {$u_{1}$};
		\draw[normalarrow] (1.9, 0) -- ++(0, \normalstep)
		node[pos = 0.85, right = 3pt, normallabel] {$u_{2}$};
		\draw[normalarrow] (5, 0) -- ++(-\normalstep, -\normalstep)
		node[midway, below = 8pt, normallabel] {$u_{3}$};
		\draw[normalarrow] (1.9, \Hfourlevel) -- ++(0, \normalstep)
		node[pos = 0.85, right = 3pt, normallabel] {$u_{4}$};
		\draw[normalarrow] (-3, 0) -- ++(\normalstep, -\normalstep)
		node[midway, below = 8pt, normallabel] {$u_{5}$};

		\draw[restrictionarrow] (\arrangementcentre + 1.75, {\Hfourlevel - 0.6}) -- (\arrangementcentre + 1.75, {\Hfourlevel - 2.2});
		\node[right = 2pt, font = \scriptsize, fill = none, inner sep = 1pt] at (\arrangementcentre + 1.75, {\Hfourlevel - 1.2})
		{$u_{i}^{\mcFc} = \widebar{\pi}_{\mcF}(u_{i})$};

		\begin{scope}[shift = {(0, \restrictionlevel)}]

			\draw[restrictedaxis] (\arrangementleft, 0) -- (\arrangementright, 0);
			\node[left = 5pt] at (\arrangementleft, 0) {$\mcA_{\mcF}$};
			\foreach \x/\i/\sign in {-3/5/1, 0/1/1, 5/3/-1} {
				\fill (\x, 0) circle (1.35pt);

				\draw[normalarrow] (\x, \restrictednormallevel) -- ++({\sign*\normalstep}, 0)
				node[midway, below = 6pt, normallabel]
				{$\widebar{u}_{\i}^{\mcFc} = \sign$};

				\fill (\x, \restrictednormallevel) circle (1.35pt);
			}

			\node[
			anchor = north east,
			yshift = -6pt,
			normallabel
			] at (-4.10, \restrictednormallevel)
			{
				$u_{i}^{\mcFc} = \widebar{u}_{i}^{\mcFc}
				\ (i \in \mcF^{\dagger})\!: \qquad
				u_{4}^{\mcFc} = 0 \qquad$
				};
		\end{scope}
	\end{tikzpicture}
}
	\caption{
		Restriction of the cooriented arrangement to $H_{\mcF} = H_{2}$, where $\mcF = \{ 2 \}$.
		The indices in $\mcF^{\dagger} = \{ 1 , 3 , 5 \}$ are active, whereas $4$ is inactive because $H_{4} \cap H_{2} = \varnothing$.
	}
	\label{fig:restriction-arrangement}
	\endgroup
\end{figure}


\subsection{Hypertoric Subvarieties}
\label{subsection:subsection-3.4-hypertoric-subvarieties}

We first fix the terminology used for embeddings.
A \textbf{closed K\"ahler embedding} is a closed holomorphic embedding $j : Y \hookrightarrow M$ satisfying $j^{\ast}\w_{M} = \w_{Y}$; for hyperk\"ahler orbifolds, a \textbf{closed hyperk\"ahler embedding} is a closed embedding which is holomorphic for each complex structure and satisfies $j^{\ast}\w_{M}^{j} = \w_{Y}^{j}$ for $j = 1 , 2 , 3$.

\begin{definition}
	\label{def:hypertoric-subvariety}
	Let $( Y , \w_{\RR}^{Y} , \w_{\CC}^{Y} , T^{n'} )$ and $( M , \w_{\RR}^{M} , \w_{\CC}^{M} , T^{n} )$ be hypertoric orbifolds.
	When explicitly stated, we allow ineffective orbifold presentations: the local orbifold groups may contain non-identity elements that act trivially on their entire charts.
	We call $Y$ a \textbf{hypertoric subvariety} of $M$ if there exists a closed hyperk\"ahler embedding $j : Y \hookrightarrow M$ and an injective homomorphism $\rho : T^{n'} \hookrightarrow T^{n}$ such that $j$ is $\rho$-equivariant.
\end{definition}

We now construct the hypertoric subvariety associated with a flat, adapting Konno's construction in \cite[\S7, in particular Claim~7.1]{Konno2000}, in which he studied the behaviour when restricting to a single hyperplane, to arbitrary flats in the orbifold setting.
The coordinate projection, the associated commutative exact diagrams, and the comparison of moment maps follow his construction.

For a flat $H_{\mcF}$, we use a hat to distinguish the coordinate subspace of $\mft^{N}$ from the residual normal space $\mft_{\mcF} \subseteq \mft^{n}$.
Set
\[
	\hat{\mft}_{\mcF} := \Span_{\RR}\{ e_{i} ~|~ i \in \mcF \} \subseteq \mft^{N}, \quad
	\mft^{\mcFc} := \mft^{N} / \hat{\mft}_{\mcF}, \quad
	q_{\mcFc} : \mft^{N} \longrightarrow \mft^{\mcFc}, \quad
	T^{\mcFc} := \mft^{\mcFc} / q_{\mcFc}( \mft_{\ZZ}^{N} ),
\]
where $q_{\mcFc}$ is the coordinate projection.
Since the normals $\{ u_{i} \}_{i \in \mcF}$ are linearly independent, the restriction $q_{\mcFc}|_{\mfk} : \mfk \ra q_{\mcFc}(\mfk)$ is injective, and so it is an isomorphism onto its image; we set $\mfk^{\mcFc} := q_{\mcFc}(\mfk)$ and $K^{\mcFc} := q_{\mcFc}(K) \subseteq T^{\mcFc}$.
The quotient Lie algebra $\mft^{\mcFc} / \mfk^{\mcFc}$ is canonically isomorphic to $\mft^{n} / \mft_{\mcF} = \mft^{\crk \mcF}$, giving the commutative exact diagram
\begin{equation}
	\label{eqn:sub-lie-algebra-seq}
	\begin{tikzcd}[ ampersand replacement = \& , column sep = large ]
		0 \ar[r] \& \mfk \ar[ r , hook , "\imath" ] \ar[ d , "q_{\mcFc}|_{\mfk}" , "\sim"'{sloped} ]
		\& \mft^{N} \ar[ r , "\pi" ] \ar[ d , "q_{\mcFc}" ]
		\& \mft^{n} \ar[r] \ar[ d , "\widebar{\pi}_{\mcF}" ] \& 0 \\
		0 \ar[r] \& \mfk^{\mcFc} \ar[ r , hook , "\imath_{\mcFc}" ]
		\& \mft^{\mcFc} \ar[ r , "\pi_{\mcFc}" ]
		\& \mft^{\crk\mcF} \ar[r] \& 0.
	\end{tikzcd}
\end{equation}
Dualising gives
\begin{equation}
	\label{eqn:sub-dual-lie-algebra-seq}
	\begin{tikzcd}[ ampersand replacement = \& , column sep = large , row sep = large ]
		0 \& \mfk^{\ast} \ar[l]
		\& ( \mft^{N} )^{\ast} \ar[ l , "\imath^{\ast}"' ]
		\& ( \mft^{n} )^{\ast} \ar[ l , "\pi^{\ast}"' ]
		\& 0 \ar[l] \\
		0 \& ( \mfk^{\mcFc} )^{\ast} \ar[l] \ar[ u , "( q_{\mcFc}|_{\mfk} )^{\ast}" , "\sim"'{sloped} ]
		\& ( \mft^{\mcFc} )^{\ast} \ar[ l , "\imath_{\mcFc}^{\ast}"' ] \ar[ u , "q_{\mcFc}^{\ast}" ]
		\& ( \mft^{\crk\mcF} )^{\ast} \ar[ l , "\pi_{\mcFc}^{\ast}"' ] \ar[ u , "\widebar{\pi}_{\mcF}^{\ast}"' ]
		\& 0 \ar[l].
	\end{tikzcd}
\end{equation}

At the group level, $q_{\mcFc}|_{K}$ may have a finite kernel.
Let $\widehat{T}_{\mcF} \subseteq T^{N}$ be the coordinate subtorus supported on $\mcF$, so that $\ker( q_{\mcFc}|_{K} ) = K \cap \widehat{T}_{\mcF} = \ker( \widehat{T}_{\mcF} \ra T^{n})$.
Under $\widehat{T}_{\mcF} \cong \mft_{\mcF} / U_{\mcF} = T_{U}^{\rk \mcF}$, this last map is the finite covering of Proposition~\ref{prop:torus-decomposition}; hence, $\ker( q_{\mcFc}|_{K} ) \cong \Gamma_{\mcF}$ and there is an exact sequence $1 \ra \Gamma_{\mcF} \ra K \ra K^{\mcFc} \ra 1$.
Since the kernel acts trivially on the coordinate subspace $T^{\ast}\CC^{\mcFc}$, the induced $K$-action there factors through the effective torus $K^{\mcFc}$.

Let $\nu^{\mcFc} \in ( \mfk^{\mcFc} )^{\ast}$ be uniquely fixed by $( q_{\mcFc}|_{\mfk} )^{\ast}(\nu^{\mcFc}) = \nu$, and let $\phi_{\HK}^{\mcFc}$ be the hyperk\"ahler moment map for the effective $K^{\mcFc}$-action on $T^{\ast}\CC^{\mcFc}$.
By the naturality of the moment maps with respect to $q_{\mcFc}|_{\mfk}$, together with $( q_{\mcFc}|_{\mfk} )^{\ast}(\nu^{\mcFc}) = \nu$, the extension by zeros of any point of $( \phi_{\HK}^{\mcFc} )^{-1}(\nu^{\mcFc} , 0)$ lies in $\phi_{\HK}^{-1}(\nu , 0)$.
A positive-dimensional stabiliser for the effective $K^{\mcFc}$-action would therefore lift to a positive-dimensional stabiliser for the ambient $K$-action on the corresponding point of $T^{\ast}\CC^{N}$, which is not possible by the regularity of $(\nu , 0)$.
Hence, $( \nu^{\mcFc} , 0 )$ is a regular value, and we define the effective quotient $M_{\mcF}^{\eff} := ( \phi_{\HK}^{\mcFc} )^{-1}(\nu^{\mcFc} , 0) / K^{\mcFc}$.

The removal of factors whose projected normals vanish follows \cite[\S7, p.~1018]{Konno2000}; we also check that it introduces no additional finite isotropy.
If $i \in \mcFc \setminus \mcF^{\dagger}$, then $u_{i}^{\mcFc} = 0$, which means by \eqref{eqn:sub-lie-algebra-seq} that the $i$-th coordinate vector lies in $\mfk^{\mcFc}$.
It follows that $K^{\mcFc}$ contains each coordinate circle indexed by $\mcFc \setminus \mcF^{\dagger}$, and hence their product, which acts independently on the inactive coordinate pairs and leaves the active coordinates unchanged.
Fix, once and for all, a base point $y_{0} \in H_{\mcF}$, and set $c_{i} := \ell_{i}(y_{0}) \neq 0$ for $i \in \mcFc \setminus \mcF^{\dagger}$.
The $i$-th coordinate factor $T^{\ast}\CC$ is then reduced by its coordinate circle at moment-map level $(c_{i} , 0)$, giving the level-set equations
\begin{equation}
	\label{eqn:inactive-point-factor}
	c_{i} = \tfrac{1}{2} (|z_{i}|^{2} - |w_{i}|^{2}), \qquad z_{i}w_{i} = 0,
\end{equation}
and its hyperk\"ahler quotient is a single point: if $c_{i} > 0$, then $w_{i} = 0$ and $|z_{i}|^{2} = 2c_{i}$; if $c_{i} < 0$, then $z_{i} = 0$ and $|w_{i}|^{2} = -2c_{i}$.
In either case, a unique coordinate is non-zero, so the corresponding coordinate circle acts freely on this level set, and deleting the inactive factors introduces no additional finite isotropy.
Reduction in stages therefore identifies $M_{\mcF}^{\eff}$ with the labelled hypertoric quotient defined only by the non-zero projected vectors $u_{i}^{\mcFc}$ with $i \in \mcF^{\dagger}$, and, by Lemma~\ref{lemma:restricted-arrangement-simple}, this active labelled arrangement is simple.

The same coarse space carries two orbifold structures: the effective orbifold $M_{\mcF}^{\eff}$, that we have just discussed, which is obtained by quotienting out the global ineffective isotropy $\Gamma_{\mcF}$, and the orbifold presentation induced from the ambient $K$-quotient, which retains this ineffective isotropy.
We write $M_{\mcF}$ for this coarse space equipped with the induced orbifold presentation whenever it is viewed as a suborbifold of $M_{\nu}$, and we use the effective orbifold $M_{\mcF}^{\eff}$ when discussing its intrinsic hypertoric geometry.
Of course, when $\Gamma_{\mcF}$ is trivial, the two orbifold structures coincide.

\begin{theorem}
	\label{thm:hypertoric-subvariety}
	Let $H_{\mcF}$ be a flat of a simple arrangement $\mcA$.
	Then the coordinate subspace $T^{\ast}\CC^{\mcFc} \subseteq T^{\ast}\CC^{N}$ induces a closed $T^{n}$-invariant subspace
	\begin{equation}
		\label{eqn:subvariety-subset}
		M_{\mcF} \cong \bigl( \phi_{\HK}^{-1}(\nu , 0) \cap T^{\ast}\CC^{\mcFc} \bigr) / K \subseteq M_{\nu},
	\end{equation}
	where $M_{\mcF}$ in \eqref{eqn:subvariety-subset} is understood with its induced, possibly ineffective, $K$-presentation.
	Its effective orbifold presentation is the labelled hypertoric orbifold $M_{\mcF}^{\eff}$ above.

	It has real dimension $4\crk(\mcF)$, and its residual torus is $T_{V}^{\crk \mcF}$.
	The inclusion preserves $\w_{\RR}$ and $\w_{\CC}$, and, after choosing the non-canonical splitting \eqref{eqn:unsat-decomp}, it is $T^{n}$-equivariant with respect to the corresponding injective homomorphism $T_{V}^{\crk \mcF} \hookrightarrow T^{n}$.
	With the induced orbifold presentation, $M_{\mcF}$ is therefore a hypertoric subvariety of $M_{\nu}$ in the sense of Definition~\ref{def:hypertoric-subvariety}.
\end{theorem}

\begin{proof}
	Since the ambient $\phi_{\HK}$-moment-map equations restrict on $T^{\ast}\CC^{\mcFc}$ to those of the $K^{\mcFc}$-action, we obtain the orbit-space identification \eqref{eqn:subvariety-subset}.
	The coordinate inclusion is
	\begin{equation}
		\label{eqn:subvariety-closed-embedding}
		\hat{\jmath} : T^{\ast}\CC^{\mcFc} \hookrightarrow T^{\ast}\CC^{N}, \qquad
		( z_{i} , w_{i} )_{i \in \mcFc} \longmapsto ( \hat{z}_{i} , \hat{w}_{i} )_{i = 1}^{N},
	\end{equation}
	where $( \hat{z}_{i} , \hat{w}_{i} ) = ( 0 , 0 )$ for $i \in \mcF$ and $( \hat{z}_{i} , \hat{w}_{i} ) = ( z_{i} , w_{i} )$ for $i \in \mcFc$.
	The map $\hat{\jmath}$ is closed and preserves the ambient hyperk\"ahler forms, and so it descends to a closed embedding $j : M_{\mcF} \hookrightarrow M_{\nu}$ of the induced quotient presentation; write $\w_{\RR}^{\mcF}$ and $\w_{\CC}^{\mcF}$ for the reduced forms on $M_{\mcF}$.
	The image of $j$ is cut out by the equations $z_{i} = w_{i} = 0$ for $i \in \mcF$, which are invariant under the residual $T^{n}$-action and thus the image is $T^{n}$-invariant.

	Choose a splitting $T^{n} \cong T_{V}^{\rk \mcF} \times T_{V}^{\crk \mcF}$ as in Proposition~\ref{prop:torus-decomposition}.
	The second factor acts effectively on $M_{\mcF}^{\eff}$, and the descended embedding $j$ is $T^{\mcFc}$-equivariant via the action induced by the coordinate inclusion $T^{\mcFc} \hookrightarrow T^{n}$.
	Since \eqref{eqn:subvariety-closed-embedding} preserves the standard forms before reduction, the reduced forms satisfy $j^{\ast}(\w_{\RR}) = \w_{\RR}^{\mcF}$ and $j^{\ast}(\w_{\CC}) = \w_{\CC}^{\mcF}$.
\end{proof}

The construction above is compatible with the restriction to a flat described in \cite[Proposition~2.1]{ProudfootWebster2007}.
Here, we retain the projected-normal labels and the effective/ineffective distinction, and track the effects of inactive factors on the finite isotropy and the inherited fibre-rotation moment-map normalisation, since these details are needed in the cut-space arguments.

For $x \in ( \mft^{\crk \mcF} )^{\ast}$ and each $i \in \mcFc$, define
\begin{equation}
	\label{eqn:subarrangement-hyperplanes}
	\lambda_{i}^{\mcFc} := \ell_{i}(y_{0}), \quad
	\ell_{i}^{\mcFc} \! (x) := \langle x ,\, u_{i}^{\mcFc} \rangle + \lambda_{i}^{\mcFc}, \quad
	G_{i} := \{ \ell_{i}^{\mcFc} = 0 \}, \quad
	\mcA^{\crk \mcF} := \{ G_{i} ~|~ i \in \mcF^{\dagger} \}.
\end{equation}
For an inactive index $i$, the set $G_{i}$ is empty and is not included in $\mcA^{\crk \mcF}$.
We nevertheless retain the non-zero constant $\ell_{i}^{\mcFc} = \lambda_{i}^{\mcFc}$, because these constants contribute to the normalisation that the $\Phi$-moment map inherits below.
The active hyperplanes $G_{i}$, with $i \in \mcF^{\dagger}$, are pairwise distinct by Lemma~\ref{lemma:restricted-arrangement-simple}.
Replacing $y_{0}$ by $y_{0} + \widebar{\pi}_{\mcF}^{\ast}(a)$, where $a \in (\mft^{\crk \mcF})^{\ast}$, adds $\langle a ,\, u_{i}^{\mcFc} \rangle$ to each $\lambda_{i}^{\mcFc}$.
Thus, the new arrangement is the translate of the old one by $-a$, and the following statements are invariant under this translation.

To follow the affine identification used in the proof of \cite[Claim~7.1(3)]{Konno2000}, define the affine map
\[
	\eta_{y_{0}} : ( \mft^{\crk \mcF} )^{\ast} \longrightarrow H_{\mcF}, \qquad
	\eta_{y_{0}}(x) := \widebar{\pi}_{\mcF}^{\ast}(x) + y_{0}.
\]
Since $\widebar{\pi}_{\mcF}^{\ast}$ identifies $( \mft^{\crk \mcF} )^{\ast}$ with the annihilator of $\mft_{\mcF}$ by \eqref{eqn:sub-dual-lie-algebra-seq}, this is an affine isomorphism onto $H_{\mcF}$, and for $x \in ( \mft^{\crk \mcF} )^{\ast}$, we have
\begin{equation}
	\label{eqn:corank-dual-subalgebra-identification-result}
	\ell_{i}(\eta_{y_{0}}(x)) =
	\langle \widebar{\pi}_{\mcF}^{\ast}(x) + y_{0} ,\, u_{i} \rangle + \lambda_{i} =
	\begin{cases}
		0,                                       & i \in \mcF,  \\[2pt]
		\ell_{i}^{\mcFc} \! (x),  & i \in \mcFc.
	\end{cases}
\end{equation}

\begin{proposition}
	\label{prop:restriction-identification}
	The affine map $\eta_{y_{0}}$ identifies the labelled arrangement $\mcA^{\crk \mcF}$ in $(\mft^{\crk \mcF})^{\ast}$ with the restriction arrangement $\mcA_{\mcF}$ on $H_{\mcF}$.
	That is, $\eta_{y_{0}}(G_{i}) = H_{i} \cap H_{\mcF}$ for every active index $i \in \mcF^{\dagger}$, and the projected vector $u_{i}^{\mcFc}$ provides the orbifold label on the corresponding restricted hyperplane.
\end{proposition}

\begin{proof}
	Equation~\eqref{eqn:corank-dual-subalgebra-identification-result} shows that, for each active index $i$, a point $x$ belongs to $G_{i}$ if and only if $\eta_{y_{0}}(x)$ belongs to $H_{i} \cap H_{\mcF}$.
	Since $\eta_{y_{0}}$ is an affine isomorphism, it follows that $\eta_{y_{0}}(G_{i}) = H_{i} \cap H_{\mcF}$.
	Moreover, the labels agree because the quotient construction retains the full projected normals $u_{i}^{\mcFc}$, including their multiplicities, rather than only their primitive parts.
\end{proof}

Recall from \eqref{eqn:normal-space-associated-lattices} that $U_{G} = \Span_{\ZZ}\{ u_{i} ~|~ i \in G \}$ is the lattice generated by the ambient normals indexed by $G$, and that $V_{\mcF} = \mft_{\mcF} \cap \mft_{\ZZ}^{n}$ is the saturation of $U_{\mcF}$ in $\mft_{\ZZ}^{n}$.

\begin{lemma}[Comparison of Local Orbifold Lattices]
	\label{lemma:flat-local-group-comparison}
	Let $W \subseteq \mft_{\ZZ}^{n}$ be a sublattice such that $U_{\mcF} + W$ and $V_{\mcF} + W$ have full rank in $\mft_{\ZZ}^{n}$.
	Then the inclusion $U_{\mcF} + W \subseteq V_{\mcF} + W$ induces an exact sequence of finite abelian groups
	\[
		0 \longrightarrow \frac{V_{\mcF} + W}{U_{\mcF} + W} \longrightarrow \frac{\mft_{\ZZ}^{n}}{U_{\mcF} + W} \longrightarrow \frac{\mft_{\ZZ}^{n}}{V_{\mcF} + W} \longrightarrow 0,
	\]
	whose left-hand group is naturally a quotient of $\Gamma_{\mcF} = V_{\mcF} / U_{\mcF}$.
	Moreover, writing $L^{\vee} := \{ \alpha \in ( \mft^{n} )^{\ast} ~|~ \langle \alpha ,\, L \rangle \subseteq \ZZ \}$ for a full-rank lattice $L \subseteq \mft_{\ZZ}^{n}$, we have $( V_{\mcF} + W )^{\vee} \subseteq ( U_{\mcF} + W )^{\vee}$.
	In particular, if $W = U_{G}$ for $G \subseteq \mcF^{\dagger}$ and $H_{\mcF \cup G}$ is a vertex of the ambient simple arrangement, then $U_{G} \cap V_{\mcF} = 0$ and the left-hand group is $\Gamma_{\mcF}$.
\end{lemma}

\begin{proof}
	The exact sequence is the quotient sequence associated with the inclusions $U_{\mcF} + W \subseteq V_{\mcF} + W \subseteq \mft_{\ZZ}^{n}$.
	The natural map $V_{\mcF} / U_{\mcF} \ra ( V_{\mcF} + W ) / ( U_{\mcF} + W )$ is surjective, which shows that the kernel is a quotient of $\Gamma_{\mcF}$.
	Furthermore, a functional integral on the larger lattice $V_{\mcF} + W$ is also integral on $U_{\mcF} + W$, which gives the stated inclusion of dual lattices.

	For the final claim, take $W = U_{G}$, where $H_{\mcF \cup G}$ is a vertex.
	Simplicity means that the normals indexed by $\mcF \cup G$ are linearly independent.
	A non-zero element of $U_{G} \cap V_{\mcF}$ would give a non-trivial real linear relation between the normals indexed by $G$ and those indexed by $\mcF$, contradicting their linear independence, thus, $U_{G} \cap V_{\mcF} = \{ 0 \}$.
	If $v = u + g \in V_{\mcF}$, with $u \in U_{\mcF}$ and $g \in U_{G}$, then $g = v - u \in V_{\mcF} \cap U_{G} = \{ 0 \}$.
	Hence, $V_{\mcF} \cap (U_{\mcF} + U_{G}) = U_{\mcF}$, and therefore $( V_{\mcF} + U_{G} ) / ( U_{\mcF} + U_{G} ) \cong V_{\mcF} / U_{\mcF} = \Gamma_{\mcF}$.
\end{proof}

Applied to $M_{\mcF}$, the inclusion $( V_{\mcF} + W )^{\vee} \subseteq ( U_{\mcF} + W )^{\vee}$ shows that every local weight for the effective presentation $M_{\mcF}^{\eff}$ is also a local weight for the presentation inherited from $M_{\nu}$.
The corresponding character of the inherited local group is obtained by composing the effective character with the natural surjection $\mft_{\ZZ}^{n} / ( U_{\mcF} + W ) \ra \mft_{\ZZ}^{n} / ( V_{\mcF} + W )$.

Now take an ambient region $\Delta_{A}$ meeting $H_{\mcF}$, and retain only its face $\Delta_{A} \cap H_{\mcF}$ on the flat. Expressing this face in intrinsic coordinates gives the \textbf{augmented sign region} $\Delta_{B}^{\crk \mcF} = \eta_{y_{0}}^{-1}(\Delta_{A} \cap H_{\mcF})$, where $B := A \cap \mcFc$.
On $H_{\mcF}$, the inequalities indexed by $\mcF$ are automatically satisfied, regardless of their chosen sign, since the corresponding affine functionals vanish there.
The intersection therefore depends only on the remaining sign choices which are recorded by $B = A \cap \mcFc$, and hence $\Delta_{A} \cap H_{\mcF} = \Delta_{B} \cap H_{\mcF}$.
In other words, for every $B \subseteq \mcFc$, \eqref{eqn:corank-dual-subalgebra-identification-result} gives
\begin{equation}
	\label{eqn:augmented-sign-region}
	\Delta_{B}^{\crk \mcF} :=
	\{ x \in ( \mft^{\crk \mcF} )^{\ast} ~|~
	\ell_{i}^{\mcFc}\!(x) \geq 0 \text{ for } i \in \mcFc \setminus B,
	\quad
	\ell_{i}^{\mcFc}\!(x) \leq 0 \text{ for } i \in B \}.
\end{equation}
Here, \emph{augmented} means that we also retain the inequalities for the inactive indices. For each such index $i$, the functional $\ell_{i}^{\mcFc} = \lambda_{i}^{\mcFc}$ is a non-zero constant, so a non-empty augmented region requires $i \in B$ if and only if $\lambda_{i}^{\mcFc} < 0$.
These forced inequalities are satisfied everywhere and do not change the region, and therefore each region of $\mcA^{\crk \mcF}$ is described by its active inequalities together with the uniquely determined inequalities for the inactive indices.

As an example, in Figure~\ref{fig:restriction-arrangement}, take the shaded ambient region $\Delta_{\emptyset}$ and retain its lower edge on $H_{2}$, between $H_{1} \cap H_{2}$ and $H_{3} \cap H_{2}$.
This edge, expressed in intrinsic coordinates and redrawn as the corresponding segment in the lower panel, is $\Delta_{\emptyset}^{1}$.
Its sign description retains the positive sign at the inactive index $4$, recording that $H_{2}$ lies entirely on the positive side of $H_{4}$.
For every $B \subseteq \mcFc$ with $\Delta_{B}^{\crk \mcF} \neq \varnothing$, we write $\mcE_{B}^{\mcF}$ for the corresponding extended-core component of $M_{\mcF}^{\eff}$ and, since the two presentations share the same underlying component, we use the same notation for that component with the presentation induced from $M_{\nu}$.

\begin{remark}
	\label{remark:subarrangement-regions}
	In particular, for every $B \subseteq \mcFc$, the previous identification gives $\eta_{y_{0}}(\Delta_{B}^{\crk \mcF}) = \Delta_{B} \cap H_{\mcF} = \Delta_{B \cup \mcF} \cap H_{\mcF}$.
	For later use, if $B \subseteq \mcFc$, set $u_{B}^{\mcFc} := -\sum_{i \in B} u_{i}^{\mcFc}$ and $\lambda_{B}^{\mcFc} := \sum_{i \in B} \lambda_{i}^{\mcFc}$.
	These formulae include inactive indices: their projected vectors vanish, but their constant terms record the normalisation inherited from the fibre-rotation moment map on the ambient hypertoric variety $M_{\nu}$.
\end{remark}

It remains to compare the ambient fibre-rotation action and moment map on $M_{\nu}$ with their intrinsic counterparts on $M_{\mcF}^{\eff}$.
In anticipation of Section~\ref{section:section-4-cutting-hypertoric-variety}, write $\Phi[z , w] := \tfrac{1}{2} \|w\|^{2}$ for the ambient fibre-rotation moment map, and let $\Phi_{\mcF}^{\std}$ denote the standard fibre-rotation moment map on the effective quotient $M_{\mcF}^{\eff}$.

For an inactive index $i$, the constant $\lambda_{i}^{\mcFc} = \ell_{i}(y_{0})$ is the value $c_{i}$ appearing in \eqref{eqn:inactive-point-factor}.
If this constant is positive, then $w_{i} = 0$, so the factor is fixed by the fibre-rotation action and contributes nothing to $\Phi$.
If it is negative, then $z_{i} = 0$ and $\tfrac{1}{2} |w_{i}|^{2} = -\lambda_{i}^{\mcFc}$, so in the latter case, the corresponding circle in $K^{\mcFc}$ cancels the rotation of $w_{i}$ in the quotient.
This circle acts only on the inactive coordinate and therefore leaves the active coordinates unchanged, thus deleting the inactive point factors leaves the standard fibre-rotation action on $M_{\mcF}^{\eff}$, while their total contribution to the ambient moment map is $C_{\mcF} := \sum_{i \in \mcFc \setminus \mcF^{\dagger}} \max\{ 0 ,\, -\lambda_{i}^{\mcFc} \}$, and so $\Phi|_{M_{\mcF}} = \Phi_{\mcF}^{\std} + C_{\mcF}$.
Furthermore, the constant $C_{\mcF}$ is independent of the choice of $y_{0}$, because every inactive restriction $\ell_{i}|_{H_{\mcF}}$ is constant and, for a non-empty augmented region $\Delta_{B}^{\crk \mcF}$, the forced inactive signs give $C_{\mcF} = -\sum_{i \in B \setminus \mcF^{\dagger}} \lambda_{i}^{\mcFc}$.

Now let $\mu_{\RR}^{\mcF}$ denote the intrinsic residual real moment map on $M_{\mcF}^{\eff}$, which takes its values in $(\mft^{\crk \mcF})^{\ast}$ and whose normalisation is determined by the constants $\lambda_{i}^{\mcFc}$ in \eqref{eqn:subarrangement-hyperplanes}.
Since $u_{i}^{\mcFc} = 0$ for each inactive index, the inactive indices contribute nothing to $u_{B}^{\mcFc}$, whereas their contribution to $-\lambda_{B}^{\mcFc}$ is $C_{\mcF}$.
Hence, on the extended-core component $\mcE_{B}^{\mcF}$, the restriction of the ambient fibre-rotation moment map is $\left. \Phi \right|_{\mcE_{B}^{\mcF}} = \langle \mu_{\RR}^{\mcF} ,\, u_{B}^{\mcFc} \rangle - \lambda_{B}^{\mcFc}$.

Finally, let $\mcS$ denote the ambient slack function (introduced in Section~\ref{section:section-4-cutting-hypertoric-variety}) and define the intrinsic slack function by $\mcS_{\mcF}^{\std}(x) := \sum_{i \in \mcF^{\dagger}} \max\{ 0 ,\, -\ell_{i}^{\mcFc} \! (x) \}$.
Equation \eqref{eqn:corank-dual-subalgebra-identification-result} then gives
\begin{equation}
	\label{eqn:slack-restriction}
	\mcS(\eta_{y_{0}}(x)) = \mcS_{\mcF}^{\std}(x) + C_{\mcF},
\end{equation}
which shows that $C_{\mcF}$ is the contribution of the inactive indices to the ambient slack along the flat.
We use \eqref{eqn:slack-restriction} in Section~\ref{section:section-4-cutting-hypertoric-variety} when comparing cuts with restrictions to flats.

%% file: sections/4-cut-spaces.tex
\section{Symplectic Cut of a Hypertoric Variety}
\label{section:section-4-cutting-hypertoric-variety}

Symplectic cutting provides the geometric basis for our approach to equivariant quantisation: it allows us to pass from the non-compact hypertoric variety $M_{\nu}$ to a family of compact K\"ahler orbifolds while retaining the residual torus action.
The fibre-rotation circle supplies the cutting action, producing the spaces $M_{\nu}^{\leq \delta}$ at positive regular levels.
The construction also retains a direct connection with the hyperplane arrangement: cutting each extended-core component truncates its moment region, and the resulting polytopes fit together along common faces to form the moment polyptych.
These cuts and their moment polyptychs supply the geometric setting for the localisation calculations in \S\S5--6; their relation to the section spaces of the original hypertoric variety is established in \S7.

\subsection{Residual Fibre-Rotation \texorpdfstring{$U_{1}$}{U1}-Action}
\label{subsection:subsection-4.1-residual-circle-action}

As mentioned previously, hypertoric varieties carry a natural $U_{1}$-action induced by rotating the cotangent fibres of $T^{\ast}\CC^{N}$ \cite[\S1]{HaradaProudfoot2004}: for $\tau \in U_{1}$, the action is $\tau \cdot ( z , w ) = ( z , \tau w )$.
It scales the holomorphic-symplectic form and preserves the real K\"ahler form: $\tau^{\ast}\w_{\CC} = \tau \w_{\CC}$ and $\tau^{\ast}\w_{\RR} = \w_{\RR}$.
With the moment-map convention fixed in Section~\ref{section:section-2-hypertoric-varieties}, this action is Hamiltonian for $\w_{\RR}$, with a moment map $\Phi : T^{\ast}\CC^{N} \ra \RR_{\geq 0}$ given by $\Phi(z , w) = \tfrac{1}{2} \|w\|^{2}$.

The fibre-rotation action preserves $\phi_{\RR}$ and satisfies $\phi_{\CC}(z, \tau w) = \tau \phi_{\CC}(z , w)$, so it preserves $\phi_{\HK}^{-1}(\nu , 0)$.
Since it also commutes with $K$, it descends to a Hamiltonian $U_{1}$-action on $M_{\nu}$, and we continue to denote the induced moment map by $\Phi$.
By the coordinate description in Lemma~\ref{lemma:half-space-coordinates}, the zero set of the descended moment map is the locus $w = 0$, namely $\Phi^{-1}(0) = \mcE_{\emptyset} = X_{\nu}$.
Harada and Proudfoot \cite[Proposition~1.3]{HaradaProudfoot2004} show that the compactness of $X_{\nu}$ implies the properness of$\Phi$ when the reducing group is abelian (which is the case here since $K$ is a torus).
The boundedness of $\Delta_{\emptyset}$ makes $X_{\nu}$ compact by Lemma~\ref{lemma:half-space-coordinates}, and hence $\Phi$ is proper.

Although the fibre-rotation $U_{1}$-action is not generally induced by a single circle subgroup of the residual torus $T^{n}$, its restriction to each extended-core component $\mcE_{A}$ agrees with the action of a one-parameter subgroup of $T^{n}$ \cite[\S3]{HaradaProudfoot2004}.
To describe it, for $A \subseteq \{ 1 , \ldots , N \}$, define $\iota_{A} : U_{1} \ra T^{N}$ by $\iota_{A}(\tau) = ( \tau_{1} , \ldots , \tau_{N} )$, where $\tau_{i} = \tau^{-1}$ for $i \in A$ and $\tau_{i} = 1$ otherwise.
The induced $U_{1}$-action on $\mcE_{A}$ is then the restriction of the residual $T^{n}$-action along $\pi \circ \iota_{A}$.
To express its generator and moment map explicitly, set
\begin{equation}
	\label{eqn:U1-action-generator}
	e_{A} := -\sum\nolimits_{i \in A} e_{i} \in \mft_{\ZZ}^{N}, \quad
	u_{A} := \pi(e_{A}) = -\sum\nolimits_{i \in A} u_{i} \in \mft_{\ZZ}^{n}, \quad
	\lambda_{A} := \sum\nolimits_{i \in A} \lambda_{i} \in \RR.
\end{equation}
The one-parameter subgroup $\pi \circ \iota_{A}$ has infinitesimal generator $u_{A} = ( \pi \circ \iota_{A} )_{\ast}(1) = \pi(e_{A})$, which we call the \textbf{restricted $U_{1}$-action generator} on $\mcE_{A}$.
By the coordinate description of $\mcE_{A}$ in Subsection~\ref{subsection:subsection-3.2-core-extended-core}, one has $w_{i} = 0$ for $i \notin A$ and $z_{i} = 0$ for $i \in A$, so that the restriction of $\Phi$ is
\begin{equation}
	\label{eqn:circle-moment-map}
	\Phi_{A} = \left. \Phi \right|_{\mcE_{A}} : \mcE_{A} \ra \RR_{\geq 0}, \qquad
	\Phi_{A}[ z , w ] = \tfrac{1}{2} \sum\nolimits_{i \in A} |w_{i}|^{2} = \big\langle \mu_{\RR}[ z , w ] ,\, u_{A} \big\rangle - \lambda_{A}.
\end{equation}
The sign in $u_{A}$ and its absence in $\lambda_{A}$ are both required for \eqref{eqn:circle-moment-map}.
Observe also that when $A = \{ i \}$, we have $u_{\{ i \}} = -u_{i}$, so the restricted generator is the negative of the corresponding arrangement normal.


\subsection{Cut Spaces}
\label{subsection:subsection-4.2-cut-spaces}

Following Lerman \cite{Lerman1995}, we form the symplectic cut of $M_{\nu}$; for the orbifold formulation and the descent of commuting group actions, compare \cite[\S5.8]{GuilleminGinzburgKarshon2002}.
The action that we use here is the one in which the cutting circle rotates the auxiliary coordinate $\zeta$ with the same weight as $w$, so that the cut retains the sub-level region $\{ \Phi \leq \delta \}$ \cite[\S1.1, p.~249]{Lerman1995}.

Equip the auxiliary copy of $\CC$ with the K\"ahler form $\w_{\aux} = ( \sqrt{-1} / 4 \pi ) d\zeta \wedge d\bar{\zeta}$; then the scalar $U_{1}$-rotation action $\tau \cdot \zeta = \tau \zeta$ has the moment map $\tfrac{1}{2} |\zeta|^{2}$ (c.f. Section~\ref{section:section-2-hypertoric-varieties}).
Let $U_{1}$ act on $M_{\nu} \times \CC$ by $\tau \cdot ( [ z , w ] , \zeta ) = ( [ z , \tau w ] , \tau \zeta )$, which has the moment map
\[
	\rho : M_{\nu} \times \CC \ra \RR_{\geq 0}, \qquad
	\rho\big( [ z , w ] , \zeta \big) = \Phi[ z , w ] + \tfrac{1}{2} |\zeta|^{2} = \tfrac{1}{2} \|w\|^{2} + \tfrac{1}{2} |\zeta|^{2}.
\]
A point $( p , \zeta ) \in M_{\nu} \times \CC$ is critical for $\rho$ if and only if $\zeta = 0$ and $p$ is critical for $\Phi$, so it follows that $\CritVal(\rho) = \CritVal(\Phi)$ and a positive level is regular for the cutting moment map if and only if it is regular for $\Phi$.

For any regular value $\delta > 0$ of $\rho$, the $U_{1}$-action is locally free on the level set $\rho^{-1}(\delta)$, and we define the \textbf{cut space} of the hypertoric variety $M_{\nu}$ at level $\delta$ to be the symplectic quotient
\[
	M_{\nu}^{\leq \delta} := \left( M_{\nu} \times \CC \right) \sslash_{\delta} U_{1} = \rho^{-1}(\delta) / U_{1}.
\]
Since the cutting action is holomorphic and Hamiltonian for the product K\"ahler structure and is locally free at a regular level, the cut space $M_{\nu}^{\leq \delta}$ is a K\"ahler orbifold of real dimension $4n$ (or complex dimension $2n$).
We continue to denote its reduced complex structure by $J_{1}$.

The cut contains an open subset, that we call the \textbf{cut interior} $M_{\nu}^{< \delta} := \{ [ [ z , w ] , \zeta ] \in M_{\nu}^{\leq \delta} ~|~ \zeta \neq 0 \}$, which is $T^{n}$-equivariantly symplectomorphic to $\{ p \in M_{\nu} ~|~ \Phi(p) < \delta \}$.
The biholomorphism type of $M_{\nu}^{< \delta}$ is constant as $\delta$ varies within a connected component of the regular values of $\Phi$ \cite[Theorems~3.2--3.4]{Hausel1998}, and if $\delta$ lies above the highest critical value of $\Phi$, then $M_{\nu}^{< \delta}$ is biholomorphic to all of $M_{\nu}$; compare \cite[Theorems~3.2--3.4 and the following remark]{Hausel1998}.
At lower regular levels, this global identification need not hold.

The complement of the cut interior is the \textbf{cut divisor} $Z_{\nu}^{\delta} := \{ [ [z , w ] , 0 ] \in M_{\nu}^{\leq \delta} \} \cong \Phi^{-1}(\delta) / U_{1}$, which has real codimension two and is not a topological boundary: $M_{\nu}^{\leq \delta}$ is an orbifold without boundary and is compact by Lemma~\ref{lemma:proper-moment-map}.
Nevertheless, we shall often use the terms \emph{cut boundary} and \emph{boundary point} for $Z_{\nu}^{\delta}$ and its points when discussing the moment polyptych below.

The residual $T^{n}$-action on $M_{\nu}$ commutes with the cutting circle, hence it descends to $M_{\nu}^{\leq \delta}$; we continue to denote its moment map by $\mu_{\RR}$, and likewise continue using $\Phi$ to denote the function on $M_{\nu}^{\leq \delta}$ induced by $\Phi$ on $M_{\nu}$.

\begin{lemma}
	\label{lemma:proper-moment-map}
	Since $\Delta_{\emptyset}$ is bounded, the cutting moment map $\rho$ is proper.
	In particular, every regular level $\delta > 0$ gives a compact cut space $M_{\nu}^{\leq \delta}$.
\end{lemma}

\begin{proof}
	By Subsection~\ref{subsection:subsection-4.1-residual-circle-action}, $\Phi : M_{\nu} \to \RR_{\geq 0}$ is proper.
	For every $R \geq 0$, continuity of $\rho$ implies that $\rho^{-1}([ 0 , R ])$ is closed in $M_{\nu} \times \CC$ and, moreover, $\rho^{-1}( [ 0 , R ] ) \subset \Phi^{-1}( [ 0 , R ] ) \times \{ \zeta \in \CC \mid \tfrac{1}{2} |\zeta|^{2} \leq R \}$.
	The right-hand side is compact because $\Phi^{-1}( [ 0 , R ] )$ is compact, since $\Phi$ is proper and the $\zeta$-factor is a closed ball in $\CC$; hence, $\rho^{-1}([ 0 , R ])$ is a closed subset of a compact set, and therefore compact.
	It follows that $\rho^{-1}(\delta)$ is compact, as is the quotient $M_{\nu}^{\leq \delta} = \rho^{-1}(\delta) / U_{1}$, being the continuous image of a compact space.
\end{proof}

The same construction applies to each extended-core component $\mcE_{A}$ with $\Delta_{A} \neq \varnothing$.
Using the restricted $U_{1}$-action and its moment map $\Phi_{A}$, we define the \textbf{cut component} to be
\[
	\mcE_{A}^{\leq \delta} := ( \mcE_{A} \times \CC ) \sslash_{\delta} U_{1} = \rho_{A}^{-1}(\delta) / U_{1},
	\quad \text{ where } \quad
	\rho_{A}( [ z , w ] , \zeta ) := \Phi_{A}[ z , w ] + \tfrac{1}{2} |\zeta|^{2}.
\]
Since the subspace $\mcE_{A} \times \CC$ is closed, $U_{1}$-invariant, and $\rho_{A} = \rho|_{\mcE_{A} \times \CC}$, the cutting action is locally free on $\rho_{A}^{-1}(\delta) = \rho^{-1}(\delta) \cap ( \mcE_{A} \times \CC )$ and $\delta$ is also a regular value of $\rho_{A}$.
The induced map $\mcE_{A}^{\leq \delta} \hookrightarrow M_{\nu}^{\leq \delta}$ is a closed embedding into a compact orbifold, and hence $\mcE_{A}^{\leq \delta}$ is compact.
The \textbf{cut-component interior} is $\mcE_{A}^{< \delta} := \{ [ [ z , w ] , \zeta ] \in \mcE_{A}^{\leq \delta} ~|~ \zeta \neq 0 \}$ and its complement is the \textbf{cut-component divisor} $Z_{A}^{\delta} := \{ [ [ z , w ] , 0 ] \in \mcE_{A}^{\leq \delta} \} \cong \Phi_{A}^{-1}(\delta) / U_{1}$.
When the context is clear, we shall refer to these simply as the \emph{component interior} and \emph{component divisor}, respectively.
Note that each component interior $\mcE_{A}^{< \delta}$ is symplectomorphic to $\{ p \in \mcE_{A} ~|~ \Phi_{A}(p) < \delta \}$.


\subsection{Moment Polyptychs}
\label{subsection:subsection-4.3-moment-polyptychs}

In general, the full cut space $M_{\nu}^{\leq \delta}$ is not toric for the residual $T^{n}$-action, but each cut extended-core component $\mcE_{A}^{\leq \delta}$ is toric.
The moment image of each component is a truncation of the corresponding arrangement region, and these truncations agree on common faces.
To record the contributions from the half-spaces where the affine-linear functionals $\ell_{i}$ are negative, we introduce the piecewise-linear \textbf{slack function}:
\begin{equation}
	\label{eqn:slack-function}
	\mcS : ( \mft^{n} )^{\ast} \longrightarrow \RR_{\geq 0}, \qquad
	\mcS(x) := \sum_{i = 1}^{N} \max\{ 0 , -\ell_{i}(x) \}.
\end{equation}

\begin{lemma}
	\label{lemma:slack-on-region}
	If $x \in \Delta_{A}$, then
	\begin{equation}
		\label{eqn:slack-on-region}
		\mcS(x) = \langle x ,\, u_{A} \rangle - \lambda_{A}.
	\end{equation}
\end{lemma}

\begin{proof}
	If $x \in \Delta_{A}$, then we have $\ell_{i}(x) \leq 0$ for $i \in A$ and $\ell_{i}(x) \geq 0$ for $i \notin A$.
	On $\Delta_{A}$, we therefore have $\mcS(x) = -\sum_{i \in A} \ell_{i}(x) = \langle x , -\sum_{i \in A} u_{i} \rangle - \sum_{i \in A} \lambda_{i} = \langle x , u_{A} \rangle - \lambda_{A}$ by \eqref{eqn:U1-action-generator}.
\end{proof}

For the next lemma, recall that $\mcE_{\nu} = \bigcup_{A} \mcE_{A}$ is the extended core.

\begin{lemma}
	\label{lemma:slack-function}
	Let $p = [ z , w ] \in M_{\nu}$ and write $x = \mu_{\RR}(p)$.
	Then $\mcS(x) \leq \Phi(p)$, and equality holds if and only if $p \in \mcE_{\nu}$.
\end{lemma}

\begin{proof}
	On $M_{\nu}$, one has $\ell_{i}(x) = \tfrac{1}{2}( |z_{i}|^{2} - |w_{i}|^{2} )$, so $\max\{ 0 , -\ell_{i}(x) \} = \tfrac{1}{2} \max\{ 0 , |w_{i}|^{2} - |z_{i}|^{2} \} \leq \tfrac{1}{2} |w_{i}|^{2}$.
	For each $i$, equality holds if $w_{i} = 0$ when both sides vanish, or if $z_{i} = 0$ when both sides equal $\tfrac{1}{2} |w_{i}|^{2}$.
	Clearly, if both coordinates are non-zero, then the inequality is strict.
	Summing over $i$ gives $\mcS(x) \leq \tfrac{1}{2} \sum_{i} |w_{i}|^{2} = \Phi(p)$, with equality holding if and only if either $w_{i} = 0$ or $z_{i} = 0$ for each $i$; that is, when $z_{i}w_{i} = 0$ for every $i$, which is the condition for $p$ to lie in the extended core $\mcE_{\nu}$.
\end{proof}

\begin{proposition}
	\label{prop:cut-component-image}
	Let $\Delta_{A} \neq \varnothing$, and let $\delta>0$ be a regular value of $\rho$.
	Then
	\[
		\mu_{\RR}(\mcE_{A}^{\leq \delta}) = \Delta_{A} \cap \{ x \in ( \mft^{n} )^{\ast} \mid \langle x ,\, u_{A} \rangle \leq \delta + \lambda_{A} \}.
	\]
	This set is a bounded intersection of finitely many closed half-spaces, and hence a polytope.
	Since all defining normals are integral, it has rational vertices whenever the constants $\lambda_{i}$ and $\delta$ are rational.
\end{proposition}

\begin{proof}
	Let $[ [ z , w ] , \zeta ] \in \mcE_{A}^{\leq \delta}$ and choose a representative $( [ z , w] , \zeta ) \in \mcE_{A} \times \CC$ such that $\rho( [ z , w ] , \zeta ) = \delta$, so that we have $\Phi[ z , w ] \leq \delta$.
	Since $[ z , w ] \in \mcE_{A}$, one has $\mu_{\RR}[ z , w ] \in \Delta_{A}$, and equation \eqref{eqn:circle-moment-map} gives $\Phi[ z , w ] = \Phi_{A}[ z , w ] = \langle \mu_{\RR}[ z , w ] , u_{A} \rangle - \lambda_{A}$, hence $\langle \mu_{\RR}[ z , w ] , u_{A} \rangle \leq \delta + \lambda_{A}$, showing one direction.

	Conversely, suppose $x \in \Delta_{A}$ satisfies $\langle x , u_{A} \rangle \leq \delta + \lambda_{A}$.
	By the toric moment-image description $\mu_{\RR}(\mcE_{A}) = \Delta_{A}$ from Subsection~\ref{subsection:subsection-3.2-core-extended-core} (see also \cite[\S2]{HaradaProudfoot2004}), choose $[ z , w ] \in \mcE_{A}$ with $\mu_{\RR}[ z , w ] = x$; then $\Phi[ z , w ] = \langle x , u_{A} \rangle - \lambda_{A} \leq \delta$.
	Picking $\zeta \in \CC$ with $\tfrac{1}{2}|\zeta|^{2} = \delta - \Phi[ z , w ]$ gives $( [ z , w] , \zeta ) \in \rho_{A}^{-1}(\delta)$, hence its $U_{1}$-equivalence class $[ [ z , w ] , \zeta ] \in \mcE_{A}^{\leq \delta}$ satisfies $\mu_{\RR}[ [ z , w ] , \zeta ] = \mu_{\RR}[ z , w ] = x$, proving the reverse inclusion.
	Compactness follows from the closed embedding $\mcE_{A}^{\leq \delta} \hookrightarrow M_{\nu}^{\leq \delta}$ and Lemma~\ref{lemma:proper-moment-map}; convexity is immediate from the half-space description in the statement of the proposition.
\end{proof}

On a region $\Delta_{A}$, Lemma~\ref{lemma:slack-on-region} identifies the truncating affine function with the slack: $\langle x , u_{A} \rangle - \lambda_{A} = \mcS(x)$.
The cut retains the part of each region on which $\mcS \leq \delta$.
For $A = \varnothing$, one has $u_{\emptyset} = 0$ and $\lambda_{\emptyset} = 0$, so the truncation is vacuous and the distinguished bounded region $\Delta_{\emptyset}$ is retained in full.

For each subset $A \subseteq \{ 1 , \ldots , N \}$ with $\Delta_{A} \neq \varnothing$, define the \textbf{truncation} $\Delta_{A}^{\leq \delta}$ and its \textbf{cut face} $\Pi_{A}^{\delta}$ by
\[
	\Delta_{A}^{\leq \delta} := \left\{ x \in \Delta_{A} ~|~ \langle x ,\, u_{A} \rangle \leq \delta + \lambda_{A}	\right\},
	\qquad
	\Pi_{A}^{\delta} := \left\{ x \in \Delta_{A} ~|~ \langle x ,\, u_{A} \rangle = \delta + \lambda_{A} \right\}.
\]
By Proposition~\ref{prop:cut-component-image}, every non-empty truncation $\Delta_{A}^{\leq \delta}$ is a compact polytope and, at a regular cut level, it is also $n$-dimensional.
Indeed, suppose that a non-empty truncation $\Delta_{A}^{\leq \delta}$ has an empty interior in the full-dimensional region $\Delta_{A}$.
Since $\Delta_{A}^{\leq \delta} = \Delta_{A} \cap \{ \mcS \leq \delta \}$ and $\mcS|_{\Delta_{A}}$ is affine, this can occur only when $\delta = \min_{\Delta_{A}} \mcS$.
The truncation is then the non-empty minimum face $\Delta_{A} \cap \{ \mcS = \delta \}$, which contains an arrangement vertex $v$ by Lemma~\ref{lemma:regions-pointed}.
Thus, $\delta = \mcS(v)$ is a critical value by Proposition~\ref{prop:critical-values-slack}, which contradicts the regularity of $\delta$.

From Remark~\ref{remark:poset-regions}, recall that, for non-empty regions $\Delta_{A}$ and $\Delta_{B}$, the indices in $A \triangle B$ label the hyperplanes that one must cross to pass from one interior into the other.
The following lemma shows that, whenever $\Delta_{A}$ and $\Delta_{B}$ intersect, their intersection is a common face obtained by imposing the equations $\ell_{i} = 0$ for $i \in A \triangle B$.
The two truncating inequalities agree on this face, so cutting either region at level $\delta$ retains the same portion of their intersection.

\begin{lemma}
	\label{lemma:region-common-face}
	For any $A , B \subseteq \{ 1 , \ldots , N \}$, we have
	\begin{equation}
		\label{eqn:region-common-face}
		\Delta_{A} \cap \Delta_{B} =
		\Delta_{A} \cap \Big( \bigcap\nolimits_{i \in A \triangle B} H_{i} \Big) =
		\Delta_{B} \cap \Big( \bigcap\nolimits_{i \in A \triangle B} H_{i} \Big).
	\end{equation}
	In particular, it follows that $\Delta_{A} \cap \Delta_{B}$ is either empty or a common face of both $\Delta_{A}$ and $\Delta_{B}$ and, on this common face, the two truncating affine-linear functionals agree: $\langle x ,\, u_{A} \rangle - \lambda_{A} = \langle x ,\, u_{B} \rangle - \lambda_{B}$.
	Moreover, the intersection $\Delta_{A}^{\leq \delta} \cap \Delta_{B}^{\leq \delta}$ is either empty or a common face of the two truncated polytopes.
\end{lemma}

\begin{proof}
	If $x \in \Delta_{A} \cap \Delta_{B}$ and $i \in A \triangle B$, then $\Delta_{A}$ and $\Delta_{B}$ lie respectively in the two opposite closed half-spaces bounded by $H_{i}$.
	Hence $x \in H_{i}^{+} \cap H_{i}^{-} = H_{i}$ and so $\ell_{i}(x) = 0$, proving one inclusion in the first equality of \eqref{eqn:region-common-face}.
	Conversely, suppose that $x \in \Delta_{A}$ and $x \in H_{i}$ for every $i \in A \triangle B$.
	Since $x \in \Delta_{A}$, we have $\ell_{i}(x) \leq 0$ for $i \in A \cap B$ and $\ell_{i}(x) \geq 0$ for $i \not\in A \cup B$.
	For $i \in A \triangle B$, the assumption $x \in H_{i}$ gives $\ell_{i}(x) = 0$, so both $\ell_{i}(x) \leq 0$ and $\ell_{i}(x) \geq 0$ hold, which means $x$ satisfies every defining inequality of $\Delta_{B}$ too, whence $x \in \Delta_{B}$ and moreover $x \in \Delta_{A} \cap \Delta_{B}$.

	On the common face, we have $\ell_{i}(x) = 0$ for every $i \in A \triangle B$, and hence
	\[
		\langle x , u_{A} \rangle - \lambda_{A} =
		-\sum_{i \in A} \ell_{i}(x) =
		-\sum_{i \in B} \ell_{i}(x) =
		\langle x , u_{B} \rangle - \lambda_{B}.
	\]
	Writing $C := \Delta_{A} \cap \Delta_{B}$, this equality gives
	\[
	\Delta_{A}^{\leq \delta} \cap \Delta_{B}^{\leq \delta} = C \cap \left\{ x ~|~ \langle x , u_{A} \rangle - \lambda_{A} \leq \delta \right\} = C \cap \left\{ x ~|~ \langle x , u_{B} \rangle - \lambda_{B} \leq \delta \right\}.
	\]
	Truncation preserves faces: if $P$ is a polyhedron and $F$ is a face of $P$, then $F \cap H^{-}$, whenever non-empty, is a face of $P \cap H^{-}$.
	Since $C$ is a face of both $\Delta_{A}$ and $\Delta_{B}$, the intersection $\Delta_{A}^{\leq \delta} \cap \Delta_{B}^{\leq \delta}$ is therefore either empty or a common face of both $\Delta_{A}^{\leq \delta}$ and $\Delta_{B}^{\leq \delta}$.
\end{proof}

The collection of non-empty truncations $\Delta_{A}^{\leq \delta}$ and their faces together form a finite polytopal complex; when the $\lambda_{i}$ and $\delta$ are rational, each of its vertices are rational.
We define the \textbf{moment polyptych} of the cut space $M_{\nu}^{\leq \delta}$ to be the polytopal complex
\[
	\Delta_{\nu}^{\leq \delta} := \bigcup_{\substack{A \subseteq \{1,\ldots , N\}\\ \Delta_{A}^{\leq \delta}\neq\varnothing}} \Faces(\Delta_{A}^{\leq \delta}),
\]
where, $\Faces(P)$ denotes the set of all faces of a polytope $P$, including $P$ itself.
The maximal cells of $\Delta_{\nu}^{\leq \delta}$ are the \textbf{components of the moment polyptych}.
At a regular level, these are the non-empty $n$-dimensional truncations $\Delta_{A}^{\leq \delta}$, so each moment-polyptych component is a truncated arrangement region.
The \textbf{boundary of the moment polyptych} is the subcomplex generated by the non-empty cut faces,
\[
	\Pi_{\nu}^{\delta} := \bigcup_{\substack{A \subseteq \{1,\ldots , N\}\\ \Pi_{A}^{\delta}\neq\varnothing}} 	\Faces(\Pi_{A}^{\delta}).
\]
We also write $\mcE_{\nu}^{\leq \delta}:=\bigcup_{A}\mcE_{A}^{\leq\delta}$ for the cut extended core of the cut space $M_{\nu}^{\leq \delta}$.

A vertex or edge of the moment polyptych means a zero- or one-dimensional cell of this polytopal complex, respectively.
When a vertex belongs to the boundary subcomplex $\Pi_{\nu}^{\delta}$, we call it a \textbf{boundary vertex}; otherwise, we call it an \textbf{interior vertex}.

We call
\[
	|\Delta_{\nu}^{\leq \delta}| = \bigcup_{ \Delta_{A}^{\leq \delta} \neq \varnothing } \Delta_{A}^{\leq \delta} \subseteq ( \mft^{n} )^{\ast}
	\qquad \text{ and} \qquad
	|\Pi_{\nu}^{\delta}| = \bigcup_{ \Pi_{A}^{\delta} \neq \varnothing } \Pi_{A}^{\delta}
\]
the \textbf{moment-polyptych support} and the \textbf{moment-polyptych boundary support}, respectively.
Thus, statements about subsets of $( \mft^{n} )^{\ast}$, moment images, convexity, and lattice points concern these supports, whereas cells, components, faces, edges, and subcomplexes refer to the corresponding polytopal complexes.

\begin{proposition}
	\label{prop:slack-support}
	Let $\delta > 0$ be a regular value of $\rho$.
	Then the moment-polyptych support is
	\begin{equation}
		\label{eqn:slack-support}
		|\Delta_{\nu}^{\leq \delta}| = \mu_{\RR}(M_{\nu}^{\leq \delta}) =
		\bigcup\nolimits_{\Delta_{A}^{\leq \delta} \neq \varnothing} \Delta_{A}^{\leq \delta} =
		\{ x \in ( \mft^{n} )^{\ast} \mid \mcS(x) \leq \delta \},
	\end{equation}
	and the moment-polyptych boundary support is
	\begin{equation}
		\label{eqn:slack-boundary}
		|\Pi_{\nu}^{\delta}| = \mu_{\RR}( Z_{\nu}^{\delta} \cap \mcE_{\nu}^{\leq \delta} ) =
		\bigcup\nolimits_{\Pi_{A}^{\delta} \neq \varnothing} \Pi_{A}^{\delta} = \{ x \in ( \mft^{n} )^{\ast} ~|~ \mcS(x) = \delta \}.
	\end{equation}
	Furthermore, since $|\Delta_{\nu}^{\leq \delta}| = \{ x \in (\mft^{n})^{\ast} ~|~ \mcS \leq \delta \}$ and $\mcS$ is a convex piecewise-linear function, the moment-polyptych support is convex.
\end{proposition}

\begin{proof}
	To show \eqref{eqn:slack-support}, let $q \in M_{\nu}^{\leq \delta}$ be represented by $( p , \zeta ) \in \rho^{-1}(\delta)$, so that $\Phi(p) \leq \delta$. By Lemma~\ref{lemma:slack-function}, we have $\mcS( \mu_{\RR}(q) ) = \mcS( \mu_{\RR}(p) ) \leq \Phi(p) \leq \delta$, whence $\mu_{\RR}( M_{\nu}^{\leq \delta} ) \subseteq \{ \mcS \leq \delta \}$.
	Conversely, suppose $x \in (\mft^{n})^{\ast}$ satisfies $\mcS(x) \leq \delta$ and let $\Delta_{A}$ be the region containing $x$.
	By Lemma~\ref{lemma:slack-on-region}, $\langle x ,\, u_{A} \rangle - \lambda_{A} = \mcS(x) \leq \delta$, hence $x \in \Delta_{A}^{\leq \delta} = \mu_{\RR}( \mcE_{A}^{\leq \delta} )$ by Proposition~\ref{prop:cut-component-image}, proving \eqref{eqn:slack-support}.

	For \eqref{eqn:slack-boundary}, if $x \in \Pi_{A}^{\delta}$, then $x \in \Delta_{A}$ and $\mcS(x) = \langle x , u_{A} \rangle - \lambda_{A} = \delta$, and so $|\Pi_{\nu}^{\delta}| \subseteq \{ \mcS = \delta \}$.
	Conversely, if $\mcS(x) = \delta$ and $x \in \Delta_{A}$, then the same identity implies $x \in \Pi_{A}^{\delta}$, whence $|\Pi_{\nu}^{\delta}| = \{ \mcS = \delta \}$.
	Finally, on each cut component $\mcE_{A}^{\leq \delta}$, the equality $\mcS \circ \mu_{\RR} = \Phi$ holds by Lemma~\ref{lemma:slack-function} and therefore the image of its boundary $Z_{A}^{\delta}$ is $\Pi_{A}^{\delta}$.
	Taking the union over $A$ yields the final equality in \eqref{eqn:slack-boundary}.
\end{proof}

Note that the equality in \eqref{eqn:slack-boundary} cannot be simplified to $|\Pi_{\nu}^{\delta}| = \mu_{\RR}(Z_{\nu}^{\delta})$ in general.
Indeed, if $q \in Z_{\nu}^{\delta}$ is represented by $(p , 0)$ with $p \not\in \mcE_{\nu}$, then $\Phi(p) = \delta$, whereas Lemma~\ref{lemma:slack-function} implies $\mcS(\mu_{\RR}(q)) = \mcS(\mu_{\RR}(p)) < \delta$.
Such points on the cut divisor therefore map to $\{ \mcS < \delta \}$, rather than to the support of the polyptych boundary $|\Pi_{\nu}^{\delta}| = \{ \mcS = \delta \}$.

\begin{proposition}
	\label{prop:critical-values-slack}
	The critical values of the fibre-rotation moment map are
	\begin{equation}
		\label{eqn:critical-values-slack}
		\CritVal(\Phi) = \{ \mcS(v) \mid v \in \Verts( \mcA ) \}.
	\end{equation}
	Furthermore, for every $\delta>0$, this implies
	\begin{equation}
		\label{eqn:regular-value-polyptych-boundary}
		\delta \text{ is a regular value of } \rho
		\qquad \text{ if and only if } \qquad
		|\Pi_{\nu}^{\delta}| \text{ contains no vertex of } \mcA.
	\end{equation}
\end{proposition}

\begin{proof}
	The critical set of $\Phi$ is $M_{\nu}^{U_{1}}$, so let $F$ be a connected component of this fixed-point set, and write $\Phi|_{F} \equiv c$.
	By \cite[Proposition~3.2]{HaradaProudfoot2004}, $F$ lies in an extended-core component $\mcE_{A}$.
	Its image $\mu_{\RR}(F)$ is a non-empty face of $\Delta_{A}$ and therefore contains an arrangement vertex $v$ by Lemma~\ref{lemma:regions-pointed}.
	Choose $q \in F$ with $\mu_{\RR}(q) = v$.
	Since $F \subseteq \mcE_{\nu}$, Lemma~\ref{lemma:slack-function} gives $c = \Phi(q) = \mcS(\mu_{\RR}(q)) = \mcS(v)$, hence every critical value is therefore of the form $\mcS(v)$	for some arrangement vertex $v$.

	Conversely, \cite[Corollary~3.3]{HaradaProudfoot2004} states that every arrangement vertex is the image of a point fixed by the fibre-rotation action, so every value $\mcS(v)$ is critical, proving \eqref{eqn:critical-values-slack}.
	Finally, by the discussion in Subsection~\ref{subsection:subsection-4.2-cut-spaces}, a positive level is regular for $\rho$ if and only if it is regular for $\Phi$.
	Combining \eqref{eqn:critical-values-slack} with $|\Pi_{\nu}^{\delta}| = \{ \mcS = \delta \}$ from Proposition~\ref{prop:slack-support} then gives \eqref{eqn:regular-value-polyptych-boundary}.
\end{proof}

Let us call the connected components of $(0,\infty)\setminus\CritVal(\Phi)$ the \textbf{cut chambers of $\Phi$} and the positive critical values of $\Phi$ the \textbf{cut walls}.
We will mainly use the unbounded \textbf{outer cut chamber}
\begin{equation}
	\label{eqn:outer-cut-chamber}
	\Cout:=(c_{\max},\infty),
	\qquad \text{ where } \qquad
	c_{\max}:=\max\CritVal(\Phi).
\end{equation}
By Proposition~\ref{prop:critical-values-slack}, the set $\CritVal(\Phi)$ is finite, so $c_{\max}$ is well-defined and every $\delta \in \Cout$ is a regular cut level.
On each core component $\mcE_{A} \subseteq \mcC$, we have $\Phi = \mcS \circ \mu_{\RR}$, and the slack function $\mcS|_{\Delta_{A}}$ attains its maximum at an arrangement vertex.
Whence, $\Phi \leq c_{\max}$ on the compact core, and $\delta > c_{\max}$ implies that $\Phi^{-1}(\delta) \cap \mcC = \varnothing$.

\begin{remark}
	Our polyptych construction is closely related to the bounded polytopal complex associated by Bielawski and Dancer to a simple hypertoric arrangement \cite[\S6, Theorems~6.5, 6.7, Proposition~6.8]{BielawskiDancer2000}.
	Their complex records the incidence geometry of the \emph{bounded regions} of the arrangement, and hence of the core.
	The moment polyptych instead retains every non-empty truncated arrangement region, together with the boundary subcomplex $\Pi_{\nu}^{\delta}$ and the possibly non-primitive cutting generators $u_{A}$; this extra cut and lattice data will determine the local orbifold groups and isotropy weights later, which are required for equivariant localisation.
	The term \emph{moment polyptych} emphasises this presentation by alluding to several polytopal panels glued along common faces, rather than just the single convex support $|\Delta_{\nu}^{\leq\delta}|$.
\end{remark}

\begin{remark}
	\label{remark:polyptych-coorientation}
	The construction also depends on the chosen coorientation (and therefore on the distinguished base region $\Delta_{\emptyset}$), as reversing the coorientation of $H_{i}$ replaces its affine functional $\ell_{i}$ by $-\ell_{i}$.
	Since $\ell_{i} = \langle x ,\, u_{i} \rangle + \lambda_{i}$, this amounts to replacing $(u_{i} , \lambda_{i})$ by $(-u_{i} , -\lambda_{i})$; the underlying hyperplane $H_{i} = \{ \ell_{i} = 0 \}$ is unchanged, whilst its two cooriented half-spaces are interchanged.
	In the hypertoric presentation, this interchanges the roles of $z_{i}$ and $w_{i}$ via $( z_{i} , w_{i} ) \mapsto (w_{i} , -z_{i} )$, which preserves the holomorphic-symplectic form.
	The labels of the arrangement are changed of course, in addition to the fibre-rotation action.
	Combinatorially, this is reflected in the change of the restricted generators $u_{A}$ on each extended-core component $\mcE_{A}$, and hence of the truncating half-spaces.
	Compare Remark~\ref{remark:hyperkahler-analogues} and the references therein.
\end{remark}


\subsection{Worked Cut Space and Moment Polyptych Examples}
\label{subsection:subsection-4.4-worked-cut-spaces-moment-polyptychs}

The following four examples make the cut spaces and their moment polyptychs explicit.
In each case, we describe the truncated arrangement regions and distinguish the polyptych from its convex support.
The $A_{r}$ family also illustrates how a non-primitive cutting generator gives rise to finite orbifold isotropy at a boundary point.
We will return to these examples in Section~\ref{section:section-6-equivariant-quantisation-hypertoric-algebraic-cutting} to compute their equivariant quantisations, using the local tangent and fibre representations determined in Section~\ref{section:section-5-localisation-formulae-isotropy-data}.
For brevity, in this subsection, we shall write $\Delta_{ij} := \Delta_{\{i, j\}}$ and $\Pi_{ij} := \Pi_{\{i, j\}}$, and similarly for larger subsets.

\begin{example}[$T^{\ast}\CC\PP^{1}$]
	\label{example:TCP1-example-1}
	We consider the one-dimensional hypertoric data $u_{1} = 1$, $u_{2} = -1$, with parameter $\nu = m \in \RR_{>0}$, and choose the lift $\lambda = ( 0 , m ) \in ( \mft^{2} )^{\ast}$.
	The map $\pi : \mft^{2} \ra \mft^{1}$ has kernel generated by $( 1 , 1 )$, so the corresponding kernel torus is $K = \ker( \pi ) \cong U_{1}$.
	Its action on $T^{\ast}\CC^{2}$ has the corresponding real and complex $K$-moment maps:
	\[
		\phi_{\RR}(z , w) = \tfrac{1}{2} ( |z_{1}|^{2} + |z_{2}|^{2} - |w_{1}|^{2} - |w_{2}|^{2} ), \qquad \phi_{\CC}(z , w) = \langle z ,\, w \rangle = z_{1}w_{1} + z_{2}w_{2}.
	\]
	The hyperk\"ahler quotient is the cotangent bundle of the complex projective line, $M_{\nu} = \phi_{\HK}^{-1}(m , 0) / K \cong T^{\ast}\CC\PP^{1}$.

	Since $\W_{\CC\PP^{1}}^{1} \cong \mcO_{\CC\PP^{1}}(-2)$, we may write $M_{\nu} \cong \Tot(\mcB)$, where $\mcB := \mcO_{\CC\PP^{1}}(-2)$.
	The fibre-rotation circle acts by scalar multiplication on the fibres of $\mcB$.
	For every $\delta > 0$, the cut compactifies each affine fibre $\mcB_{p} \cong \CC$, with $p \in \CC\PP^{1}$, to $\PP(\mcB_{p} \oplus \CC) \cong \CC\PP^{1}$, and these fibrewise compactifications assemble to form the Hirzebruch surface:
	\[
		M_{\nu}^{\leq \delta} \cong \PP(\mcB\oplus\mcO_{\CC\PP^{1}}) =
		\PP( \mcO_{\CC\PP^{1} }(-2) \oplus \mcO_{ \CC\PP^{1} } ) \cong \FF_{2}.
	\]
	The chart where the $\mcO_{\CC\PP^{1}}$-coordinate is non-zero is biholomorphic to $\Tot(\mcB)$, while the locus where that coordinate vanishes is the cut divisor at infinity.
	Thus, for every $\delta > 0$, the cut is biholomorphic to the same Hirzebruch surface $\FF_{2}$; varying $\delta$ changes the K\"ahler class induced by symplectic reduction, rather than the underlying complex projective-bundle structure.

	Now let us describe the moment polyptych of the cut space.
	Identifying $( \mft^{1} )^{\ast} \cong \RR$, write $x$ for the residual moment coordinate.
	The affine-linear functionals defining the arrangement are
	\[
		\ell_{1}(x) = \langle x ,\, u_{1} \rangle + \lambda_{1} = x, \qquad
		\ell_{2}(x) = \langle x ,\, u_{2} \rangle + \lambda_{2} = m - x,
	\]
	which determine the two hyperplanes $H_{1} = \{ 0 \}$, $H_{2} = \{ m \}$, as well as the base region $\Delta_{\emptyset} = H_{1}^{+} \cap H_{2}^{+} = [ 0 , m ]$.
	Under the residual real moment map, this is the image of the zero section $X_{\nu} = \CC\PP^{1} \subset T^{\ast}\CC\PP^{1} = M_{\nu}$.
	The two unbounded cotangent directions extend this interval to the left and right, giving the regions $\Delta_{1} = (-\infty , 0]$ and $\Delta_{2} = [m , +\infty)$, and cutting at level $\delta$ truncates these two unbounded regions.
	Since $u_{1} = +1$ and $u_{2} = -1$, the corresponding generators are $u_{ \{ 1 \} } = -u_{1} = -1$ and $u_{ \{ 2 \} } = -u_{2} = +1$, hence the moment polyptych $\Delta_{\nu}^{\leq \delta}$ has three components $\Delta_{1}^{\leq \delta} = [ -\delta , 0 ]$, $\Delta_{\emptyset}^{\leq \delta} = [ 0 , m ]$, and $\Delta_{2}^{\leq \delta} = [ m , m + \delta ]$, with support $|\Delta_{\nu}^{\leq \delta}| = \mu_{\RR}(M_{\nu}^{\leq \delta}) = [ -\delta , m + \delta ]$.
	The resulting moment polyptych is shown in
	Figure~\ref{fig:polyptych-TCP1}.

	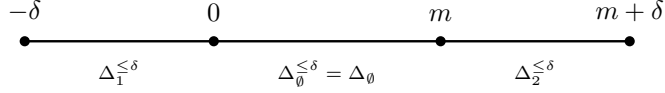
\begin{figure}[ht]
		\centering
		\begin{tikzpicture}[x=1cm,y=1cm]

			\coordinate (bL) at (-4,0);
			\coordinate (v1) at (-1.5,0);
			\coordinate (v2) at (1.5,0);
			\coordinate (bR) at (4,0);

			\draw[polyboundary] (bL) -- (bR);

			\foreach \p in {bL,v1, v2,bR}
			\node[vertex] at (\p) {};

			\node[above=4pt] at (bL) {$-\delta$};
			\node[above=4pt] at (v1) {$0$};
			\node[above=4pt] at (v2) {$m$};
			\node[above=4pt] at (bR) {$m+\delta$};

			\node[region , below=6pt] at (-2.75 , 0) {$\Delta_{1}^{\leq \delta}$};
			\node[region , below=6pt] at (0 , 0) {$\Delta_{\emptyset}^{\leq \delta} = \Delta_{\emptyset}$};
			\node[region , below=6pt] at (2.75 , 0) {$\Delta_{2}^{\leq \delta}$};
		\end{tikzpicture}

		\caption{
			The moment polyptych $\Delta_{\nu}^{\leq \delta}$ for $M_{\nu} \cong T^{\ast}\CC\PP^{1}$.
			}
		\label{fig:polyptych-TCP1}
	\end{figure}
\end{example}


\begin{example}[$A_{r}$-Resolution Family]
	\label{example:A2-resolution-example-1}
	Fix $r \in \ZZ_{> 0}$, and consider the one-dimensional hypertoric data $u_{1} = 1$, $u_{2} = \cdots = u_{r + 1} = -1$, with parameter $\nu = ( m_{1} , \ldots , m_{r} ) \in \RR_{> 0}^{r}$, where $0 < m_{1} < \cdots < m_{r}$, and choose the lift $\lambda = ( 0 , m_{1} , \ldots , m_{r} ) \in ( \mft^{r + 1} )^{\ast}$.
	The map $\pi : \mft^{r + 1} \ra \mft^{1}$ has kernel generated by $e_{1} + e_{j + 1}$ for $1 \leq j \leq r$, so the kernel torus is $K = \ker( \pi ) \cong T^{r}$.
	If an element of $K$ is written as $( k_{1} , \ldots , k_{r} )$ in the corresponding coordinates, then its action on $T^{\ast}\CC^{r + 1}$ is
	\[
		\begin{split}
			( k_{1} , \ldots , k_{r} ) \cdot & ( z_{1} , \ldots , z_{r + 1} , w_{1} , \ldots , w_{r + 1} ) \\
			                                 & =
			(
			k_{1} \cdots k_{r} z_{1} , k_{1} z_{2} , \ldots , k_{r} z_{r + 1} ,
			k_{1}^{-1} \cdots k_{r}^{-1} w_{1} , k_{1}^{-1}w_{2} , \ldots , k_{r}^{-1} w_{r + 1}
			).
		\end{split}
	\]
	The corresponding real and complex $K$-moment maps are
	\[
		\phi_{\RR}(z , w) = \tfrac{1}{2} ( | z_{1} |^{2} + | z_{j + 1} |^{2} - | w_{1} |^{2} - | w_{j + 1} |^{2} )_{ 1 \leq j \leq r },
		\quad
		\phi_{\CC}(z , w) = ( z_{1}w_{1} + z_{j + 1}w_{j + 1} )_{ 1 \leq j \leq r }.
	\]
	The hyperk\"ahler quotient $M_{\nu} = \phi_{\HK}^{-1}( \nu , 0 ) / K$ is the smooth hypertoric surface that is the minimal resolution of the $A_{r}$ Kleinian singularity $\CC^{2} / \ZZ_{r + 1}$.
	The fibre-rotation circle acts on $M_{\nu}$ by $\tau \cdot [ z , w ] = [ z , \tau w ]$, with the moment map $\Phi[ z , w ] = \tfrac{1}{2}\|w\|^{2}$.

	To describe the moment polyptych of the cut space, identify $( \mft^{1} )^{\ast} \cong \RR$ and write $x$ for the residual moment coordinate.
	The affine-linear functionals defining the arrangement are
	\[
		\ell_{1}(x) = \langle x ,\, u_{1} \rangle + \lambda_{1} = x, \qquad
		\ell_{i + 1}(x) = \langle x ,\, u_{i + 1} \rangle + \lambda_{i + 1} = m_{i} - x, \qquad 1 \leq i \leq r,
	\]
	so the arrangement consists of the $r + 1$ points $0 < m_{1} < \cdots < m_{r}$.
	The case $r = 1$ is the $T^{\ast}\CC\PP^{1}$ example above, and $r = 2$ is the $A_{2}$-resolution.

	The slack function is $\mcS(x) = \max\{ 0 ,\, -x \} + \sum_{i = 1}^{r} \max\{ 0 ,\, x - m_{i} \}$, so $\mcS(0) = \mcS(m_{1}) = 0$ and, for each $2 \leq j \leq r$, we have $\mcS(m_{j}) = ( j - 1 )m_{j} - \sum_{i < j}m_{i}$.
	These values increase with $j$, since $\mcS(m_{j + 1}) - \mcS(m_{j}) = j( m_{j + 1} - m_{j} ) > 0$, and hence $c_{\max} = \mcS(m_{r}) = ( r - 1 )m_{r} - \sum_{i < r}m_{i}$.

	For the remainder of this example, assume that $\delta > c_{\max}$, so that $\delta$ is a regular level in the outer cut chamber, and the resulting cut $M_{\nu}^{\leq \delta}$ is a compact K\"ahler orbifold surface.
	For $r \geq 2$, the compact core has $r$ components, corresponding to $[0 , m_{1}]$ and the intervals $[m_{i - 1} , m_{i}]$ for $2 \leq i \leq r$.
	Unlike in the first example, the cut therefore cannot be described by the same fibrewise compactification over a single compact base component.
	At this outer-chamber level, the bounded regions remain unchanged, and the only two unbounded regions are truncated.

	On the left-hand region $\Delta_{1} = (-\infty , 0]$, the generator and constant are $u_{\{1\}} = -1$ and $\lambda_{\{1\}} = 0$, so $\Delta_{1}^{\leq \delta} = [-\delta , 0]$.
	On the rightmost region, we have $A = \{ 2 , \ldots, r + 1 \}$, so $u_{A} = r$ and $\lambda_{A} = \sum_{i = 1}^{r} m_{i}$, which gives the cutting inequality $rx \leq \lambda_{A} + \delta$.
	The moment polyptych therefore has the $r + 2$ components
	\[
		\Delta_{1}^{\leq \delta} = [ -\delta , 0 ], \quad
		\Delta_{\emptyset}^{\leq \delta} = [ 0 , m_{1} ], \quad
		\Delta_{2 \ldots i}^{\leq \delta} = [ m_{i-1} , m_{i} ] \quad
		(2 \leq i \leq r),
		\quad
		\Delta_{A}^{\leq \delta} = [ m_{r} , (\lambda_{A} + \delta) / r ],
	\]
	with support $|\Delta_{\nu}^{\leq \delta}| = [ -\delta , ( \lambda_{A} + \delta ) / r ]$.

	For $r \geq 2$, the right endpoint is located at $(\lambda_{A} + \delta) / r$, as opposed to $m_{r} + \delta$.
	This reflects the fact that the restricted fibre-rotation generator $u_{A} = r$ is non-primitive on $\mcE_{A}$, which produces the local isotropy $\ZZ / r\ZZ$ at the rightmost boundary vertex.
	On the other hand, the left-hand cut vertex has the primitive cutting generator $u_{\{1\}} = -1$, while each arrangement vertex has a single primitive active normal, and hence they only have trivial local isotropy.

	The moment polyptych for $r = 2$ is shown in
	Figure~\ref{fig:polyptych-A2}; for $r \geq 3$, the polyptych differs only in having $r - 2$ further interior vertices at $m_{3} , \ldots , m_{r}$ and the right-hand endpoint $( \lambda_{A} + \delta ) / r$ in place of $( m_{1} + m_{2} + \delta ) / 2$.

	\begin{figure}[ht]
		\centering
		\begin{tikzpicture}[x=1cm , y=1cm]

			\coordinate (bL) at (-4.4 , 0);
			\coordinate (v1) at (-1.8 , 0);
			\coordinate (v2) at ( 0.4 , 0);
			\coordinate (v3) at ( 2.2 , 0);
			\coordinate (bR) at ( 4.6 , 0);

			\draw[polyboundary] (bL) -- (bR);

			\foreach \p in {bL,v1,v2,v3,bR}
			\node[vertex] at (\p) {};

			\node[above=4pt] at (bL) {$-\delta$};
			\node[above=4pt] at (v1) {$0$};
			\node[above=4pt] at (v2) {$m_{1}$};
			\node[above=4pt] at (v3) {$m_{2}$};
			\node[above=4pt] at (bR)
			{$( m_{1} + m_{2} + \delta ) / 2$};

			\node[region, below=8pt] at (-3.1 , 0) {$\Delta_{1}^{\leq \delta}$};
			\node[region, below=8pt] at (-0.7 , 0) {$\Delta_{\emptyset}^{\leq \delta} = \Delta_{\emptyset}$};
			\node[region, below=8pt] at ( 1.3 , 0) {$\Delta_{2}^{\leq \delta}$};
			\node[region, below=8pt] at ( 3.4 , 0) {$\Delta_{23}^{\leq \delta}$};

		\end{tikzpicture}
		\caption{
				The moment polyptych $\Delta_{\nu}^{\leq \delta}$ for the cut of the $A_{2}$-resolution, with $\delta>m_{2} - m_{1}$.
				The rightmost endpoint reflects the non-primitive restricted fibre-rotation generator $u_{\{ 2 , 3 \}} = 2$.
				The filled dots mark the vertices.
			}
		\label{fig:polyptych-A2}
	\end{figure}
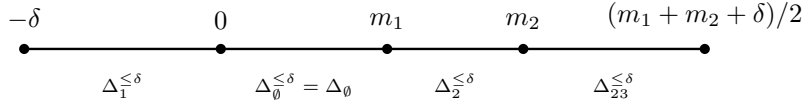
\end{example}


\begin{example}[$T^{\ast}\CC\PP^{2}$]
	\label{example:TCP2-example-1}
	Next, we consider the two-dimensional hypertoric data $u_{1} = ( 1 , 0 )$, $u_{2} = ( 0 , 1 )$, and $u_{3} = ( -1 , -1 )$, with parameter $\nu = m \in \RR_{>0}$, and choose the lift $\lambda = ( 0 , 0 , m ) \in ( \mft^{3} )^{\ast}$.
	The map $\pi : \mft^{3} \ra \mft^{2}$ has kernel generated by $( 1 , 1 , 1 )$, so the kernel torus is $K = \ker( \pi ) \cong U_{1}$.
	It acts on $T^{\ast}\CC^{3}$ with corresponding real and complex moment maps
	\[
		\phi_{\RR}(z , w) = \tfrac{1}{2} ( \|z\|^{2} - \|w\|^{2} ),
		\quad
		\phi_{\CC}(z , w) = \langle z ,\, w \rangle.
	\]
	The hyperk\"ahler quotient is $M_{\nu} = \phi_{\HK}^{-1}(m , 0) / K$, which is the cotangent bundle of the complex projective plane, $M_{\nu} \cong T^{\ast}\CC\PP^{2} \cong \Tot(\Omega_{\CC\PP^{2}}^{1})$.

	Let $\mcB := \Omega_{\CC\PP^{2}}^{1}$, so that $M_{\nu} \cong \Tot(\mcB)$.
	As in Example~\ref{example:TCP1-example-1}, fibre rotation acts by scalar multiplication on the cotangent fibres.
	For every $\delta > 0$, the cut compactifies each fibre $\mcB_{p} \cong \CC^{2}$, with $p \in \CC\PP^{2}$, to $\PP( \mcB_{p} \oplus \CC) \cong \CC\PP^{2}$, hence, as a complex manifold,
	\[
		M_{\nu}^{\leq\delta}
		\cong
		\PP(\mcB\oplus\mcO_{\CC\PP^{2}})
		\cong
		\PP(\Omega_{\CC\PP^{2}}^{1}\oplus\mcO_{\CC\PP^{2}}).
	\]
	Again, varying $\delta$ changes the reduced K\"ahler class on this fixed complex projective bundle.

	Let us describe its moment polyptych.
	Identifying $( \mft^{2} )^{\ast} \cong \RR^{2}$, write $x = ( x_{1} , x_{2} )$.
	The affine-linear functionals defining the arrangement are $\ell_{1}(x) = x_{1}$, $\ell_{2}(x) = x_{2}$, and
	$\ell_{3}(x) = m - x_{1} - x_{2}$, so the bounded distinguished base region is the $m$-fold dilate of the standard simplex $\Delta_{\emptyset} = \{ \ell_{i} \geq 0 ~|~ 1 \leq i \leq 3 \} = \Conv\{(0 , 0) , (m , 0) , (0 , m)\}$, i.e., the moment image of the zero section $X_{\nu} = \CC\PP^{2}$.
	The six unbounded regions consist of three regions adjacent across the edges of the simplex and three corner regions.
	Cutting truncates all six, and Proposition~\ref{prop:slack-support} gives the hexagonal support
	\[
		|\Delta_{\nu}^{\leq \delta}| = \mu_{\RR}(M_{\nu}^{\leq \delta}) = \Conv\{(-\delta , 0) , (0 , -\delta) , (m+\delta , -\delta) , (m+\delta , 0) , (0 , m+\delta) , (-\delta , m+\delta)\}.
	\]
	The original simplex is the central bounded cell, and the subdivision of the hexagon into this cell and the six truncated outer regions is shown in Figure~\ref{fig:polyptych-TCP2}.

	\begin{figure}[ht]
		\centering
		\begin{tikzpicture}[scale=1.1]

			\tikzset{
				polyboundary/.style={line width=0.9pt},
				polyedge/.style={line width=0.45pt},
				region/.style={font=\scriptsize , inner sep=1pt},
				hyperplane/.style={dashed, line width=0.45pt},
			}

			\def\mval{3.2}
			\def\dval{1.0}

			\coordinate (v12) at (0, 0);
			\coordinate (v23) at (\mval, 0);
			\coordinate (v13) at (0, \mval);

			\coordinate (bL)  at (-\dval, 0);
			\coordinate (bB)  at (0, -\dval);
			\coordinate (bBR) at ({\mval + \dval}, -\dval);
			\coordinate (bR)  at ({\mval + \dval}, 0);
			\coordinate (bT)  at (0, {\mval + \dval});
			\coordinate (bTL) at (-\dval, {\mval + \dval});

			\fill[gray!10] (v12) -- (v23) -- (v13) -- cycle;
			\draw[polyboundary]
			(bL) -- (bB) -- (bBR) -- (bR)
			-- (bT) -- (bTL) -- cycle;
			\draw[polyedge] (bB)  -- (bT);
			\draw[polyedge] (bL)  -- (bR);
			\draw[polyedge] (bTL) -- (bBR);

			\foreach \p in {v12,v23,v13,bL,bB,bBR,bR,bT,bTL}
			\fill (\p) circle (1.35pt);

			\node[region] at (0.95 , 0.92) {$\Delta_{\emptyset}$};
			\node[font=\scriptsize]
			at ($(bL)!0.5!(bTL)+(-0.38 , 0)$)
			{$\Delta_{1}^{\leq \delta}$};

			\node[font=\scriptsize]
			at ($(bB)!0.5!(bBR)+(0 , -0.38)$)
			{$\Delta_{2}^{\leq \delta}$};

			\node[font=\scriptsize]
			at ($(bR)!0.5!(bT)+(0.28 , 0.28)$)
			{$\Delta_{3}^{\leq \delta}$};

			\node[font=\scriptsize]
			at ($(bL)!0.5!(bB)+(-0.24 , -0.24)$)
			{$\Delta_{12}^{\leq \delta}$};

			\node[font=\scriptsize]
			at ($(bT)!0.5!(bTL)+(0 , 0.32)$)
			{$\Delta_{13}^{\leq \delta}$};

			\node[font=\scriptsize]
			at ($(bBR)!0.5!(bR)+(0.38 , 0)$)
			{$\Delta_{23}^{\leq \delta}$};

		\end{tikzpicture}
		\caption{
			The moment polyptych $\Delta_{\nu}^{\leq \delta}$ for $M_{\nu} \cong T^{\ast}\CC\PP^{2}$.
			The outer hexagon is its support, the internal subdivision records the seven maximal cells, and the filled dots mark the vertices.
		}
		\label{fig:polyptych-TCP2}
	\end{figure}
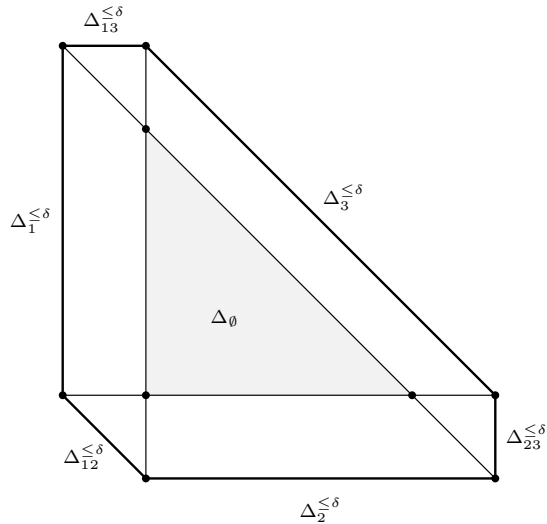
\end{example}


\begin{example}[$T^{\ast}(\CC\PP^{1} \times \CC\PP^{1})$]
	\label{example:TCP1xCP1-example-1}
	Lastly, we consider the two-dimensional hypertoric data $u_{1} = ( 1 , 0 )$, $u_{2} = ( 0 , 1 )$, $u_{3} = ( -1 , 0 )$, and $u_{4} = ( 0 , -1 )$, with parameter $\nu = ( m_{1} , m_{2} ) \in \RR_{>0}^{2}$, and choose the lift $\lambda = ( 0 , 0 , m_{1} , m_{2} ) \in ( \mft^{4} )^{\ast}$.
	The map $\pi : \mft^{4} \ra \mft^{2}$ has kernel generated by $( 1 , 0 , 1 , 0 )$ and $( 0 , 1 , 0 , 1 )$, so the kernel torus is $K = \ker( \pi ) \cong U_{1}^{2}$ and acts on $T^{\ast}\CC^{4}$ with corresponding real and complex moment maps
	\[
		\begin{split}
			\phi_{\RR}(z , w) &= \tfrac{1}{2} ( |z_{1}|^{2} + |z_{3}|^{2} - |w_{1}|^{2} - |w_{3}|^{2} , |z_{2}|^{2} + |z_{4}|^{2} - |w_{2}|^{2} - |w_{4}|^{2} ), \\
			\phi_{\CC}(z , w) &= ( z_{1}w_{1} + z_{3}w_{3} , z_{2}w_{2} + z_{4}w_{4} ).
		\end{split}
	\]
	The hyperk\"ahler quotient is $M_{\nu} = \phi_{\HK}^{-1}(( m_{1} , m_{2} ) , 0) / K \cong T^{\ast}(\CC\PP^{1} \times \CC\PP^{1})$ and, after setting $X_{\nu} = \CC\PP^{1} \times \CC\PP^{1}$, we have $M_{\nu} \cong T^{\ast}X_{\nu} \cong \Tot(\Omega_{X_{\nu}}^{1})$, and $\Omega_{X_{\nu}}^{1} \cong \mcO(-2 , 0) \oplus \mcO(0 , -2)$.

	Set $\mcB := \W_{X_{\nu}}^{1}$.
	As in the preceding cotangent-bundle examples, the fibre-rotation action is the scalar multiplication on the cotangent fibres.
	For every $\delta > 0$, the cut compactifies each $\CC^{2}$-fibre to $\CC\PP^{2}$ and gives, as a complex manifold,
	\[
		M_{\nu}^{\leq\delta}
		\cong
		\PP(\mcB\oplus\mcO_{X_{\nu}})
		\cong
		\PP(\mcO(-2 , 0)\oplus\mcO(0 , -2)\oplus\mcO).
	\]
	As before, varying $\delta$ changes the reduced K\"ahler class rather than the entire complex projective-bundle model.

	We now describe the moment polyptych.
	Identifying $( \mft^{2} )^{\ast} \cong \RR^{2}$, write $x = ( x_{1} , x_{2} )$.
	The affine-linear functionals defining the arrangement are $\ell_{1}(x) = x_{1}$, $\ell_{2}(x) = x_{2}$, $\ell_{3}(x) = m_{1} - x_{1}$, and $\ell_{4}(x) = m_{2} - x_{2}$.
	Thus, the compact base region is the rectangle $\Delta_{\emptyset} = [ 0 , m_{1} ] \times [ 0 , m_{2} ]$, which is the real moment-map image of the zero section $X_{\nu} = \CC\PP^{1} \times \CC\PP^{1}$.
	There are four unbounded side regions adjacent across its edges and four corner regions.
	Cutting these eight regions gives the octagonal support
	\[
		\scalebox{0.88}{
			$|\Delta_{\nu}^{\leq \delta}| = \mu_{\RR}(M_{\nu}^{\leq \delta}) = \Conv\{ (-\delta , 0) , (0 , -\delta) , (m_{1} , -\delta) , (m_{1} + \delta , 0), (m_{1} + \delta , m_{2}) , (m_{1},m_{2} + \delta) , (0 , m_{2} + \delta) , (-\delta , m_{2}) \}$.
			}
	\]
	The original rectangle is the central bounded cell, and Figure~\ref{fig:polyptych-TCP1xCP1} shows the subdivision of the octagonal support into this rectangle and the eight truncated outer cells.

	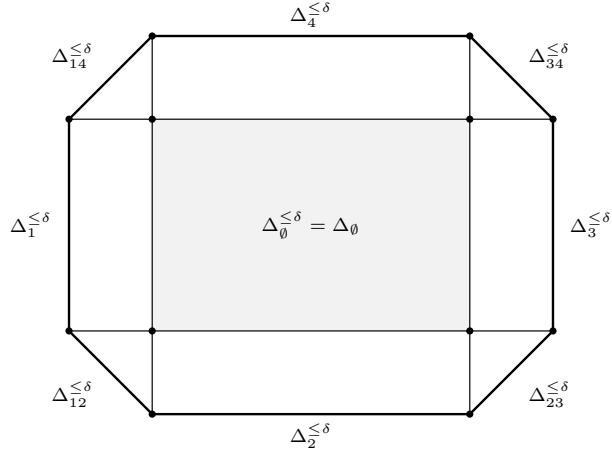
\begin{figure}[ht]
		\centering
		\begin{tikzpicture}[scale=1.0]

			\tikzset{
				polyboundary/.style={line width=0.9pt},
				polyedge/.style={line width=0.45pt},
				region/.style={font=\scriptsize , inner sep=1pt},
			}

			\def\moneval{4.2}
			\def\mtwoval{2.8}
			\def\dval{1.1}

			\coordinate (v12) at (0, 0);
			\coordinate (v23) at (\moneval, 0);
			\coordinate (v34) at (\moneval, \mtwoval);
			\coordinate (v14) at (0, \mtwoval);

			\coordinate (bL1) at (-\dval, 0);
			\coordinate (bB1) at (0, -\dval);
			\coordinate (bB2) at (\moneval, -\dval);
			\coordinate (bR1) at ({\moneval + \dval}, 0);
			\coordinate (bR2) at ({\moneval + \dval}, \mtwoval);
			\coordinate (bT2) at (\moneval, {\mtwoval + \dval});
			\coordinate (bT1) at (0, {\mtwoval + \dval});
			\coordinate (bL2) at (-\dval, \mtwoval);

			\fill[gray!10] (v12) -- (v23) -- (v34) -- (v14) -- cycle;

			\draw[polyboundary]
			(bL1) -- (bB1) -- (bB2) -- (bR1) -- (bR2) -- (bT2) -- (bT1) -- (bL2) -- cycle;

			\draw[polyedge] (bB1) -- (bT1);
			\draw[polyedge] (bB2) -- (bT2);
			\draw[polyedge] (bL1) -- (bR1);
			\draw[polyedge] (bL2) -- (bR2);

			\foreach \p in {v12, v23, v34, v14, bL1, bB1, bB2, bR1, bR2, bT2, bT1, bL2}
			\fill (\p) circle (1.35pt);

			\node[region] at ({0.5*\moneval}, {0.5*\mtwoval}) {$\Delta_{\emptyset}^{\leq \delta} = \Delta_{\emptyset}$};

			\node[font=\scriptsize, anchor=east]  at (-1.20, {0.5*\mtwoval}) {$\Delta_{1}^{\leq \delta}$};
			\node[font=\scriptsize, anchor=north] at ({0.5*\moneval}, -1.10) {$\Delta_{2}^{\leq \delta}$};
			\node[font=\scriptsize, anchor=west]  at ({\moneval+1.20}, {0.5*\mtwoval}) {$\Delta_{3}^{\leq \delta}$};
			\node[font=\scriptsize, anchor=south] at ({0.5*\moneval}, {\mtwoval + 1.10}) {$\Delta_{4}^{\leq \delta}$};

			\node[font=\scriptsize, anchor=north east] at (-0.65, -0.55) {$\Delta_{12}^{\leq \delta}$};
			\node[font=\scriptsize, anchor=north west] at ({\moneval + 0.65}, -0.55) {$\Delta_{23}^{\leq \delta}$};
			\node[font=\scriptsize, anchor=south west] at ({\moneval + 0.65}, {\mtwoval + 0.55}) {$\Delta_{34}^{\leq \delta}$};
			\node[font=\scriptsize, anchor=south east] at (-0.65, {\mtwoval + 0.55}) {$\Delta_{14}^{\leq \delta}$};
		\end{tikzpicture}
		\caption{
			The moment polyptych $\Delta_{\nu}^{\leq \delta}$ for $M_{\nu} \cong T^{\ast}(\CC\PP^{1} \times \CC\PP^{1})$.
			The outer octagon is its support, the internal subdivision records the nine maximal cells, and the filled dots mark the vertices.
		}
		\label{fig:polyptych-TCP1xCP1}
	\end{figure}
\end{example}


\subsection{Properties of Cut Spaces}
\label{subsection:subsection-4.5-cut-space-properties}

These examples illustrate the local geometry of the truncations defining the moment polyptych.
Returning to the general setting, let us establish two structural properties that are required in our localisation argument.
First, the restriction of the moment polyptych to a flat is compatible with cutting the corresponding hypertoric subvariety.
Second, this dimensional reduction can then be used to determine when finite stabilisers necessarily occur on the cut boundary.

Given a flat $H_{\mcF} \in L(\mcA)$, define the \textbf{restriction of the moment polyptych to the flat} $H_{\mcF}$ by
\[
	\Delta_{\nu}^{\leq \delta}|_{H_{\mcF}} :=
	\bigcup_{\substack{ A \subseteq \{ 1 , \ldots , N \} \\ \Delta_{A}^{\leq \delta} \cap H_{\mcF} \neq \varnothing } }
	\Faces\left( \Delta_{A}^{\leq \delta} \cap H_{\mcF} \right).
\]
Since each $\ell_{i}$ with $i \in \mcF$ is either everywhere non-negative or everywhere non-positive on $\Delta_{A}$, the equalities $\ell_{i} = 0$ defining $H_{\mcF} = \bigcap_{i \in \mcF} H_{i}$ cut out a face of $\Delta_{A}^{\leq \delta}$, and thus every non-empty intersection $\Delta_{A}^{\leq \delta} \cap H_{\mcF}$ is a face of $\Delta_{A}^{\leq \delta}$.
Hence, $\Delta_{\nu}^{\leq \delta}|_{H_{\mcF}}$ is a polytopal subcomplex of $\Delta_{\nu}^{\leq \delta}$ with support $|\Delta_{\nu}^{\leq \delta}| \cap H_{\mcF}$ in $(\mft^{n})^{\ast}$.

For every $B \subseteq \mcFc$, let $\Delta_{B}^{\crk\mcF}$ be the augmented sign region \eqref{eqn:augmented-sign-region}.
We denote its \emph{ambient truncation} by
\begin{equation}
	\label{eqn:ambient-truncation}
	\Delta_{B}^{\crk \mcF, \leq \delta} := \Delta_{B}^{\crk \mcF} \cap
	\left\{ x \in ( \mft^{\crk \mcF} )^{\ast} ~|~ \langle x ,\, u_{B}^{\mcFc} \rangle \leq \delta + \lambda_{B}^{\mcFc} \right\},
\end{equation}
where $u_{B}^{\mcFc}$ and $\lambda_{B}^{\mcFc}$ are as fixed in Remark~\ref{remark:subarrangement-regions}.

\begin{remark}[Ambient and Intrinsic Normalisation]
	The terms \emph{ambient} and \emph{intrinsic} refer to the normalisation of the fibre-rotation moment map, rather than to the vector space in which a truncation is written.

	On the \emph{ambient} hypertoric variety, the fibre-rotation and residual real moment maps are $\Phi : M_{\nu} \ra \RR$ and $\mu_{\RR} : M_{\nu} \ra (\mft^{n})^{\ast}$, respectively.
	Their restrictions to the hypertoric subvariety $M_{\mcF} \subseteq M_{\nu}$ retain this \emph{ambient} normalisation; in particular, the residual moment image lies in the flat $H_{\mcF} \subseteq (\mft^{n})^{\ast}$.

	\emph{Intrinsically}, on the other hand, the effective hypertoric orbifold $M_{\mcF}^{\eff}$ has the residual real moment map $\mu_{\RR}^{\mcF} : M_{\mcF}^{\eff} \ra (\mft^{\crk\mcF})^{\ast}$ for the $T^{\crk \mcF}$-action and the standard fibre-rotation moment map $\Phi_{\mcF}^{\std} : M_{\mcF}^{\eff} \ra \RR$.
	The affine isomorphism $\eta_{y_{0}} : (\mft^{\crk\mcF})^{\ast} \overset{\sim}{\longrightarrow} H_{\mcF}$ identifies the \emph{intrinsic} moment-coordinate space $(\mft^{\crk\mcF})^{\ast}$ with the \emph{ambient} flat $H_{\mcF} \subseteq (\mft^{n})^{\ast}$.

	Under the identification of the underlying coarse spaces, we have $\Phi|_{M_{\mcF}} = \Phi_{\mcF}^{\std} + C_{\mcF}$.
	Thus, although the ambient truncations $\Delta_{B}^{\crk\mcF,\leq\delta}$ live in $(\mft^{\crk\mcF})^{\ast}$, they use the level $\delta$ inherited from the ambient Hamiltonian $\Phi$.
	That is to say, they are the intrinsic truncations at level 	$\delta - C_{\mcF}$.
\end{remark}

\begin{proposition}
	\label{prop:cut-subspace-subpolyptych}
	For a subset $\mcF \subseteq \{ 1 , \ldots , N \}$, let $H_{\mcF} \in L(\mcA)$ be the corresponding flat, and let $j : M_{\mcF} \hookrightarrow M_{\nu}$ be the associated hypertoric subvariety with the orbifold presentation induced from $M_{\nu}$.
	Let $\delta > 0$ be a regular value of the ambient cutting moment map $\rho : M_{\nu} \times \CC \ra \RR$.
	Then cutting via the restricted ambient fibre-rotation moment map $\Phi|_{M_{\mcF}} : M_{\mcF} \ra \RR$ at the level $\delta$ gives
	\[
		M_{\mcF}^{\leq \delta} := ( M_{\mcF} \times \CC )\sslash_{\delta} U_{1} = \bigl( \rho^{-1}(\delta) \cap ( M_{\mcF} \times \CC ) \bigr) / U_{1}.
	\]

	On the effective hypertoric orbifold $M_{\mcF}^{\eff}$, the same underlying cut is obtained intrinsically from $\Phi_{\mcF}^{\std} : M_{\mcF}^{\eff} \ra \RR$ at the shifted level $\delta - C_{\mcF}$.
	Let $\Delta_{\mcF}^{\leq \delta}$ be the polytopal complex in $(\mft^{\crk \mcF})^{\ast}$ generated by the non-empty ambient truncations $\Delta_{B}^{\crk \mcF, \leq \delta}$ of \eqref{eqn:ambient-truncation}, with support
	\[
		|\Delta_{\mcF}^{\leq \delta}| = \bigcup\nolimits_{ \substack{ B \subseteq \mcFc \\ \Delta_{B}^{\crk \mcF, \leq \delta} \neq \varnothing } }
		\Delta_{B}^{\crk \mcF, \leq \delta} \subseteq (\mft^{\crk \mcF})^{\ast}.
	\]
	Then:
	\begin{enumerate}[(i)]
		\item the inclusion $j : M_{\mcF}\hookrightarrow M_{\nu}$
		descends to a closed K\"ahler embedding $j^{\leq \delta} : M_{\mcF}^{\leq \delta} \hookrightarrow M_{\nu}^{\leq \delta}$ after cutting;
		\label{prop:cut-subspace-subpolyptych-1}
		\item for the chosen base point $y_{0} \in H_{\mcF}$, the affine isomorphism $\eta_{y_{0}} : (\mft^{\crk \mcF})^{\ast} \overset{\sim}{\longrightarrow} H_{\mcF} \subseteq (\mft^{n})^{\ast}$ identifies the moment polyptych $\Delta_{\mcF}^{\leq \delta}$, written in intrinsic coordinates but using the ambient cut level $\delta$, with the subcomplex obtained by restricting the ambient moment polyptych to $H_{\mcF}$:
		\[
			\eta_{y_{0}} : \Delta_{\mcF}^{\leq\delta} \overset{\sim}{\longrightarrow} \left. \Delta_{\nu}^{\leq \delta} \right|_{H_{\mcF}} \subseteq \Delta_{\nu}^{\leq \delta}.
		\]
		In particular, on supports, we have $\eta_{y_{0}} \bigl( |\Delta_{\mcF}^{\leq \delta}| \bigr) = |\Delta_{\nu}^{\leq \delta}| \cap H_{\mcF} \subseteq(\mft^{n})^{\ast}$.
		\label{prop:cut-subspace-subpolyptych-2}
	\end{enumerate}
\end{proposition}

\begin{proof}
	The coordinate inclusion $T^{\ast}\CC^{\mcFc} \hookrightarrow T^{\ast}\CC^{N}$ from \eqref{eqn:subvariety-closed-embedding} is closed, preserves the hyperk\"ahler forms, and is equivariant for the fibre-rotation action.
	The induced embedding $j : M_{\mcF} \hookrightarrow M_{\nu}$ is therefore $U_{1}$-equivariant, and the product embedding $j \times \Identity_{\CC} : M_{\mcF} \times \CC \hookrightarrow M_{\nu} \times \CC$ restricts to a closed $U_{1}$-invariant inclusion of cut levels.
	Since the ambient action is locally free at the regular level $\delta$, so is its restriction, and quotienting gives a closed embedding $j^{\leq \delta} : M_{\mcF}^{\leq \delta} \hookrightarrow M_{\nu}^{\leq \delta}$.
	The product embedding $j \times \Identity_{\CC}$ preserves the K\"ahler forms, so after reduction the K\"ahler form on $M_{\mcF}^{\leq \delta}$ is $(j^{\leq \delta})^{\ast} \w_{\RR}^{\leq \delta}$, proving \ref{prop:cut-subspace-subpolyptych-1}.

	For \ref{prop:cut-subspace-subpolyptych-2}, recall that the affine isomorphism $\eta_{y_{0}} : (\mft^{\crk \mcF})^{\ast} \overset{\sim}{\rightarrow} H_{\mcF} \subseteq (\mft^{n})^{\ast}$ identifies the intrinsic coordinate space with the ambient flat.
	Fix $B \subseteq \mcFc$ and set $A := B \cup \mcF$, so that Proposition~\ref{prop:restriction-identification} and Remark~\ref{remark:subarrangement-regions} give $\eta_{y_{0}}(\Delta_{B}^{\crk \mcF}) = \Delta_{A} \cap H_{\mcF}$.
	When non-empty, the region $\Delta_{B}^{\crk \mcF}$ in $(\mft^{\crk \mcF})^{\ast}$ corresponds to a face of the ambient region $\Delta_{A}$ in $(\mft^{n})^{\ast}$.
	For $x \in (\mft^{\crk \mcF})^{\ast}$, put $y := \eta_{y_{0}}(x) \in H_{\mcF}$ so that, by \eqref{eqn:corank-dual-subalgebra-identification-result}, we have $\ell_{i}(y) = 0$ for $i \in \mcF$ and $\ell_{i}(y) = \ell_{i}^{\mcFc}(x)$ for $i \in \mcFc$. Therefore,
	\[
		\langle y ,\, u_{A} \rangle - \lambda_{A} = -\sum\nolimits_{i \in A} \ell_{i}(y) = -\sum\nolimits_{i \in B} \ell_{i}^{\mcFc}\!(x) = \langle x ,\, u_{B}^{\mcFc} \rangle - \lambda_{B}^{\mcFc},
	\]
	hence the truncating inequality in the intrinsic coordinate $x$ agrees with the ambient truncating inequality at $y = \eta_{y_{0}}(x)$.
	Thus, from \eqref{eqn:ambient-truncation}, we get
	\begin{equation}
		\label{eqn:subregion-identification}
		\eta_{y_{0}}(\Delta_{B}^{\crk \mcF,\leq \delta}) = \Delta_{B \cup \mcF}^{\leq \delta} \cap H_{\mcF}.
	\end{equation}

	To express the same truncation using the intrinsic normalisation, recall that the ambient slack function $\mcS : (\mft^{n})^{\ast} \ra \RR_{\geq 0}$ and the intrinsic slack 	function $\mcS_{\mcF}^{\std} : (\mft^{\crk \mcF})^{\ast} \ra \RR_{\geq 0}$ satisfy $\mcS(\eta_{y_{0}}(x)) = \mcS_{\mcF}^{\std}(x) + C_{\mcF}$ by \eqref{eqn:slack-restriction}.
	Hence,
	\[
		\Delta_{B}^{\crk \mcF,\leq \delta} = \left\{ x \in \Delta_{B}^{\crk \mcF} ~|~ \mcS_{\mcF}^{\std}(x) \leq \delta - C_{\mcF} \right\}.
	\]
	On each extended-core component $\mcE_{B}^{\mcF} \subseteq M_{\mcF}^{\eff}$, the moment-map formulae of Subsection~\ref{subsection:subsection-3.4-hypertoric-subvarieties} give $\Phi_{\mcF}^{\std} = \mcS_{\mcF}^{\std} \circ \mu_{\RR}^{\mcF}$, where $\mu_{\RR}^{\mcF} : M_{\mcF}^{\eff} \ra (\mft^{\crk \mcF})^{\ast}$ is the intrinsic residual real moment map.
	Then we see that $\Delta_{\mcF}^{\leq \delta}$ is the moment polyptych obtained by cutting $M_{\mcF}^{\eff}$ at level $\delta - C_{\mcF}$ using $\Phi_{\mcF}^{\std} : M_{\mcF}^{\eff} \ra \RR$, and this intrinsic cut has the same underlying coarse space as $M_{\mcF}^{\leq \delta}$.
	By \ref{prop:cut-subspace-subpolyptych-1}, $M_{\mcF}^{\leq \delta}$ is a closed subspace of the compact ambient cut $M_{\nu}^{\leq \delta}$, and is therefore compact.
	Moreover, this argument applies even when
	$H_{\mcF} \cap \Delta_{\emptyset} = \varnothing$.

	It remains to check that \eqref{eqn:subregion-identification} accounts for every polytope generating $\Delta_{\nu}^{\leq \delta}|_{H_{\mcF}}$.
	Take any ambient truncation $\Delta_{A^{\prime}}^{\leq \delta} \subseteq (\mft^{n})^{\ast}$ meeting $H_{\mcF}$, and set $B := A^{\prime} \cap \mcFc$.
	For each $i \in \mcF$, every point of $H_{\mcF}$ lies on $H_{i}$ and therefore belongs to both closed half-spaces bounded by $H_{i}$.
	Then since $B = A^{\prime} \cap \mcFc$, the inequalities indexed by $\mcFc$ are identical in the definitions of $\Delta_{A^{\prime}}$ and $\Delta_{B \cup \mcF}$: both require $\ell_{i} \leq 0$ for $i \in B$ and $\ell_{i} \geq 0$ for $i \in \mcFc \setminus B$.
	Hence, $\Delta_{A^{\prime}} \cap H_{\mcF} = \Delta_{B \cup \mcF} \cap H_{\mcF}$, and as both of these regions are truncated by the same ambient inequality $\mcS \leq \delta$, we obtain
	\[
		\Delta_{A^{\prime}}^{\leq \delta} \cap H_{\mcF}
		= \Delta_{B \cup \mcF}^{\leq \delta} \cap H_{\mcF}
		= \eta_{y_{0}}(\Delta_{B}^{\crk \mcF,\leq \delta}).
	\]
	It follows that the images of the non-empty ambient truncations that generate $\Delta_{\mcF}^{\leq \delta}$ also generate the restricted ambient complex $\Delta_{\nu}^{\leq \delta}|_{H_{\mcF}}$.
	As $\eta_{y_{0}}$ is an affine isomorphism, it preserves faces and their incidences, which yields the claimed isomorphism of polytopal complexes.
	Finally, taking the union of their cells gives the support identity $\eta_{y_{0}} ( |\Delta_{\mcF}^{\leq \delta}| ) = |\Delta_{\nu}^{\leq \delta}| \cap H_{\mcF} \subseteq (\mft^{n})^{\ast}$.
\end{proof}

A regular symplectic cut by a locally free circle action is an orbifold.
However, a non-trivial finite stabiliser does not by itself make the coarse quotient singular: a finite group generated by pseudo-reflections has a smooth quotient.
For the cuts considered here, reducibility of the core forces a cyclic stabiliser which acts non-trivially in two complex directions, so the local group contains no non-trivial pseudo-reflections; this is the content of the following theorem.

\begin{theorem}
	\label{thm:reducible-core-orbifold}
	Let $M_{\nu}$ be a smooth hypertoric manifold of real dimension $4n$, and suppose that its core $\mcC$ is reducible.
	For every cut level $\delta \in \Cout = ( c_{\max} , \infty )$, there exists a boundary point of $M_{\nu}^{\leq \delta}$ whose local orbifold structure group is cyclic of order at least two and whose coarse local quotient is singular.
	In particular, the cut space $M_{\nu}^{\leq \delta}$ is not a smooth manifold.
\end{theorem}

First, however, let us consider the simpler case when $M_{\nu}$ has real dimension four (i.e., when $n = 1$), since then the general case in Theorem~\ref{thm:reducible-core-orbifold} will follow from it by induction via restrictions to a four-dimensional hypertoric submanifold.

\begin{lemma}
	\label{lemma:one-dimensional-reducible-core-orbifold}
	Let $M_{\nu}$ be a smooth hypertoric manifold of real dimension four, and suppose that its core $\mcC$ is reducible.
	Then, for every cut level $\delta \in \Cout = (c_{\max} , \infty)$, the cut space $M_{\nu}^{\leq \delta}$ contains a boundary point with cyclic local orbifold group $\mu_{k}$ of some order $k \geq 2$.
	On a local uniformising chart, this group acts non-trivially on two complex tangent directions, and hence the coarse local quotient is singular.
\end{lemma}

\begin{proof}[Proof of Lemma~\ref{lemma:one-dimensional-reducible-core-orbifold}]
	As $n = 1$, fix an identification $( \mft^{n} )^{\ast} \cong \RR$.
	Since $\mcA$ is simple, the hyperplanes $H_{1} , \ldots , H_{N}$ are $N$ distinct points (i.e., arrangement vertices) whose bounded regions are the $N - 1$ closed intervals between each consecutive pair of points in $\RR$.
	Moreover, there are at least two bounded regions as the core $\mcC$ is reducible, so $N \geq 3$.
	It follows that the distinguished bounded interval $\Delta_{\emptyset}$ shares an endpoint with a further bounded interval.
	If $H_{j}$ is that common endpoint, then the neighbouring region differs from $\Delta_{\emptyset}$ only in the sign of $\ell_{j}$, which we write as $\Delta_{j} := \Delta_{\{ j \}}$.
	Let $X_{\nu} = \mcE_{\emptyset}$ and $\mcE_{j} = \mcE_{\{ j \}}$ be the corresponding compact core components, so $X_{\nu} \cup \mcE_{j} \subseteq \mcC$.
	The one-dimensional geometrical picture is shown in Figure~\ref{fig:adjacent-region}: the adjacent bounded intervals $\Delta_{\emptyset}$ and $\Delta_{j}$ are separated by $H_{j}$, whose cooriented normal vector $u_{j}$ points towards $\Delta_{\emptyset}$.

	\begin{figure}[!htbp]
		\centering
		\captionsetup[subfigure]{labelfont=normalfont}

	    \newcommand{\drawCommonChain}[1]{%
	        \coordinate (hL) at (-2.00 , 0);
	        \coordinate (hJ) at (1.00 , 0);
	        \coordinate (hR) at (2.70 , 0);
	        \coordinate (hA) at (4.25 , 0);

	        \foreach \h in {hL,hJ,hR,hA}{%
	            \draw[hyperplane] ($ (\h)+(0 , 0.70) $) -- ($ (\h)-(0 , 0.70) $);
	        }
	        \node[below=2pt] at ($(hJ)+(0 , -0.62)$) {$H_{j}$};

	        \draw[region] (hL) -- (hJ);
	        \draw[region] (hJ) -- (hR);
	        \node[below=4pt] at ($(hL)!0.5!(hJ)$) {$\Delta_{\emptyset}$};
	        \node[below=4pt] at ($(hJ)!0.5!(hR)$) {$\Delta_{j}$};

	        \draw[continuation] (hL) -- ++(-2.00 , 0);
	        \draw[continuation] (hR) -- (hA);
	        \node[below=4pt] at ($(hR)!0.5!(hA)$) {$\cdots$};
	        \draw[continuation] (hA) -- ++(3.20 , 0);
	        \node[below=4pt] at ($(hA)+(0.95 , 0)$) {$#1$};

	        \draw[normal] ($(hJ)+(0 , 0.52)$) -- ++(-0.80 , 0);
	        \node[above left=2pt] at ($(hJ)+(-0.38 , 0.52)$) {$u_{j}$};
	    }

	    \begin{subfigure}[t]{0.49\linewidth}
	        \centering
	        \begin{tikzpicture}[
	            x=0.60cm, y=0.86cm,
	            hyperplane/.style={line width=0.5pt},
	            region/.style={line width=0.9pt},
	            continuation/.style={densely dashed, line width=0.5pt},
	            normal/.style={->, line width=0.6pt},
	            every node/.style={font=\small}
	        ]

	            \path[use as bounding box] (-4.1 , -1.28) rectangle (8.15 , 1.15);
	            \drawCommonChain{\Delta_{A}}
	        \end{tikzpicture}
	        \caption{Before cutting.}
	        \label{fig:adjacent-region}
	    \end{subfigure}%
	    \hfill
	    \begin{subfigure}[t]{0.49\linewidth}
	        \centering
	        \begin{tikzpicture}[
	            x=0.60cm, y=0.86cm,
	            hyperplane/.style={line width=0.5pt},
	            region/.style={line width=0.9pt},
	            continuation/.style={densely dashed, line width=0.5pt},
	            normal/.style={->, line width=0.6pt},
	            cutline/.style={line width=0.7pt},
	            every node/.style={font=\small}
	        ]
	            \path[use as bounding box] (-4.1 , -1.28) rectangle (8.15 , 1.15);
	            \drawCommonChain{\Delta_{A}^{\leq \delta}}

	            \coordinate (cutA) at (6.00 , 0);
	            \draw[cutline] ($(cutA)+(0 , 0.70)$) -- ($(cutA)-(0 , 0.70)$);
	            \fill (cutA) circle (1.4pt);
	            \node[below=2pt] at ($(cutA)+(0 , -0.67)$) {$\Pi_{A}^{\delta}$};
	            \draw[normal] ($(cutA)+(0 , 0.52)$) -- ++(1.60 , 0);
	            \node[above right=2pt] at ($(cutA)+(0.95 , 0.52)$) {$u_{A}$};
	        \end{tikzpicture}
	        \caption{After cutting at $\delta$.}
	        \label{fig:non-primitive}
	    \end{subfigure}

	    \caption{
	        The one-dimensional configuration in the proof of Lemma~\ref{lemma:one-dimensional-reducible-core-orbifold}.
	        The bounded regions $\Delta_{\emptyset}$ and $\Delta_{j}$ are separated by $H_{j}$, whose normal $u_{j}$ points towards $\Delta_{\emptyset}$.
	        The right panel adds the cut point $\Pi_{A}^{\delta}$ in the relative interior of the terminal unbounded region $\Delta_{A}$, with cutting normal $u_{A}$ pointing outwards from the retained side.
	    }
	    \label{fig:arrangement-chain-and-cut}
	\end{figure}
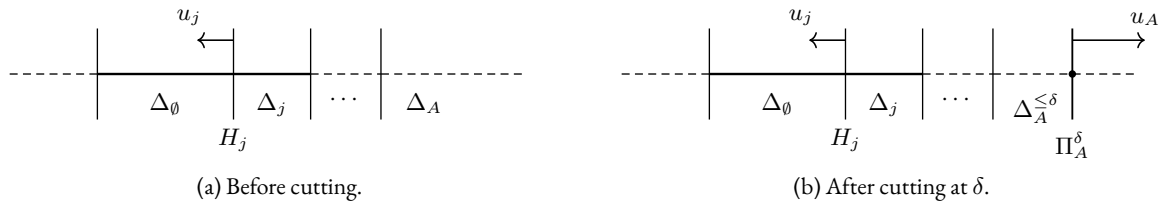

	In dimension one, the regions encountered on either side of $\Delta_{\emptyset}$ form a chain in the poset of regions $( \mcR(\mcA) , \preceq )$ from Remark~\ref{remark:poset-regions}, and the two chains meet only at the minimum $\Delta_{\emptyset}$.
	The neighbouring bounded region $\Delta_{j}$ lies in exactly one of these chains, and we let $\Delta_{A}$ be the maximal, and hence unbounded, region in that chain corresponding to a non-compact extended-core component $\mcE_{A}$.
	The cooriented primitive normals of the hyperplanes crossed along this chain all point towards $\Delta_{\emptyset}$ so, after identifying $\mft^{1}_{\ZZ} \cong \ZZ$ in a way that is compatible with this orientation, they all have the same sign and \eqref{eqn:U1-action-generator} gives $|u_{A}| = |-\sum_{i \in A} u_{i}| = |A|$.
	Because $\Delta_{j}$ is bounded, the chain from $\Delta_{\emptyset}$ to its unbounded terminal region $\Delta_{A}$ crosses at least two hyperplanes, so $|u_{A}| = |A| \geq 2$ and $u_{A}$ is non-primitive.
	The resulting truncation of $\Delta_{A}$ to $\Delta_{A}^{\leq \delta}$ is illustrated in Figure~\ref{fig:non-primitive}.

	Write $u_{A} = \pm k \widebar{u}$, where $k \geq 2$ and $\widebar{u}$ is the positive primitive generator of $\mft_{\ZZ}^{1} \cong \ZZ$.
	Since $\delta \in \Cout$, every arrangement vertex $v$ satisfies $\mcS(v) \leq c_{\max} < \delta$, whereas the defining cut equation gives $\mcS = \delta$ at $\Pi_{A}^{\delta}$.
	The cut point therefore cannot be an arrangement vertex and, since the boundary points of the one-dimensional region $\Delta_{A}$ are arrangement vertices, it lies in the relative interior of $\Delta_{A}$.

	Choose $[ z , w ] \in \Phi_{A}^{-1}(\delta) \subseteq \mcE_{A}$ with $\mu_{\RR}[ z , w ] \in \Pi_{A}^{\delta}$; the toric orbit--face correspondence then implies that the residual $T^{1}$-action is free at $[ z , w ]$.
	With respect to the generator $\widebar{u}$, the homomorphism $\pi \circ \iota_{A} : U_{1} \ra T^{1}$ inducing the restricted cutting action from Subsection~\ref{subsection:subsection-4.1-residual-circle-action} is the $k$-fold covering $\tau \mapsto \tau^{\pm k}$.
	The stabiliser of $[ z , w ]$ under this cutting action is therefore its kernel, $\ker(\pi \circ \iota_{A}) = \mu_{k} := \{ \tau \in U_{1} ~\mid~ \tau^{k} = 1 \} \cong \ZZ / k \ZZ$.

	The representative $( [ z , w ] , 0 ) \in \rho_{A}^{-1}(\delta) \subseteq \rho^{-1}(\delta)$ has the same stabiliser for the diagonal cutting action as its auxiliary coordinate $\zeta$ is zero.
	Its quotient class $p := [ [ z , w ] , 0 ] \in Z_{A}^{\delta} \subseteq Z_{\nu}^{\delta} \subseteq M_{\nu}^{\leq \delta}$ therefore has the local orbifold structure group $\Gamma_{p} := \mu_{k}$, since $M_{\nu} \times \CC$ is smooth and the cutting action is locally free at the regular level $\delta$.
	Unlike $[ z , w ]$ before reduction, $p$ is fixed by the residual $T^{1}$-action: surjectivity of $\pi \circ \iota_{A} : U_{1} \ra T^{1}$ ensures that, for every $t \in T^{1}$, the point $t \cdot ( [ z , w ] , 0 )$ lies in the same cutting orbit as $( [ z , w ] , 0 )$.

	It remains to determine the action of $\Gamma_{p} = \mu_{k}$ on a local uniformising chart of the cut.
	Let $q = [ z , w ] \in \mcE_{A}$ be the point chosen above, and let $V_{q} \subseteq T_{q}M_{\nu}$ be the complex tangent line bundle to the complexified cutting-circle orbit through $q$.
	Since the cutting circle is abelian, every $g \in \mu_{k}$ acts trivially on the orbit line $V_{q}$.
	To determine its action transverse to the orbit, i.e., on the quotient line $T_{q}M_{\nu} / V_{q}$, consider the holomorphic-symplectic form $\w_{\CC}$ on $M_{\nu}$.
	Since $M_{\nu}$ has complex dimension two and $\w_{\CC}$ is non-degenerate, the fibre-rotation identity $\tau^{\ast}\w_{\CC} = \tau\w_{\CC}$, applied to $\tau = g \in \mu_{k}$, gives $\det_{\CC}( (d g)_{q} ) = g$.
	The action on $V_{q}$ is trivial, so the induced action on $T_{q}M_{\nu} / V_{q}$ must therefore be multiplication by $g$.
	The auxiliary tangent line $\CC_{\zeta}$ at $\zeta = 0$ also carries the character $g \mapsto g$.

	Since K\"ahler reduction removes the complex orbit direction $V_{q}$, on a local uniformising chart of the cut at $p$, the $\mu_{k}$-representation is $( T_{q}M_{\nu} / V_{q} ) \oplus \CC_{\zeta}$, where $g \cdot ( a , b ) = ( g a , g b )$.
	Hence, every non-identity element of $\mu_{k}$ acts non-trivially on both complex tangent directions, so no non-identity element is a pseudo-reflection.
	The Chevalley--Shephard--Todd theorem \cite{ShephardTodd1954, Chevalley1955} therefore implies that the coarse local quotient is singular.
\end{proof}

The higher-dimensional case reduces to Lemma~\ref{lemma:one-dimensional-reducible-core-orbifold} by restricting to a flat of rank $n - 1$, as follows.

\begin{proof}[Proof of Theorem~\ref{thm:reducible-core-orbifold}]
	Consider the distinguished compact core component $X_{\nu} = \mcE_{\emptyset}$.
	By \cite[Corollary~3.6]{HaradaProudfoot2004}, the core $\mcC$ is a deformation retract of $M_{\nu}$ and is therefore connected.
	Furthermore, since $\mcC$ is reducible and is the finite union of its compact irreducible components, the intersection graph of these components is connected.
	It follows that there exists a compact core component $\mcE_{A} \neq \mcE_{\emptyset}$ such that $\mcE_{A} \cap \mcE_{\emptyset} \neq \varnothing$, so their moment polytopes have a non-empty intersection.

	Choose a vertex $v$ of the common face $\Delta_{A} \cap \Delta_{\emptyset}$, which is also a vertex of $\Delta_{\emptyset}$.
	Then $v$ is an arrangement vertex by Lemma~\ref{lemma:regions-pointed}, and simplicity implies that exactly $n$ hyperplanes pass through $v$, so $|\mcI_{v}| = n$.
	Furthermore, because $v \in \Delta_{A} \cap \Delta_{\emptyset}$, every $i \in A$ satisfies both $\ell_{i}(v) \leq 0$ and $\ell_{i}(v) \geq 0$, so $\ell_{i}(v) = 0$ and $A \subseteq \mcI_{v}$.
	Since $A \neq \varnothing$, choose $j \in A$ and set $\mcJ_{v , j} := \mcI_{v} \setminus \{ j \}$.
	Then the regions $\Delta_{A}$ and $\Delta_{\emptyset}$ lie in opposite closed half-spaces determined by $H_{j}$, since their respective sign conditions are $\ell_{j} \leq 0$ and $\ell_{j} \geq 0$.

	By simplicity, $H_{\mcJ_{v , j}}$ is an affine line in $(\mft^{n})^{\ast}$ through $v$ and, along this line, each functional $\ell_{i}$ with $i \in \mcJ_{v , j}$ vanishes identically whereas $\ell_{j}$ is non-constant since $u_{j} \not\in \Span_{\RR}\{ u_{i} \mid i \in \mcJ_{v , j} \}$ by simplicity.
	The inequalities corresponding to the indices outside $\mcI_{v}$ remain strict in a neighbourhood of $v$.
	Since $\Delta_{A}$ and $\Delta_{\emptyset}$ are bounded and impose opposite inequalities at $H_{j}$, it follows that $\Delta_{A} \cap H_{\mcJ_{v , j}}$ and $\Delta_{\emptyset} \cap H_{\mcJ_{v , j}}$ are non-degenerate bounded intervals on opposite sides of $v$ that meet at $v$.

	Let us first check that the effective and ambient-induced orbifold presentations of $M_{\mcJ_{v , j}}$ coincide.
	Since $M_{\nu}$ is smooth, the arrangement $\mcA$ is smooth by \cite[Theorem~3.2]{BielawskiDancer2000} and the normals $\{ u_{i} \mid i \in \mcI_{v} \}$ form a $\ZZ$-basis of $\mft_{\ZZ}^{n}$.
	As $\mcJ_{v , j} = \mcI_{v} \setminus \{j\}$, this gives $\mft_{\ZZ}^{n} = U_{\mcJ_{v , j}} \oplus \ZZ u_{j}$.
	Thus, $U_{\mcJ_{v , j}}$ is saturated: it contains every ambient lattice point in its real span $\mft_{\mcJ_{v , j}}$.
	Therefore, $V_{\mcJ_{v , j}} = \mft_{\mcJ_{v , j}} \cap \mft_{\ZZ}^{n} = U_{\mcJ_{v , j}}$, so $\Gamma_{\mcJ_{v , j}} = V_{\mcJ_{v , j}} / U_{\mcJ_{v , j}}$ is trivial.
	By the construction in Subsection~\ref{subsection:subsection-3.4-hypertoric-subvarieties}, it follows that the effective quotient $M_{\mcJ_{v , j}}^{\eff}$ has the same orbifold structure as the ambient-induced subvariety $M_{\mcJ_{v , j}} \subseteq M_{\nu}$; compare Lemma~\ref{lemma:flat-local-group-comparison}.

	We next check that the restricted arrangement on $H_{\mcJ_{v , j}}$ is smooth.
	It is simple by Lemma~\ref{lemma:restricted-arrangement-simple}, so it suffices to show that every active projected normal is primitive.
	For any $i \not\in \mcJ_{v , j}$ with $H_{i} \cap H_{\mcJ_{v , j}} \neq \varnothing$, this intersection is an ambient arrangement vertex.
	Smoothness of $\mcA$ therefore makes the corresponding normals a lattice basis, giving $\mft_{\ZZ}^{n} = U_{\mcJ_{v , j}} \oplus \ZZ u_{i}$.
	Thus, the image of $u_{i}$ generates the rank-one quotient lattice $\mft_{\ZZ}^{n} / U_{\mcJ_{v , j}}$ and is therefore primitive.

	By Theorem~\ref{thm:hypertoric-subvariety}, the associated hypertoric subvariety $M_{\mcJ_{v , j}} \subseteq M_{\nu}$ has real dimension $4\crk(\mcJ_{v , j}) = 4$.
	Proposition~\ref{prop:restriction-identification} identifies its intrinsic labelled arrangement $\mcA^{\crk \mcJ_{v , j}}$ in $(\mft^{\crk \mcJ_{v , j}})^{\ast}$ with this smooth restricted arrangement $\mcA_{\mcJ_{v , j}}$ on $H_{\mcJ_{v , j}}$.
	The effective quotient is therefore smooth; since the effective and ambient-induced presentations coincide, $M_{\mcJ_{v , j}}$ is a smooth four-dimensional hypertoric submanifold of $M_{\nu}$.

	For the flat $H_{\mcJ_{v , j}}$ in particular, the ambient and intrinsic fibre-rotation normalisations coincide since $v \in \Delta_{\emptyset} \cap H_{\mcJ_{v , j}}$, and so $\mcS(v) = 0$.
	Writing $x_{v}:=\eta_{y_{0}}^{-1}(v) \in (\mft^{\crk \mcJ_{v , j}})^{\ast}$, equation~\eqref{eqn:slack-restriction} gives $0 = \mcS(v) = \mcS_{\mcJ_{v , j}}^{\std}(x_{v}) + C_{\mcJ_{v , j}}$.
	Both terms on the right are non-negative, and hence $C_{\mcJ_{v , j}} = 0$, so it follows that $\Phi|_{M_{\mcJ_{v , j}}} = \Phi_{\mcJ_{v , j}}^{\std} : M_{\mcJ_{v , j}} \ra \RR$.

	Moreover, every vertex of the intrinsic one-dimensional arrangement maps under $\eta_{y_{0}} : ( \mft^{\crk\mcJ_{v , j}} )^{\ast} \overset{\sim}{\ra} H_{\mcJ_{v , j}} \subseteq ( \mft^{n} )^{\ast}$ to an ambient arrangement vertex, since it is obtained by intersecting $H_{\mcJ_{v , j}}$ with one further active hyperplane.
	Moreover, because $C_{\mcJ_{v , j}} = 0$, the intrinsic and ambient slack values agree at these vertices.
	The interval $\Delta_{\emptyset} \cap H_{\mcJ_{v , j}}$ corresponds to the distinguished bounded region of the intrinsic arrangement, so Proposition~\ref{prop:critical-values-slack} applies and gives $\max \CritVal( \Phi_{\mcJ_{v , j}}^{\std} ) \leq c_{\max} < \delta$.
	Thus, $\delta$ lies in the outer cut chamber of the four-dimensional hypertoric manifold $M_{\mcJ_{v , j}}$.
	Since $\Gamma_{\mcJ_{v , j}}$ is trivial, its effective and ambient-induced presentations agree, and Proposition~\ref{prop:cut-subspace-subpolyptych} identifies its intrinsic cut at level $\delta$ with the ambient-induced cut $M_{\mcJ_{v , j}}^{\leq \delta}$.

	Under $\eta_{y_{0}}$, the intervals $\Delta_{\emptyset} \cap H_{\mcJ_{v , j}}$ and $\Delta_{A} \cap H_{\mcJ_{v , j}}$ correspond to two adjacent bounded regions of the intrinsic one-dimensional arrangement.
	The core of $M_{\mcJ_{v , j}}$ is therefore reducible, so Lemma~\ref{lemma:one-dimensional-reducible-core-orbifold} 	applies at the outer-chamber level $\delta$.
	Let $p \in M_{\mcJ_{v , j}}^{\leq \delta}$ be the boundary point supplied by the lemma, and choose a representative $(q , 0) \in \rho^{-1}(\delta) \cap ( M_{\mcJ_{v , j}} \times \CC )$ of $p$ with $U_{1}$-stabiliser $\mu_{k} \subseteq U_{1}$, where $k \geq 2$.
	The embedding $M_{\mcJ_{v , j}}^{\leq \delta} \hookrightarrow M_{\nu}^{\leq \delta}$ of Proposition~\ref{prop:cut-subspace-subpolyptych} is induced by the $U_{1}$-equivariant inclusion of the corresponding cut level sets, so $(q , 0)$ has the same stabiliser when regarded as a point of the ambient level $\rho^{-1}(\delta) \subseteq M_{\nu} \times \CC$.
	Since $M_{\nu} \times \CC$ is smooth and the cutting action is locally free on this level, the ambient cut has the same local group $\Gamma_{p} = \mu_{k}$ at the image of $p$, which we continue to denote by $p$.

	To determine the ambient tangent representation, choose a local uniformising chart $(\widetilde{U} , \Gamma_{p})$ for $M_{\nu}^{\leq \delta}$ at $p$, compatible with the embedded cut.
	Let $\widetilde{V} \subseteq \widetilde{U}$ be the $\Gamma_{p}$-invariant complex submanifold representing $M_{\mcJ_{v , j}}^{\leq \delta}$, and fix a lift $\widetilde{p} \in \widetilde{V}$ of $p$.
	The calculation in the proof of Lemma~\ref{lemma:one-dimensional-reducible-core-orbifold} shows that the action on the two-dimensional complex tangent space $T_{\widetilde{p}} \widetilde{V}$ is $g \cdot (a , b) = (g a , g b)$ for $g \in \mu_{k}$.
	A $\Gamma_{p}$-invariant Hermitian metric gives an invariant splitting $T_{\widetilde{p}} \widetilde{U} = T_{\widetilde{p}} \widetilde{V} \oplus N_{\widetilde{p}}$, where $N_{\widetilde{p}}$ is the complex normal space to $\widetilde{V}$ in $\widetilde{U}$ at $\widetilde{p}$.
	For every $g \in \mu_{k} \setminus \{1\}$, the fixed subspace of $T_{\widetilde{p}} \widetilde{V}$ is zero, so the fixed subspace in $T_{\widetilde{p}} \widetilde{U}$ lies completely in $N_{\widetilde{p}}$ and has complex codimension at least two.
	The ambient tangent representation is faithful, since its restriction to $T_{\widetilde{p}} \widetilde{V}$ is faithful, and no non-identity element of $\Gamma_{p}$ is a pseudo-reflection.
	As $\Gamma_{p}$ is non-trivial, the Chevalley--Shephard--Todd theorem \cite{ShephardTodd1954, Chevalley1955} implies that the ambient coarse local quotient is singular at $p$, proving the theorem.
\end{proof}

Theorem~\ref{thm:reducible-core-orbifold} rules out smooth cut spaces when the core is reducible.
Combining this obstruction with the cotangent-bundle criterion of Bielawski and Dancer \cite[Theorem~7.2]{BielawskiDancer2000} gives the following characterisation of smooth cut spaces for hypertoric varieties.

\begin{corollary}
	\label{corollary:smooth-cut-space-characterisation}
	Let $M_{\nu}$ be a smooth hypertoric manifold of real dimension $4n$, and let $X_{\nu} = \mcE_{\emptyset} \subseteq M_{\nu}$ be its distinguished compact toric subvariety with moment polytope $\Delta_{\emptyset} \subseteq (\mft^{n})^{\ast}$.
	Let $U_{\nu} \subseteq M_{\nu}$ be the $T^{n}$-invariant open subset identified with $T^{\ast}X_{\nu}$ by \cite[Theorem~7.1]{BielawskiDancer2000}, so that $T^{\ast}X_{\nu} \cong U_{\nu} \hookrightarrow M_{\nu}$.
	Recall that $\rho : M_{\nu} \times \CC \ra \RR_{\geq 0}$ is the cutting moment map and that $\Cout = (c_{\max} , \infty)$ is the outer cut chamber.
	Then the following are equivalent:
	\begin{enumerate}[(i)]
		\item for some $\delta \in \Cout$, the cut $M_{\nu}^{\leq \delta}$ is a smooth manifold;
		\item for every $\delta > 0$, the level $\delta$ is regular for $\rho$ and the cut $M_{\nu}^{\leq \delta}$ is a smooth manifold;
		\item $U_{\nu} = M_{\nu}$.
	\end{enumerate}
	Whenever these conditions hold, $X_{\nu}$ is a product of projective spaces.
\end{corollary}

\begin{proof}
	We first prove (i) $\implies$ (iii).
	Suppose that $M_{\nu}^{\leq \delta}$ is smooth for some $\delta \in \Cout$.
	By Theorem~\ref{thm:reducible-core-orbifold}, the core $\mcC \subseteq M_{\nu}$ is irreducible, so $\Delta_{\emptyset}$ is the unique bounded region of $\mcA$.
	We shall show that every arrangement vertex lies in $\Delta_{\emptyset}$, so let $v \in (\mft^{n})^{\ast}$ be such an arrangement vertex.
	By \cite[Corollary~3.3]{HaradaProudfoot2004}, there is a fibre-rotation fixed point $q \in M_{\nu}$ with $\mu_{\RR}(q) = v$, where $\mu_{\RR} : M_{\nu} \ra (\mft^{n})^{\ast}$ is the residual real moment map.
	Every fibre-rotation fixed point lies in the core, which is the union of the unstable orbifolds of the fixed components \cite[Corollary~3.6]{HaradaProudfoot2004}, and so $q$ belongs to a compact core component whose bounded moment polytope contains $v$.
	Since $\Delta_{\emptyset}$ is the only bounded region, we have $v \in \Delta_{\emptyset}$.

	To apply the cotangent-bundle criterion, we also need to show that every hyperplane of $\mcA$ supports a facet of $\Delta_{\emptyset}$.
	Fix $i \in \{1,\ldots , N\}$, and write $u_{j}^{(i)} := u_{j} + \RR u_{i} \in \mft^{n} / \RR u_{i}$ for each $j \neq i$, so that $u_{j}^{(i)}$ represents the linear parts of the restricted affine functionals $\ell_{j}|_{H_{i}}$ on $H_{i} \subseteq (\mft^{n})^{\ast}$.
	Boundedness of $\Delta_{\emptyset}$ implies that the ambient normals $\{u_{j}\}$ positively span $\mft^{n}$, so their images $\{u_{j}^{(i)}\}_{j \neq i}$ positively span the $(n-1)$-dimensional quotient $\mft^{n} / \RR u_{i}$.
	Choose distinct indices $j_{1},\ldots , j_{n-1}$ for which the projected normals are linearly independent.
	The corresponding restricted affine functionals on $H_{i}$ have independent linear parts, so their common zero set is a single point: $\{v'\} = H_{i} \cap H_{j_{1}} \cap \cdots \cap H_{j_{n-1}} \subseteq (\mft^{n})^{\ast}$, and thus $v'$ is an ambient arrangement vertex and lies in $\Delta_{\emptyset}$ by the previous paragraph.

	Near $v'$, the only active inequalities defining $\Delta_{\emptyset}$ are $\ell_{j} \geq 0$ for $j \in \mcI_{v'}$; all remaining inequalities are strict.
	Simplicity makes the active normals linearly independent and so, since $i \in \mcI_{v'}$, imposing $\ell_{i} = 0$ therefore gives an $(n-1)$-dimensional face of $\Delta_{\emptyset}$.
	Hence, $H_{i}$ supports a facet, and $\mcA$ is the facet arrangement of $\Delta_{\emptyset}$.

	Smoothness of $M_{\nu}$ gives the smoothness condition on $\mcA$ in \cite[Theorem~3.2]{BielawskiDancer2000}, so $\Delta_{\emptyset}$ is the Delzant polytope of the smooth compact toric variety $X_{\nu}$.
	All arrangement vertices lie in $\Delta_{\emptyset}$, which is the equivalent form of condition~(ii) given in the proof of \cite[Theorem~7.2]{BielawskiDancer2000}.
	Theorems~7.1 and~7.2 of \cite{BielawskiDancer2000} therefore give $U_{\nu} = M_{\nu} \cong T^{\ast}X_{\nu}$ and show that $X_{\nu}$ is a product of projective spaces, which shows that (i) $\implies$ (iii) and the additional conclusion.

	Now let us show that (iii) $\implies$ (ii).
	The identification $M_{\nu} = U_{\nu} \cong T^{\ast}X_{\nu}$ from \cite[Theorem~7.1]{BielawskiDancer2000} identifies the coordinate $w$ with a cotangent covector and intertwines the fibre-rotation action with the scalar multiplication on the cotangent fibres.
	Writing $b \in X_{\nu}$ for a base point and $\alpha \in T_{b}^{\ast}X_{\nu}$ for a cotangent covector, the cutting-circle action on $T^{\ast}X_{\nu} \times \CC$ is $\tau \cdot (b,\alpha,\zeta) = (b,\tau\alpha,\tau\zeta)$ for $\tau \in U_{1}$.
	Fix $\delta > 0$, and use this identification to regard $\rho^{-1}(\delta)$ as a subset of $T^{\ast}X_{\nu} \times \CC$.
	The zero section corresponds to $X_{\nu} = \Phi^{-1}(0)$, so a point $(b,\alpha,\zeta) \in \rho^{-1}(\delta)$ cannot have both $\alpha = 0$ and $\zeta = 0$.
	If $\tau \in U_{1}$ stabilises this point, then $\tau\alpha = \alpha$ and $\tau\zeta = \zeta$; since at least one of $\alpha$ and $\zeta$ is non-zero, it follows that $\tau = 1$, hence the cutting-circle action is free on every positive level set of $\rho$.
	Whence, every $\delta > 0$ is a regular value and $M_{\nu}^{\leq \delta} = \rho^{-1}(\delta)/U_{1}$ is a smooth manifold, proving (iii) $\implies$ (ii).
	Finally, (ii) $\implies$ (i) follows by choosing any $\delta > c_{\max}$.
\end{proof}

%% file: sections/5-localisation-formulae-isotropy-data.tex
\section{Equivariant Localisation Formulae and Isotropy Data}
\label{section:section-5-localisation-formulae-isotropy-data}

Let us move on from the geometry of the compact cuts to the calculation of their equivariant indices.
The equivariant orbifold localisation formula reduces this global calculation to one involving local data at the fixed points: their tangent representations, finite orbifold structure groups, and the actions on the fibres of the pre-quantum line orbibundle.
We first fix the equivariant conventions and recall the localisation formula, before then showing that every local input for the hypertoric cut can be recovered from the moment polyptych and its associated lattice data.

\subsection{Equivariant Localisation Formulae}
\label{subsection:subsection-5.1-equivariant-localisation-formulae}

Let $M$ be a compact K\"ahler manifold with K\"ahler form $\w$, and let $\mcL \ra M$ be a holomorphic pre-quantum line bundle, normalised so that $F_{\nabla} = (2 \pi / \sqrt{-1})\,\w$ and $c_{1}(\mcL) = [\w]$.
Its Riemann--Roch number, or holomorphic Euler characteristic, is $\chi(M ; \mcL) = \sum_{q \geq 0}(-1)^{q} \dim_{\CC} H^{q}(M ; \mcL)$.
The Hirzebruch--Riemann--Roch formula expresses this number as $\chi(M ; \mcL) = \int_{M} \Td(TM) \wedge \Ch(\mcL)$ and, when the higher cohomology vanishes, the alternating sum reduces to $\chi(M ; \mcL) = \dim_{\CC}H^{0}(M ; \mcL)$.

Suppose that $\dim_{\CC}M = D$, and that a compact torus $T$ acts holomorphically on $M$ with a moment map $\mu : M \ra \mft^{\ast}$ and a lift to $\mcL$ preserving the pre-quantum connection.
Since $T$ is compact and connected, this action also preserves the Hermitian structure of $\mcL$.
We use the geometric linearisation fixed in Section~\ref{section:section-2-hypertoric-varieties} and the pullback action on sections given by \eqref{eqn:pullback-action-on-sections}, i.e., a geometric fibre character is inverted when regarded as a character in the induced representation on sections.
In the equivariant formulas below, we use a logarithmic variable $\xi \in \mft_{\CC}$: the element $e^{\xi} \in T_{\CC}$ is characterised by $\chi_{\alpha}(e^{\xi}) = e^{\langle \alpha ,\, \xi \rangle}$ for every $\alpha \in \mft_{\ZZ}^{\ast}$.
Thus, for $\eta \in \mft$, the element $[\eta] \in T$ of Section~\ref{section:section-2-hypertoric-varieties} is written as $e^{2\pi\sqrt{-1}\eta}$ in this notation.

For a local tangent line bundle $V_{\alpha}$ of geometric weight $\alpha$, the equivariant first Chern class is
\begin{equation}
	\label{eqn:equivariant-first-Chern-class-convention}
	c_{1}^{T}(V_{\alpha})(\xi)
	= -\langle \alpha ,\, \xi \rangle,
\end{equation}
so that the respective fixed-point denominator is $1 - e^{\langle \alpha , \xi \rangle}$.
For a pre-quantum lift normalised by the moment map $\mu$, the geometric action on the fibre over a fixed point $p$ has weight $-\mu(p)$.
We write $\beta_{p} = \mu(p)$ for the corresponding pullback fibre weight, which enters the equivariant character.

Each cohomology group $H^{q}(M ; \mcL)$ is then a finite-dimensional $T$-representation, so we can define the equivariant holomorphic Euler characteristic by
\begin{equation}
	\label{eqn:holomorphic-Euler-characteristic}
	\chi_{T}( M ; \mcL ) := \sum\nolimits_{q \geq 0}(-1)^{q} [ H^{q}( M ; \mcL) ] \in R(T).
\end{equation}
Suppose that $M^{T}$ is finite and write, at each fixed point $p$, $T_{p}M \cong \bigoplus_{j = 1}^{D}V_{\alpha_{p , j}}$ with $\alpha_{p , j} \neq 0$.
The equivariant index theorem expresses the character of $\chi_{T}( M ; \mcL)$, defined in \eqref{eqn:holomorphic-Euler-characteristic}, as the integral of the equivariant Riemann--Roch class $\Td_{T}(TM) \wedge \Ch_{T}(\mcL)$.
Atiyah--Bott--Berline--Vergne localisation evaluates this integral as a sum over the fixed points, dividing each contribution by the equivariant Euler class of $T_{p}M$.
With \eqref{eqn:equivariant-first-Chern-class-convention}, this division cancels the linear numerators in the Todd class, leaving one factor $(1 - e^{\langle \alpha_{p, j} , \xi \rangle})^{-1}$ for each tangent line bundle, and the equivariant Chern character of the pre-quantum fibre contributes $e^{\langle \mu(p) , \xi \rangle}$.
We therefore obtain the holomorphic fixed-point formula; compare \cite[Corollary~I.10; see also Theorem~I.7]{GuilleminGinzburgKarshon2002}:
\begin{equation}
	\label{eqn:equivariant-HRR}
	\chi_{T}(M ; \mcL)( e^{\xi} ) = \sum_{p \in M^{T}}
	\frac{
		e^{\langle \mu(p) ,\, \xi \rangle}
	}{
		\prod_{j = 1}^{D}
		\left( 1 - e^{\langle \alpha_{p , j} ,\, \xi \rangle} \right)
	},
\end{equation}
for $\xi \in \mft_{\CC}$ away from the denominator hyperplanes where $e^{\langle \alpha_{p , j} , \xi \rangle} = 1$.
Although the individual summands may have poles, their sum is the character of a finite-dimensional virtual representation and is regular at the identity.
Furthermore, if the higher cohomology vanishes, then this character is the character of $H^{0}(M ; \mcL)$.

Of course, the cut spaces considered below need not be smooth manifolds.
Indeed, Theorem~\ref{thm:reducible-core-orbifold} shows that, even when $M_{\nu}$ is smooth, a reducible core $\mcC$ forces every cut $M_{\nu}^{\leq \delta}$ with $\delta \in \Cout$ to have non-trivial cyclic isotropy at the cut boundary and a singular coarse local quotient there, so we will generally require the orbifold version of the previous formula.
At an orbifold fixed point $p$, the $T$-action on a local uniformising chart may lift only to a finite covering torus $\widetilde{T}_{p} \ra T$.
This makes the local tangent and fibre weights characters of $\widetilde{T}_{p}$; that is, they may be regarded as rational weights with respect to the global lattice $\mft_{\ZZ}^{\ast}$.

\begin{theorem}[Equivariant Orbifold Kawasaki-Riemann--Roch Formula]
	\label{thm:equivariant-KRR-index-theorem}
	Let $M$ be a compact K\"ahler orbifold of complex dimension $D$, equipped with a holomorphic action of a torus $T$, and let $\mcL \ra M$ be a $T$-equivariant holomorphic line orbibundle.
	Suppose that the fixed-point set $M^{T}$ is finite.

	For each $p \in M^{T}$, let $\Gamma_{p}$ be the local orbifold structure group.
	Assume that, on a local uniformising chart, the $T$-action lifts to a finite covering torus $\widetilde{T}_{p} \ra T$ whose action commutes with that of $\Gamma_{p}$, and that the tangent representation splits into common $\widetilde{T}_{p}$- and $\Gamma_{p}$-weight spaces $T_{p}M \cong \bigoplus_{j = 1}^{D} V_{\alpha_{p , j}}$ with $\alpha_{p , j} \neq 0$.
	For $g \in \Gamma_{p}$, let $\chi_{p , j}^{\tang}(g) \in U_{1}$ be the \emph{geometric} eigenvalue of $g$ on the tangent line bundle $V_{\alpha_{p , j}}$, let $\beta_{p}$ be the pullback torus weight of the local fibre, and let $\chi_{p}^{\fib}(g) \in U_{1}$ be its local \emph{pullback} character; i.e., $\chi_{p}^{\fib}(g)$ is the inverse of the geometric eigenvalue of $g$ on the local fibre $\mcL_{p}$.
	We say that $\xi \in \mft_{\CC}$ is \textbf{generic} if
	\begin{equation}
		\label{eqn:orbifold-generic-xi}
		\chi_{p, j}^{\tang}(g) e^{\langle \alpha_{p , j} ,\, \xi \rangle} \neq 1
	\end{equation}
	for every $p \in M^{T}$, $g \in \Gamma_{p}$, and $j = 1 , \ldots , D$.

	Then, for all generic $\xi \in \mft_{\CC}$, the \emph{Kawasaki localisation formula} gives
	\begin{equation}
		\label{eqn:equivariant-KRR-character-formula}
		\chi_{T}(M ; \mcL)( e^{\xi} ) = \sum_{p \in M^{T}}
		\frac{1}{|\Gamma_{p}|}
		\sum_{g \in \Gamma_{p}}
		\frac{
			\chi_{p}^{\fib}(g) e^{\langle \beta_{p} ,\, \xi \rangle}
		}{
			\prod\nolimits_{j = 1}^{D}
			[ 1 - \chi_{p, j}^{\tang}(g) e^{\langle \alpha_{p , j} ,\, \xi \rangle} ]
		}.
	\end{equation}
\end{theorem}

On the compact cut spaces, algebraic and holomorphic orbifold sheaf cohomology agree, so we use $H^{q}(M ; \mcL)$ for both of these groups.
Theorem~\ref{thm:equivariant-KRR-index-theorem} is the isolated-fixed-point specialisation of the equivariant orbifold index theorem of Vergne \cite{Vergne1996}; see also \cite[Theorems~3.4--3.5]{Meinrenken1998}, together with Kawasaki's orbifold Riemann--Roch theorem \cite{Kawasaki1979}.
The finite-group average over the local structure group is written explicitly in the isolated-fixed-point case in \cite[\S10.1 and Remark~10.4]{CannasDaSilva1996}; compare also \cite[Corollary~I.10 and (I.13)]{GuilleminGinzburgKarshon2002}.
The cited formulae are initially stated either near the identity or under a stronger genericity hypothesis; both sides are meromorphic functions of $\xi$, so the equality extends by meromorphic continuation to every $\xi$ satisfying \eqref{eqn:orbifold-generic-xi}.
We call the equivariant index character $\chi_{T}(M ; \mcL)(e^{\xi})$ computed by \eqref{eqn:equivariant-KRR-character-formula} the \emph{Kawasaki character}.

For the cut space of a hypertoric variety, $\Gamma_{p}$ is the kernel of the local covering homomorphism $\widetilde{T}_{p} \ra T^{n}$, and its action on the chart is the restriction of the covering-torus action.
The decomposition into one-dimensional $\widetilde{T}_{p}$-weight spaces is therefore also a decomposition into common $\widetilde{T}_{p}$- and $\Gamma_{p}$-eigenspaces, as required in Theorem~\ref{thm:equivariant-KRR-index-theorem}.
In the local calculations below, we retain the local orbifold groups inherited from $M_{\nu}$ when restricting to a hypertoric subvariety $M_{\mcF} \subseteq M_{\nu}$, as described in Subsection~\ref{subsection:subsection-3.4-hypertoric-subvarieties}.
Some elements of these groups may act trivially on the restricted charts but non-trivially in directions normal to the subvariety.
Passing to the effective quotient $M_{\mcF}^{\eff}$ removes the elements acting trivially on the restricted charts.
Characters calculated on this effective quotient pull back to the inherited local groups along the homomorphisms in Lemma~\ref{lemma:flat-local-group-comparison}.
When applying the Kawasaki localisation formula to $M_{\nu}^{\leq \delta}$, however, we use the full local group $\Gamma_{p}$ at each fixed point and its representations on the full tangent space and the pre-quantum fibre of the cut.


\subsection{Isotropy Data of the Cut Spaces}
\label{subsection:subsection-5.2-isotropy-data-cut-spaces}

We specialise Theorem~\ref{thm:equivariant-KRR-index-theorem} to the cut spaces $M_{\nu}^{\leq \delta}$.
Although the full cut space need not be toric for the residual $T^{n}$-action, it is a compact K\"ahler orbifold.
The fixed points, the local tangent weights, and the finite local groups are nevertheless determined by the moment polyptych, using only the toric geometry of its individual cut components.

\begin{lemma}
	\label{lemma:cut-space-fixed-points}
	Let $M_{\nu}^{\leq \delta}$ be the cut space at a regular value $\delta > 0$.
	Then, every residual $T^{n}$-fixed point of $M_{\nu}^{\leq \delta}$ lies in the cut extended core $\mcE_{\nu}^{\leq \delta} = \bigcup_{A} \mcE_{A}^{\leq \delta}$.
	Moreover, the residual moment map restricts to a bijection $\mu_{\RR} : ( M_{\nu}^{\leq \delta} )^{T^{n}} \overset{\sim}{\ra} \Verts ( \Delta_{\nu}^{\leq \delta} )$ and, under this bijection:
	\begin{enumerate}[(i)]
		\item fixed points in the cut interior $M_{\nu}^{<\delta}$ correspond to the interior vertices of $\Delta_{\nu}^{\leq \delta}$;
		      \label{lemma:cut-space-fixed-points-1}
		\item fixed points in the cut boundary $Z_{\nu}^{\delta}$ correspond to the boundary vertices of $\Delta_{\nu}^{\leq \delta}$, that is, to the vertices of $\Pi_{\nu}^{\delta}$.
		      \label{lemma:cut-space-fixed-points-2}
	\end{enumerate}
\end{lemma}

\begin{proof}
	For a point $q = [ z , w ] \in M_{\nu}$, set $I(q) := \{ i ~|~ z_{i} = w_{i} = 0 \}$.
	The Lie algebra of the residual stabiliser of $q$ is $\mft_{q} = \Span_{\RR}\{ u_{i} ~|~ i \in I(q) \}$ and, moreover, $\mu_{\RR}(q) \in H_{i}$ for every $i \in I(q)$.

	First, for \ref{lemma:cut-space-fixed-points-1}, suppose that $p \in M_{\nu}^{< \delta}$ is a residual fixed point, and let $q \in \{ \Phi < \delta \} \subset M_{\nu}$ be the corresponding point under the $T^{n}$-equivariant identification of the cut interior.
	Then $q$ is also fixed, so $\mft_{q} = \mft^{n}$.
	Furthermore, simplicity gives $|I(q)| \leq n$, while the stabiliser formula shows that the normals indexed by $I(q)$ span $\mft^{n}$.
	Hence, $|I(q)| = \rk( I(q) ) = n$.

	The dual exact sequence \eqref{eqn:dual-lie-algebra-seq} shows that the weights $\{ \alpha_{j} \}_{j \not\in I(q)}$ form a basis of $\mfk^{\ast}$, and it follows that the complex moment-map equation $\sum_{j \not\in I(q)} (z_{j}w_{j}) \alpha_{j} = 0$ gives $z_{j}w_{j} = 0$ for every $j \not\in I(q)$.
	The same equality also holds for $j \in I(q)$ by definition, so $q \in \mu_{\CC}^{-1}(0) = \mcE_{\nu}$.

	On an extended-core component containing $q$, the toric orbit--face correspondence identifies $v = \mu_{\RR}(q)$ as an arrangement vertex.
	Moreover, since $\mcS(v) = \Phi(q) < \delta$, it is an interior vertex of the moment polyptych.
	Conversely, every interior vertex belongs to a toric cut component, where the same orbit--face correspondence supplies a residual fixed point above it.
	This establishes the correspondence in \ref{lemma:cut-space-fixed-points-1}; injectivity will be verified below.

	For \ref{lemma:cut-space-fixed-points-2}, let $p \in Z_{\nu}^{\delta}$ be a residual $T^{n}$-fixed point, represented by $q = [ z , w ] \in \Phi^{-1}(\delta) \subseteq M_{\nu}$.
	Since $p$ is fixed after quotienting by fibre rotation, $T^{n} \cdot q \subseteq U_{1} \cdot q$, and hence $\dim \mft_{q} \geq n - 1$.
	If $\dim \mft_{q} = n$, the argument for residual fixed points above gives $q \in \mcE_{\nu}$ and an arrangement vertex $v = \mu_{\RR}(q)$ with $\mcS(v) = \Phi(q) = \delta$.
	This contradicts the regularity of $\delta$ by Proposition~\ref{prop:critical-values-slack}, hence $\dim \mft_{q} = n - 1$.

	Choose an infinitesimal generator $\xi$ of $\mft^{n} / \mft_{q}$, normalised so that its action at $q$ agrees with the fibre-rotation generator, and lift it to $\hat{\xi} \in \mft^{N}$.
	Their actions on the representative $(z , w) \in \phi_{\HK}^{-1}(\nu , 0)$ differ by an infinitesimal $K$-action, so there exists a $\kappa \in \mfk$ such that $\hat{\xi} + \kappa$ acts at $(z , w)$ as the fibre-rotation generator.
	Write $\hat{\xi} + \kappa = \sum_{i = 1}^{N} c_{i} e_{i}$, where $c_{i} \in \RR$.
	Then, on the $i$-th coordinate pair, the infinitesimal coordinate-torus action has velocity $2\pi \sqrt{-1}( c_{i}z_{i} , -c_{i}w_{i} )$, whereas fibre-rotation action has velocity $2\pi\sqrt{-1} ( 0 , w_{i} )$.
	Their equality at $( z , w )$ therefore implies $c_{i}z_{i} = 0$ and $-c_{i}w_{i} = w_{i}$ for each $i = 1, \ldots, N$.
	If both $z_{i}$ and $w_{i}$ were non-zero, these equations would force $c_{i} = 0$ and $c_{i} = -1$, a contradiction.
	Hence, $z_{i}w_{i} = 0$ for every $i$, and so $q \in \mcE_{\nu}$.

	The point $p$ therefore lies on a toric cut component $\mcE_{A}^{\leq \delta}$, where the toric orbit--face correspondence identifies its moment image with a vertex of $\Delta_{A}^{\leq \delta}$.
	Since $p$ lies on the cut divisor $Z_{\nu}^{\delta}$, this vertex belongs to the cut face $\Pi_{A}^{\delta}$ and is therefore a boundary vertex.
	Conversely, every boundary vertex is attained on a toric cut component by Proposition~\ref{prop:cut-component-image}, and the same correspondence supplies a residual fixed point on its component divisor, establishing the correspondence in \ref{lemma:cut-space-fixed-points-2}.

	To verify injectivity independently of the choice of cut component, fix a vertex $v$ and observe that the equations $|z_{i}|^{2} - |w_{i}|^{2} = 2\ell_{i}(v)$ and $z_{i}w_{i} = 0$ determine $|z_{i}|$ and $|w_{i}|$ for every $i$.
	The remaining freedom consists of phases forming a single $T^{N}$-orbit, which descends to a single residual $T^{n}$-orbit in the cut after quotienting by $K$ and, at a boundary vertex, by the cutting circle.
	By the vertex argument from before, this residual $T^{n}$-orbit is zero-dimensional and hence consists of a single point, proving injectivity and completing the proof.
\end{proof}

Recall from Subsection~\ref{subsection:subsection-4.3-moment-polyptychs} that an edge of the moment polyptych $\Delta_{\nu}^{\leq \delta}$ is a one-dimensional cell of it as a polytopal complex.
We say that such an edge \emph{emanates from} a vertex $v$ if it contains $v$.

\begin{theorem}
	\label{thm:vertex-edges}
	Let $p \in M_{\nu}^{\leq \delta}$ be a residual $T^{n}$-fixed point, and let $v = \mu_{\RR}(p)$.
	Then:
	\begin{enumerate}[(i)]
		\item if $v$ is an interior vertex, then $|\mcI_{v}| = \rk(\mcI_{v}) = n$.
		      Furthermore, for every $j \in \mcI_{v}$, there are exactly two edges emanating from $v$ which are not contained in $H_{j}$, one on either side, and the resulting $2n$ edges are pairwise distinct and exhaust all edges emanating from $v$;
		      \label{theorem:vertex-edges-1}
		\item if $v$ is a boundary vertex, then $|\mcI_{v}| = \rk(\mcI_{v}) = n - 1$.
		      Furthermore, for every $j \in \mcI_{v}$, there are exactly two edges emanating from $v$ which are not contained in $H_{j}$, one on either side, and the resulting $2n - 2$ edges are pairwise distinct.
		      Moreover, there is exactly one additional distinct edge that is tangent to the affine line $H_{\mcI_{v}} = \cap_{j \in \mcI_{v}} H_{j}$, and the resulting $2n - 1$ edges exhaust all edges emanating from $v$.
		      \label{theorem:vertex-edges-2}
	\end{enumerate}
\end{theorem}

\begin{proof}
	Since $\delta > 0$ is regular, Proposition~\ref{prop:critical-values-slack} and \eqref{eqn:slack-boundary} imply that no cut face $\Pi_{A}^{\delta} = \Delta_{A} \cap \{\mcS = \delta\}$ contains an arrangement vertex.
	Every edge emanating from $v$ belongs to a polyptych component $\Delta_{A}^{\leq \delta} \subseteq (\mft^{n})^{\ast}$, which near $v$ has a simplicial tangent cone defined by $n$ active inequalities with linearly independent normals.
	The active arrangement inequalities are $\pm\ell_{i}(x) \geq 0$ for $i \in \mcI_{v}$, with the signs determined by $A$; at a boundary vertex, the cutting inequality $\langle x ,\, u_{A} \rangle \leq \delta + \lambda_{A}$ is active as well.
	Each edge emanating from $v$ within this component is obtained by keeping $n-1$ of these inequalities as equalities and allowing the remaining inequality to become strict.
	We use this description in the two cases below.

	For \ref{theorem:vertex-edges-1}, suppose that $v$ is an interior vertex, so $|\mcI_{v}| = \rk(\mcI_{v}) = n$ and the normals $\{u_{i} \mid i \in \mcI_{v}\}$ are linearly independent.
	For $j \in \mcI_{v}$, keeping $\ell_{k} = 0$ for every $k \in \mcJ_{v , j} := \mcI_{v} \setminus \{j\}$ within a component $\Delta_{A}^{\leq \delta}$ containing $v$ determines an edge along the affine line $H_{\mcJ_{v , j}}$, on the side of $H_{j}$ chosen for the region $\Delta_{A}$.
	The independence of these normals ensures that regions containing $v$ occur on both sides of $H_{j}$, giving one edge in each direction along this line.
	These two edges are contained in every $H_{k}$ with $k \in \mcJ_{v , j}$ but not in $H_{j}$, so an edge associated with one index cannot also be associated with another.
	As every edge arises in this way, the resulting $2n$ edges are all distinct and exhaust the edges emanating from $v$, proving \ref{theorem:vertex-edges-1}.

	For \ref{theorem:vertex-edges-2}, let $v$ be a boundary vertex and choose a component $\Delta_{A}^{\leq \delta}$ with $v \in \Pi_{A}^{\delta}$.
	Since $v$ is not an arrangement vertex, the $n$ active equalities at $v$ consist of $n-1$ arrangement equalities together with the cutting equality: $\ell_{k}(x) = 0$ for each $k \in \mcI_{v}$ and $\langle x ,\, u_{A} \rangle = \delta + \lambda_{A}$.
	Thus, $|\mcI_{v}| = \rk(\mcI_{v}) = n - 1$, and the normals $\{u_{k} ~|~ k \in \mcI_{v}\} \cup \{u_{A}\}$ form an $\RR$-basis of $\mft^{n}$.
	In particular, $H_{\mcI_{v}} \subseteq (\mft^{n})^{\ast}$ is an affine line, and
	\[
		H_{\mcI_{v}} \cap \left\{ x \in (\mft^{n})^{\ast} ~|~ \langle x ,\, u_{A} \rangle = \delta + \lambda_{A} \right\} = \{v\},
	\]
	intersecting transversely.
	We first identify the edge along this line, and then the remaining edges, which lie in the cut faces.

	Keeping $\ell_{k}(x) = 0$ for each $k \in \mcI_{v}$ gives the edge $\Delta_{A}^{\leq \delta} \cap H_{\mcI_{v}}$, and its points other than $v$ satisfy $\langle x ,\, u_{A} \rangle < \delta + \lambda_{A}$ (or equivalently, $\mcS(x) < \delta$), so the edge points from $v$ into the retained side.
	To show that this edge is independent of the component, suppose that $\Delta_{B}^{\leq \delta}$ also contains $v$.
	Then $A \triangle B \subseteq \mcI_{v}$, and \eqref{eqn:U1-action-generator} gives
	\[
		\begin{aligned}
			( \langle x ,\, u_{A} \rangle - \lambda_{A} ) - ( \langle x ,\, u_{B} \rangle - \lambda_{B} )
			 & = -\sum\nolimits_{k \in A} \ell_{k}(x) + \sum\nolimits_{k \in B} \ell_{k}(x)                          \\
			 & = -\sum\nolimits_{k \in A \setminus B} \ell_{k}(x) + \sum\nolimits_{k \in B \setminus A} \ell_{k}(x),
		\end{aligned}
	\]
	using the disjoint decompositions $A=(A\cap B)\sqcup(A\setminus B)$, $B=(A\cap B)\sqcup(B\setminus A)$, and since the terms indexed by $k \in A \cap B$ cancel out.
	For $x \in H_{\mcI_{v}}$, every term on the right-hand side vanishes because $( A \setminus B ) \sqcup ( B \setminus A ) = A \triangle B \subseteq \mcI_{v}$, so the two cutting inequalities agree along $H_{\mcI_{v}}$.
	Moreover, the defining inequalities of $\Delta_{A}$ and $\Delta_{B}$ differ only for indices in $A \triangle B$, since $\ell_{k} = 0$ on $H_{\mcI_{v}}$; hence $\Delta_{A} \cap H_{\mcI_{v}} = \Delta_{B} \cap H_{\mcI_{v}}$.
	It follows that $\Delta_{A}^{\leq \delta} \cap H_{\mcI_{v}} = \Delta_{B}^{\leq \delta} \cap H_{\mcI_{v}}$, so every component containing $v$ yields the same edge along $H_{\mcI_{v}}$, and hence, there is exactly one such edge emanating from $v$; when $n = 1$, this is the only edge, so \ref{theorem:vertex-edges-2} holds in the one-dimensional case.

	To identify the remaining edges, first suppose that $n = 2$.
	Then $\mcI_{v} = \{j\}$, and the only arrangement regions containing $v$ are $\Delta_{A}$ and the adjacent region $\Delta_{B}$ on the opposite side of $H_{j}$, where $B = A \triangle \{j\}$.
	Since $v \in \Delta_{A} \cap \Delta_{B}$, Lemma~\ref{lemma:region-common-face} gives $\langle v ,\, u_{B} \rangle - \lambda_{B} = \langle v ,\, u_{A} \rangle - \lambda_{A} = \delta$, so $v \in \Pi_{A}^{\delta} \cap \Pi_{B}^{\delta}$.
	In each of the two components, retaining the cutting equality and allowing $\ell_{j}$ to become non-zero presents its cut face as an edge emanating from $v$, i.e., the two edges are $\Pi_{A}^{\delta}$ and $\Pi_{B}^{\delta}$, with $\ell_{j} > 0$ on one and $\ell_{j} < 0$ on the other away from $v$.

	For $n \geq 3$, we reduce to this planar case (i.e., when $n = 2$) by fixing $j \in \mcI_{v}$ and restricting to $H_{\mcJ_{v , j}}$, where $\mcJ_{v , j} = \mcI_{v} \setminus \{j\}$.
	The equalities $\ell_{k}(x) = 0$ for $k \in \mcJ_{v , j}$ define $H_{\mcJ_{v , j}}$ as an affine plane in $(\mft^{n})^{\ast}$, and the only line of the restricted arrangement $\mcA_{\mcJ_{v, j}}$ through $v$ is $H_{j} \cap H_{\mcJ_{v , j}} = H_{\mcI_{v}}$.
	The corresponding cut subvariety $M_{\mcJ_{v , j}}^{\leq \delta} \subseteq M_{\nu}^{\leq \delta}$ has real dimension eight, and Proposition~\ref{prop:cut-subspace-subpolyptych} identifies its moment polyptych with the restricted subcomplex $\Delta_{\nu}^{\leq \delta}|_{H_{\mcJ_{v , j}}}$.
	Figure~\ref{fig:boundary-vertex-restrictions} illustrates this planar restriction and the further restriction to $H_{\mcI_{v}}$ for $n = 3$.

	\begin{figure}[htbp]
		\centering
		\begingroup
		\def\mval{2}
		\def\dval{1}
		\pgfmathsetmacro{\outerval}{\mval + \dval}

		\def\figscale{0.85}
		\pgfmathsetmacro{\labelscale}{1/\figscale}
		\def\projx{0.75}
		\def\projy{0.433}
		\def\projz{1.00}
		\def\panelgap{1.35}
		\def\labelpos{0.5}
		\def\labelgap{0.30}
		\def\vertexradius{1.10pt}
		\def\edgewidth{0.45pt}
		\def\highlightwidth{1.20pt}

		\pgfmathsetmacro{\planeshift}{(\mval + 3*\dval)*\projx + \panelgap}
		\pgfmathsetmacro{\lineshift}{\planeshift + \outerval*\projx + 2*\panelgap}
		\pgfmathsetmacro{\labelheight}{(\outerval - \labelpos*\dval)*\projz}
		\pgfmathsetmacro{\ambientlabelx}{\dval*\projx +
			(\outerval*\projz + \dval*\projy - \labelheight)*\projx/(\projz + \projy)
			+ \labelgap}
		\pgfmathsetmacro{\planelabelx}{
			(\outerval*\projz - \labelheight)*\projx/(\projz + \projy)
			+ \labelgap}

		\tikzset{
			polyprojection/.style={
					x={(-\projx cm , -\projy cm)},
					y={(\projx cm , -\projy cm)},
					z={(0cm,\projz cm)}
				},
			polyedge/.style={draw=black,line width=\edgewidth},
			polyhidden/.style={polyedge,dashed},
			polyplane/.style={polyedge,fill=gray,opacity=0.25},
			polycore/.style={polyedge,fill=gray,opacity=0.50},
			polycutedge/.style={draw=black,line width=\highlightwidth},
			polyinward/.style={draw=black,line width=\highlightwidth},
			polyinteriorvertex/.style={polyedge,fill=gray},
			polyboundaryvertex/.style={polyedge,fill=black},
			polylabel/.style={inner sep=2pt,transform shape,scale=\labelscale},
		}

		\newcommand{\polyptychcoordinates}{%
			\coordinate[overlay] (o)    at (0 , 0,0);
			\coordinate[overlay] (aX)   at (\mval,0,0);
			\coordinate[overlay] (aY)   at (0,\mval,0);
			\coordinate[overlay] (y)    at (0 , 0,\mval);
			\coordinate[overlay] (bX)   at (-\dval,0,0);
			\coordinate[overlay] (bY)   at (0 , -\dval,0);
			\coordinate[overlay] (bZ)   at (0 , 0 , -\dval);
			\coordinate[overlay] (pX)   at (\outerval,0,0);
			\coordinate[overlay] (pY)   at (0,\outerval,0);
			\coordinate[overlay] (v)    at (0 , 0,\outerval);
			\coordinate[overlay] (pXnY) at (\outerval , -\dval,0);
			\coordinate[overlay] (pXnZ) at (\outerval,0 , -\dval);
			\coordinate[overlay] (pYnX) at (-\dval,\outerval,0);
			\coordinate[overlay] (pYnZ) at (0,\outerval , -\dval);
			\coordinate[overlay] (pZnX) at (-\dval,0,\outerval);
			\coordinate[overlay] (pZnY) at (0 , -\dval,\outerval);
		}

		\newcommand{\polyptychplanesection}{%
			\draw[polyplane]
			(bY) -- (bZ) -- (pYnZ) -- (pY) -- (v) -- (pZnY) -- cycle;
			\draw[polycore] (o) -- (aY) -- (y) -- cycle;
		}

		\scalebox{\figscale}{%
			\begin{tikzpicture}[polyprojection]
				\begin{scope}[name prefix=ambient-]
					\polyptychcoordinates
					\polyptychplanesection

					\draw[polyhidden] (bX) -- (bY) -- (pZnY) -- (pZnX) -- cycle;
					\draw[polyhidden] (bZ) -- (bY) -- (pXnY) -- (aX);
					\draw[polyhidden] (bZ) -- (bX) -- (pYnX) -- (aY);
					\draw[polyhidden] (pX) -- (pXnZ) -- (bZ) -- (pYnZ) -- (pY);
					\draw[polyhidden] (bX) -- (pX);
					\draw[polyhidden] (bY) -- (pY);
					\draw[polyhidden] (bZ) -- (v);
					\draw[polyhidden] (pXnZ) -- (pZnX);
					\draw[polyhidden] (pYnZ) -- (pZnY);
					\draw[polyhidden] (pXnY) -- (pYnX);

					\draw[polyedge] (v) -- (pZnX) -- (pZnY) -- cycle;
					\draw[polyedge] (v) -- (pX) -- (pXnY) -- (pZnY) -- cycle;
					\draw[polyedge] (v) -- (pY) -- (pYnX) -- (pZnX) -- cycle;
					\draw[polyedge]
					(pXnY) -- (pXnZ) -- (pYnZ) -- (pYnX) -- (pY) -- (pX) -- cycle;
					\draw[polyedge] (o) -- (aY) -- (y) -- cycle;

					\draw[polyinward] (v) -- (y);
					\foreach \p in {o,aX,aY,y}
					\draw[polyinteriorvertex] (\p) circle (\vertexradius);
					\foreach \p in {bX,bY,bZ,pX,pY,v,pXnY,pXnZ,pYnX,pYnZ,pZnX,pZnY}
					\draw[polyboundaryvertex] (\p) circle (\vertexradius);
					\node[polylabel,above=2.pt] at (v) {$v$};
					\node[polylabel,anchor=west] at (\ambientlabelx cm,\labelheight cm)
					{$\Delta_{\nu}^{\leq \delta}$};
				\end{scope}

				\begin{scope}[xshift=\planeshift cm,name prefix=plane-]
					\polyptychcoordinates
					\polyptychplanesection

					\draw[polyedge] (bZ) -- (v);
					\draw[polyedge] (pZnY) -- (pYnZ);
					\draw[polyedge] (bY) -- (pY);

					\draw[polycutedge] (pZnY) -- (v) -- (pY);
					\draw[polyinward] (v) -- (y);

					\foreach \p in {o,aY,y}
					\draw[polyinteriorvertex] (\p) circle (\vertexradius);

					\foreach \p in {bY,bZ,pY,pYnZ,v,pZnY}
					\draw[polyboundaryvertex] (\p) circle (\vertexradius);
					\node[polylabel,above=2.5pt] at (v) {$v$};
					\node[polylabel,anchor=west] at (\planelabelx cm,\labelheight cm)
					{$\Delta_{\nu}^{\leq \delta}|_{H_{\mcJ_{v , j}}}$};
				\end{scope}

				\begin{scope}[xshift=\lineshift cm,name prefix=line-]
					\polyptychcoordinates
					\draw[polyedge] (bZ) -- (v);
					\draw[polyinward] (v) -- (y);

					\foreach \p in {o,y}
					\draw[polyinteriorvertex] (\p) circle (\vertexradius);
					\foreach \p in {bZ,v}
					\draw[polyboundaryvertex] (\p) circle (\vertexradius);
					\node[polylabel,above=2.5pt] at (v) {$v$};
					\node[polylabel,anchor=west] at (\labelgap cm,\labelheight cm)
					{$\Delta_{\nu}^{\leq \delta}|_{H_{\mcI_{v}}}$};
				\end{scope}
			\end{tikzpicture}%
		}
		\endgroup
		\caption{
			The boundary-vertex restriction argument for $n = 3$.
			From left to right, the ambient moment polyptych $\Delta_{\nu}^{\leq \delta}$ is restricted first to the plane $H_{\mcJ_{v , j}}$ and then to the line $H_{\mcI_{v}} = H_{j} \cap H_{\mcJ_{v , j}}$.
		}
		\label{fig:boundary-vertex-restrictions}
	\end{figure}
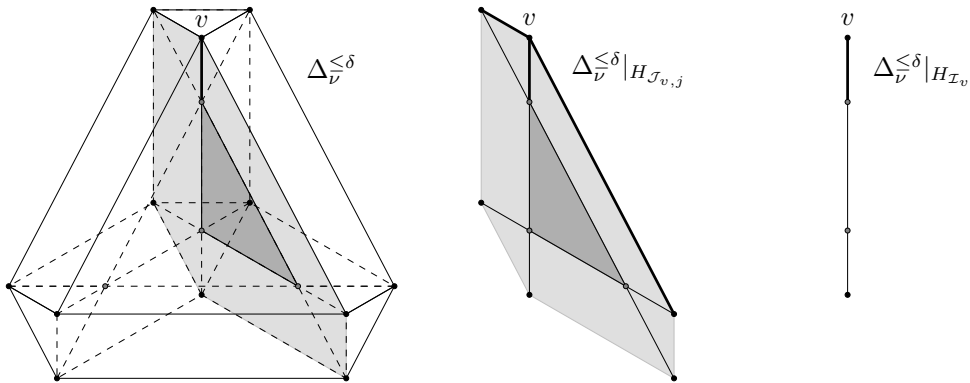

	Within a component $\Delta_{A}^{\leq \delta}$ containing $v$, retaining the equalities $\ell_{k}(x) = 0$ for each $k \in \mcJ_{v , j}$ and $\langle x ,\, u_{A} \rangle = \delta + \lambda_{A}$ gives the cut face $\Pi_{A}^{\delta} \cap H_{\mcJ_{v , j}}$ of the restricted component $\Delta_{A}^{\leq \delta} \cap H_{\mcJ_{v , j}}$.
	The planar argument therefore gives two such edges, one on each side of $H_{j} \cap H_{\mcJ_{v , j}}$; since the restricted complex $\Delta_{\nu}^{\leq \delta}|_{H_{\mcJ_{v , j}}}$ is a subcomplex of $\Delta_{\nu}^{\leq \delta}$, these are also edges of the ambient moment polyptych $\Delta_{\nu}^{\leq \delta}$.

	For each $j \in \mcI_{v}$, the two edges just determined are contained in $H_{k}$ for every $k \in \mcJ_{v , j}$, but neither is contained in $H_{j}$.
	If either edge also belonged to the pair associated with another index $j^{\prime} \neq j$, it would also have to be contained in $H_{j}$, a contradiction.
	Hence, no edge occurs in two different pairs of hyperplanes.
	These $2(n - 1)$ edges lie in $\{ \mcS = \delta \}$, whereas the previously identified edge along $H_{\mcI_{v}}$ lies in $\{ \mcS < \delta \}$ away from $v$, and therefore we get $2(n - 1) + 1 = 2n - 1$ distinct edges.
	Within each component containing $v$, an edge is obtained by imposing $n - 1$ of the $n$ active equalities.
	Omitting the cutting equality gives the edge along $H_{\mcI_{v}}$; omitting one arrangement equality gives an edge in one of the pairs already constructed.
	These $2n - 1$ edges therefore exhaust every edge emanating from $v$, which proves \ref{theorem:vertex-edges-2}.
\end{proof}

\begin{notation}[Edges Emanating from a Vertex]
	\label{notation:vertex-edges}
	Let $v$ be a vertex of $\Delta_{\nu}^{\leq \delta}$, and let $p$ be the residual fixed point with $\mu_{\RR}(p) = v$.
	For $j \in \mcI_{v}$, the two edges of Theorem~\ref{thm:vertex-edges} not contained in $H_{j}$ are called the \textbf{non-cutting edges} at $v$; we write $\varepsilon_{v , j}$ for the one on the side on which $\ell_{j}$ increases, and $\varphi_{v , j}$ for the one on which it decreases.
	At a boundary vertex, the edge tangent to $H_{\mcI_{v}}$ is called the \textbf{cutting edge} $\varkappa_{v}$; it is the edge dual to the cut facet, and it leaves the cut boundary and runs into the interior.
	Edges are indexed by the vertex $v$ from which they emanate, whereas the local tangent weights defined below are vectors in $( \mft_{\QQ}^{n} )^{\ast}$ indexed by the fixed point $p$; the edge--weight pairs are therefore $( \varepsilon_{v , j} , \varrho_{p , j} )$, $( \varphi_{v , j} , \varsigma_{p , j} )$, and $( \varkappa_{v} , \vartheta_{p} )$.
\end{notation}

We now normalise the local tangent weights.
Since each edge tangent space is one-dimensional, there exist unique vectors $\varrho_{p , j} \in T_{v}\varepsilon_{v , j}$ and $\varsigma_{p , j} \in T_{v}\varphi_{v , j}$ such that
\begin{equation}
	\label{eqn:noncutting-weight-normalisation}
	\langle \varrho_{p , j} ,\, u_{j} \rangle = 1, \qquad
	\langle \varsigma_{p , j} ,\, u_{j} \rangle = -1.
\end{equation}
These are local weights in $( \mft_{\QQ}^{n} )^{\ast}$; they need not be global $T^{n}$-characters. The interior edge--weight configuration is illustrated in Figure~\ref{fig:weights-interior-vertex}.

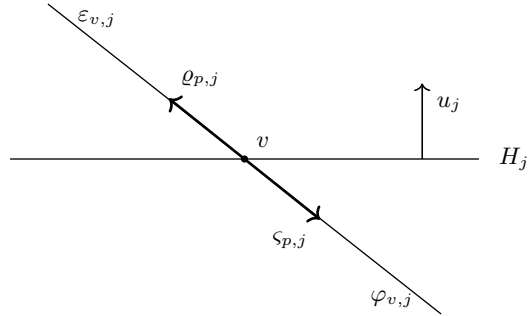
\begin{figure}[H]
	\centering
	\begin{tikzpicture}[
			hyper/.style={line width=0.5pt},
			tangent/.style={->, line width=1pt},
			normal/.style={->, line width=0.6pt},
			every node/.style={font=\small}
		]

		\coordinate (v) at (0 , 0);

		\draw[hyper] (-3.1 , 0) -- (3.1 , 0);
		\node[right=4pt] at (3.1 , 0) {$H_j$};

		\draw[hyper] (-2.6 , 2.05) -- (2.6 , -2.05);

		\node[above=2pt]  at (-1.95,1.55) {$\varepsilon_{v , j}$};
		\node[below=2pt] at (1.95 , -1.55) {$\varphi_{v , j}$};

		\draw[normal] (2.35,0) -- ++(0 , 1.0);
		\node[right=2pt] at (2.35,0.72) {$u_j$};

		\draw[tangent] (v) -- ++(-1.00,0.80);
		\draw[tangent] (v) -- ++(1.00 , -0.80);

		\node[above right=1pt] at (-1.00,0.80) {$\varrho_{p , j}$};
		\node[below left=1pt]  at (1.00 , -0.80) {$\varsigma_{p , j}$};

		\fill (v) circle (1.35pt);
		\node[above right=1pt] at (v) {$v$};

	\end{tikzpicture}
	\caption{
		An edge pair $\{ \varepsilon_{v , j} , \varphi_{v , j} \}$ meeting at an interior vertex $v$.
		The coorientation of the hyperplane $H_{j}$, determined by the normal $u_{j}$, distinguishes the corresponding local tangent directions $\varrho_{p , j}$ and $\varsigma_{p , j}$.
		}
	\label{fig:weights-interior-vertex}
\end{figure}

At a boundary vertex, choose a cut face $\Pi_{A}^{\delta}$ containing $v$, then there exists a unique vector $\vartheta_{p} \in T_{v} \varkappa_{v}$ such that
\begin{equation}
	\label{eqn:cutting-weight-normalisation}
	\langle \vartheta_{p} ,\, u_{j} \rangle = 0
	\quad \text{ for each } j \in \mcI_{v}, \qquad
	\langle \vartheta_{p} ,\, u_{A} \rangle = -1.
\end{equation}
As the vectors $\{ u_{j} \}_{j \in \mcI_{v}} \cup \{ u_{A} \}$ form an $\RR$-basis of $\mft^{n}$, \eqref{eqn:cutting-weight-normalisation} determines $\vartheta_{p}$ uniquely.
If $v$ also lies in $\Pi_{B}^{\delta}$, Lemma~\ref{lemma:region-common-face} gives $A \triangle B \subseteq \mcI_{v}$, hence $u_{A} - u_{B} \in U_{\mcI_{v}}$.
Since $\vartheta_{p}$ annihilates every active arrangement normal vector, it has the same pairing with $u_{A}$ and $u_{B}$ and so $\vartheta_{p}$ is independent of the chosen cut face.
The boundary edge--weight configuration is illustrated in Figure~\ref{fig:weights-boundary-vertex}.

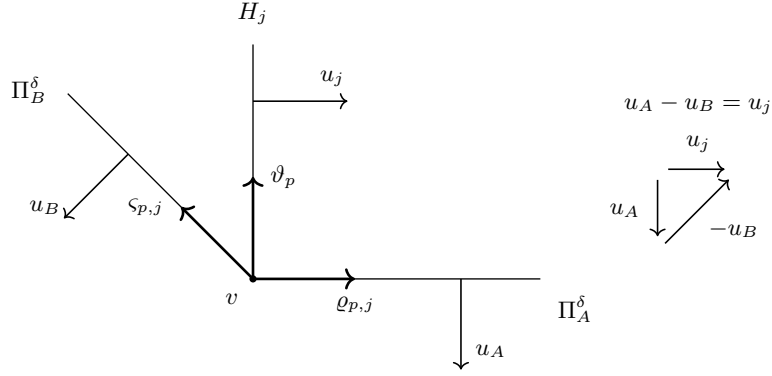
\begin{figure}[H]
	\centering
	\begin{tikzpicture}[
			boundary/.style={line width=0.5pt},
			tangent/.style={->, line width=1pt},
			normal/.style={->, line width=0.6pt},
			every node/.style={font=\small}
		]

		\coordinate (v) at (0 , 0);

		\draw[boundary] (v) -- ++(-2.45,2.45);
		\node[left=3pt] at (-2.5 , 2.5)
		{$\Pi_{B}^\delta$};

		\draw[boundary] (v) -- ++(3.80,0);
		\node[below right=3pt] at (3.80,0)
		{$\Pi_{A}^\delta$};

		\draw[boundary] (v) -- ++(0 , 3.10);
		\node[above=4pt] at (0 , 3.10)
		{$H_j$};

		\draw[normal] (0 , 2.35) -- ++(1.25,0);
		\node[above=2pt] at (1.05,2.35)
		{$u_{j}$};

		\draw[normal] (2.75,0) -- ++(0 , -1.20);
		\node[right=2pt] at (2.75 , -0.95)
		{$u_{A}$};

		\draw[normal] (-1.65,1.65) -- ++(-0.85 , -0.85);
		\node[left=2pt] at (-2.35,0.95)
		{$u_{B}$};

		\begin{scope}[shift={(5.35,1.45)}, scale=0.90]

			\coordinate (sumstart) at (0 , 0);
			\coordinate (summid)   at (0 , -1.20);
			\coordinate (sumend)   at (1.20,0);

			\node at (0.60,0.95) {$u_{A} - u_{B} = u_{j}$};

			\draw[normal, shorten <=4pt, shorten >=5.5pt]
			(sumstart) -- (summid)
			node[midway,left=3pt] {$u_{A}$};

			\draw[normal, shorten <=4pt, shorten >=5.5pt]
			(summid) -- (sumend)
			node[midway,below right=0.25pt] {$-u_{B}$};

			\draw[normal, shorten <=4pt, shorten >=5.5pt]
			(sumstart) -- (sumend)
			node[midway,above=3pt] {$u_{j}$};

		\end{scope}

		\draw[tangent] (v) -- ++(1.35,0);
		\node[below=3pt] at (1.35,0)
		{$\varrho_{p , j}$};

		\draw[tangent] (v) -- ++(-0.95,0.95);
		\node[left=3pt] at (-0.95,0.95)
		{$\varsigma_{p , j}$};

		\draw[tangent] (v) -- ++(0 , 1.35);
		\node[right=3pt] at (0 , 1.35)
		{$\vartheta_{p}$};

		\fill (v) circle (1.35pt);
		\node[below left=2pt] at (v)
		{$v$};
	\end{tikzpicture}
	\caption{
		An edge triple $\{ \varepsilon_{v , j} , \varphi_{v , j} , \varkappa_{v} \}$ at a boundary vertex.
		Here, $B = A \cup \{j\}$, so $u_{A} - u_{B} = u_{j}$, as shown by the vector sum.
		The inward cutting weight $\vartheta_{p}$ annihilates $u_{j}$, so it has pairing $-1$ with both cutting normals, as in \eqref{eqn:cutting-weight-normalisation}.
	}
	\label{fig:weights-boundary-vertex}
\end{figure}

The following elementary cut calculation on $T^{\ast}\CC$ illustrates how two complex tangent directions at a boundary fixed point can carry the same residual-circle weight.

\begin{lemma}[Elementary Cut Model on $T^{\ast}\CC$]
	\label{lemma:local-boundary-cut-model}
	Consider $T^{\ast}\CC = \CC^{2}_{( z , w )}$ with the residual $T^{1}$-action $t \cdot ( z , w ) = ( tz , t^{-1}w )$ and the fibre-rotation action $\tau \cdot ( z , w ) = ( z , \tau w )$.
	Form the diagonal cut with an auxiliary coordinate $\zeta$ at a cut level $\delta > 0$.
	Then the cut is modelled locally by $\{ \tfrac{1}{2} |w|^{2} + \tfrac{1}{2} |\zeta|^{2} = \delta \} / U_{1} \cong \CC_{z} \times \CC\PP^{1}_{[ w : \zeta ]}$, and the cut divisor is $\{ \zeta = 0 \}$.
	Moreover, at the residual fixed point $( z , [ w : \zeta ] ) = ( 0 , [ 1 : 0 ] )$, the two holomorphic tangent coordinates $z$ and $q := \zeta / w$ both have residual $T^{1}$-weight $+1$.
\end{lemma}

\begin{proof}
	The cutting circle acts diagonally on $( w , \zeta )$, and the quotient of the level set $\{ ( w , \zeta ) \in \CC^{2} ~|~ \tfrac{1}{2} |w|^{2} + \tfrac{1}{2} |\zeta|^{2} = \delta \}$ by $U_{1}$ is the projective line $\CC\PP_{[ w : \zeta ]}^{1}$.
	Since the cutting circle leaves $z$ unchanged, locally the full cut is $\CC_{z} \times \CC\PP_{ [ w : \zeta ] }^{1}$.
	The residual $T^{1}$-action acts by $t \cdot ( z , w , \zeta ) = ( tz , t^{-1}w , \zeta )$ and commutes with the cutting-circle action.
	On the chart with $w \neq 0$ near $[ 1 : 0 ]$, the local quotient coordinate is $q = \zeta / w$, and hence $z \mapsto tz$ and $q \mapsto tq$; thus, the two tangent weights coincide and equal $+1$.
\end{proof}

For a compact symplectic toric orbifold, the labelled facet normals and their dual edge directions determine the local tangent weights at a fixed point \cite{LermanTolman1997}, and this description applies to each toric cut component $\mcE_{A}^{\leq \delta} \subseteq M_{\nu}^{\leq \delta}$; in the full tangent representation of the cut, however, the cutting weight $\vartheta_{p}$ occurs twice at a boundary fixed point $p \in Z_{\nu}^{\delta}$.
Throughout the following proposition and its proof, we describe each restricted cut subvariety using the local charts of $M_{\nu}^{\leq \delta}$ restricted to that subvariety.
We keep the full local groups, even when some of their elements act trivially on the restricted charts, and the tangent and normal spaces considered are taken at lifts of the relevant points in these charts.

\begin{proposition}
	\label{prop:isotropy-weights}
	Let $p \in M_{\nu}^{\leq \delta}$ be a residual $T^{n}$-fixed point, and let $v = \mu_{\RR}(p)$.
	Then, as a representation of the local covering torus, and hence of the local orbifold group $\Gamma_{p}$, we have the following splitting
	\begin{equation}
		\label{eqn:isotropy-decomposition}
		T_{p}M_{\nu}^{\leq \delta} \cong \begin{cases}
			\displaystyle
			\bigoplus\nolimits_{j \in \mcI_{v}}
			\left( V_{\varrho_{p , j}} \oplus V_{\varsigma_{p , j}} \right),                                     & \text{ if } p \in M_{\nu}^{<\delta}, \\[3mm]
			\displaystyle
			\bigoplus\nolimits_{j \in \mcI_{v}}
			\left( V_{\varrho_{p , j}} \oplus V_{\varsigma_{p , j}} \right) \oplus V_{\vartheta_{p}}^{\oplus 2}, & \text{ if } p \in Z_{\nu}^{\delta}.
		\end{cases}
	\end{equation}
\end{proposition}

\begin{proof}
	The calculation for interior fixed points is inspired by Harada and Holm's identification of tangent weights from the incident arrangement edges in the smooth hypertoric setting \cite[proof of Proposition~3.2]{HaradaHolm2005}.
	Suppose first that $p \in M_{\nu}^{<\delta}$, and let $q \in \{\Phi < \delta\} \subseteq M_{\nu}$ be the corresponding fixed point under the canonical $T^{n}$-equivariant symplectomorphism.
	This identifies the symplectic tangent representations at $p$ and $q$, whose weights agree when computed using invariant compatible complex structures on local uniformising charts \cite[Lemma~A.1]{LermanTolman1997}.
	We may therefore calculate the weights at $q$ on $M_{\nu}$.

	Choose an extended-core component $\mcE_{A}$ containing $q$.
	On a local uniformising chart, a Hermitian metric invariant under the local covering torus splits $T_{q}M_{\nu}$ into the tangent and normal representations of $\mcE_{A}$ as $T_{q}M_{\nu} \cong T_{q}\mcE_{A} \oplus N_{q}\mcE_{A}$.
	Moreover, since $\mcE_{A}$ is $\w_{\CC}$-Lagrangian and the residual torus preserves $\w_{\CC}$, contraction with $\w_{\CC}$ identifies $N_{q}\mcE_{A} \cong T_{q}^{\ast}\mcE_{A}$, and so $T_{q}M_{\nu} \cong T_{q}\mcE_{A} \oplus T_{q}^{\ast}\mcE_{A}$.
	The toric tangent representation $T_{q}\mcE_{A}$ contains exactly one weight from each pair $\{\varrho_{p , j} , \varsigma_{p , j}\}$, for $j \in \mcI_{v}$, by \cite[\S6]{LermanTolman1997}.
	At an interior vertex, the edge directions and \eqref{eqn:noncutting-weight-normalisation} give $\varsigma_{p , j} = -\varrho_{p , j}$, and the dual representation $V_{\varsigma_{p , j}} \cong V_{\varrho_{p , j}}^{\ast}$ supplies the other weight in each pair, so transferring these weights back to the cut interior yields
	\[
		T_{p}M_{\nu}^{\leq \delta} \cong \bigoplus\nolimits_{j \in \mcI_{v}} (V_{\varrho_{p , j}} \oplus V_{\varsigma_{p , j}}),
	\]
	which proves the interior case of \eqref{eqn:isotropy-decomposition}.

	Now let $p \in Z_{\nu}^{\delta}$, choose a cut face $\Pi_{A}^{\delta}$ containing $v$.
	We separate the tangent directions along $M_{\mcI_{v}}^{\leq \delta} \subseteq M_{\nu}^{\leq \delta}$ from its normal directions.
	Before cutting, $M_{\mcI_{v}} \subseteq M_{\nu}$ is defined upstairs by $z_{i} = w_{i} = 0$ for $i \in \mcI_{v}$, and is therefore $T^{n}$-invariant.
	On a compatible local uniformising chart at $p$, an invariant Hermitian metric gives the following splitting:
	\begin{equation}
		\label{eqn:tangent-space-decomposition}
		T_{p}M_{\nu}^{\leq \delta} \cong N_{p}
		\big( M_{\mcI_{v}}^{\leq \delta} \subset M_{\nu}^{\leq \delta} \big) \oplus T_{p}M_{\mcI_{v}}^{\leq \delta},
	\end{equation}
	which is equivariant for the local covering torus.

	To calculate the tangent summand $T_{p}M_{\mcI_{v}}^{\leq \delta}$, consider its effective presentation, which has real dimension four.
	By Proposition~\ref{prop:cut-subspace-subpolyptych}, this is obtained by cutting $M_{\mcI_{v}}^{\eff}$ at level $\delta - C_{\mcI_{v}}$ using $\Phi_{\mcI_{v}}^{\std}$, since $\Phi|_{M_{\mcI_{v}}} = \Phi_{\mcI_{v}}^{\std} + C_{\mcI_{v}}$.
	More explicitly, the local group $\Gamma_{p}$ of the inherited chart and the local group $\Gamma_{p, \mcI_{v}}^{\eff}$ of the effective restricted cut have the respective lattice descriptions
	\[
		\Gamma_{p} \cong \left. \mft_{\ZZ}^{n} \middle/ ( U_{\mcI_{v}} + \ZZ u_{A} ), \right.
		\qquad
		\Gamma_{p, \mcI_{v}}^{\eff} \cong \left. \mft_{\ZZ}^{n} \middle/ ( V_{\mcI_{v}} + \ZZ u_{A} ), \right.
	\]
	so we see that replacing $U_{\mcI_{v}}$ by its saturation $V_{\mcI_{v}}$ removes the ineffective isotropy is present on the restricted chart.
	Lemma~\ref{lemma:flat-local-group-comparison}, applied with $\mcF = \mcI_{v}$ and $W = \ZZ u_{A}$, gives $(V_{\mcI_{v}} + \ZZ u_{A})^{\vee} \subseteq (U_{\mcI_{v}} + \ZZ u_{A})^{\vee}$.
	The underlying lattice inclusion $U_{\mcI_{v}} + \ZZ u_{A} \subseteq V_{\mcI_{v}} + \ZZ u_{A}$ induces the natural surjective homomorphism $\Gamma_{p} \ra \Gamma_{p, \mcI_{v}}^{\eff}$, along which each character of $\Gamma_{p, \mcI_{v}}^{\eff}$ pulls back to a character of $\Gamma_{p}$ by composition.
	The dual-lattice inclusion above expresses the corresponding pullback to the torus acting on the inherited chart, so the local weights calculated on the effective cut define characters of the covering torus.

	On the local covering circle of the effective cut, Lemma~\ref{lemma:local-boundary-cut-model} gives two one-dimensional weight spaces carrying the same primitive cutting character.
	Under the pullback $\Gamma_{p} \ra \Gamma_{p, \mcI_{v}}^{\eff}$, both carry the character with weight $\vartheta_{p}$, characterised by $\langle \vartheta_{p} ,\, u_{j} \rangle = 0$ for $j \in \mcI_{v}$ and $\langle \vartheta_{p} ,\, u_{A} \rangle = -1$ in \eqref{eqn:cutting-weight-normalisation}.
	Therefore, we have
	\[
		T_{p}M_{\mcI_{v}}^{\leq \delta} \cong V_{\vartheta_{p}}^{\oplus 2}
	\]
	as a representation of the covering torus acting on the	inherited chart, and hence of $\Gamma_{p}$ by restriction.

	For the normal summand in \eqref{eqn:tangent-space-decomposition}, fix $j \in \mcI_{v}$ and set $\mcJ_{v , j} := \mcI_{v} \setminus\{j\}$.
	The cut subvariety $M_{\mcJ_{v , j}}^{\leq \delta}$ has real dimension eight and contains $M_{\mcI_{v}}^{\leq \delta}$.
	The two toric cut components through $p$ in $M_{\mcJ_{v , j}}^{\leq \delta}$ have moment polytopes on opposite sides of $H_{j} \cap H_{\mcJ_{v , j}}$ in the flat $H_{\mcJ_{v , j}} \subseteq (\mft^{n})^{\ast}$ and, by Theorem~\ref{thm:vertex-edges}, their incident edge pairs at $v$ are respectively $\{\varepsilon_{v , j} , \varkappa_{v}\}$ and $\{\varphi_{v , j} , \varkappa_{v}\}$.
	The local toric-orbifold model \cite[\S6]{LermanTolman1997} identifies their tangent weights with the corresponding edge directions.
	For the restriction to $M_{\mcJ_{v, j}}^{\leq \delta}$, we apply Lemma~\ref{lemma:flat-local-group-comparison} with $\mcF = \mcJ_{v , j}$ and $W = \ZZ u_{j} + \ZZ u_{A}$.
	Since $U_{\mcJ_{v , j}} + W = U_{\mcI_{v}} + \ZZ u_{A}$, the resulting characters pull back to the same covering torus as above, giving the respective tangent representations $V_{\varrho_{p , j}} \oplus V_{\vartheta_{p}}$ and $V_{\varsigma_{p , j}} \oplus V_{\vartheta_{p}}$ on the restricted charts.

	The normalisations \eqref{eqn:noncutting-weight-normalisation} and \eqref{eqn:cutting-weight-normalisation} give $\langle \varrho_{p , j} ,\, u_{j} \rangle = 1$, $\langle \varsigma_{p , j} ,\, u_{j} \rangle = -1$, and $\langle \vartheta_{p} ,\, u_{j} \rangle = 0$.
	Since $T_{p}M_{\mcI_{v}}^{\leq \delta}$ consists of two copies of $V_{\vartheta_{p}}$, the lines $V_{\varrho_{p , j}}$ and $V_{\varsigma_{p , j}}$ give two distinct weight lines in the complex two-dimensional normal space, and hence
	\[
		N_{p} ( M_{\mcI_{v}}^{\leq \delta} \subset M_{\mcJ_{v , j}}^{\leq \delta} ) \cong V_{\varrho_{p , j}} \oplus V_{\varsigma_{p , j}}.
	\]
	These covering-torus weights are distinct as $j$ varies: the pair indexed by $j$ has pairings $+1$ and $-1$ with $u_{j}$, while both weight pairings vanish on every $u_{k}$ with $k \in \mcJ_{v , j} = \mcI_{v} \setminus \{j\}$.
	Since the normal space to $M_{\mcI_{v}}^{\leq \delta}$ in the ambient cut has complex dimension $2(n - 1)$, the resulting $2(n - 1)$ weight lines exhaust it:
	\[
		N_{p} ( M_{\mcI_{v}}^{\leq \delta} \subset M_{\nu}^{\leq \delta} ) \cong \bigoplus\nolimits_{j \in \mcI_{v}} ( V_{\varrho_{p , j}} \oplus V_{\varsigma_{p , j}} ).
	\]
	Combining the two summands in \eqref{eqn:tangent-space-decomposition} proves the boundary case of \eqref{eqn:isotropy-decomposition}.
\end{proof}

Thus, a boundary vertex has only one cutting edge $\varkappa_{v}$, even though its associated weight $\vartheta_{p}$ occurs twice in $T_{p}M_{\nu}^{\leq \delta}$.
The remaining tangent data required for the Kawasaki localisation formula \eqref{eqn:equivariant-KRR-character-formula} is the local group $\Gamma_{p}$ and its action on these isotropy weight spaces, which we now identify.

Let $p \in M_{\nu}^{\leq \delta}$ be a residual fixed point with $v = \mu_{\RR}(p)$, and choose a toric cut component $\mcE_{A}^{\leq \delta}$ containing $p$.
The facets of its moment polytope $\Delta_{A}^{\leq \delta}$ coming from the arrangement have labelled normals $\pm u_{j}$ and label one, since the $u_{j}$ are primitive.
At a boundary vertex, the cut face $\Pi_{A}^{\delta}$ is a facet with potentially a non-primitive normal $u_{A} = k_{A}\widebar{u}_{A}$, where $k_{A} \in \ZZ_{>0}$ and $\widebar{u}_{A}$ is primitive.
This follows the convention in Subsection~\ref{subsection:subsection-3.3-flats-arrangement} where the cut face $\Pi_{A}^{\delta}$ has label $k_{A}$ and labelled normal $u_{A}$.
Since changing the sign of a normal does not affect the lattice it generates, we define the \textbf{local normal lattice} by
\begin{equation}
	\label{eqn:local-normal-lattice}
	\Lambda_{v} :=
	\begin{cases}
		U_{\mcI_{v}},
		 & \text{ if } v \text{ is an interior vertex},                            \\[1mm]
		U_{\mcI_{v}} + \ZZ u_{A},
		 & \text{ if } v \text{ is a boundary vertex and } v \in \Pi_{A}^{\delta}.
	\end{cases}
\end{equation}
At a boundary vertex, the local normal lattice $\Lambda_{v}$ is independent of the chosen cut face: if $v \in \Pi_{A}^{\delta} \cap \Pi_{B}^{\delta}$, then $u_{A} - u_{B} \in U_{\mcI_{v}}$ implies $U_{\mcI_{v}} + \ZZ u_{A} = U_{\mcI_{v}} + \ZZ u_{B}$.
By Theorem~\ref{thm:vertex-edges}, the vectors $\{u_{j} ~|~ j \in \mcI_{v}\}$ form an $\RR$-basis of $\mft^{n}$ at an interior vertex, whereas a boundary vertex yields such a basis after adjoining $u_{A}$.
Thus, $\Lambda_{v} \subseteq \mft_{\ZZ}^{n}$ has full rank and a finite index.

For the toric cut component $\mcE_{A}^{\leq \delta}$, \cite[Lemma~6.6]{LermanTolman1997} identifies the local orbifold group at $p$ with the quotient
\begin{equation}
	\label{eqn:local-orbifold-group}
	\Gamma_{p} \cong \mft_{\ZZ}^{n} / \Lambda_{v}.
\end{equation}
Restricting the locally free quotient construction of $M_{\nu}^{\leq \delta}$ to the invariant subspace defining $\mcE_{A}^{\leq \delta}$ does not change the stabiliser of a representative of $p$, and thus $\Gamma_{p}$ is also the local orbifold group of $M_{\nu}^{\leq \delta}$ at $p$.
The associated local covering torus $\widetilde{T}_{v}^{n} := \mft^{n} / \Lambda_{v}$ is a finite cover of $T^{n} = \mft^{n} / \mft_{\ZZ}^{n}$ with kernel $\Gamma_{p}$ and character lattice
\begin{equation}
	\label{eqn:local-dual-lattice}
	\Lambda_{v}^{\vee} := \{ \alpha \in ( \mft^{n} )^{\ast} ~|~ \langle \alpha ,\, \lambda \rangle \in \ZZ \text{ for every } \lambda \in \Lambda_{v} \}.
\end{equation}

\begin{lemma}[Local Group Action on a Weight Space]
	\label{lemma:local-cover-kernel-character}
	Consider the covering $\widetilde{T}_{v}^{n} = \mft^{n}/\Lambda_{v} \ra T^{n} = \mft^{n}/\mft_{\ZZ}^{n}$, whose deck group is $\Gamma_{p} \cong \mft_{\ZZ}^{n}/\Lambda_{v}$.
	Let $\alpha \in \Lambda_{v}^{\vee}$ be a weight of $\widetilde{T}_{v}^{n}$, and let $g \in \Gamma_{p}$ be represented by $\gamma_{g} \in \mft_{\ZZ}^{n}$.
	Then $g$ acts geometrically on the one-dimensional representation $V_{\alpha}$ of weight $\alpha$ by multiplication by
	\begin{equation}
		\label{eqn:deck-group-weight-character}
		\exp( 2 \pi \sqrt{-1} \, \langle \alpha ,\, \gamma_{g} \rangle ).
	\end{equation}
\end{lemma}

\begin{proof}
	The kernel of the covering consists of the classes of $\mft_{\ZZ}^{n}$ modulo $\Lambda_{v}$, so $g$ is the class of $\gamma_{g}$ in $\widetilde{T}_{v}^{n}$.
	Evaluating the character with weight $\alpha$ on this class gives \eqref{eqn:deck-group-weight-character}, as in the local covering-torus construction of \cite[\S4]{KarshonSternbergWeitsman2007}.
	If $\gamma_{g}^{\prime}$ is another representative, then $\gamma_{g}^{\prime} - \gamma_{g} \in \Lambda_{v}$ and hence $\langle \alpha ,\, \gamma_{g}^{\prime} \rangle - \langle \alpha ,\, \gamma_{g} \rangle \in \ZZ$, so the exponential is independent of this choice and \eqref{eqn:deck-group-weight-character} is well-defined.
\end{proof}

Let us check that the tangent weights in Proposition~\ref{prop:isotropy-weights} belong to $\Lambda_{v}^{\vee}$.
The normalisations and edge directions give $\langle \varrho_{p , j} ,\, u_{j} \rangle = 1$ and $\langle \varsigma_{p , j} ,\, u_{j} \rangle = -1$, while both weights vanish on $u_{k}$ for $k \in \mcI_{v} \setminus \{j\}$.
They therefore pair integrally with $U_{\mcI_{v}}$, which proves the claim at an interior vertex.

At a boundary vertex, take a non-cutting weight $\alpha \in \{\varrho_{p , j} , \varsigma_{p , j}\}$ and choose a cut face $\Pi_{B}^{\delta}$ containing the corresponding edge.
Then $\langle \alpha ,\, u_{B} \rangle = 0$, and $u_{A} - u_{B} \in U_{\mcI_{v}}$ gives
\[
	\langle \alpha ,\, u_{A} \rangle = \langle \alpha ,\, u_{B} \rangle + \langle \alpha ,\, u_{A} - u_{B} \rangle = \langle \alpha ,\, u_{A} - u_{B} \rangle \in \ZZ.
\]
By \eqref{eqn:cutting-weight-normalisation}, the cutting weight satisfies $\langle \vartheta_{p} ,\, u_{j} \rangle = 0$ for every $j \in \mcI_{v}$ and $\langle \vartheta_{p} ,\, u_{A} \rangle = -1$.
Each tangent weight therefore pairs integrally with the generators of $\Lambda_{v}$ and hence belongs to $\Lambda_{v}^{\vee}$, as defined in \eqref{eqn:local-dual-lattice}.

For $g \in \Gamma_{p}$, choose a representative $\gamma_{g} \in \mft_{\ZZ}^{n}$.
By Lemma~\ref{lemma:local-cover-kernel-character}, the action of $g$ on $V_{\varrho_{p , j}}$ and $V_{\varsigma_{p , j}}$ is multiplication by the respective scalars
\begin{equation}
	\label{eqn:group-characters}
	\rho_{p , j}(g) := e^{2\pi\sqrt{-1}\langle \varrho_{p , j} ,\, \gamma_{g} \rangle}, \qquad
	\sigma_{p , j}(g) := e^{2\pi\sqrt{-1}\langle \varsigma_{p , j} ,\, \gamma_{g} \rangle},
\end{equation}
and, at a boundary fixed point, its action on both copies of $V_{\vartheta_{p}}$ is multiplication by
\begin{equation}
	\label{eqn:group-character-boundary}
	\theta_{p}(g) := e^{2\pi\sqrt{-1}\langle \vartheta_{p} ,\, \gamma_{g} \rangle}.
\end{equation}
These eigenvalues are independent of the choice of representative $\gamma_{g}$, since all the weights belong to $\Lambda_{v}^{\vee}$.
Moreover, they are the geometric tangent eigenvalues $\chi_{p , j}^{\tang}(g)$, along with their multiplicities, that appear in Theorem~\ref{thm:equivariant-KRR-index-theorem}.

Together with the tangent weights, these eigenvalues determine every denominator factor in the Kawasaki localisation formula \eqref{eqn:equivariant-KRR-character-formula}.
Each one-dimensional summand in \eqref{eqn:isotropy-decomposition} contributes a factor $1 - \chi_{p , j}^{\tang}(g) e^{\langle \alpha_{p , j} ,\, \xi \rangle}$ to the denominator, so the two copies of $V_{\vartheta_{p}}$ at a boundary fixed point contribute $(1 - \theta_{p}(g) e^{\langle \vartheta_{p} ,\, \xi \rangle} )^{2}$.
The genericity condition \eqref{eqn:orbifold-generic-xi} therefore requires $\chi_{p , j}^{\tang}(g)e^{\langle \alpha_{p , j} ,\, \xi \rangle} \neq 1$ for every fixed point $p$, every tangent summand, and every $g \in \Gamma_{p}$, rather than merely $e^{\langle \alpha_{p , j} ,\, \xi \rangle} \neq 1$.
The next subsection determines the fibre data needed to complete the localisation formulae \eqref{eqn:cut-localisation-interior} and \eqref{eqn:cut-localisation-boundary}.


\subsection{Pre-Quantum Line Orbibundle}
\label{subsection:subsection-5.3-pre-quantum-cut-space}

To complete the local data for the Kawasaki localisation formula, we now construct the $T^{n}$-equivariant pre-quantum line orbibundle on the cut and determine its fibre characters.
Fix a regular integral cut level $d > 0$, and retain the integral lift $\lambda$ from Section~\ref{section:section-2-hypertoric-varieties}, together with the pre-quantum line orbibundle $\mcL_{\nu} \ra M_{\nu}$ of \eqref{eqn:prequantum-orbibundle}, whose Chern curvature is $F_{\nabla_{\nu}} = ( 2\pi / \sqrt{-1} ) \w_{\RR}$.
The fibre-rotation action lifts to the trivial line bundle $L_{\nu} = \mcN \times \CC \ra \mcN$ via $\tau \cdot ( (z , w) , a ) = ( (z , \tau w) , a )$ for $\tau \in U_{1}$, which descends to $\mcL_{\nu} \ra M_{\nu}$ since it commutes with the $K$-action.

The auxiliary line $\CC_{\zeta}$ carries the K\"ahler form $\w_{\aux} = ( \sqrt{-1} / 4\pi ) d\zeta\wedge d\bar\zeta$ from Section~\ref{section:section-4-cutting-hypertoric-variety}.
Since $d$ is integral, $\varpi_{d}(\tau) = \tau^{d}$ defines a character $\varpi_{d} : U_{1} \ra U_{1}$, which we use to twist the geometric fibre action on the auxiliary line bundle $\CC_{\zeta} \times \CC_{Y} \ra \CC_{\zeta}$.
We equip this line bundle with the Hermitian metric $\| ( \zeta , Y ) \|^{2} = |Y|^{2} \exp( -|\zeta|^{2} / 2 )$ and the lift $\tau \cdot ( \zeta , Y ) = ( \tau \zeta , \tau^{d} Y )$ for $\tau \in U_{1}$, and denote the resulting pre-quantum line bundle by $\mcO_{\CC}(\varpi_{d}) \ra \CC_{\zeta}$; its Chern connection $\nabla_{\aux}$ has curvature $F_{\nabla_{\aux}} = ( 2\pi / \sqrt{-1} ) \w_{\aux}$.

Taking the product of these lifts gives the diagonal cutting-circle action on $\mcL_{\nu} \boxtimes \mcO_{\CC}( \varpi_{d} ) \ra M_{\nu} \times \CC_{\zeta}$, which commutes with the residual $T^{n}$-action.
The untwisted product lift is compatible with the cutting moment map $\rho : M_{\nu} \times \CC_{\zeta} \ra \RR$, given by $\rho( [z , w] , \zeta ) = \Phi[ z , w ] + \tfrac{1}{2}|\zeta|^{2}$.
Twisting the auxiliary fibre action by $\varpi_{d}$ changes the compatible moment-map normalisation from $\rho$ to $\rho - d$ by the Kostant formula.

Since $(\rho - d)|_{ \rho^{-1}(d) } = 0$, the product connection restricted to this level set is basic for the cutting circle and descends under reduction.
Furthermore, the regularity of $d$ ensures that the cutting $U_{1}$-action is locally free on $\rho^{-1}(d)$, so the quotient line bundle
\begin{equation}
	\label{eqn:cut-prequantum-orbibundle}
	\mcL_{\nu}^{\leq d} := \left.
	\left( \left. \left( \mcL_{\nu} \boxtimes \mcO_{\CC}(\varpi_{d}) \right) \middle|_{\rho^{-1}(d)} \right. \right) \middle/ U_{1}
	\right. \longrightarrow M_{\nu}^{\leq d}
\end{equation}
is a $T^{n}$-equivariant holomorphic pre-quantum line orbibundle.
The product Hermitian metric also descends, and the descended connection $\nabla^{\leq d}$ satisfies $F_{\nabla^{\leq d}} = ( 2\pi / \sqrt{-1} ) \w_{\RR}^{\leq d}$; compare the pre-quantum cutting constructions in \cite[\S3]{DGMW1995} and \cite[\S4]{Meinrenken1998}.

\begin{lemma}
	\label{lemma:cut-bundle-locality}
	Under the canonical $T^{n}$-equivariant symplectomorphism $M_{\nu}^{< d} \cong \{ q \in M_{\nu} ~|~ \Phi(q) < d \}$, there is a canonical $T^{n}$-equivariant isomorphism of Hermitian pre-quantum line orbibundles with connection,
	\[
		\left. \mcL_{\nu}^{\leq d} \right|_{ M_{\nu}^{< d} } \cong \left. \mcL_{\nu} \right|_{ \{ \Phi < d \} }.
	\]
\end{lemma}

\begin{proof}
	This is the locality property of the pre-quantum symplectic-cut construction \cite[Theorem~4.5]{Meinrenken1998}.
\end{proof}

Let us now determine the local fibre characters of $\mcL_{\nu}^{\leq d}$ at a residual fixed point $p \in (M_{\nu}^{\leq d})^{T^{n}}$, with $v = \mu_{\RR}(p) \in (\mft^{n})^{\ast}$.
Recall from \eqref{eqn:local-dual-lattice} that $\Lambda_{v}^{\vee}$ is the character lattice of the local covering torus $\widetilde{T}_{v}^{n} = \mft^{n} / \Lambda_{v}$.
The integrality of $\lambda$ and $d$ gives
\begin{equation}
	\label{eqn:vertex-in-local-dual-lattice}
	v \in \Lambda_{v}^{\vee} = \{ \alpha \in (\mft^{n})^{\ast} ~|~ \langle \alpha ,\, \eta \rangle \in \ZZ	\text{ for every } \eta \in \Lambda_{v}	\}.
\end{equation}
Indeed, for every $j \in \mcI_{v}$, the equality $\ell_{j}(v) = 0$ gives $\langle v ,\, u_{j} \rangle = -\lambda_{j} \in \ZZ$.
At a boundary vertex $v \in \Pi_{A}^{d}$, the cutting equality also gives $\langle v,\, u_{A} \rangle = d + \lambda_{A} \in \ZZ$.
The vectors $u_{j}$ for $j \in \mcI_{v}$, together with $u_{A}$ at a boundary vertex, generate the local normal lattice $\Lambda_{v}$ of \eqref{eqn:local-normal-lattice} and hence $v$ pairs integrally with $\Lambda_{v}$, proving \eqref{eqn:vertex-in-local-dual-lattice}.

Pre-quantum reduction preserves the normalisation of the residual moment map $\mu_{\RR}$, so the Kostant formula implies that the local covering torus $\widetilde{T}_{v}^{n}$ acts on the local fibre of $\mcL_{\nu}^{\leq d}$ at $p$ with geometric weight $-v$.
Under the pullback convention \eqref{eqn:pullback-action-on-sections}, the weight entering the Kawasaki localisation formula is therefore $\beta_{p} = v \in \Lambda_{v}^{\vee}$.
This applies to both interior and boundary fixed points; for interior ones, it also follows from Lemma~\ref{lemma:cut-bundle-locality}.

For $g \in \Gamma_{p}$ represented by $\gamma_{g} \in \mft_{\ZZ}^{n}$, Lemma~\ref{lemma:local-cover-kernel-character}, applied to the weight $-v$, shows that $g$ acts geometrically on this local fibre by multiplication by $\exp( -2 \pi \sqrt{-1} \langle v ,\, \gamma_{g} \rangle )$.
Similar to before, under the pullback action on sections, the corresponding value of the fibre character in \eqref{eqn:equivariant-KRR-character-formula} is the inverse, namely
\begin{equation}
	\label{eqn:local-cut-fibre-character}
	\chi_{p}^{\fib}(g) = \exp( 2 \pi \sqrt{-1} \langle \mu_{\RR}(p) ,\, \gamma_{g} \rangle ),
\end{equation}
and this value is independent of the representative $\gamma_{g}$ because $v \in \Lambda_{v}^{\vee}$.
Together with the tangent decomposition in Proposition~\ref{prop:isotropy-weights} and the local group $\Gamma_{p}$ in \eqref{eqn:local-orbifold-group}, these fibre characters finally complete the local data required by Theorem~\ref{thm:equivariant-KRR-index-theorem} for $(M_{\nu}^{\leq d} , \mcL_{\nu}^{\leq d})$.

%% file: sections/6-hypertoric-equivariant-quantisation.tex
\section{Equivariant Quantisation of Hypertoric Cut Spaces}
\label{section:section-6-equivariant-quantisation-hypertoric-algebraic-cutting}


The local data established in Section~\ref{section:section-5-localisation-formulae-isotropy-data} finally allows us to compute the equivariant indices of the compact hypertoric cuts.
Although the full cut need not be toric for the residual $T^{n}$-action, the moment polyptych determines its fixed points and the local representations that enter the orbifold Kawasaki--Riemann--Roch formula.
First, we assemble the localisation formula and then evaluate it in the examples of Subsection~\ref{subsection:subsection-4.4-worked-cut-spaces-moment-polyptychs}.
Section~\ref{section:section-7-algebraic-cutting-combinatorial-quantisation} relates these indices to the circle-graded quantisation of the original non-compact hypertoric variety, proves the higher-cohomology vanishing needed to identify them with characters of section spaces, and describes the resulting filtration combinatorially.

\subsection{Equivariant Localisation Formula}
\label{subsection:subsection-6.1-equivariant-quantisation-hypertoric}

At a regular integral cut level $d > 0$, the cut space $M_{\nu}^{\leq d}$ is a compact K\"ahler orbifold of complex dimension $2n$, with K\"ahler form $\w_{\RR}^{\leq d}$.
It is equipped with the residual $T^{n}$-action and the $T^{n}$-equivariant pre-quantum line orbibundle $\mcL_{\nu}^{\leq d} \ra M_{\nu}^{\leq d}$.
We write $\chi_{T^{n}}(M_{\nu}^{\leq d} ; \mcL_{\nu}^{\leq d}) \in R(T^{n})$ for its equivariant holomorphic Euler characteristic, as in \eqref{eqn:holomorphic-Euler-characteristic}.
Although this is initially a virtual $T^{n}$-representation, Corollary~\ref{cor:outer-chamber-vanishing} in Section~\ref{section:section-7-algebraic-cutting-combinatorial-quantisation} identifies it with the representation on $H^{0}(M_{\nu}^{\leq d} ; \mcL_{\nu}^{\leq d})$ when $d$ lies in the outer cut chamber.

We now obtain one of the main results of the paper: an explicit formula for the equivariant holomorphic Euler characteristic $\chi_{T^{n}}(M_{\nu}^{\leq d} ; \mcL_{\nu}^{\leq d})$ in terms of the moment polyptych $\Delta_{\nu}^{\leq d}$ and its associated lattice data.
The following theorem brings together the fixed-point, tangent, and fibre calculations of Section~\ref{section:section-5-localisation-formulae-isotropy-data} to determine every interior and boundary contribution, including the actions of the finite local groups $\Gamma_{p}$ and the multiplicity-two cutting weight at boundary fixed points.

\begin{theorem}[Equivariant Localisation Formula for Orbifold Cut Spaces]
	\label{theorem:character-hypertoric-orbifold}
	Let $M_{\nu}$ be the non-empty hypertoric orbifold constructed in Section~\ref{section:section-2-hypertoric-varieties} at a generic parameter $\nu$, with the arrangement cooriented so that the distinguished base region $\Delta_{\emptyset}$ is non-empty and bounded as in Subsection~\ref{subsection:subsection-3.1-hyperplane-arrangements}.
	Assume that the fixed lift $\lambda$ is integral, $\lambda \in ( \mft_{\ZZ}^{N} )^{\ast}$, so that $\nu = \imath_{\ZZ}^{\ast}(\lambda) \in \mfk_{\ZZ}^{\ast}$.
	Let $d \in \ZZ_{>0}$ be a regular value of the cutting moment map, let $\mcL_{\nu}^{\leq d} \to M_{\nu}^{\leq d}$ be the cut pre-quantum line orbibundle from \eqref{eqn:cut-prequantum-orbibundle}, and let $\xi \in \mft_{\CC}^{n}$ be generic in the sense of \eqref{eqn:orbifold-generic-xi}, so that every denominator factor displayed below is non-zero.

	Then the equivariant Kawasaki character is
	\begin{equation}
		\label{eqn:character-hypertoric-orbifold}
		\chi_{T^{n}}(M_{\nu}^{\leq d} ; \mcL_{\nu}^{\leq d})( e^{\xi} ) = \sum\nolimits_{p \in ( M_{\nu}^{\leq d} )^{T^{n}}} \chi_{p}(e^{\xi}),
	\end{equation}
	where $\chi_{p}(e^{\xi})$ denotes the local orbifold contribution of the fixed point $p$.
	For a fixed point $p$, write $v := \mu_{\RR}(p)$ so that $\Gamma_{p} \cong \mft_{\ZZ}^{n} / \Lambda_{v}$ and, for each $g \in \Gamma_{p}$, let $\gamma_{g} \in \mft_{\ZZ}^{n}$ be a representative.
	If $p \in M_{\nu}^{<d}$ is an interior fixed point, then
	\begin{equation}
		\label{eqn:cut-localisation-interior}
		\chi_{p}(e^{\xi}) = \frac{1}{|\Gamma_{p}|}
		\sum_{g \in \Gamma_{p}}
		\frac{
			e^{2 \pi \sqrt{-1}\langle v ,\, \gamma_{g} \rangle}
			e^{\langle v ,\, \xi \rangle}
		}{
			\prod\nolimits_{j \in \mcI_{v}}
			\left( 1 - \rho_{p , j}(g)e^{\langle \varrho_{p , j} ,\, \xi \rangle} \right)
			\left( 1 - \sigma_{p , j}(g)e^{\langle \varsigma_{p , j} ,\, \xi \rangle} \right)
		}
	\end{equation}
	and, if $p \in Z_{\nu}^{d}$ is a boundary fixed point, then
	\begin{equation}
		\label{eqn:cut-localisation-boundary}
		\chi_{p}(e^{\xi}) = \frac{1}{|\Gamma_{p}|}
		\sum_{g \in \Gamma_{p}}
		\frac{
			e^{2 \pi \sqrt{-1}\langle v ,\, \gamma_{g} \rangle}
			e^{\langle v ,\, \xi \rangle}
		}{
			\left( 1 - \theta_{p}(g)e^{\langle \vartheta_{p} ,\, \xi \rangle} \right)^{2}
			\prod\nolimits_{j \in \mcI_{v}}
			\left( 1 - \rho_{p , j}(g)e^{\langle \varrho_{p , j} ,\, \xi \rangle} \right)
			\left( 1 - \sigma_{p , j}(g)e^{\langle \varsigma_{p , j} ,\, \xi \rangle} \right)
		}.
	\end{equation}
	Note that these expressions are independent of the choice of $\gamma_{g} \in \mft_{\ZZ}^{n}$.
	Here, $\rho_{p , j}(g)$, $\sigma_{p , j}(g)$, and $\theta_{p}(g)$ are the geometric tangent eigenvalues from \eqref{eqn:group-characters} and \eqref{eqn:group-character-boundary}.
	Moreover, under the pullback action on sections, the fibre factor $e^{ 2\pi\sqrt{-1} \langle v ,\, \gamma_{g} \rangle}$ is the character value $\chi_{p}^{\fib}(g)$ from \eqref{eqn:local-cut-fibre-character}.
\end{theorem}

\begin{proof}
	By Lemma~\ref{lemma:cut-space-fixed-points}, the residual fixed points are isolated and correspond to the vertices $v$ of $\Delta_{\nu}^{\leq d}$ and hence, at the point $p$ above $v$, the local group is $\Gamma_{p} \cong \mft_{\ZZ}^{n}/\Lambda_{v}$ by \eqref{eqn:local-orbifold-group}.
	Lemma~\ref{lemma:local-cover-kernel-character} identifies $\Gamma_{p}$ with the kernel of the local covering $\widetilde{T}_{v}^{n} = \mft^{n}/\Lambda_{v} \ra T^{n}$.
	The action of $\Gamma_{p}$ on the local chart is obtained by restricting the covering-torus action, so each element of $\Gamma_{p}$ acts by a scalar on every isotropy weight space in \eqref{eqn:isotropy-decomposition}, which gives the required decomposition by Theorem~\ref{thm:equivariant-KRR-index-theorem}.

	Proposition~\ref{prop:isotropy-weights} gives the tangent decomposition \eqref{eqn:isotropy-decomposition}, whose summands, together with the geometric eigenvalues in \eqref{eqn:group-characters} and \eqref{eqn:group-character-boundary}, determine the denominator factors.
	At a boundary fixed point, the summand $V_{\vartheta_{p}}^{\oplus 2}$ gives the squared factor $(1 - \theta_{p}(g) e^{\langle \vartheta_{p} ,\, \xi \rangle})^{2}$, the pre-quantum normalisation gives the pullback fibre weight $\beta_{p} = v = \mu_{\RR}(p)$, and \eqref{eqn:local-cut-fibre-character} gives $\chi_{p}^{\fib}(g) = e^{2\pi\sqrt{-1} \langle v,\, \gamma_{g} \rangle}$.

	Substituting these data into \eqref{eqn:equivariant-KRR-character-formula} proves \eqref{eqn:cut-localisation-interior} and \eqref{eqn:cut-localisation-boundary}.
	The expressions are independent of the representatives $\gamma_{g}$ because all the tangent and fibre weights involved lie in $\Lambda_{v}^{\vee}$.
\end{proof}

For the smooth cut spaces characterised by Corollary~\ref{corollary:smooth-cut-space-characterisation}, the local groups $\Gamma_{p}$ are trivial, and the previous theorem simplifies to the following localisation formula.

\begin{corollary}[Equivariant Localisation Formula for Smooth Cut Spaces]
	\label{corollary:character-smooth-hypertoric-cut}
	Under the hypotheses of Theorem~\ref{theorem:character-hypertoric-orbifold}, suppose that $M_{\nu}^{\leq d}$ is a smooth manifold with trivial local orbifold groups.
	Then $\mcL_{\nu}^{\leq d} \ra M_{\nu}^{\leq d}$ is an ordinary pre-quantum line bundle, and the contribution of a residual fixed point $p$, with $v = \mu_{\RR}(p)$, to \eqref{eqn:character-hypertoric-orbifold} is
	\begin{equation}
		\label{eqn:smooth-cut-local-contribution}
		\chi_{p}(e^{\xi}) =
		\begin{cases}
			\displaystyle
			\frac{
				e^{\langle v ,\, \xi \rangle}
			}{
				\prod\nolimits_{j \in \mcI_{v}}
				\left(1 - e^{\langle \varrho_{p , j} ,\, \xi \rangle} \right)
				\left(1 - e^{\langle \varsigma_{p , j},\, \xi \rangle} \right)
			},
			\qquad & p \in M_{\nu}^{<d}, \\[5mm]
			\displaystyle
			\frac{
				e^{\langle v ,\, \xi \rangle}
			}{
				\left(1 - e^{\langle \vartheta_{p} ,\, \xi \rangle} \right)^2
				\prod\nolimits_{j \in \mcI_{v}}
				\left(1-e^{\langle \varrho_{p , j}, \, \xi \rangle} \right)
				\left(1-e^{\langle \varsigma_{p , j} ,\, \xi \rangle} \right)
			},
			\qquad & p \in Z_{\nu}^{d}.
		\end{cases}
	\end{equation}
	Moreover, $\Lambda_{v} = \mft_{\ZZ}^{n}$ at every vertex, so the local covering torus is the residual torus $T^{n}$.
	In particular, this formula applies for every cut level $d > 0$ whenever the equivalent conditions of Corollary~\ref{corollary:smooth-cut-space-characterisation} hold.
\end{corollary}

\begin{proof}
	Since $\Gamma_{p} = \{1\}$, equation~\eqref{eqn:local-orbifold-group} gives $\Lambda_{v} = \mft_{\ZZ}^{n}$.
	The unique element of $\Gamma_{p}$ is represented by $0 \in \mft_{\ZZ}^{n}$, so all finite-group tangent and fibre character values in \eqref{eqn:cut-localisation-interior} and \eqref{eqn:cut-localisation-boundary} equal $1$, and the averages reduce to the expressions in \eqref{eqn:smooth-cut-local-contribution}.
	The final assertion follows from Corollary~\ref{corollary:smooth-cut-space-characterisation}.
\end{proof}

We now return to the boundary contributions in the orbifold localisation formula \eqref{eqn:cut-localisation-boundary}.
When $n = 1$, the tangent representation at a boundary fixed point consists of two copies of the cutting weight, so each summand in the finite-group average has a squared denominator.
The local group is cyclic of order $k$, and although the individual summands may involve fractional powers arising from the local covering circle, these cancel upon averaging.

Lemma~\ref{lemma:boundary-vertex-closed-form} evaluates this average as a rational function in the residual character variable $t \in T^{1}$, determined by $k$ and the coordinate $v$ of the respective boundary vertex.
This expression allows us to calculate the boundary contributions in the worked examples below without repeating the finite-group summation.

To state the lemma, which uses Theorem~\ref{theorem:character-hypertoric-orbifold}, we continue to assume that the lift is integral and fix a regular integral cut level $d > 0$.
For a boundary vertex $v \in \Pi_{A}^{d}$, write $p_{v} \in Z_{\nu}^{d}$ for the corresponding residual fixed point, with $\mu_{\RR}(p_{v}) = v$.
Henceforth, including in the worked examples, we index its local data by $v$, writing $\Gamma_{v}$, $\vartheta_{v}$, $\theta_{v}$, $\chi_{v}$, $\chi_{v}^{\fib}$, and $\chi_{v , j}^{\tang}$.

\begin{lemma}[Rank-One Boundary Contribution]
	\label{lemma:boundary-vertex-closed-form}
	Suppose that $n = 1$, and let $v \in \Pi_{A}^{d}$ be a boundary vertex.
	Choose the primitive lattice generator $\widebar{u} \in \mft_{\ZZ}^{1}$ such that $u_{A} = k\widebar{u}$, with $k \in \ZZ_{>0}$, and identify $(\mft^{1})^{\ast}$ with $\RR$ via the linear isomorphism $x \mapsto \langle x ,\, \widebar{u} \rangle$, and continue to write $v$ for the real coordinate $\langle v ,\, \widebar{u} \rangle$ of the vertex.
	Let $t = e^{\xi}$ be the $T^{1}$-character variable for this choice of generator.
	Write
	\begin{equation}
		\label{eqn:boundary-vertex-division}
		d + \lambda_{A} = ak + b, \qquad a , b \in \ZZ, \quad 0 \leq b < k,
	\end{equation}
	so that $v = a + b/k$ and $a = \lfloor v \rfloor$.
	Then $\Gamma_{v} \cong \ZZ / k\ZZ$, and
	\begin{equation}
		\label{eqn:boundary-vertex-closed-form}
		\chi_{v}(t) = \sum\nolimits_{\ell \geq 0} ( k\ell + b + 1 )\,t^{a - \ell} = t^{a} \, \frac{ ( b + 1 ) + ( k - b - 1 )\,t^{-1} }{ ( 1 - t^{-1} )^{2} }.
	\end{equation}
	The contribution $\chi_{v}(t)$ contains only integral powers of $t$: the fractional powers appearing in the individual terms cancel in the average.
\end{lemma}

\begin{proof}
	By Theorem~\ref{thm:vertex-edges}\ref{theorem:vertex-edges-2}, we have $\mcI_{v} = \varnothing$, so $\Lambda_{v} = \ZZ u_{A}$ and $\Gamma_{v} = \mft_{\ZZ}^{1} / \ZZ u_{A} \cong \ZZ / k\ZZ$ is generated by the class of $\widebar{u}$.
	For $j = 0 , \ldots , k - 1$, let $g_{j} := [ j\widebar{u} ] \in \Gamma_{v}$ and choose $\gamma_{j} := j\widebar{u} \in \mft_{\ZZ}^{1}$ as its representative.
	The normalisation \eqref{eqn:cutting-weight-normalisation} gives $\langle \vartheta_{v} ,\, u_{A} \rangle = -1$, so $\vartheta_{v} = -1/k$ relative to the chosen coordinate, and hence $t^{\vartheta_{v}} = t^{-1/k}$.
	Moreover, $v \in \Pi_{A}^{d}$ gives $kv = \langle v ,\, u_{A} \rangle = d + \lambda_{A}$, whence $v = a + b/k$.
	By \eqref{eqn:U1-action-generator}, $\lambda_{A} = \sum_{i \in A} \lambda_{i}$.
	Since the fixed lift $\lambda$ is integral, $\lambda_{A}$ is an integer, so \eqref{eqn:boundary-vertex-division} is a Euclidean division of integers.
	It follows that
	\[
	\theta_{v}(g_{j}) = e^{-2\pi\sqrt{-1} \, j/k}, \qquad
	\chi_{v}^{\fib}(g_{j}) = e^{2\pi\sqrt{-1} \langle v ,\, \gamma_{j} \rangle} = e^{2\pi\sqrt{-1} \, jb/k},
	\]
	so substituting these character values into the boundary localisation formula \eqref{eqn:cut-localisation-boundary} of Theorem~\ref{theorem:character-hypertoric-orbifold}, and using $\mcI_{v} = \varnothing$, gives
	\[
		\chi_{v}(t) = \frac{t^{v}}{|\Gamma_{v}|} \sum_{g \in \Gamma_{v}} \frac{ \chi_{v}^{\fib}(g) }{ ( 1 - \theta_{v}(g) e^{ \langle \vartheta_{v} , \xi  \rangle } )^{2} }
		= \frac{ t^{ \, a + b/k} }{ k } \sum_{j = 0}^{k - 1} \frac{ e^{ 2 \pi \sqrt{-1} \, jb/k } }{ \big( 1 - e^{ -2 \pi \sqrt{-1} \, j/k } \, t^{-1/k} \big)^{2} }.
	\]
	Expanding $(1 - z)^{-2} = \sum_{m \geq 0}(m + 1)z^{m}$ and interchanging the finite sum with the formal power series, we find by character orthogonality for $\ZZ / k\ZZ$ that 	the inner average vanishes unless $m \equiv b \pmod{k}$.
	Writing the surviving indices as $m = b + k\ell$ with $\ell \geq 0$, and using $t^{\, a + b/k} \, t^{-( b + k\ell ) / k} = t^{\, a - \ell}$, gives the first equality in \eqref{eqn:boundary-vertex-closed-form}; summing $\sum_{\ell \geq 0}t^{-\ell}$ and $\sum_{\ell \geq 0} \ell \, t^{-\ell}$ gives the second.
\end{proof}

\begin{remark}
	\label{remark:boundary-vertex-specialisations}
	For $k = 1$, we have $b = 0$ and $a = v$, so \eqref{eqn:boundary-vertex-closed-form} reduces to $\chi_{v}(t) = t^{v}(1 - t^{-1})^{-2}$, the contribution of a boundary vertex with trivial local group.
	More generally, $b = 0$ holds if and only if $v$ is integral, or equivalently when $\chi_{v}^{\fib}$ is trivial.
	The residue $b$ therefore measures the non-integrality of the vertex, rather than a failure of fractional powers to cancel: the averaged contribution has integral powers for every $b$, with leading exponent $\lfloor v \rfloor$, equal to $v$ when $b = 0$.
\end{remark}

\begin{remark}[Geometric and Pullback Character Conventions]
	\label{remark:convention-sensitivity}
	The rank-one calculation provides a consistency check on the use of the geometric action and the induced pullback action on sections in \eqref{eqn:cut-localisation-boundary}.
	As fixed in Subsection~\ref{subsection:subsection-5.1-equivariant-localisation-formulae}, the denominator uses the geometric tangent eigenvalue $\theta_{v}(g)$, whereas the numerator uses the pullback fibre character value $\chi_{v}^{\fib}(g)$.
	By \eqref{eqn:pullback-action-on-sections}, this fibre character value is the inverse of the geometric eigenvalue of $g$ on the pre-quantum fibre.
	Replacing only $\chi_{v}^{\fib}(g)$ by its inverse would therefore insert the geometric fibre eigenvalue where the formula requires its pullback counterpart, while leaving the tangent factors and the torus factor $t^{v}$ unchanged.

	In the proof of Lemma~\ref{lemma:boundary-vertex-closed-form}, the surviving congruence changes from $m \equiv b \pmod{k}$ to $m \equiv - b \pmod{k}$.
	The expression is unchanged only when $2b \equiv 0 \pmod{k}$, that is, when $b = 0$ or when $k$ is even and $b = k/2$.
	In particular, local groups of order at most two cannot detect this substitution.
	If $2b \not\equiv 0 \pmod{k}$, every exponent in the formal expansion of the altered boundary contribution would lie in $2b / k + \ZZ$ and therefore be non-integral.
	Since the overall index belongs to $R(T^{1}) = \ZZ[t , t^{-1}]$ by \eqref{eqn:holomorphic-Euler-characteristic}, these fractional-power terms would have to cancel against contributions from other vertices for the altered fixed-point sum to give the correct index.

	The $A_{r}$ family of Example~\ref{example:A2-resolution-example-1} shows that such cancellation need not be possible.
	At a positive regular integral level $d \in \Cout$, the rightmost boundary vertex $v$ has local group $\Gamma_{v} \cong \ZZ / r\ZZ$, so $k = r$, whereas all other vertices have trivial local groups and their contributions are rational expressions in integral powers of $t$.
	For $r \geq 3$ and $2b \not\equiv 0 \pmod{r}$, no other vertex contribution can therefore cancel the non-integral powers introduced at the right boundary.
	The altered sum subsequently fails to lie in $\ZZ[t , t^{-1}]$, due to the geometric fibre eigenvalue being used in the place of the required pullback fibre character value.

	On the other hand, simultaneously inverting all finite-group character values, while leaving $t^{v}$ and $t^{-1/k}$ unchanged, merely reindexes the average by $g \mapsto g^{-1}$ and leaves it unchanged.
	This simultaneous inversion is only a change of indexing, whereas the fibre-only substitution mixes the geometric and pullback conventions.
\end{remark}


\subsection{Worked Equivariant Characters}
\label{subsection:subsection-6.2-worked-equivariant-character}

We now apply Theorem~\ref{theorem:character-hypertoric-orbifold} to the cut spaces in Subsection~\ref{subsection:subsection-4.4-worked-cut-spaces-moment-polyptychs}.
For the smooth cotangent-bundle examples, we use Corollary~\ref{corollary:character-smooth-hypertoric-cut}; the $A_{r}$-resolution family will show how non-trivial local groups enter the boundary contributions for $r \geq 2$, and requires Theorem~\ref{theorem:character-hypertoric-orbifold}.
For each vertex $v \in \Verts(\Delta_{\nu}^{\leq d})$, let $p_{v}$ be the unique residual fixed point with $\mu_{\RR}(p_{v}) = v$.
As in Lemma~\ref{lemma:boundary-vertex-closed-form}, we index its local data by $v$ and write $\chi_{v}$ for its contribution.
Thus,
\begin{equation}
	\label{eqn:cut-character-shorthand}
	\chi_{d}(e^{\xi}) := \chi_{T^{n}}(M_{\nu}^{\leq d} ; \mcL_{\nu}^{\leq d})(e^{\xi}) = \sum\nolimits_{v \in \Verts(\Delta_{\nu}^{\leq d})} \chi_{v}(e^{\xi}).
\end{equation}

For a local tangent weight $\alpha \in \Lambda_{v}^{\vee}$, write $t^{\alpha} := e^{\langle \alpha ,\, \xi \rangle}$.
The weight $\alpha$ determines a character of the local covering torus $\widetilde{T}_{v}^{n} = \mft^{n} / \Lambda_{v}$, whose character lattice is given in \eqref{eqn:local-dual-lattice}, and if $\alpha \in (\mft_{\ZZ}^{n})^{\ast}$, this character descends further to the residual torus $T^{n}$.

The individual terms in the finite-group average defining $\chi_{v}$ may involve fractional powers, but averaging over $\Gamma_{v}$ gives a rational expression in the characters of the residual torus $T^{n}$.
The remaining poles cancel when the vertex contributions are summed in \eqref{eqn:cut-character-shorthand}, giving a Laurent polynomial since the total index belongs to $R(T^{n})$.
In rank one, we write $t = e^{\xi}$, and in rank two, we write $(t_{1} , t_{2}) = (e^{\xi_{1}} , e^{\xi_{2}})$.
After simplifying the total character, evaluation at the identity means setting $t = 1$ or $(t_{1} , t_{2}) = (1 , 1)$, respectively.

Throughout the examples, $d > 0$ is a regular integral level in the outer cut chamber.
Corollary~\ref{cor:outer-chamber-vanishing} in Section~\ref{section:section-7-algebraic-cutting-combinatorial-quantisation} will show that $H^{q}( M_{\nu}^{\leq d} ; \mcL_{\nu}^{\leq d} ) = 0$ for every $q > 0$.
Hence, $\chi_{d}$ is the character of $H^{0}( M_{\nu}^{\leq d} ; \mcL_{\nu}^{\leq d} )$, and its value at the identity is the dimension of that section space.
Corollary~\ref{cor:weighted-lattice-point-character} later computes the same character alternatively as a weighted lattice-point sum on the level-$d$ moment polyptych.

For the first example, we use the rank-one Brion identity for the interval $[0 , m]$, where $m \in \ZZ_{\geq0}$:
\begin{equation}
	\label{eqn:rank-one-Brion-sign-check}
	\frac{1}{1 - t} + \frac{t^{m}}{1 - t^{-1}} = 1 + t + \cdots + t^{m}.
\end{equation}


\begin{example}[$T^{\ast}\CC\PP^{1}$, continued]
	\label{example:TCP1-example-2}
	We continue Example~\ref{example:TCP1-example-1} with $m , d \in \ZZ_{>0}$ and $\delta = d$.
	The cut moment polyptych has components $[-d , 0]$, $[0 , m]$, and $[m , m + d]$, with support $|\Delta_{\nu}^{\leq d}| = [-d , m + d]$.
	Figure~\ref{fig:isotropy-TCP1} shows the local tangent weights at the four residual fixed points of the cut; at each boundary fixed point, the cutting weight occurs with multiplicity two.

	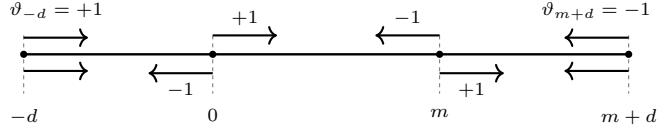
\begin{figure}[ht]
		\centering
		\begin{tikzpicture}[x=1cm , y=1cm]

			\tikzset{
				polyboundary/.style={line width=0.9pt},
				weight/.style={->,line width=0.8pt},
				cutweight/.style={->,line width=0.8pt},
				fixedpoint/.style={font=\scriptsize},
				weightlabel/.style={font=\scriptsize},
			}

			\coordinate (b1) at (-4 , 0);
			\coordinate (v1) at (-1.5 , 0);
			\coordinate (v2) at ( 1.5 , 0);
			\coordinate (b2) at ( 4 , 0);

			\foreach \p in {v1, v2}
			\draw[dashed, dash pattern=on 1.5pt off 1.5pt, line width=0.35pt, black!55]
			($(\p)+(0 , 0.25)$) -- ($(\p)+(0 , -0.52)$);

			\foreach \p in {b1, b2}
			\draw[dashed, dash pattern=on 1.5pt off 1.5pt, line width=0.35pt, black!55]
			($(\p)+(0 , 0.25)$) -- ($(\p)+(0 , -0.52)$);

			\draw[polyboundary] (b1) -- (b2);
			\foreach \p in {b1,v1,v2,b2}
			\fill (\p) circle (1.35pt);

			\draw[weight]
			($(v1)+(0 , 0.25)$) -- ++(0.85 , 0)
			node[midway , above , weightlabel] {$+1$};

			\draw[weight]
			($(v1)+(0 , -0.25)$) -- ++(-0.85 , 0)
			node[midway , below , weightlabel] {$-1$};

			\draw[weight]
			($(v2)+(0 , 0.25)$) -- ++(-0.85 , 0)
			node[midway , above , weightlabel] {$-1$};

			\draw[weight]
			($(v2)+(0 , -0.25)$) -- ++(0.85 , 0)
			node[midway , below , weightlabel] {$+1$};

			\draw[cutweight]
			($(b1)+(0 , 0.22)$) -- ++(0.85 , 0);
			\draw[cutweight]
			($(b1)+(0 , -0.22)$) -- ++(0.85 , 0);

			\node[weightlabel , above=3pt] at ($(b1)+(0.43 , 0.22)$)
			{$\vartheta_{-d} = +1$};

			\draw[cutweight]
			($(b2)+(0 , 0.22)$) -- ++(-0.85 , 0);
			\draw[cutweight]
			($(b2)+(0 , -0.22)$) -- ++(-0.85 , 0);

			\node[weightlabel , above=3pt] at ($(b2)+(-0.43 , 0.22)$)
			{$\vartheta_{m + d} = -1$};

			\node[fixedpoint , below=17pt] at (b1) {$-d$};
			\node[fixedpoint , below=17pt] at (v1) {$0$};
			\node[fixedpoint , below=17pt] at (v2) {$m$};
			\node[fixedpoint , below=17pt] at (b2) {$m + d$};

		\end{tikzpicture}

		\caption{
			Local tangent weights at the four residual fixed points of the cut of $T^{\ast}\CC\PP^{1}$.
			At each boundary vertex, the cutting weight occurs with multiplicity two, represented by the pair of parallel arrows.
			}
		\label{fig:isotropy-TCP1}
	\end{figure}

	All four local groups are trivial, so Corollary~\ref{corollary:character-smooth-hypertoric-cut} gives the vertex contributions
	\[
		\chi_{-d}(t) = \frac{t^{-d}}{( 1 - t )^{2}}, \quad
		\chi_{0}(t) = \frac{1}{( 1 - t )( 1 - t^{-1} )}, \quad \chi_{m}(t) = \frac{t^{m}}{( 1 - t )( 1 - t^{-1} )}, \quad
		\chi_{m + d}(t) = \frac{t^{m + d}}{( 1 - t^{-1} )^{2}},
	\]
	and regrouping their sum gives
	\begin{equation}
		\label{eqn:TCP1-example-2-character}
		\chi_{d}(t) =
		\left( \frac{1}{1 - t} + \frac{t^{m + d}}{1 - t^{-1}} \right)
		\left( \frac{t^{-d}}{1 - t} + \frac{1}{1 - t^{-1}} \right) =
		\left( \sum_{r = 0}^{m + d} t^{r} \right)
		\left( \sum_{q = 0}^{d} t^{-q} \right).
	\end{equation}
	The two rational factors in \eqref{eqn:TCP1-example-2-character} are the Brion sums for $[0 , m + d]$ and $[-d , 0]$, respectively, so \eqref{eqn:rank-one-Brion-sign-check} gives the finite sums in the second equality.
	Evaluating the resulting Laurent polynomial at $t = 1$ gives $\dim_{\CC} H^{0}(M_{\nu}^{\leq d} ; \mcL_{\nu}^{\leq d}) = (m + d + 1)(d + 1)$.
	Note for later that the factorisation in \eqref{eqn:TCP1-example-2-character} is the localisation counterpart of the base--fibre description of the cut polarisation which will be discussed in Remark~\ref{remark:projective-bundle-TCP1-computation}.
\end{example}


\begin{example}[$A_{2}$-Resolution, continued]
	\label{example:A2-resolution-example-2}
	We continue the $A_{2}$-resolution example from Example~\ref{example:A2-resolution-example-1}, with $m_{1} , m_{2} , d \in \ZZ_{>0}$, $0 < m_{1} < m_{2}$, and $\delta = d$.
	The arrangement vertices satisfy $\mcS(0) = \mcS(m_{1}) = 0$ and $\mcS(m_{2}) = m_{2} - m_{1}$, so $c_{\max} = m_{2} - m_{1}$ and the outer-chamber condition is $d \in \Cout = ( m_{2} - m_{1} , \infty )$.

	To avoid floor functions in this calculation, we further assume $d - (m_{2} - m_{1}) = 2s$, with $s \in \ZZ_{>0}$.
	The rightmost boundary vertex is then $(m_{1} + m_{2} + d) / 2 = m_{2} + s \in \ZZ$, so the pullback fibre character at the corresponding fixed point is trivial.
	Note however that this parity assumption is not required for the existence of the cut pre-quantum line orbibundle or for the cancellation of fractional powers in the local average: Lemma~\ref{lemma:boundary-vertex-closed-form} gives a rational expression in integral powers of $t$ for every residue $b$.

	The cut moment polyptych has the components $[ -d , 0 ]$, $[ 0 , m_{1} ]$, $[ m_{1} , m_{2} ]$, and $[ m_{2} , m_{2} + s ]$, and its support is $|\Delta_{\nu}^{\leq d}| = [ -d , m_{2} + s ]$.
	At the rightmost boundary, the non-primitive cutting normal $u_{\{ 2 , 3 \}} = 2$ gives $\Gamma_{m_{2} + s} \cong \ZZ / 2\ZZ$; Figure~\ref{fig:isotropy-A2} shows the local tangent weights at all five residual fixed points.

	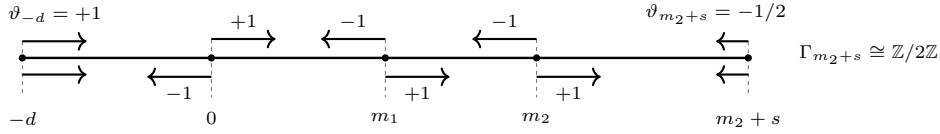
\begin{figure}[H]
		\centering
		\begin{tikzpicture}[x=1cm , y=1cm]

			\tikzset{
				polyboundary/.style={line width=0.9pt},
				weight/.style={->,line width=0.8pt},
				cutweight/.style={->,line width=0.8pt},
				fixedpoint/.style={font=\scriptsize},
				weightlabel/.style={font=\scriptsize},
				grouplabel/.style={font=\scriptsize},
			}

			\coordinate (b1)  at (-4.8 , 0);
			\coordinate (v1)  at (-2.3 , 0);
			\coordinate (v2)  at ( 0.0 , 0);
			\coordinate (v3)  at ( 2.0 , 0);
			\coordinate (b23) at ( 4.8 , 0);

			\foreach \p in {v1,v2,v3}
			\draw[dashed, dash pattern=on 1.5pt off 1.5pt, line width=0.35pt, black!55]
			($(\p)+(0 , 0.25)$) -- ($(\p)+(0 , -0.52)$);

			\foreach \p in {b1,b23}
			\draw[dashed, dash pattern=on 1.5pt off 1.5pt, line width=0.35pt, black!55]
			($(\p)+(0 , 0.22)$) -- ($(\p)+(0 , -0.52)$);

			\draw[polyboundary] (b1) -- (b23);

			\foreach \p in {b1,v1,v2,v3,b23}
			\fill (\p) circle (1.35pt);

			\draw[cutweight]
			($(b1)+(0 , 0.22)$) -- ++(0.85 , 0);
			\draw[cutweight]
			($(b1)+(0 , -0.22)$) -- ++(0.85 , 0);

			\node[weightlabel , above=3pt]
			at ($(b1)+(0.43 , 0.22)$)
			{$\vartheta_{-d} = +1$};

			\draw[weight]
			($(v1)+(0 , 0.25)$) -- ++(0.85 , 0)
			node[midway , above , weightlabel] {$+1$};

			\draw[weight]
			($(v1)+(0 , -0.25)$) -- ++(-0.85 , 0)
			node[midway , below , weightlabel] {$-1$};

			\draw[weight]
			($(v2)+(0 , 0.25)$) -- ++(-0.85 , 0)
			node[midway , above , weightlabel] {$-1$};

			\draw[weight]
			($(v2)+(0 , -0.25)$) -- ++(0.85 , 0)
			node[midway , below , weightlabel] {$+1$};

			\draw[weight]
			($(v3)+(0 , 0.25)$) -- ++(-0.85 , 0)
			node[midway , above , weightlabel] {$-1$};

			\draw[weight]
			($(v3)+(0 , -0.25)$) -- ++(0.85 , 0)
			node[midway , below , weightlabel] {$+1$};

			\draw[cutweight]
			($(b23)+(0 , 0.22)$) -- ++(-0.425 , 0);
			\draw[cutweight]
			($(b23)+(0 , -0.22)$) -- ++(-0.425 , 0);

			\node[weightlabel , above=3pt]
			at ($(b23)+(-0.43 , 0.22)$)
			{$\vartheta_{m_{2} + s} = -1/2$};

			\node[grouplabel, anchor=west] at ($(b23)+(0.55 , 0.05)$)
			{$\Gamma_{m_{2} + s} \cong \ZZ/ 2\ZZ$};

			\node[fixedpoint , below=17pt] at (b1)  {$-d$};
			\node[fixedpoint , below=17pt] at (v1)  {$0$};
			\node[fixedpoint , below=17pt] at (v2)  {$m_{1}$};
			\node[fixedpoint , below=17pt] at (v3)  {$m_{2}$};
			\node[fixedpoint , below=17pt] at (b23) {$m_{2} + s$};

		\end{tikzpicture}

		\caption{
				Local tangent weights at the five residual fixed points of the cut of the $A_{2}$-resolution.
				At the fixed point above the right boundary vertex, the cutting normal $u_{\{ 2 , 3 \}} = 2$ gives the tangent weight $\vartheta_{m_{2} + s} = -1/2$ with multiplicity two and the local orbifold group $\Gamma_{m_{2} + s} \cong \ZZ / 2\ZZ$.
			}
		\label{fig:isotropy-A2}
	\end{figure}

	The first four vertices $-d$, $0$, $m_{1}$, and $m_{2}$ have trivial local groups, so their contributions are
	\[
		\chi_{-d}(t) = \frac{t^{-d}}{(1 - t)^{2}}, \qquad
		\chi_{v}(t) = \frac{t^{v}}{(1 - t)(1 - t^{-1})}, \qquad
		v \in \{ 0 , m_{1} , m_{2} \}.
	\]
	At the remaining vertex $m_{2} + s$, the non-identity element $g \in \Gamma_{m_{2} + s} \cong \ZZ / 2\ZZ$ has tangent eigenvalue $\theta_{m_{2} + s}(g) = -1$.
	Since $m_{2} + s \in \ZZ$, the pullback fibre character value is $\chi_{m_{2} + s}^{\fib}(g) = e^{2\pi\sqrt{-1} (m_{2} + s)} = 1$, so the finite-group average gives
	\[
		\chi_{m_{2} + s}(t) = \frac{1}{2}\frac{t^{m_{2} + s}}{( 1 - t^{-1 / 2} )^{2}} + \frac{1}{2}\frac{t^{m_{2} + s}}{( 1 + t^{-1 / 2} )^{2}} = \frac{t^{m_{2} + s}(1 + t^{-1})}{( 1 - t^{-1} )^{2}},
	\]
	and summing the five contributions gives
	\[
		\chi_{d}(t) = \frac{t^{-d}}{( 1 - t )^{2}} + \frac{1 + t^{m_{1}} + t^{m_{2}}}{( 1 - t )( 1 - t^{-1} )} + \frac{t^{m_{2} + s}(1 + t^{-1})}{( 1 - t^{-1} )^{2}}.
	\]
	To compare this character with the base--fibre factorisation in the previous $T^{\ast}\CC\PP^{1}$ example, let $\widetilde{\chi}_{d}(t)$ denote the character in \eqref{eqn:TCP1-example-2-character} at the same cut level $d$, with $m$ replaced by $m_{1}$.
	The contributions at $-d$, $0$, and $m_{1}$ agree, so they cancel in the difference; using $m_{1} + d = m_{2} + 2s$ in the remaining terms gives
	\[
		\begin{split}
			\chi_{d}(t) - \widetilde\chi_{d}(t)
			 & = \frac{t^{m_{2}}}{(1 - t)(1 - t^{-1})} +\frac{t^{m_{2} + s}(1 + t^{-1}) - t^{m_{2} + 2s}}{(1 - t^{-1})^{2}} \\
			 & = -t^{m_{2} + 1}\frac{(1 - t^{s})(1 - t^{s + 1})}{(1 - t)^{2}}.
		\end{split}
	\]
	Applying the finite geometric-series identity to $(1 - t^{s}) / (1 - t)$ and $(1 - t^{s + 1}) / (1 - t)$ yields
	\[
		\chi_{d}(t) = \left( \sum_{r = 0}^{m_{1} + d}t^{r} \right)
		\left( \sum_{q = 0}^{d} t^{-q} \right) - t^{m_{2}}
		\left( \sum_{k = 1}^{s} t^{k} \right)
		\left( \sum_{l = 0}^{s} t^{l} \right).
	\]
	Finally, evaluating at $t = 1$, we obtain
	\[
		\begin{split}
			\dim_{\CC} H^{0}( M_{\nu}^{\leq d} ; \mcL_{\nu}^{\leq d} )
			 & = ( m_{1} + d + 1 )( d + 1 ) - s( s + 1 )                                                  \\
			 & = ( m_{1} + d + 1 )( d + 1 ) - \frac{ ( d + m_{1} - m_{2} )( d + m_{1} - m_{2} + 2 ) }{4}.
		\end{split}
	\]
	The term $(m_{1} + d + 1)(d + 1) = \widetilde{\chi}_{d}(1)$ is the dimension for the comparison cut of $T^{\ast}\CC\PP^{1}$, whose moment-polyptych support is $[-d , m_{1} + d] = [-d , m_{2} + 2s]$.
	In the $A_{2}$ example, the cutting normal $u_{\{2 , 3\}} = 2$ instead places the rightmost endpoint at $m_{2} + s$.
	The difference $\chi_{d} - \tilde{\chi}_{d}$ records changes in both the weight multiplicities and their support; evaluation at $t = 1$ shows that the dimension decreases by $s(s + 1)$.
\end{example}

\begin{remark}[The $A_{r}$-Resolution Family, continued]
	\label{remark:Ar-localisation-family}
	The previous calculation is the $r = 2$ case of the $A_{r}$-family in Example~\ref{example:A2-resolution-example-1}.
	For $r \geq 1$ and a positive regular integral level $d \in \Cout$, the leftmost boundary at $-d$ has $|u_{\{1\}}|=1$, while the arrangement vertices $0 , m_{1} , \ldots , m_{r}$ have trivial local groups and tangent weights $\pm 1$.
	The additional arrangement vertices therefore contribute terms of the same form, while the rightmost boundary contribution depends on $r$ through its local group $\ZZ / r\ZZ$.
	For the rightmost region, take $A = \{ 2 , \ldots , r + 1 \}$, so that $\lambda_{A} = \sum_{i = 1}^{r} m_{i}$.
	Write $d + \lambda_{A} = ar + b$, with $a, b \in \ZZ$ and $0 \leq b < r$.
	Lemma~\ref{lemma:boundary-vertex-closed-form}, with $k = r$, gives the rightmost boundary contribution, and hence
	\begin{equation}
		\label{eqn:Ar-character}
		\chi_{d}^{(r)}(t) = \frac{t^{-d}}{( 1 - t )^{2}} + \frac{1 + \sum_{i = 1}^{r} t^{m_{i}}}{( 1 - t )( 1 - t^{-1} )} + t^{a} \, \frac{( b + 1 ) + ( r - b - 1 )t^{-1}}{( 1 - t^{-1} )^{2}}.
	\end{equation}

	For $r = 1$, one has $b = 0$ and $a = m_{1} + d$, and \eqref{eqn:Ar-character} recovers the $T^{\ast}\CC\PP^{1}$ character of Example~\ref{example:TCP1-example-2}.
	For $r = 2$, the parity assumption in Example~\ref{example:A2-resolution-example-2} is just the condition $b = 0$, and \eqref{eqn:Ar-character} recovers the character computed there.

	For $r = 3$, the rightmost boundary vertex $v$ has local group $\Gamma_{v} \cong \ZZ / 3\ZZ$.
	Write $d + m_{1} + m_{2} + m_{3} = 3a + b$, with $a , b \in \ZZ$ and $0 \leq b < 3$.
	If $b \neq 0$, or equivalently $d + m_{1} + m_{2} + m_{3} \not\equiv 0 \pmod{3}$, Remark~\ref{remark:convention-sensitivity} shows that substituting the geometric fibre eigenvalue for the pullback character value $\chi_{v}^{\fib}(g)$, while leaving the other factors unchanged, produces exponents in $2b / 3 + \ZZ$.
	Since all other vertex contributions are rational expressions in integral powers of $t$, these non-integral powers cannot cancel, so $A_{3}$ is the smallest member of the family that detects this substitution.
\end{remark}


\begin{example}[$T^{\ast}\CC\PP^{2}$, continued]
	\label{example:TCP2-example-2}
	We continue Example~\ref{example:TCP2-example-1} with $m , d \in \ZZ_{>0}$ and $\delta = d$.
	The local tangent weights at the nine residual fixed points are shown in Figure~\ref{fig:isotropy-TCP2}.

	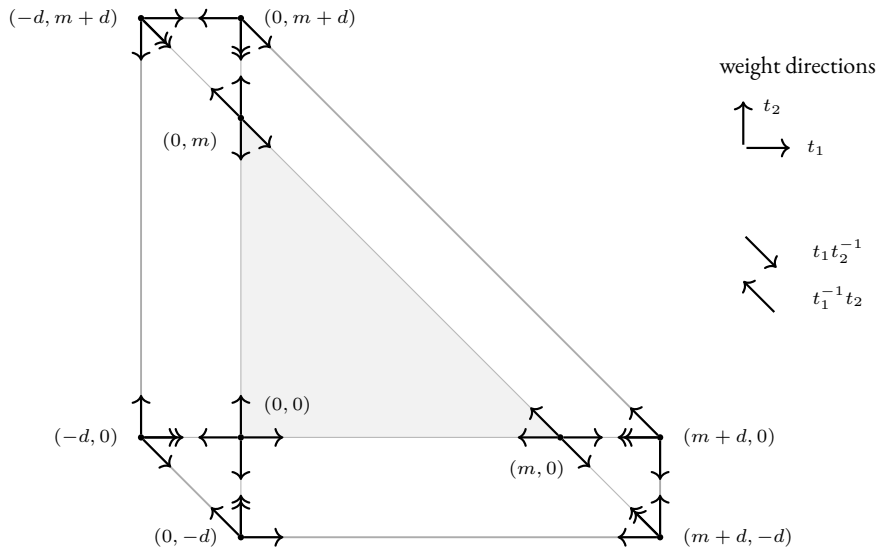
\begin{figure}[ht]
		\centering
		\begin{tikzpicture}[scale=0.88]

			\tikzset{
				polyboundary/.style={draw=gray!65, line width=0.72pt},
				polyedge/.style={draw=gray!55, line width=0.42pt},
				weight/.style={->,draw=black , line width=0.82pt , shorten <=1.2pt},
				fixedpoint/.style={font=\scriptsize},
				weightlabel/.style={font=\scriptsize},
			}

			\def\mval{4.8}
			\def\dval{1.5}

			\def\alen{0.64}
			\def\diag{0.45}

			\def\alenTwo{0.55}
			\def\diagTwo{0.39}

			\coordinate (v12) at (0 , 0);
			\coordinate (v23) at (\mval , 0);
			\coordinate (v13) at (0,\mval);

			\coordinate (bL)  at (-\dval , 0);
			\coordinate (bB)  at (0 , -\dval);
			\coordinate (bBR) at ({\mval+\dval} , -\dval);
			\coordinate (bR)  at ({\mval+\dval},0);
			\coordinate (bT)  at (0,{\mval+\dval});
			\coordinate (bTL) at (-\dval,{\mval+\dval});

			\fill[gray!10]
			(v12) -- (v23) -- (v13) -- cycle;

			\draw[polyboundary]
			(bL) -- (bB) -- (bBR) -- (bR)
			-- (bT) -- (bTL) -- cycle;

			\draw[polyedge] (bB)  -- (bT);
			\draw[polyedge] (bL)  -- (bR);
			\draw[polyedge] (bTL) -- (bBR);

			\node[fixedpoint , above right=5pt]       at (v12)  {$( 0 , 0 )$};
			\node[fixedpoint , below=5pt]       at ($(v23)+(-0.35 , 0)$)  {$( m , 0 )$};
			\node[fixedpoint , left=5pt]       at ($(v13)+(0 , -0.35)$)  {$( 0 , m )$};
			\node[fixedpoint , left=5pt]       at (bL)  {$( -d , 0 )$};
			\node[fixedpoint , left=5pt]      at (bB)  {$( 0 , -d )$};
			\node[fixedpoint , right=5pt] at (bBR) {$( m + d , -d )$};
			\node[fixedpoint , right=5pt]      at (bR)  {$( m + d , 0 )$};
			\node[fixedpoint , right=5pt] at (bT)  {$( 0 , m + d )$};
			\node[fixedpoint , left=5pt] at (bTL) {$( -d , m + d )$};

			\draw[weight] (v12) -- ++(\alen , 0);
			\draw[weight] (v12) -- ++(0,\alen);
			\draw[weight] (v12) -- ++(-\alen , 0);
			\draw[weight] (v12) -- ++(0 , -\alen);

			\draw[weight] (v23) -- ++(\alen , 0);
			\draw[weight] (v23) -- ++(-\alen , 0);
			\draw[weight] (v23) -- ++(-\diag,\diag);
			\draw[weight] (v23) -- ++(\diag , -\diag);

			\draw[weight] (v13) -- ++(0,\alen);
			\draw[weight] (v13) -- ++(0 , -\alen);
			\draw[weight] (v13) -- ++(-\diag,\diag);
			\draw[weight] (v13) -- ++(\diag , -\diag);

			\draw[weight] (bL) -- ++(\alen , 0);
			\draw[weight] (bL) -- ++(\alenTwo , 0);
			\draw[weight] (bL) -- ++(0,\alen);
			\draw[weight] (bL) -- ++(\diag , -\diag);

			\draw[weight] (bB) -- ++(0,\alen);
			\draw[weight] (bB) -- ++(0,\alenTwo);
			\draw[weight] (bB) -- ++(\alen , 0);
			\draw[weight] (bB) -- ++(-\diag,\diag);

			\draw[weight] (bR) -- ++(-\alen , 0);
			\draw[weight] (bR) -- ++(-\alenTwo , 0);
			\draw[weight] (bR) -- ++(-\diag,\diag);
			\draw[weight] (bR) -- ++(0 , -\alen);

			\draw[weight] (bBR) -- ++(-\diag,\diag);
			\draw[weight] (bBR) -- ++(-\diagTwo,\diagTwo);
			\draw[weight] (bBR) -- ++(0,\alen);
			\draw[weight] (bBR) -- ++(-\alen , 0);

			\draw[weight] (bT) -- ++(0 , -\alen);
			\draw[weight] (bT) -- ++(0 , -\alenTwo);
			\draw[weight] (bT) -- ++(\diag , -\diag);
			\draw[weight] (bT) -- ++(-\alen , 0);

			\draw[weight] (bTL) -- ++(\diag , -\diag);
			\draw[weight] (bTL) -- ++(\diagTwo , -\diagTwo);
			\draw[weight] (bTL) -- ++(\alen , 0);
			\draw[weight] (bTL) -- ++(0 , -\alen);

			\coordinate (L) at (7.55 , 4.35);

			\node[font=\small , anchor=west] at (7.05 , 5.55) {weight directions};

			\draw[weight] (L) -- ++(0.70 , 0);
			\node[weightlabel , anchor=west] at ($(L)+(0.82 , 0)$) {$t_1$};

			\draw[weight] (L) -- ++(0 , 0.70);
			\node[weightlabel , anchor=west] at ($(L)+(0.16 , 0.62)$) {$t_2$};

			\coordinate (D1) at (7.55 , 3.05);
			\coordinate (D2) at (8.05 , 1.85);

			\draw[weight] (D1) -- ++(0.50 , -0.50);
			\node[weightlabel , anchor=west]
			at (8.45 , 2.80) {$t_1t_2^{-1}$};

			\draw[weight] (D2) -- ++(-0.50 , 0.50);
			\node[weightlabel , anchor=west]
			at (8.45 , 2.10) {$t_1^{-1}t_2$};

			\foreach \p in {v12,v23,v13,bL,bB,bBR,bR,bT,bTL}
			\fill[black] (\p) circle (1.35pt);
		\end{tikzpicture}

		\caption{
			Local tangent weights at the nine residual fixed points of the cut of $T^{\ast}\CC\PP^{2}$.
			The arrows indicate the geometric tangent-weight directions.
			At each boundary vertex, the cutting weight occurs with multiplicity two and is represented by two closely spaced arrowheads on the same ray.
		}
		\label{fig:isotropy-TCP2}
	\end{figure}

	To express the factorisation of the total character conveniently, we write each vertex contribution as a product of one $\mathsf{A}$-term and one $\mathsf{B}$-term.
	We first define
	\[
		\mathsf{A}_{12} := \frac{1}{( 1 - t_{1} )( 1 - t_{2} )}, \quad
		\mathsf{A}_{23} := \frac{t_{1}^{m + d}}{( 1 - t_{1}^{-1} )( 1 - t_{1}^{-1}t_{2} )}, \quad
		\mathsf{A}_{13} := \frac{t_{2}^{m + d}}{( 1 - t_{2}^{-1} )( 1 - t_{1}t_{2}^{-1} )},
	\]
	whose sum is the Brion expression for the lattice-point character of the triangle with vertices $(0 , 0)$, $(m+d , 0)$, and $(0 , m+d)$:
	\[
		P_{m + d}(t_{1} , t_{2})
		:= \mathsf{A}_{12} + \mathsf{A}_{23} + \mathsf{A}_{13} = \sum\nolimits_{\substack{r , q \geq 0 \\ r + q \leq m + d}}t_{1}^{r}t_{2}^{q}.
	\]
	For the triangle with vertices $(0 , 0)$, $(-d , 0)$, and $(0 , -d)$, we similarly define
	\[
		\mathsf{B}_{0}(t_{1} , t_{2}) := \frac{1}{( 1 - t_{1}^{-1} )( 1 - t_{2}^{-1} )}, \quad
		\mathsf{B}_{1}(t_{1} , t_{2}) := \frac{t_{1}^{-d}}{( 1 - t_{1} )( 1 - t_{1}t_{2}^{-1} )}, \quad
		\mathsf{B}_{2}(t_{1} , t_{2}) := \frac{t_{2}^{-d}}{( 1 - t_{2} )( 1 - t_{1}^{-1}t_{2} )},
	\]
	and set
	\begin{equation}
		\label{eqn:Q-rational-function}
		Q_{d}(t_{1} , t_{2})
		:= \mathsf{B}_{0}(t_{1} , t_{2}) + \mathsf{B}_{1}(t_{1} , t_{2}) + \mathsf{B}_{2}(t_{1} , t_{2}) = \sum\nolimits_{\substack{k , l \geq 0\\ k + l \leq d}}t_{1}^{-k}t_{2}^{-l}.
	\end{equation}
	With the tangent weights in Figure~\ref{fig:isotropy-TCP2}, Corollary~\ref{corollary:character-smooth-hypertoric-cut} gives the following factorisations of the nine vertex contributions:
	\begin{align*}
		\chi_{( 0 , 0 )} & = \mathsf{A}_{12}\mathsf{B}_{0}, & \chi_{( -d , 0 )}    & = \mathsf{A}_{12}\mathsf{B}_{1}, & \chi_{( 0 , -d )}     & = \mathsf{A}_{12}\mathsf{B}_{2}, \\
		\chi_{( m , 0 )} & = \mathsf{A}_{23}\mathsf{B}_{1}, & \chi_{( m + d , 0 )} & = \mathsf{A}_{23}\mathsf{B}_{0}, & \chi_{( m + d , -d )} & = \mathsf{A}_{23}\mathsf{B}_{2}, \\
		\chi_{( 0 , m )} & = \mathsf{A}_{13}\mathsf{B}_{2}, & \chi_{( 0 , m + d )} & = \mathsf{A}_{13}\mathsf{B}_{0}, & \chi_{( -d , m + d )} & = \mathsf{A}_{13}\mathsf{B}_{1},
	\end{align*}
	where the arguments $(t_{1},t_{2})$ have been suppressed.
	Summing the nine products gives the factorisation $\chi_{d}(t_{1} , t_{2}) = P_{m + d}(t_{1} , t_{2}) Q_{d}(t_{1} , t_{2})$.
	As in Example~\ref{example:TCP1-example-2}, this regrouping gives the localisation form of the base--fibre description of the cut polarisation; compare the projective-bundle computation in Remark~\ref{remark:projective-bundle-TCP2-computation}.
	Evaluating these finite sums at $(t_{1},t_{2})=(1 , 1)$ yields the dimension formula for the $T^{\ast}\CC\PP^{2}$ cut:
	\[
		\dim_{\CC} H^{0}( M_{\nu}^{\leq d} ; \mcL_{\nu}^{\leq d} ) = \binom{m + d + 2}{2} \binom{d + 2}{2}.
	\]
\end{example}


\begin{example}[$T^{\ast}(\CC\PP^{1} \times \CC\PP^{1})$, continued]
	\label{example:TCP1xCP1-example-2}
	We continue Example~\ref{example:TCP1xCP1-example-1} with $m_{1},m_{2},d \in \ZZ_{>0}$ and $\delta = d$, using the function $Q_{d}$ from \eqref{eqn:Q-rational-function} in the previous calculation.
	Here, the moment polyptych $\Delta_{\nu}^{\leq d}$ has octagonal support and twelve vertices rather than nine.
	The local tangent weights at these vertices are shown in Figure~\ref{fig:isotropy-TCP1xCP1}.

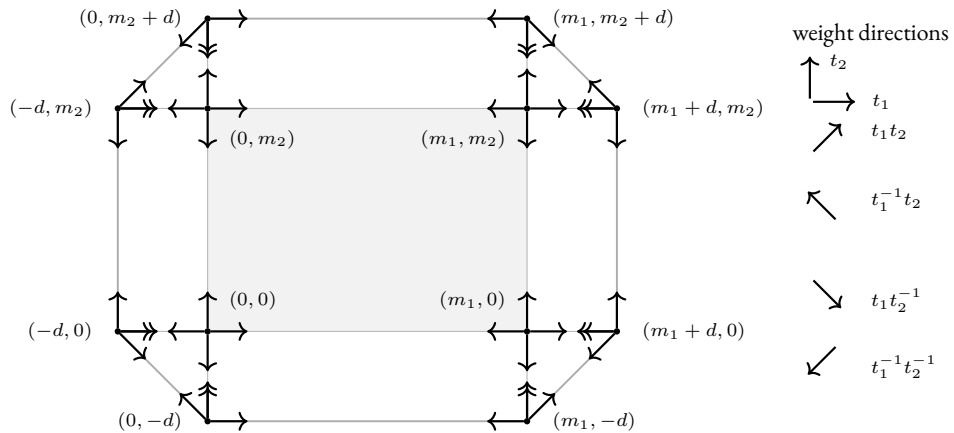
\begin{figure}[ht]
    \centering
    \begin{tikzpicture}[scale=0.88]
        \tikzset{
            polyboundary/.style={draw=gray!65 , line width=0.72pt},
            polyedge/.style={draw=gray!55 , line width=0.42pt},
            weight/.style={->,draw=black , line width=0.82pt , shorten <=1.2pt},
            fixedpoint/.style={font=\scriptsize},
            weightlabel/.style={font=\scriptsize},
        }

        \def\moneval{4.8}
        \def\mtwoval{3.35}
        \def\dval{1.35}

        \def\alen{0.60}
        \def\alenTwo{0.51}
        \def\diag{0.42}
        \def\diagTwo{0.36}

        \coordinate (v00) at (0 , 0);
        \coordinate (v10) at (\moneval , 0);
        \coordinate (v11) at (\moneval,\mtwoval);
        \coordinate (v01) at (0,\mtwoval);

        \coordinate (bL1) at (-\dval , 0);
        \coordinate (bB1) at (0 , -\dval);
        \coordinate (bB2) at (\moneval , -\dval);
        \coordinate (bR1) at ({\moneval+\dval},0);
        \coordinate (bR2) at ({\moneval+\dval},\mtwoval);
        \coordinate (bT2) at (\moneval,{\mtwoval+\dval});
        \coordinate (bT1) at (0,{\mtwoval+\dval});
        \coordinate (bL2) at (-\dval,\mtwoval);

        \fill[gray!10] (v00) -- (v10) -- (v11) -- (v01) -- cycle;
        \draw[polyboundary]
            (bL1) -- (bB1) -- (bB2) -- (bR1) -- (bR2)
            -- (bT2) -- (bT1) -- (bL2) -- cycle;
        \draw[polyedge] (bB1) -- (bT1);
        \draw[polyedge] (bB2) -- (bT2);
        \draw[polyedge] (bL1) -- (bR1);
        \draw[polyedge] (bL2) -- (bR2);

        \node[fixedpoint , anchor=south west] at ($(v00)+(0.18 , 0.18)$) {$(0 , 0)$};
        \node[fixedpoint , anchor=south east] at ($(v10)+(-0.18 , 0.18)$) {$(m_1 , 0)$};
        \node[fixedpoint , anchor=north east] at ($(v11)+(-0.18 , -0.18)$) {$(m_1 , m_2)$};
        \node[fixedpoint , anchor=north west] at ($(v01)+(0.18 , -0.18)$) {$(0 , m_2)$};

        \node[fixedpoint , left=6pt]  at (bL1) {$(-d , 0)$};
        \node[fixedpoint , left=6pt] at (bB1) {$(0 , -d)$};
        \node[fixedpoint , right=6pt] at (bB2) {$(m_1 , -d)$};
        \node[fixedpoint , right=6pt] at (bR1) {$(m_1+d , 0)$};
        \node[fixedpoint , right=6pt] at (bR2) {$(m_1+d , m_2)$};
        \node[fixedpoint , right=6pt] at (bT2) {$(m_1 , m_2+d)$};
        \node[fixedpoint , left=6pt] at (bT1) {$(0 , m_2+d)$};
        \node[fixedpoint , left=6pt]  at (bL2) {$(-d , m_2)$};

        \foreach \p in {v00,v10,v11,v01}{
            \draw[weight] (\p) -- ++(\alen , 0);
            \draw[weight] (\p) -- ++(-\alen , 0);
            \draw[weight] (\p) -- ++(0,\alen);
            \draw[weight] (\p) -- ++(0 , -\alen);
        }

        \draw[weight] (bL1) -- ++(\alen , 0);
        \draw[weight] (bL1) -- ++(\alenTwo , 0);
        \draw[weight] (bL1) -- ++(0 , \alen);
        \draw[weight] (bL1) -- ++(\diag , -\diag);

        \draw[weight] (bB1) -- ++(0 , \alen);
        \draw[weight] (bB1) -- ++(0 , \alenTwo);
        \draw[weight] (bB1) -- ++(\alen , 0);
        \draw[weight] (bB1) -- ++(-\diag , \diag);

        \draw[weight] (bB2) -- ++(0 , \alen);
        \draw[weight] (bB2) -- ++(0 , \alenTwo);
        \draw[weight] (bB2) -- ++(-\alen , 0);
        \draw[weight] (bB2) -- ++(\diag , \diag);

        \draw[weight] (bR1) -- ++(-\alen , 0);
        \draw[weight] (bR1) -- ++(-\alenTwo , 0);
        \draw[weight] (bR1) -- ++(0 , \alen);
        \draw[weight] (bR1) -- ++(-\diag , -\diag);

        \draw[weight] (bR2) -- ++(-\alen , 0);
        \draw[weight] (bR2) -- ++(-\alenTwo , 0);
        \draw[weight] (bR2) -- ++(0 , -\alen);
        \draw[weight] (bR2) -- ++(-\diag , \diag);

        \draw[weight] (bT2) -- ++(0 , -\alen);
        \draw[weight] (bT2) -- ++(0 , -\alenTwo);
        \draw[weight] (bT2) -- ++(-\alen , 0);
        \draw[weight] (bT2) -- ++(\diag , -\diag);

        \draw[weight] (bT1) -- ++(0 , -\alen);
        \draw[weight] (bT1) -- ++(0 , -\alenTwo);
        \draw[weight] (bT1) -- ++(\alen , 0);
        \draw[weight] (bT1) -- ++(-\diag , -\diag);

        \draw[weight] (bL2) -- ++(\alen , 0);
        \draw[weight] (bL2) -- ++(\alenTwo , 0);
        \draw[weight] (bL2) -- ++(0 , -\alen);
        \draw[weight] (bL2) -- ++(\diag , \diag);

        \node[font=\small , anchor=west] at (8.65 , 4.45) {weight directions};
        \coordinate (L) at (9.05 , 3.45);
        \draw[weight] (L) -- ++(0.68 , 0);
        \node[weightlabel , anchor=west] at ($(L)+(0.80 , 0)$) {$t_1$};
        \draw[weight] (L) -- ++(0 , 0.68);
        \node[weightlabel , anchor=west] at ($(L)+(0.16 , 0.60)$) {$t_2$};

        \coordinate (D1) at (9.07 , 2.67);
        \coordinate (D2) at (9.47 , 1.65);
        \coordinate (D3) at (9.07 , 0.80);
        \coordinate (D4) at (9.47 , -0.20);

        \draw[weight] (D1) -- ++(0.45 , 0.45);
        \node[weightlabel , anchor=west] at (9.83 , 3.00) {$t_1t_2$};
        \draw[weight] (D2) -- ++(-0.45 , 0.45);
        \node[weightlabel , anchor=west] at (9.83 , 1.98) {$t_1^{-1}t_2$};
        \draw[weight] (D3) -- ++(0.45 , -0.45);
        \node[weightlabel , anchor=west] at (9.83 , 0.47) {$t_1t_2^{-1}$};
        \draw[weight] (D4) -- ++(-0.45 , -0.45);
        \node[weightlabel , anchor=west] at (9.83 , -0.53) {$t_1^{-1}t_2^{-1}$};

        \foreach \p in {v00,v10,v11,v01,bL1,bB1,bB2,bR1,bR2,bT2,bT1,bL2}
            \fill[black] (\p) circle (1.35pt);
    \end{tikzpicture}

    \caption{
        Local tangent weights at the twelve residual fixed points of the cut of $T^{\ast}(\CC\PP^{1} \times \CC\PP^{1})$.
        The arrows indicate geometric tangent-weight directions.
        At each boundary vertex, the cutting weight occurs with multiplicity two, represented by two closely spaced arrowheads on the same ray.
    }
    \label{fig:isotropy-TCP1xCP1}
\end{figure}

	The vertex contributions can be grouped into four triples associated with the base vertices $(0 , 0)$, $(m_{1} , 0)$, $(0 , m_{2})$, and $(m_{1} , m_{2})$, whose rational factors are
	\[
		\begin{split}
			 & \mathsf{A}_{00} := \frac{1}{( 1 - t_{1} )( 1 - t_{2} )},
			\qquad
			\mathsf{A}_{10} := \frac{t_{1}^{m_{1}}}{( 1 - t_{1}^{-1} )( 1 - t_{2} )},    \\
			 & \mathsf{A}_{01} := \frac{t_{2}^{m_{2}}}{( 1 - t_{1} )( 1 - t_{2}^{-1} )},
			\qquad
			\mathsf{A}_{11} := \frac{t_{1}^{m_{1}}t_{2}^{m_{2}}}{( 1 - t_{1}^{-1} )( 1 - t_{2}^{-1} )}.
		\end{split}
	\]
	In the fibre sum $Q_{d}$ of \eqref{eqn:Q-rational-function}, a base vertex with first coordinate $m_{1}$ requires the substitution $t_{1} \mapsto t_{1}^{-1}$, while a base vertex with second coordinate $m_{2}$ requires $t_{2} \mapsto t_{2}^{-1}$.
	These substitutions follow the signs of the tangent weights on the base rectangle; summing the three contributions in each group gives
	\[
		\chi_{d}(t_{1} , t_{2}) = \mathsf{A}_{00}Q_{d}(t_{1} , t_{2}) + \mathsf{A}_{10}Q_{d}(t_{1}^{-1} , t_{2})
		+ \mathsf{A}_{01}Q_{d}(t_{1} , t_{2}^{-1}) + \mathsf{A}_{11}Q_{d}(t_{1}^{-1} , t_{2}^{-1}).
	\]
	Expanding each $Q_{d}$ using \eqref{eqn:Q-rational-function} and grouping the terms for each fixed pair $(k , l)$ yields
	\[
		\begin{split}
			\chi_{d}( t_{1} , t_{2} )
			 & = \sum\nolimits_{ \substack{k , l \geq 0 \\ k + l \leq d} }
			\left(
			\frac{ t_{1}^{-k} }{1 - t_{1}} + \frac{t_{1}^{m_{1} + k}}{1 - t_{1}^{-1} }
			\right)
			\left(
			\frac{ t_{2}^{-l} }{ 1 - t_{2} } + \frac{ t_{2}^{m_{2} + l} }{ 1 - t_{2}^{-1} }
			\right)                                                        \\
			 & = \sum\nolimits_{ \substack{k , l \geq 0 \\ k + l \leq d} }
			t_{1}^{-k}t_{2}^{-l}
			\left( \sum_{r = 0}^{m_{1} + 2k} t_{1}^{r} \right)
			\left( \sum_{q = 0}^{m_{2} + 2l} t_{2}^{q} \right),
		\end{split}
	\]
	where each rational factor on the first line is a translated rank-one Brion sum, evaluated using \eqref{eqn:rank-one-Brion-sign-check}.
	Evaluating at $( t_{1} , t_{2} ) = ( 1 , 1 )$ gives
	\[
		\begin{split}
			\dim_{\CC} H^{0}( M_{\nu}^{\leq d} ; \mcL_{\nu}^{\leq d} )
			 & = \sum_{k = 0}^{d} \sum_{l = 0}^{d - k}
			( m_{1} + 2k + 1 )( m_{2} + 2l + 1 )                                                                          \\
			 & = \frac{ ( d + 1 )( d + 2 ) ( d^{2} + 2d(m_{1} + m_{2}) + 3d + 3m_{1}m_{2} + 3m_{1} + 3m_{2} + 3 ) }{ 6 }.
		\end{split}
	\]
	In both rank-two examples, every arrangement vertex lies in $\Delta_{\emptyset}$, so $M_{\nu} = U_{\nu} \cong T^{\ast}X_{\nu}$ \cite[Theorems~7.1 and~7.2]{BielawskiDancer2000} and the cuts are smooth by Corollary~\ref{corollary:smooth-cut-space-characterisation}.
	These cuts are projective bundles over their distinguished zero sections and have trivial local groups; the calculations above assemble their characters from the corresponding base and fibre Brion sums.
	For $T^{\ast}\CC\PP^{2}$, this gives the product $P_{m + d}Q_{d}$, whereas for $T^{\ast}(\CC\PP^{1} \times \CC\PP^{1})$ it gives the displayed sum of products indexed by $k + l \leq d$.
	On the other hand, in Example~\ref{example:A2-resolution-example-2}, the additional interior and rightmost orbifold boundary vertex contributions give a correction to the character of $T^{\ast}\CC\PP^{1}$ in comparison.
\end{example}

%% file: sections/7-algebraic-cutting-combinatorial-quantisation.tex
\section{Algebraic Cutting and Combinatorial Quantisation}
\label{section:section-7-algebraic-cutting-combinatorial-quantisation}

At positive regular integral levels in the outer cut chamber, the cut quantisation admits three descriptions: the localisation index of Section~\ref{section:section-6-equivariant-quantisation-hypertoric-algebraic-cutting}, an algebraic description, and a combinatorial description.
Here, we consider the other two by realising the cut as a GIT quotient, identifying its section space with a finite filtration piece of $H^{0}( M_{\nu} ; \mcL_{\nu} )$, and then proving the higher-cohomology vanishing needed to identify that space with the index.
We then compute the residual-weight multiplicities from the coordinate ring, obtaining a weighted lattice-point formula for the same equivariant character computed by the localisation formula of Theorem~\ref{theorem:character-hypertoric-orbifold}.

\subsection{Algebraic Cutting}
\label{subsection:subsection-7.1-algebraic-cutting}

The localisation formula computes the equivariant index of each compact cut.
To relate these cuts to the quantisation of the original non-compact hypertoric variety $M_{\nu}$, we decompose $H^{0}( M_{\nu} ; \mcL_{\nu} )$ into its fibre-rotation weight spaces, which we call the \emph{fibre-rotation grading}.
The algebraic cutting method used here expresses the cut level $d \in \ZZ_{> 0}$ as a condition on these weights, producing a filtration whose pieces are identified with $H^{0}(M_{\nu}^{\leq d} ; \mcL_{\nu}^{\leq d})$ at the positive regular integral levels in the outer cut chamber.

In \cite{EdidinGraham1998}, Edidin and Graham formulate an algebraic counterpart to Lerman's symplectic cut \cite{Lerman1995}, and they were the first ones to call it algebraic cutting.
In their construction, one forms a GIT quotient of $M \times \mathbb{A}^{1}$ by an anti-diagonal one-dimensional torus action, with the cut level encoded by a choice of linearisation.
Here, we use the diagonal cutting convention used in Subsection~\ref{subsection:subsection-4.2-cut-spaces} (i.e., opposite to theirs), so that the sub-level region $\{ \Phi \leq d \}$ is retained.
Since $M_{\nu}$ is already realised as a GIT quotient and fibre rotation admits a holomorphic $\CC^{\ast}$-complexification, its algebraic cut can be formulated using graded section rings.
This description allows us to compare the cutting-invariant sections on $M_{\nu} \times \CC$ with the sections of the pre-quantum line orbibundle on the cut space.

Fix a generic integral parameter $\nu \in \mfk_{\ZZ}^{\ast}$, retaining the fixed integral lift $\lambda$ and the pre-quantum line orbibundle $\mcL_{\nu} \ra M_{\nu}$ from Section~\ref{section:section-2-hypertoric-varieties}.
We consider its section ring
\[
	R := \bigoplus\nolimits_{m \geq 0} H^{0}( M_{\nu} ; \mcL_{\nu}^{\otimes m} ),
\]
whose degree $m$ records the tensor power.
The hypertoric variety $M_{\nu}$ is semiprojective in the sense fixed in Section~\ref{section:section-2-hypertoric-varieties}, and $\mcL_{\nu}$ is relatively ample.
By Proposition~\ref{prop:positive-ray-semi-invariants}~---~proved in Subsection~\ref{subsection:subsection-7.3-combinatorial-interpretation} independently of what follows~---~$R$ is isomorphic as a graded $\CC$-algebra to the semi-invariant ring $R_{\nu}$ of \eqref{eqn:GIT-hypertoric-definition}.
It follows that $R$ is a finitely generated graded $\CC$-algebra and $M_{\nu} \cong \Proj R$ as a projective scheme over $\Spec R_{0}$.

The fibre-rotation action extends holomorphically to $\CC^{\ast}$ and lifts to $\mcL_{\nu}$ by the complexification of the $U_{1}$-lift fixed in Subsection~\ref{subsection:subsection-5.3-pre-quantum-cut-space}.
On the trivial line bundle $L_{\nu} = \mcN \times \CC \ra \mcN$, the lift acts by $\tau \cdot ((z , w) , a)=((z , \tau w) , a)$ for $\tau \in \CC^{\ast}$, leaving the fibre coordinate $a$ unchanged.
After descent, it acts on each $R_{m} = H^{0}(M_{\nu} ; \mcL_{\nu}^{\otimes m})$ by the pullback convention \eqref{eqn:pullback-action-on-sections}, and commutes with the residual $T^{n}$-action.

On the auxiliary affine line $\CC_{\zeta}$, the geometric action is $\tau \cdot \zeta = \tau\zeta$.
For an integer $d$, let $\varpi_{d} : \CC^{\ast} \ra \CC^{\ast}$ be the character $\varpi_{d}(\tau) = \tau^{d}$, and equip the trivial line bundle $\mcO_{\CC}(\varpi_{d}) \ra \CC_{\zeta}$ with geometric fibre weight $+d$, as in the auxiliary construction of Subsection~\ref{subsection:subsection-5.3-pre-quantum-cut-space}.
Write $y(\zeta) = \zeta$ for the coordinate function and $\one_{d}$ for the nowhere-vanishing constant section of $\mcO_{\CC}(\varpi_{d})$ in its standard trivialisation so that, under the pullback action \eqref{eqn:pullback-action-on-sections}, they have weights $+1$ and $-d$, respectively.
The section ring of this line bundle is
\[
	S := \bigoplus\nolimits_{m \geq 0} H^{0}( \CC ; \mcO_{\CC}(\varpi_{d})^{\otimes m} ),
\]
and the section $\one_{d}^{\otimes m}$ trivialises the $m$-th tensor power, so $S_{m} \cong \CC[y]\cdot\one_{d}^{\otimes m}$.
Multiplication is induced by tensor product, making $S$ a polynomial algebra over $\CC[y]$ generated by $\one_{d}$ in section-ring degree one; note that the coefficient $y$ has section-ring degree zero.
In particular, $\Proj S \cong \Spec S_{0} = \Spec \CC[y] \cong \CC$.

To form the section ring of the product line orbibundle $\mcL_{\nu} \boxtimes \mcO_{\CC}(\varpi_{d}) \ra M_{\nu} \times \CC$, we pair the same tensor power in each factor by forming the \textbf{Segre product}
\[
	R~\#_{\CC}~S :=\bigoplus\nolimits_{m \geq 0}R_{m} \otimes_{\CC} S_{m}.
\]
For its invariant-ring description, let an auxiliary $\CC^{\ast}$ act on $R \otimes_{\CC} S$ with weight $a - b$ on $R_{a} \otimes_{\CC} S_{b}$.
Clearly its invariants are the terms with $a = b$, so $R~\#_{\CC}~S = (R \otimes_{\CC} S)^{\CC^{\ast}}$ and it is graded by this common degree.
Moreover, since $R \otimes_{\CC} S$ is finitely generated and the auxiliary torus is reductive, the Segre product is finitely generated as well.
The construction is studied under the name Segre product in \cite[Chapter~4]{GotoWatanabe1978}; see also \cite[\S2.2]{BrionProcesi1990} and, as the \emph{Cartesian product}, \cite[Exercise~II.5.11]{Hartshorne1977}.
Note, however, that those references impose degree-zero hypotheses in the statements cited, whereas here $R_{0}$ need not be equal to $\CC$; only the definition above and invariant-ring description are used below.

For the section rings $R$ and $S$ above, the common degree-$m$ piece is
\[
	( R~\#_{\CC}~S )_{m} \cong H^{0}( M_{\nu} \times \CC ; \mcL_{\nu}^{\otimes m} \boxtimes \mcO_{\CC}(\varpi_{d})^{\otimes m} ).
\]
We now take the invariants for the cutting $\CC^{\ast}$-action:
\[
	( R~\#_{\CC}~S )^{\CC^{\ast}} = \bigoplus_{m \geq 0}
	( R~\#_{\CC}~S )_{m}^{\CC^{\ast}} = \bigoplus_{m \geq 0}
	H^{0}
	\left( M_{\nu} \times \CC ; \big( \mcL_{\nu} \boxtimes \mcO_{\CC}( \varpi_{d} ) \big)^{\otimes m} \right)^{\CC^{\ast}},
\]
which is the invariant section ring used to define the algebraic cut.

\begin{definition}[Algebraic Cut]
	\label{def:algebraic-cut}
	For an integer $d$, the \textbf{algebraic cut} of $M_{\nu}$ at level $d$ is the GIT quotient of $M_{\nu} \times \CC$ for the cutting $\CC^{\ast}$-action linearised on $\mcL_{\nu} \boxtimes \mcO_{\CC}(\varpi_{d})$,
	\begin{equation}
		\label{eqn:algebraic-cut}
		M_{\nu}^{\alg , \leq d} := \Proj \left( ( R~\#_{\CC}~S )^{\CC^{\ast}} \right).
	\end{equation}
\end{definition}
Since the cutting $\CC^{\ast}$ is reductive, the invariant ring $( R~\#_{\CC}~S )^{\CC^{\ast}}$ is finitely generated as well.
For a regular integral level $d>0$, the algebraic cut agrees with the symplectic cut of Subsection~\ref{subsection:subsection-4.2-cut-spaces}.
Indeed, the cutting $\CC^{\ast}$-action is the complexification of the $U_{1}$-action on $M_{\nu}\times\CC$, and the character $\varpi_{d}$ specifies the moment-map normalisation $\rho-d$ used in Subsection~\ref{subsection:subsection-5.3-pre-quantum-cut-space}.
The positive reals $\RR_{>0} \subset \CC^{\ast}$ act by the gradient flow of $\rho$, so $\rho$ is strictly increasing along every non-trivial $\RR_{>0}$-orbit.
Since $\rho$ is proper by Lemma~\ref{lemma:proper-moment-map} and the orbit maps are algebraic, the limit of $\tau \cdot (p , \zeta)$ as $\tau \to 0$ exists for every point and is a fixed point, at which $\rho$ is a critical value of $\Phi$; as $\tau \to \infty$, the limit exists if and only if $\zeta = 0$ and $p \in \mcC$ \cite[Corollary~3.6]{HaradaProudfoot2004}, and otherwise $\rho \to \infty$ along the orbit.
By Proposition~\ref{prop:positive-ray-semi-invariants}, $( R~\#_{\CC}~S )^{\CC^{\ast}}$ is the ring of semi-invariants on $\mcN \times \CC$ for the torus $K_{\CC} \times \CC^{\ast}$ and the characters $m(\nu , d)$, so $M_{\nu}^{\alg , \leq d}$ is the corresponding GIT quotient of the affine variety $\mcN \times \CC$.
By the form of the Kempf--Ness theorem for characters due to King \cite[Proposition~3.21 and Corollary~3.22]{Nakajima1999}, applied to the closed invariant subvariety $\mcN \times \CC$ of $T^{\ast}\CC^{N} \times \CC$ as in \cite[Theorem~3.24]{Nakajima1999}, a point of $\mcN \times \CC$ is semistable if and only if the closure of its $K_{\CC} \times \CC^{\ast}$-orbit meets $( \phi_{\RR}^{-1}(\nu) \times \CC ) \cap \rho^{-1}(d)$; reducing first by $K$ through \eqref{eqn:hypertoric-kempf-ness}, this says that the closure of the $\CC^{\ast}$-orbit of $(p , \zeta)$ in $M_{\nu} \times \CC$ meets $\rho^{-1}(d)$.
This closure is the orbit together with its limits, at which $\rho \neq d$, so $(p , \zeta)$ is semistable if and only if its orbit meets $\rho^{-1}(d)$, in which case it meets $\rho^{-1}(d)$ in a single $U_{1}$-orbit.
The limits of a semistable orbit are fixed points with $\rho \neq d$, hence unstable, so every semistable orbit is closed in the semistable locus, the GIT quotient is geometric, and the inclusion of $\rho^{-1}(d)$ into the semistable locus induces
\[
	M_{\nu}^{\alg , \leq d} \cong \rho^{-1}(d)/U_{1} = M_{\nu}^{\leq d},
\]
with the complex structure on $\rho^{-1}(d)/U_{1}$ induced from this quotient, as in \cite[Theorems~3.2 and~3.4]{Hausel1998}.
Under this identification, the descended linearised orbibundle is the cut pre-quantum line orbibundle $\mcL_{\nu}^{\leq d}$ constructed in Subsection~\ref{subsection:subsection-5.3-pre-quantum-cut-space}.
We henceforth write $M_{\nu}^{\leq d}$ for the common algebraic and symplectic cut at a positive regular integral level $d$.
At an integral critical level, \eqref{eqn:algebraic-cut} still defines the corresponding algebraic GIT quotient, although it need not be a compact K\"ahler orbifold.

\begin{remark}
	The construction above applies more generally to a polarised semiprojective variety equipped with an algebraic $\CC^{\ast}$-action and a compatible linearisation, provided that the associated section ring is finitely generated.
	Our section-space and combinatorial results that follow, however, use the additional structure of hypertoric varieties.
	In particular, we show below that the unstable locus at an outer-chamber cut level is $\mcC \times \{0\}$, whose codimension provides the condition needed for the extension isomorphism used in Theorem~\ref{theorem:algebraic-cut-weight-space}.
	The hypertoric coordinate ring then allows us to compute the residual $T^{n}$-weight multiplicities explicitly in Theorem~\ref{theorem:combinatorial-multiplicity}.
\end{remark}

The cutting linearisation makes the bound on fibre-rotation weights explicit.
For $s_{i} \in H^{0}(M_{\nu} ; \mcL_{\nu})_{i}$ and $j \in \ZZ_{\geq 0}$, the product $s_{i}y^{j} \otimes \one_{d}$ is a section of $\mcL_{\nu} \boxtimes \mcO_{\CC}(\varpi_{d})$ on $M_{\nu} \times \CC$.
Since $y$ and $\one_{d}$ have pullback weights $+1$ and $-d$, its cutting weight is $i + j - d$, and hence $s_{i} y^{j} \otimes \one_{d}$ has cutting weight zero if and only if $i + j - d = 0$.
A non-negative exponent $j$ exists only when $i \leq d$, in which case $j = d - i$.
For every integer $d \geq 0$, we therefore define the \textbf{cut-filtration piece}
\[
	F_{d} H^{0}( M_{\nu} ; \mcL_{\nu} ) := \bigoplus\nolimits_{i \leq d} H^{0}( M_{\nu} ; \mcL_{\nu} )_{i}.
\]
These spaces form an increasing filtration that is exhaustive because the $\CC^{\ast}$-action is rational, so every section is a finite sum of weight vectors.
Theorem~\ref{theorem:combinatorial-multiplicity} will show that the grading is non-negative and that each $F_{d}$ is finite-dimensional.
For $d \geq 1$, the successive quotients recover the individual fibre-rotation weight spaces:
\[
	\frac{ F_{d}H^{0}( M_{\nu} ; \mcL_{\nu} ) }{ F_{d - 1}H^{0}( M_{\nu} ; \mcL_{\nu} ) } \cong H^{0}( M_{\nu} ; \mcL_{\nu} )_{d}.
\]

To compare global invariant sections with sections on the cut, we first show that sections extend uniquely from the semistable locus to $M_{\nu} \times \CC$ in the outer cut chamber.
Fix $d \in \ZZ_{>0} \cap \Cout$, where $\Cout = ( c_{\max} , \infty )$ is defined in \eqref{eqn:outer-cut-chamber}, and write
\[
	U_{d} := ( M_{\nu} \times \CC )^{\sstable}(\varpi_{d}),
	\qquad
	\mcG_{d} := \mcL_{\nu} \boxtimes \mcO_{\CC}(\varpi_{d}) \lra M_{\nu} \times \CC.
\]
Here, $U_{d}$ is the semistable locus for the cutting $\CC^{\ast}$-action; the variation of this locus with the cut level is discussed further in Subsection~\ref{subsection:subsection-7.2-higher-cohomology-vanishing}.
As shown after Definition~\ref{def:algebraic-cut}, a point is semistable if and only if its $\CC^{\ast}$-orbit meets $\rho^{-1}(d)$.
For a cut above the highest critical value, Hausel identifies the unstable locus with the downward Morse flow of the circle moment map, contained in the slice $\{ \zeta = 0 \} \subseteq M_{\nu} \times \CC$; see \cite[Remark following Theorem~3.4]{Hausel1998}.
Harada and Proudfoot identify this downward flow with the hypertoric core $\mcC \subset M_{\nu}$ \cite[Corollary~3.6]{HaradaProudfoot2004}, so the unstable locus is
\[
	( M_{\nu} \times \CC )^{\ustable}(\varpi_{d}) = ( M_{\nu} \times \CC ) \setminus U_{d} = \mcC \times \{0\}.
\]
Since $\dim_{\CC}M_{\nu} = 2n$ and $\dim_{\CC} \mcC = n$, the unstable locus has complex codimension $(2n + 1) - n = n + 1 \geq 2$.

To show that sections of $\mcG_{d}$ over $U_{d}$ extend uniquely across the unstable locus $( M_{\nu} \times \CC )^{\ustable}(\varpi_{d})$ to $M_{\nu} \times \CC$, we work with the sheaf of locally invariant sections on the coarse space.
By Subsection~\ref{subsection:subsection-2.2-algebraic-hypertoric}, $M_{\nu}^{\coarse}$ is normal; since $\CC$ is smooth and carries no orbifold structure, the product $( M_{\nu} \times \CC )^{\coarse} \cong M_{\nu}^{\coarse} \times \CC$ is normal as well.
Let $q : M_{\nu} \times \CC \lra M_{\nu}^{\coarse} \times \CC$ be the coarse-space map, set $U_{d}^{\coarse} := q(U_{d})$, and let $\mcG_{d}^{\coarse}$ be the sheaf of locally invariant sections of $\mcG_{d}$ on $M_{\nu}^{\coarse} \times \CC$.
On a local uniformising chart with a finite structure group $\Gamma$, the restriction of $\mcG_{d}$ is a $\Gamma$-equivariant locally free rank-one sheaf, and $\mcG_{d}^{\coarse}$ is its sheaf of $\Gamma$-invariant sections.
Since finite-group invariants of a locally free sheaf on a normal chart are reflexive over the normal invariant ring \cite[Proposition~4]{BroerChuai2010}, $\mcG_{d}^{\coarse}$ is a rank-one reflexive sheaf on $M_{\nu}^{\coarse} \times \CC$.

The coarse-space map is locally a finite quotient, so it preserves the codimension of the unstable locus.
Therefore,
\[
	\codim_{ M_{\nu}^{\coarse} \times \CC }	\left( (M_{\nu}^{\coarse} \times \CC) \setminus U_{d}^{\coarse} \right) = n + 1 \geq 2.
\]
Writing $\iota: U_{d}^{\coarse} \hookrightarrow M_{\nu}^{\coarse} \times \CC$ for the open inclusion, the extension of reflexive sheaves across a closed subset of codimension at least two in a normal variety gives $\mcG_{d}^{\coarse} \cong \iota_{\ast} \big( \mcG_{d}^{\coarse}|_{U_{d}^{\coarse}} \big)$ by \cite[Proposition~1.6]{Hartshorne1980}.
Thus, restriction of sections gives the following $\CC^{\ast}\times T^{n}$-equivariant isomorphism in orbifold notation:
\[
	H^{0}( M_{\nu} \times \CC ; \mcG_{d} ) \overset{\sim}{\longrightarrow} H^{0}( U_{d} ; \mcG_{d}|_{U_{d}} ),
\]
and then taking the cutting $\CC^{\ast}$-invariants yields
\begin{equation}
	\label{eqn:extension-isomorphism}
	H^{0}(M_{\nu} \times \CC ; \mcG_{d})^{\CC^{\ast}} \overset{\sim}{\longrightarrow} H^{0}(U_{d} ; \mcG_{d}|_{U_{d}})^{\CC^{\ast}}.
\end{equation}

The following theorem identifies the cut-filtration pieces with invariant sections and, at positive regular integral levels in the outer cut chamber, with sections of the cut pre-quantum line orbibundles.
For consecutive such levels, it recovers the individual fibre-rotation weight space $H^{0}(M_{\nu} ; \mcL_{\nu})_{d}$ as a quotient of the section spaces at levels $d$ and $d-1$.

\begin{theorem}
	\label{theorem:algebraic-cut-weight-space}
	Let $M_{\nu}$ be the non-empty hypertoric orbifold constructed in Section~\ref{section:section-2-hypertoric-varieties} at a generic integral parameter $\nu$, with the fixed integral lift $\lambda$.
	Assume that the distinguished base region $\Delta_{\emptyset}$ is non-empty and bounded as in Subsection~\ref{subsection:subsection-3.1-hyperplane-arrangements}.
	Equip $M_{\nu}$ with the fibre-rotation $\CC^{\ast}$-action and $\mcL_{\nu}$ with the compatible linearisation fixed above.
	Then, for every integer $d \geq 0$, there is a natural $T^{n}$-equivariant isomorphism
	\begin{equation}
		\label{eqn:filtration-invariant-sections}
		F_{d}H^{0}(M_{\nu} ; \mcL_{\nu}) \overset{\sim}{\longrightarrow} H^{0} \left( M_{\nu} \times \CC ; \mcL_{\nu} \boxtimes \mcO_{\CC}(\varpi_{d}) \right)^{\CC^{\ast}}.
	\end{equation}
	If, in addition, $d > 0$ is a regular integral level in the outer cut chamber, then restriction to the semistable locus and descent induce a natural $T^{n}$-equivariant isomorphism
	\begin{equation}
		\label{eqn:filtration-cut-sections}
		\Psi_{d} : F_{d}H^{0}(M_{\nu} ; \mcL_{\nu}) \overset{\sim}{\longrightarrow} H^{0}(M_{\nu}^{\leq d} ; \mcL_{\nu}^{\leq d}).
	\end{equation}

	Moreover, if $d$ and $d-1$ are both positive regular integral levels in the outer cut chamber, define $j_{d} := \Psi_{d} \circ \big( F_{d - 1} \hookrightarrow F_{d} \big) \circ \Psi_{d - 1}^{-1}$.
	Then
	\[
		j_{d}: H^{0}(M_{\nu}^{\leq d - 1} ; \mcL_{\nu}^{\leq d - 1}) \hookrightarrow H^{0}(M_{\nu}^{\leq d} ; \mcL_{\nu}^{\leq d})
	\]
	is a natural $T^{n}$-equivariant injection.
	After identifying the two cuts with their common complex orbifold in the outer cut chamber, Lemma~\ref{lemma:variation-polarisation-lemma} below identifies $j_{d}$ with multiplication by the canonical section of the cut divisor.
	Whence,
	\begin{equation}
		\label{eqn:algebraic-cut-quantisation}
		H^{0}(M_{\nu} ; \mcL_{\nu})_{d} \cong \frac{ H^{0}( M_{\nu}^{\leq d} ; \mcL_{\nu}^{\leq d} ) }{ j_{d} H^{0}( M_{\nu}^{\leq d - 1} ; \mcL_{\nu}^{\leq d - 1} ) }
	\end{equation}
	as residual $T^{n}$-representations.
\end{theorem}

\begin{proof}
	We first prove \eqref{eqn:filtration-invariant-sections} by describing the invariant sections on $M_{\nu} \times \CC$.
	Recall that $\one_{d}$ is the constant section of $\mcO_{\CC}(\varpi_{d})$, with pullback weight $-d$.
	Since the auxiliary factor $\CC$ is affine, the K\"unneth formula \cite[Proposition~9.2.4]{Kempf1993} gives
	\[
		H^{0}( M_{\nu} \times \CC ; \mcL_{\nu} \boxtimes \mcO_{\CC}(\varpi_{d}) ) \cong
		H^{0}( M_{\nu} ; \mcL_{\nu} ) \otimes_{\CC} ( \CC[y] \cdot \one_{d} ).
	\]
	Under the pullback convention \eqref{eqn:pullback-action-on-sections}, the auxiliary section $y$ has weight $+1$.
	Hence, if $s_{i} \in H^{0}( M_{\nu} ; \mcL_{\nu} )_{i}$, then the product section $s_{i}y^{j} \otimes \one_{d}$ has cutting weight $i + j - d$.
	It is therefore invariant if and only if $j = d - i \geq 0$, and every invariant section is uniquely a finite sum of terms of this form, so
	\[
		H^{0}( M_{\nu} \times \CC ; \mcL_{\nu} \boxtimes \mcO_{\CC}(\varpi_{d}) )^{\CC^{\ast}} \cong
		\bigoplus\nolimits_{i \leq d} H^{0}( M_{\nu} ; \mcL_{\nu} )_{i} y^{d - i} \otimes \one_{d}.
	\]
	Thus, the map defined on weight vectors by $s_{i} \mapsto s_{i}y^{d - i} \otimes \one_{d}$ gives the isomorphism \eqref{eqn:filtration-invariant-sections}.

	For \eqref{eqn:filtration-cut-sections}, suppose that $d>0$ is a regular integral level in the outer cut chamber.
	By \eqref{eqn:extension-isomorphism}, restricting from $M_{\nu} \times \CC$ to the semistable locus $U_{d}$ does not change the invariant section space.
	The equivariant line orbibundle $\mcG_{d}|_{U_{d}} = ( \mcL_{\nu} \boxtimes \mcO_{\CC}(\varpi_{d}) )|_{U_{d}}$ descends to $\mcL_{\nu}^{\leq d} \ra M_{\nu}^{\leq d}$
	from \eqref{eqn:cut-prequantum-orbibundle}, so its invariant sections can be identified with $H^{0}( M_{\nu}^{\leq d} ; \mcL_{\nu}^{\leq d} )$.
	Composing this restriction and descent with \eqref{eqn:filtration-invariant-sections} gives the isomorphism $\Psi_{d}$ in \eqref{eqn:filtration-cut-sections}.

	It remains to identify the inclusion between consecutive filtration pieces geometrically.
	Suppose that $d - 1$ and $d$ are both positive regular integral levels in the outer cut chamber.
	Under the $T^{n}$-equivariant isomorphism \eqref{eqn:filtration-invariant-sections} at these two levels, a weight vector $s_{i}$, with $i \leq d - 1$, is represented by $s_{i}y^{d - 1 - i} \otimes \one_{d - 1}$ and $s_{i}y^{d - i} \otimes \one_{d}$, respectively.
	Under the natural identification $\mcO_{\CC}(\varpi_{d - 1}) \otimes \mcO_{\CC}(\varpi_{1}) \cong \mcO_{\CC}(\varpi_{d})$, the two representatives satisfy
	\[
		s_{i}y^{d - i} \otimes \one_{d} = ( s_{i}y^{d - 1 - i} \otimes \one_{d - 1} ) \otimes ( y \otimes \one_{1} ),
	\]
	and multiplication by $y \otimes \one_{1}$ represents the inclusion $F_{d - 1} \hookrightarrow F_{d}$ before descent.
	The section $y \otimes \one_{1}$ has cutting weight $1 - 1 = 0$ and hence descends on the common semistable locus.
	Since $y(\zeta) = \zeta$, the proof of Lemma~\ref{lemma:variation-polarisation-lemma} identifies this descent with the canonical section $s_{\zeta}$ of the cut divisor.
	Under the isomorphisms $\Psi_{d - 1}$ and $\Psi_{d}$, the inclusion therefore becomes multiplication by $s_{\zeta}$, which is the map $j_{d}$.

	Finally, $F_{d}H^{0}( M_{\nu} ; \mcL_{\nu} ) / F_{d - 1}H^{0}( M_{\nu} ; \mcL_{\nu} ) \cong H^{0}( M_{\nu} ; \mcL_{\nu} )_{d}$.
	Applying the isomorphisms $\Psi_{d-1}$ and $\Psi_{d}$ from \eqref{eqn:filtration-cut-sections} to the denominator and numerator gives \eqref{eqn:algebraic-cut-quantisation}.
	Since each construction commutes with the residual $T^{n}$-action, the identifications are $T^{n}$-equivariant.
\end{proof}

\begin{remark}
	The isomorphism \eqref{eqn:filtration-invariant-sections} between $F_{d}H^{0}(M_{\nu} ; \mcL_{\nu})$ and the invariant section space is valid for every integer $d \geq 0$.
	The outer-chamber hypothesis is used only in the extension step \eqref{eqn:extension-isomorphism}, which allows us to identify those invariant sections with sections on the geometric cut.
	We make no corresponding claim about extensions for arbitrary inner chambers; in rank one, the unstable complement may have codimension one.
\end{remark}

\begin{remark}
	The truncation of circle-weight multiplicities under symplectic cutting was established for compact pre-quantised Hamiltonian circle actions by Duistermaat, Guillemin, Meinrenken, and Wu \cite[Theorem~1]{DGMW1995}; see also Meinrenken \cite[Corollary~4.6]{Meinrenken1998}.
	In our current non-compact hypertoric setting, Theorem~\ref{theorem:algebraic-cut-weight-space} establishes a corresponding result at the level of section spaces: for positive regular integral levels in the outer cut chamber, restriction and descent identify $H^{0}( M_{\nu}^{\leq d} ; \mcL_{\nu}^{\leq d} )$ with the finite filtration piece $F_{d}H^{0}( M_{\nu} ; \mcL_{\nu} )$, as residual $T^{n}$-representations.
	Thus, rather than only determining the multiplicities in an equivariant index, the construction gives an explicit identification of the cut section space with the kept fibre-rotation weight spaces.
\end{remark}


\subsection{Cut Chambers and Higher-Cohomology Vanishing}
\label{subsection:subsection-7.2-higher-cohomology-vanishing}

Theorem~\ref{thm:equivariant-KRR-index-theorem} computes the equivariant holomorphic Euler characteristic $\chi_{T^{n}}(M_{\nu}^{\leq d} ; \mcL_{\nu}^{\leq d})$ defined in \eqref{eqn:holomorphic-Euler-characteristic}.
To identify this class in $R(T^{n})$ with $[H^{0}(M_{\nu}^{\leq d} ; \mcL_{\nu}^{\leq d})]$, it suffices to prove that $H^{q}(M_{\nu}^{\leq d} ; \mcL_{\nu}^{\leq d})=0$ for every $q>0$.
In order to accomplish this, a sufficient condition is obtained by comparing the cut polarisations at levels $d$ and $d + n + 1$ in the same cut chamber.

Throughout this subsection, let $d>0$ be a regular integral cut level, and recall that $\dim_{\CC}M_{\nu} = 2n$.
For convenience, write
\[
	Y_{d}:=M_{\nu}^{\leq d}, \qquad
	D_{d}:=Z_{\nu}^{d}, \qquad
	\mcL_{d}:=\mcL_{\nu}^{\leq d}, \qquad
	\w_{d}:=\w_{\RR}^{\leq d}.
\]
Thus, $Y_{d}$ is a compact K\"ahler orbifold of complex dimension $2n$, with K\"ahler form $\w_{d}$, cut divisor $D_{d}\subset Y_{d}$, and pre-quantum line orbibundle $\mcL_{d}\ra Y_{d}$.

To compare the cut polarisations at different levels, we work within a fixed cut chamber.
Recall from Subsection~\ref{subsection:subsection-4.3-moment-polyptychs} that these chambers are the connected components of $( 0 , \infty ) \setminus \CritVal(\Phi)$.
By Proposition~\ref{prop:critical-values-slack}, the walls are the finitely many values $\mcS(v)$ at the arrangement vertices, and the outer cut chamber is $\Cout = ( c_{\max} , \infty )$, as in \eqref{eqn:outer-cut-chamber}.
Recall also that $U_{d}:=(M_{\nu}\times\CC)^{\sstable}(\varpi_{d})$ is the semistable locus for the cutting linearisation.
For $u\in M_{\nu}\times\CC$, the criterion used in Subsection~\ref{subsection:subsection-7.1-algebraic-cutting} states that $u \in U_{d}$ if and only if $( \CC^{\ast} \cdot u ) \cap \rho^{-1}(d) \neq \varnothing$.
This condition remains unchanged as the level varies within a cut chamber; compare \cite[Theorems~3.2--3.3]{Hausel1998} and \cite[\S\S1.2--1.3]{BrionProcesi1990}, and see \cite{Thaddeus1996, DolgachevHu1998} for variation of geometric invariant theory more generally.

For positive regular integral levels $d_{0}$ and $d_{1}$ in the same chamber, we therefore have $U_{d_{0}} = U_{d_{1}}$ and an identification $Y_{d_{0}}\cong Y_{d_{1}}$ of complex orbifolds, although the descended line orbibundles $\mcL_{d_{0}}$ and $\mcL_{d_{1}}$ can differ.
On this common quotient, Lemma~\ref{lemma:variation-polarisation-lemma} below describes the change in polarisation by tensoring with a power of $\mcO_{Y_{d}}(D_{d})$, while Lemma~\ref{lemma:canonical-bundle-formula} that follows expresses the orbifold canonical bundle in terms of the same divisor.
Together, these identities provide the line-orbibundle comparison needed for the vanishing argument.

\begin{lemma}
	\label{lemma:variation-polarisation-lemma}
	Let $d > 0$ and $k \geq 0$ be integers such that $d$ and $d + k$ are regular cut levels lying in the same cut chamber.
	Then, under the identification $Y_{d} \cong Y_{d+k}$, the cut line orbibundles satisfy $\mcL_{d+k} \cong \mcL_{d}\otimes\mcO_{Y_{d}}(kD_{d})$.
\end{lemma}

\begin{proof}
	Let $U:=U_{d} = U_{d+k}$ be the common semistable locus, so that $U/\CC^{\ast}\cong Y_{d}\cong Y_{d+k}$.
	For $r \in \ZZ$, let $\CC_{r}$ denote the one-dimensional representation with the geometric action $\tau\cdot v=\tau^{r}v$ for $\tau \in \CC^{\ast}$.
	The equivariant trivial line bundle $U\times\CC_{r}\ra U$ descends to the line orbibundle $\mcP_{r}:=U\times_{\CC^{\ast}}\CC_{r}\ra Y_{d}$, with $(u , v)\sim(\tau u,\tau^{r}v)$.

	The auxiliary coordinate restricts to an equivariant map $\zeta:U\ra\CC_{1}$, with $\zeta(\tau\cdot u) = \tau\zeta(u)$, and therefore defines a section $s_{\zeta}\in H^{0}(Y_{d} ; \mcP_{1})$ by $s_{\zeta}([u])=[u,\zeta(u)]$.
	On each local uniformising chart, $\zeta$ is a coordinate transverse to the invariant hypersurface $\{\zeta=0\}$, so $s_{\zeta}$ vanishes to first order along $D_{d}$ and $\div(s_{\zeta})=D_{d}$ as an orbifold divisor.
	The orbifold divisor--line-bundle correspondence therefore gives $\mcP_{1} \cong \mcO_{Y_{d}}(D_{d})$, and hence $\mcP_{k} \cong \mcP_{1}^{\otimes k} \cong \mcO_{Y_{d}}(kD_{d})$.

	The cut line orbibundle $\mcL_{d}$ is obtained by descending $(\mcL_{\nu}\boxtimes\mcO_{\CC}(\varpi_{d}))|_{U}$.
	Since $\varpi_{d+k} = \varpi_{d}\varpi_{k}$, replacing $d$ by $d+k$ tensors this equivariant line orbibundle with $U\times\CC_{k}$, whose descent is $\mcP_{k}$, and descending to the common quotient gives $\mcL_{d+k} \cong \mcL_{d} \otimes \mcP_{k} \cong \mcL_{d} \otimes \mcO_{Y_{d}}(kD_{d})$, as required.
\end{proof}

Taking first Chern classes in Lemma~\ref{lemma:variation-polarisation-lemma} and using the pre-quantum normalisation $c_{1}(\mcL_{d})=[\w_{d}]$ gives $[\w_{d + k}] = [\w_{d}] + k[D_{d}]$ under the identification $Y_{d} \cong Y_{d + k}$ within the chamber.
This is the Duistermaat--Heckman variation formula, with the change in class measured by the cut divisor; compare \cite[Theorem~7.10 with Corollary~7.7]{Hausel1998}.

For the vanishing argument, we next express the orbifold canonical bundle $K_{Y_{d}}^{\orb}$ in terms of the cut divisor $D_{d}$.
A corresponding formula appears in Hausel's Higgs-bundle compactification \cite[Corollary~7.7]{Hausel1998}, where the proof uses the contact structure on the cut divisor together with adjunction.
Here, we obtain the identity by contracting a holomorphic top form on $M_{\nu} \times \CC$ with the vector field generated by the cutting action, whilst working on local orbifold uniformising charts.

\begin{lemma}
	\label{lemma:canonical-bundle-formula}
	The orbifold canonical bundle of $Y_{d}$ satisfies $K_{Y_{d}}^{\orb} \cong \mcO_{Y_{d}}( -( n + 1 )D_{d} )$.
\end{lemma}

\begin{proof}
	Let $\W:=\w_{\CC}$ be the holomorphic-symplectic form on $M_{\nu}$.
	We work on local uniformising charts of $M_{\nu}$ and their products with $\CC$, using forms invariant under the finite local groups.
	Observe that the pullback by the cutting action satisfies $\tau^{\ast}\W = \tau\W$ and $\tau^{\ast}d\zeta = \tau d\zeta$.
	Furthermore, since $\dim_{\CC}M_{\nu} = 2n$, the form $\W^{n}$ is a nowhere-vanishing holomorphic top form on $M_{\nu}$ of weight $n$.
	Thus, $\W^{n} \wedge d\zeta$ is a nowhere-vanishing holomorphic top form on $M_{\nu} \times \CC$, and satisfies $\tau^{\ast}( \W^{n} \wedge d\zeta ) = \tau^{n + 1}( \W^{n} \wedge d\zeta )$.

	Let $\mcV$ be the holomorphic vector field generated by the cutting $\CC^{\ast}$-action on $U_{d}$.
	Local freeness implies that $\mcV$ is nowhere zero, so contracting the top form gives the nowhere-vanishing holomorphic $2n$-form $\beta := \imath_{\mcV}(\W^{n} \wedge d\zeta)$ on $U_{d}$.
	Its horizontality and cutting weight are expressed by $\imath_{\mcV}\beta = 0$ and $\tau^{\ast}\beta = \tau^{n + 1} \beta$.

	With $\mcP_{n + 1} := U_{d} \times_{\CC^{\ast}}\CC_{n + 1}$, the descent convention in the proof of Lemma~\ref{lemma:variation-polarisation-lemma} therefore identifies $\beta$ with a nowhere-vanishing holomorphic section of $K_{Y_{d}}^{\orb} \otimes \mcP_{n + 1}$, and therefore $K_{Y_{d}}^{\orb} \otimes \mcP_{n + 1} \cong \mcO_{Y_{d}}$.
	The same proof gives $\mcP_{n + 1} \cong \mcO_{Y_{d}}((n + 1)D_{d})$, and hence $K_{Y_{d}}^{\orb} \cong \mcP_{n + 1}^{-1} \cong \mcO_{Y_{d}}(-(n + 1)D_{d})$.
	Equivalently, $(K_{Y_{d}}^{\orb})^{-1} \cong \mcO_{Y_{d}}((n + 1)D_{d})$, so $c_{1}^{\orb}(TY_{d})=(n + 1)[D_{d}]$.
\end{proof}

The shift by $n + 1$ now has a precise role: if $d$ and $d + n + 1$ lie in the same cut chamber, then Lemmas~\ref{lemma:variation-polarisation-lemma}~and~\ref{lemma:canonical-bundle-formula} give $\mcL_{d} \cong K_{Y_{d}}^{\orb} \otimes \mcL_{d + n + 1}$ on the common complex orbifold quotient.
This expresses $\mcL_{d}$ as an adjoint line orbibundle, so the positivity of the shifted polarisation $\mcL_{d + n + 1}$ allows us to apply Braun's orbifold Nadel vanishing theorem \cite[\S8, Theorem~3, pp.~869--870]{Braun2021} in the smooth-metric case.
We use the orbifold formulation because the cut pre-quantum line orbibundle need not descend to an ordinary line bundle on the coarse space.

\begin{proposition}
	\label{prop:Nadel-vanishing}
	With $\dim_{\CC} M_{\nu} = 2n$, let $d > 0$ be an integral regular cut level.
	Then, if $d$ and $d + n + 1$ lie in the same cut chamber, we have $H^{q}( M_{\nu}^{\leq d} ; \mcL_{\nu}^{\leq d} ) = 0$ for every $q > 0$.
	Hence,
	\[
		\chi_{T^{n}}( M_{\nu}^{\leq d} ; \mcL_{\nu}^{\leq d} ) =
		[ H^{0}( M_{\nu}^{\leq d} ; \mcL_{\nu}^{\leq d} ) ],
	\]
	in $R(T^{n})$, and so $\chi_{T^{n}}( M_{\nu}^{\leq d} ; \mcL_{\nu}^{\leq d} )( 1 ) = \dim_{\CC} H^{0}( M_{\nu}^{\leq d} ; \mcL_{\nu}^{\leq d} )$.
\end{proposition}

\begin{proof}
	Since $d$ and $d + n + 1$ lie in the same cut chamber, their common semistable locus identifies $Y_{d}$ with $Y_{d + n + 1}$ as a complex orbifold.
	Applying Lemma~\ref{lemma:variation-polarisation-lemma} with $k = n + 1$ and using Lemma~\ref{lemma:canonical-bundle-formula} gives $\mcL_{d}\cong K_{Y_{d}}^{\orb}\otimes\mcL_{d + n + 1}$ on this common quotient.
	The reduced form $\w_{d + n + 1}$ is K\"ahler for the common complex structure, and the descended smooth Hermitian metric on $\mcL_{d + n + 1}$ has Chern curvature $\sqrt{-1} F_{\nabla^{\leq d + n + 1}} = 2\pi \w_{d + n + 1}$ by \eqref{eqn:prequantum-curvature-convention} and the construction in Subsection~\ref{subsection:subsection-5.3-pre-quantum-cut-space}.
	This satisfies Braun's positive lower curvature bound with the background K\"ahler form $\w_{d + n + 1}$ and positive constant $2\pi$.

	Since the metric is smooth, its singular weight can be taken to be $\phi=0$, so the associated multiplier ideal is $\mcO_{Y_{d}}$.
	Braun's orbifold Nadel vanishing theorem \cite[\S8, Theorem~3, pp.~869--870]{Braun2021} therefore gives $H^{q}( Y_{d} ; K_{Y_{d}}^{\orb} \otimes \mcL_{d + n + 1} ) = 0$ for every $q > 0$.
	Moreover, the identification of $\mcL_{d}$ above gives $H^{q}(Y_{d} ; \mcL_{d})=0$ for every $q > 0$.
	Lastly, the definition \eqref{eqn:holomorphic-Euler-characteristic} then provides the stated identities for the equivariant Euler characteristic and its value at the identity.
\end{proof}

Proposition~\ref{prop:Nadel-vanishing} gives a sufficient criterion for the vanishing: the shift $d \mapsto d + n + 1$ must remain within the same cut chamber.
Clearly, in the outer cut chamber $\Cout = ( c_{\max} , \infty )$ of \eqref{eqn:outer-cut-chamber}, this condition is automatic, since $d \in \Cout$ implies that $d + n + 1 \in \Cout$.
In particular, cutting at a sufficiently large positive integral level yields a pre-quantum line orbibundle whose higher cohomology vanishes.
The following corollary therefore identifies the localisation index of Section~\ref{section:section-6-equivariant-quantisation-hypertoric-algebraic-cutting} with the character of the section space at every positive regular integral level in the outer cut chamber.

\begin{corollary}
	\label{cor:outer-chamber-vanishing}
	If $d > 0$ is a regular integral level in the outer cut chamber, then $H^{q}(M_{\nu}^{\leq d} ; \mcL_{\nu}^{\leq d}) = 0$ for every $q > 0$.
	In particular, $\chi_{T^{n}}(M_{\nu}^{\leq d} ; \mcL_{\nu}^{\leq d}) = [ H^{0}(M_{\nu}^{\leq d} ; \mcL_{\nu}^{\leq d}) ]$ in $R(T^{n})$.
\end{corollary}


\subsection{Combinatorial Interpretation}
\label{subsection:subsection-7.3-combinatorial-interpretation}

At positive regular integral levels in the outer cut chamber, Theorem~\ref{theorem:algebraic-cut-weight-space} identifies $F_{d}H^{0}(M_{\nu} ; \mcL_{\nu})$ with $H^{0}(M_{\nu}^{\leq d} ; \mcL_{\nu}^{\leq d})$.
In this subsection, we compute the residual $T^{n}$-weight multiplicities in this filtration directly from the coordinate ring of the hypertoric quotient, obtaining a weighted lattice-point description on the moment polyptych.

The following proposition provides the comparison between the polynomial semi-invariants and global sections used in Subsection~\ref{subsection:subsection-7.1-algebraic-cutting}.
Its proof uses the ambient Lawrence toric quotient and does not depend on the algebraic-cut construction.

\begin{proposition}[Semi-Invariant Comparison]
	\label{prop:positive-ray-semi-invariants}
	Let $P := \CC[ z_{1} , \ldots , z_{N} , w_{1} , \ldots , w_{N} ]$, and let
	\[
		J_{\CC}P := I( \phi_{\CC}^{-1}(0) ) = \langle \langle \phi_{\CC} ,\, \kappa \rangle ~|~ \kappa \in \mfk_{\CC} \rangle_{P} \subseteq P
	\]
	be the ideal generated by components of the complex moment map, and observe that $\CC[\mcN] \cong P / J_{\CC}P$, where $\mcN = \phi_{\CC}^{-1}(0)$.
	For every $m \geq 0$, set
	\[
		\CC[\mcN]_{m\nu} := ( P / J_{\CC}P )_{m\nu} :=
		\{ \bar{f} \in P / J_{\CC}P ~|~ k^{\ast} \bar{f} = \chi_{\nu}(k)^{m} \bar{f} \quad \! \text{for every } k \in K_{\CC} \}.
	\]
	Then there exists a natural $T^{n}$-equivariant isomorphism
	\begin{equation}
		\label{eqn:positive-ray-semi-invariants}
		\CC[\mcN]_{m\nu} \cong ( P / J_{\CC}P )_{m\nu}
		\overset{\sim}{\longrightarrow}
		H^{0}(M_{\nu} ; \mcL_{\nu}^{\otimes m}).
	\end{equation}
	Furthermore, these isomorphisms are compatible with multiplication, so they assemble into an isomorphism of graded $\CC$-algebras
	\[
		\bigoplus\nolimits_{m \geq 0} ( P / J_{\CC}P )_{m\nu}
		\overset{\sim}{\lra}
		\bigoplus\nolimits_{m \geq 0} H^{0}( M_{\nu} ; \mcL_{\nu}^{\otimes m} ).
	\]
\end{proposition}

\begin{proof}
	The unstable locus in the complex moment fibre $\mcN$ may have codimension one, as in \cite[Example~1.2]{BellamyCrawSchedler2024}, so we cannot argue that sections extend directly across that locus.
	Therefore, we first instead prove that the complement of $(T^{\ast}\CC^{N})^{\sstable}(\nu)$ has codimension at least two, and then work on the ambient Lawrence quotient.

	The Hilbert--Mumford criterion \cite[Lemmas~4.2--4.3]{BellamyKuwabara2012} describes instability in terms of the cone generated by the weights of the non-zero coordinates.
	A codimension-one unstable component would force the generic point of a coordinate hyperplane to be unstable.
	Consider a coordinate hyperplane, say $\{z_{i} = 0\}$; the case $\{w_{i} = 0\}$ is identical.
	At its generic point, all coordinates $z_{j}$ and $w_{j}$ with $j \neq i$ are non-zero, so the weights $\pm \alpha_{j}$ for $j \neq i$ occur.
	Moreover, the weights $\{ \alpha_{j} ~|~ j \neq i \}$ still span $\mfk^{\ast}$.
	Indeed, if $\kappa \in \mfk$ is annihilated by every $\alpha_{j}$ with $j \neq i$, then $\imath(\kappa) = ce_{i}$ for some $c \in \RR$.
	Since $\pi \circ \imath = 0$, we obtain $cu_{i} = 0$, which forces $c = 0$, and hence $\kappa = 0$.
	Since both signs $\pm \alpha_{j}$ occur among the remaining cotangent coordinates, these weights generate $\mfk^{\ast}$ as a convex cone.
	In particular, this cone contains the linearisation character $\nu$, so the generic point of each coordinate hyperplane is semistable.
	Since the unstable locus is a union of coordinate subspaces, it follows that it has codimension at least two.

	Now consider the ambient Lawrence toric quotient
	\[
		W_{\nu} := T^{\ast}\CC^{N} \sslash_{\nu} K_{\CC} = \Proj \left( \bigoplus\nolimits_{m \geq 0} P_{m\nu} \right) \! ,
	\]
	and let $\mcA_{\nu}\ra W_{\nu}$ be the line orbibundle obtained by descending the trivial line bundle on $(T^{\ast}\CC^{N})^{\sstable}(\nu)$, with geometric fibre action given by $\chi_{\nu}$.
	Recall that descent identifies sections of $\mcA_{\nu}^{\otimes m}$ with semi-invariants of character $m\nu$ on this semistable locus.
	Since its complement has codimension at least two in the normal affine space $T^{\ast}\CC^{N}$, those semi-invariants extend uniquely to polynomial semi-invariants, and hence we get $H^{0}( W_{\nu} ; \mcA_{\nu}^{\otimes m} ) \cong P_{m\nu}$ for every $m \geq 0$.

	We also need the higher cohomology to vanish for these tensor powers on $W_{\nu}$.
	For each $m \geq 0$, the tensor power $\mcA_{\nu}^{\otimes m}$ is a line orbibundle on the canonical toric orbifold $W_{\nu}$, and its coarse pushforward is the rank-one reflexive sheaf $\mcO(E_{m})$, where $E_{m}$ is an integral torus-invariant Weil divisor; it is $\QQ$-Cartier because the coarse Lawrence toric variety is simplicial.
	Since the local stabilisers are finite and we work in characteristic zero, taking invariants is exact, so the coarse and orbifold cohomology groups agree.

	Let $\pi_{\aff} : W_{\nu} \ra \Spec P^{K_{\CC}}$ denote the projective toric morphism to the affinisation, which is surjective since it is also dominant.
	Under $W_{\nu} = \Proj ( \bigoplus\nolimits_{m \geq 0} P_{m\nu} )$, the sheaf $\mcO(E_{m})$ is the twisting sheaf $\mcO_{W_{\nu}}(m)$, and since this graded ring is finitely generated over $P^{K_{\CC}}$, there is a $k > 0$ for which $\mcO(E_{k})$ is invertible and $\pi_{\aff}$-ample \cite[Proposition~4.6.18]{Grothendieck1961}.
	As $kE_{m}$ is linearly equivalent to $mE_{k}$, every $E_{m}$ with $m \geq 0$ is $\pi_{\aff}$-nef.
	Hence, Fujino's relative toric vanishing theorem \cite[Theorem~5.3]{Fujino2007} gives $R^{q}(\pi_{\aff})_{\ast} \mcO(E_{m}) = 0$ for every $q > 0$.
	Since the affinisation is affine, the Leray spectral sequence and the previous cohomology comparison give $H^{q}( W_{\nu} ; \mcA_{\nu}^{\otimes m} ) = 0$ for every $q > 0$ and $m \geq 0$.

	Now write $r = \dim_{\CC}K_{\CC} = N - n$.
	After choosing a basis of $\mfk_{\CC}$, let $f_{1} , \ldots , f_{r}$ be the corresponding components of $\phi_{\CC}$.
	Each $f_{a}$ is a linear combination of the products $z_{i}w_{i}$, and therefore has $K_{\CC}$-character zero.
	Hausel and Sturmfels show that these moment-map equations generate the homogeneous prime complete-intersection ideal denoted by $\operatorname{Circ}(B)$ in \cite[\S6]{HauselSturmfels2002}.
	The regular-sequence property is also recorded in \cite[Lemma~4.7]{BellamyKuwabara2012}, where the condition $u_{i} \neq 0$ is the non-zero-row hypothesis used there.
	It follows that the sequence $f_{1} , \ldots , f_{r}$ forms a regular sequence, and $M_{\nu} \hookrightarrow W_{\nu}$ is the complete intersection cut out by their descended sections.

	Reductivity of $K_{\CC}$ makes the restriction $P_{m\nu}\ra(P/J_{\CC}P)_{m\nu}$ surjective for every $m \geq 0$, and therefore we have $\mcN^{\sstable}(\nu) = \mcN \cap ( T^{\ast}\CC^{N} )^{\sstable}(\nu)$.
	The inclusion of these semistable loci induces $M_{\nu}\hookrightarrow W_{\nu}$.
	Furthermore, since $\mcA_{\nu}$ and $\mcL_{\nu}$ both descend via $\chi_{\nu}$, we have $\mcL_{\nu} = \mcA_{\nu}|_{M_{\nu}}$.

	We regard the descended moment map as a section of the trivial vector bundle with fibre $\mfk_{\CC}^{\ast}$, so its Koszul complex uses the exterior powers of $(\mfk_{\CC}^{\ast})^{\vee} = \mfk_{\CC}$.
	Tensoring the Koszul resolution of $\mcO_{M_{\nu}}$ with $\mcA_{\nu}^{\otimes m}$ then gives
	\[
		0 \longrightarrow \mcA_{\nu}^{\otimes m} \otimes \bigwedge^{r}\mfk_{\CC} \longrightarrow \cdots \longrightarrow \mcA_{\nu}^{\otimes m} \otimes \mfk_{\CC} \longrightarrow \mcA_{\nu}^{\otimes m} \longrightarrow \mcL_{\nu}^{\otimes m} \longrightarrow 0.
	\]
	All the terms involve the same tensor power $\mcA_{\nu}^{\otimes m}$ because each $f_{a}$ has trivial $K_{\CC}$-character.
	The preceding toric vanishing therefore makes the global-section functor exact on this resolution, so the restriction map $H^{0}( W_{\nu} ; \mcA_{\nu}^{\otimes m} ) \ra H^{0}( M_{\nu} ; \mcL_{\nu}^{\otimes m} )$ is surjective with kernel $( J_{\CC}P )_{m\nu}$ and yields $H^{0}( M_{\nu} ; \mcL_{\nu}^{\otimes m} ) \cong P_{m\nu} / (J_{\CC}P)_{m\nu} = ( P / J_{\CC}P )_{m\nu}$.

	The restriction and descent maps commute with the residual $T^{n}$-action and with the fibre-rotation action, using the lift from Subsection~\ref{subsection:subsection-5.3-pre-quantum-cut-space} that leaves the fibre coordinate of $L_{\nu} = \mcN \times \CC$ unchanged.
	Thus, \eqref{eqn:positive-ray-semi-invariants} is $T^{n}$-equivariant and preserves fibre-rotation weights.
	Moreover, descent and restriction also preserve multiplication: the product of two semi-invariants of characters $m\nu$ and $m^{\prime}\nu$ corresponds to the product of two sections of $\mcL_{\nu}^{\otimes m}$ and $\mcL_{\nu}^{\otimes m^{\prime}}$.
	The isomorphisms therefore assemble into the stated graded-algebra isomorphism.
\end{proof}

\begin{remark}
	Proposition~\ref{prop:positive-ray-semi-invariants} concerns the characters $m\nu$ with $m \in \ZZ_{\geq 0}$.
	Descent identifies global sections with semi-invariants on the semistable locus, but these need not extend to polynomial semi-invariants on the level set $\mcN = \phi_{\CC}^{-1}(0)$.
	Indeed, already for the hypertoric presentation of $T^{\ast}\CC\PP^{1}$, the unstable locus can have codimension one, and its removal can enlarge the semi-invariant ring; see \cite[Example~1.2]{BellamyCrawSchedler2024}.
	The proposition does not claim that the semi-invariants of arbitrary $K_{\CC}$-characters extend from the semistable locus to $\mcN$.
\end{remark}

We now compute $F_{d}H^{0}( M_{\nu} ; \mcL_{\nu} )$ using the polynomial semi-invariants in Proposition~\ref{prop:positive-ray-semi-invariants}.
Consider a monomial $z^{a}w^{b} \in P$, with $a, b \in \ZZ_{\geq 0}^{N}$, and retain the fixed integral lift $\lambda$ of $\nu$ from Section~\ref{section:section-2-hypertoric-varieties}, as well as the affine-linear functionals $\ell_{i}$ defined in \eqref{eqn:hyperplane}.
A monomial $z^{a}w^{b}$ has $K_{\CC}$-character $\nu$ if and only if its exponent difference satisfies $\imath_{\ZZ}^{\ast}(a - b) = \nu$.
Since $\nu = \imath_{\ZZ}^{\ast}(\lambda)$, this is equivalent to $a - b - \lambda \in \ker( \imath_{\ZZ}^{\ast} ) = \Image(\pi_{\ZZ}^{\ast})$.
Thus, there is a unique residual weight $x \in ( \mft_{\ZZ}^{n} )^{\ast}$ such that
\begin{equation}
	\label{eqn:exponent-differences}
	a_{i} - b_{i} = \langle x ,\, u_{i} \rangle + \lambda_{i} = \ell_{i}(x), \qquad
	\text{for each } i = 1 , \ldots , N.
\end{equation}
By \eqref{eqn:positive-ray-semi-invariants}, the corresponding class in $P / J_{\CC}P$ is then a section of $\mcL_{\nu}$ of residual $T^{n}$-weight $x$.

So fix a residual weight $x$.
For each $i$, the coordinatewise minimal non-negative exponents satisfying $a_{i}-b_{i} = \ell_{i}(x)$ are $a_{i}^{\min} = \max(0,\ell_{i}(x))$ and $b_{i}^{\min} = \max(0 , -\ell_{i}(x))$.
Substituting these choices into $z^{a}w^{b}$ gives the distinguished monomial
\begin{equation}
	\label{eqn:minimal-monomial}
	p_{x} := \prod_{i = 1}^{N} z_{i}^{\max( 0 , \, \ell_{i}(x) ) } w_{i}^{\max( 0 , \, -\ell_{i}(x) ) },
\end{equation}
which we call the \textbf{minimal monomial} associated with the residual weight $x$.

Any other pair of exponents $( a_{i} , b_{i} ) \in \ZZ_{\geq 0}^{2}$ satisfying $a_{i} - b_{i} = \ell_{i}(x)$ must differ from the minimal choice by adding the same non-negative integer to both entries.
That is to say, for each $i = 1 , \ldots , N$, there is a unique integer $s_{i} \in \ZZ_{\geq 0}$ such that
\[
	a_{i} = \max( 0 , \, \ell_{i}(x) ) + s_{i}
	\qquad \text{and} \qquad
	b_{i} = \max( 0 , \, -\ell_{i}(x) ) + s_{i}.
\]
Then, using the identity $\ell_{i}(x) = \max( 0 , \ell_{i}(x) ) - \max( 0 , -\ell_{i}(x) )$ and \eqref{eqn:exponent-differences}, we may write
\[
	s_{i} = a_{i} - \max( 0 , \, \ell_{i}(x) ) = b_{i} - \max( 0 , \, -\ell_{i}(x) ).
\]
Every monomial of $K_{\CC}$-character $\nu$ and residual weight $x$ therefore factors uniquely as $p_{x}$ times a monomial in the residual-invariant products $z_{i}w_{i}$:
\[
	z^{a}w^{b} = \prod_{i = 1}^{N}
	z_{i}^{\max( 0 , \, \ell_{i}(x) ) + s_{i} } w_{i}^{\max( 0 , \, -\ell_{i}(x) ) + s_{i} } = p_{x} \cdot \prod_{i = 1}^{N} ( z_{i}w_{i} )^{s_{i}},
\]
and we see that $p_{x}$ records the part of the monomial forced by the residual weight $x$, while the exponents $s_{i}$ measure the additional invariant factors.

The same factorisation also describes the entire residual-weight-$x$ subspace before imposing the moment-map equations.
Set $h_{i}:=z_{i}w_{i}$ and $B_{0}:=\CC[h_{1}, \ldots ,h_{N}]$, and let $P_{x}\subseteq P$ be the span of the monomials satisfying \eqref{eqn:exponent-differences}.
Since multiplication by $h_{i}$ preserves the exponent differences, $P_{x}$ is a $B_{0}$-module.
Furthermore, the unique factorisation identifies its monomial basis with $p_{x}$ times the monomial basis of $B_{0}$, giving the $B_{0}$-module isomorphism $B_{0} \overset{\sim}{\ra} P_{x}$, $f \mapsto p_{x}f$.
Thus, $P_{x} = p_{x}B_{0}$ is free of rank one, a description we will use when quotienting by the complex moment-map equations in the variables $h_{i}$.

To impose the cutting condition, adjoin the auxiliary coordinate $y$.
Then a monomial of residual weight $x$ on $T^{\ast}\CC^{N} \times \CC$ has the unique normal form
\begin{equation}
	\label{eqn:monomial-normal-form}
	Q = z^{a}w^{b} y^{c} = p_{x} \cdot \left( \prod_{i = 1}^{N} ( z_{i} w_{i} )^{s_{i}} \right) y^{c}, \qquad c \in \ZZ_{\geq 0}.
\end{equation}
Under the pullback convention \eqref{eqn:pullback-action-on-sections}, the cutting weights are $\wt_{\CC^{\ast}}(z_{i}) = 0$ and $\wt_{\CC^{\ast}}(w_{i}) = \wt_{\CC^{\ast}}(y) = 1$.
On $P[y]$, we call this weight the \textbf{cut degree},
$\deg_{\cut}:=\wt_{\CC^{\ast}}$; note that it is distinct from the tensor-power degree in the section rings of Subsection~\ref{subsection:subsection-7.1-algebraic-cutting}.
In particular, each $h_{i} = z_{i}w_{i}$ has residual weight zero but cut degree one.

The exponent of $w_{i}$ in $p_{x}$ is $\max\{ 0 ,\, -\ell_{i}(x) \}$, and therefore
\[
	\deg_{\cut}(p_{x}) = \sum\nolimits_{i = 1}^{N} \max\{ 0 ,\, -\ell_{i}(x) \} = \mcS(x),
\]
where $\mcS$ is the slack function of \eqref{eqn:slack-function}.
Thus, $\mcS(x)$ is the minimum cut degree required to realise the residual weight $x$; every additional $h_{i}$- or $y$-factor increases the degree by one.
For the monomial in \eqref{eqn:monomial-normal-form}, this gives
\[
	\deg_{\cut}(Q) = \wt_{\CC^{\ast}}(Q) = \mcS(x) + \sum\nolimits_{i = 1}^{N} s_{i} + c.
\]
Recall that the constant section $\one_{d}$ of $\mcO_{\CC}(\varpi_{d})$ has pullback weight $-d$.
The section $Q\otimes\one_{d}$ is therefore cutting-invariant only if $s_{1} + \ldots + s_{N} + c = d - \mcS(x)$.
Clearly, there are no solutions when $\mcS(x) > d$; otherwise, counting monomials of degree $d - \mcS(x)$ in the $N + 1$ variables $h_{1} , \ldots , h_{N} , y$ gives $\tbinom{d - \mcS(x) + N}{N}$.
This count is taken before imposing the complex moment-map equations; to obtain the section multiplicity, we must determine the remaining independent classes.

The complex moment-map equations are linear in the $h_{i}$ variables.
Let $J_{\CC} \subseteq B_{0}$ be the ideal generated by the components of $\phi_{\CC}(z , w) = \imath^{\ast}(h_{1} , \ldots , h_{N}) = \sum_{i = 1}^{N} h_{i}\alpha_{i}$.
That is, after choosing a basis $\kappa_{1} , \ldots , \kappa_{r}$ of $\mfk_{\CC}$, where $r = \dim_{\CC} \mfk_{\CC}$, the ideal $J_{\CC}$ is generated by the homogeneous linear forms
\[
	g_{a} := \langle \phi_{\CC} ,\, \kappa_{a} \rangle = \sum\nolimits_{i = 1}^{N} \langle \alpha_{i} ,\, \kappa_{a} \rangle h_{i}, \qquad a = 1 , \ldots , r,
\]
where we recall that the extension $J_{\CC}P$ is the complex moment-map ideal used in Proposition~\ref{prop:positive-ray-semi-invariants}.

Now equip $P$ with the exponent-difference grading given by the character lattice $( \mft_{\ZZ}^{N} )^{\ast}$, determined by $\deg_{\diff}(z_{i}) = \epsilon_{i}$ and $\deg_{\diff}(w_{i}) = -\epsilon_{i}$, where $\epsilon_{1} , \ldots , \epsilon_{N}$ is the dual basis introduced in Subsection~\ref{subsection:subsection-2.1-differential-hypertoric}.
In this grading, we have $\deg_{\diff}(z^{a}w^{b}) = \sum_{i = 1}^{N} ( a_{i} - b_{i} )\epsilon_{i}$, and $P_{x}$ is the homogeneous component of degree $\lambda + \pi_{\ZZ}^{\ast}(x)$.
Since the generators of $J_{\CC}$ have exponent-difference degree zero, the component of degree $\lambda + \pi_{\ZZ}^{\ast}(x)$ satisfies $(J_{\CC}P) \cap P_{x} = J_{\CC}P_{x} = p_{x}J_{\CC}$, where the last equality uses $P_{x} = p_{x}B_{0}$.
Multiplication by $p_{x}$ therefore induces the isomorphism $B_{0}/J_{\CC} \overset{\sim}{\ra} (P / J_{\CC}P)_{x}$ given by $\widebar{f} \mapsto \widebar{p_{x}f}$.
Thus, $(P / J_{\CC}P )_{x}$, the component corresponding to the residual weight $x$, is a free rank-one module over $B_{0} / J_{\CC}$, generated by the class of $p_{x}$.

Since $\imath^{\ast} : ( \mft^{N} )^{\ast} \ra \mfk^{\ast}$ is surjective and $\dim_{\CC} \mfk_{\CC} = r$, the generators $g_{1} , \ldots , g_{r}$ are $r$ independent linear forms in the $h_{1} , \ldots , h_{N}$ variables.
Choose classes $q_{1}, \ldots , q_{n}$ to form a basis of the quotient $\Span\{ h_{1} , \ldots , h_{N} \} / \Span\{ g_{1} , \ldots , g_{r} \}$, whose dimension is $N - r = n$.
Then $B_{0}/J_{\CC} \cong \CC[q_{1},\ldots , q_{n}]$, where each $q_{j}$ is represented by a homogeneous linear combination of the variables $h_{i} = z_{i}w_{i}$ and therefore has cut degree one.
After adjoining the auxiliary coordinate function $y$, we obtain $( ( P / J_{\CC} P )[y] )_{x} \cong p_{x} \CC[ q_{1} , \ldots , q_{n} , y ]$.
Hence, the residual-weight-$x$, cut-degree-$d$ subspace has the basis
\begin{equation}
	\label{eqn:weight-x-basis}
	\{ p_{x} q_{1}^{k_{1}} \cdots q_{n}^{k_{n}} y^{c} \mid k_{1} , \ldots , k_{n} , c \in \ZZ_{\geq 0} , ~k_{1} + \cdots + k_{n} + c = d - \mcS(x) \},
\end{equation}
and this set is empty when $\mcS(x) > d$.

It is the basis \eqref{eqn:weight-x-basis} that determines both which residual weights occur as well as their multiplicities.
The minimum cut degree is $\mcS(x)$, and the remaining degree is distributed among the $n + 1$ variables $q_{1},\ldots,q_{n},y$.
The following theorem also establishes the non-negativity and finite-dimensionality used in Subsection~\ref{subsection:subsection-7.1-algebraic-cutting}.

\begin{theorem}
	\label{theorem:combinatorial-multiplicity}
	For each $d \in \ZZ_{\geq 0}$, the residual $T^{n}$-weight $x \in ( \mft_{\ZZ}^{n} )^{\ast}$ occurs in the filtration piece $F_{d}H^{0}( M_{\nu} ; \mcL_{\nu} )$ with multiplicity
	\begin{equation}
		\label{eqn:multiplicity}
		\mult_{d}(x) = \begin{cases}
			\binom{d - \mcS(x) + n}{n}, \qquad & \mcS(x) \leq d, \\
			0, \qquad                          & \mcS(x) > d,
		\end{cases}
	\end{equation}
	and no negative fibre-rotation weights occur in $H^{0}( M_{\nu} ; \mcL_{\nu} )$.
	Furthermore, the support of $\mult_{d}$ is the finite set\break $\supp(\mult_{d}) = \{ x \in (\mft_{\ZZ}^{n})^{\ast} \mid \mcS(x) \leq d \}$, and
	\[
		F_{d} H^{0}( M_{\nu} ; \mcL_{\nu} ) \cong
		\bigoplus_{x \in ( \mft_{\ZZ}^{n} )^{\ast}} \CC_{x}^{\oplus \mult_{d}(x)},
		\quad \text{hence} \quad
		\dim_{\CC} F_{d} H^{0}( M_{\nu} ; \mcL_{\nu} ) =
		\sum_{x \in ( \mft_{\ZZ}^{n} )^{\ast}} \mult_{d}(x),
	\]
	where $\CC_{x}$ denotes the one-dimensional residual $T^{n}$-representation of weight $x$.

	When $d > 0$ is a regular cut level, Proposition~\ref{prop:slack-support} identifies this support with $|\Delta_{\nu}^{\leq d}| \cap ( \mft_{\ZZ}^{n} )^{\ast}$.
	In particular, if $d > 0$ is an integral regular cut level in the outer cut chamber, then Theorem~\ref{theorem:algebraic-cut-weight-space} identifies $F_{d} H^{0}( M_{\nu} ; \mcL_{\nu} )$ with $H^{0}( M_{\nu}^{\leq d} ; \mcL_{\nu}^{\leq d} )$, and the same multiplicity and dimension formulae hold for sections of $\mcL_{\nu}^{\leq d} \ra M_{\nu}^{\leq d}$.
\end{theorem}

\begin{proof}
	Proposition~\ref{prop:positive-ray-semi-invariants} with $m = 1$ identifies $(P/J_{\CC}P)_{\nu}$ with $H^{0}(M_{\nu} ; \mcL_{\nu})$.
	By \eqref{eqn:filtration-invariant-sections}, adjoining $y$ and tensoring the cut-degree-$d$ component with $\one_{d}$ gives the invariant sections that correspond to $F_{d}H^{0}(M_{\nu} ; \mcL_{\nu})$.

	Fix a residual weight $x$.
	The basis \eqref{eqn:weight-x-basis} shows that its multiplicity is the number of non-negative integer solutions of $k_{1} + \cdots + k_{n} + c = d - \mcS(x)$.
	Clearly, if $\mcS(x) > d$, there are no solutions; otherwise, counting monomials of degree $d - \mcS(x)$ in $n + 1$ variables gives
	\[
		\mult_{d}(x) = \binom{d - \mcS(x) + n}{n},
	\]
	proving \eqref{eqn:multiplicity}.
	The same description shows that no negative fibre-rotation weights occur: $\mcS(x) \geq 0$ always holds, and the additional generators have positive cut degree.

	Choose a positive regular cut level $\delta^{\prime} > d$, which is possible because $\Phi$ has only finitely many critical values by Proposition~\ref{prop:critical-values-slack}.
	By Proposition~\ref{prop:slack-support}, the set $\{ \mcS \leq \delta^{\prime} \}$ is compact.
	Since $\mcS$ is continuous, $\{ \mcS \leq d \}$ is a closed subset of this compact set and therefore contains only finitely many lattice points.
	The filtration piece is consequently finite-dimensional, and summing its weight spaces gives the stated decomposition and dimension formula.
	At a positive regular level $d$, Proposition~\ref{prop:slack-support} identifies its weight support with the lattice points in $|\Delta_{\nu}^{\leq d}|$ and, at a positive regular integral level in the outer cut chamber, Theorem~\ref{theorem:algebraic-cut-weight-space} identifies the filtration piece with $H^{0}( M_{\nu}^{\leq d} ; \mcL_{\nu}^{\leq d} )$.
	The same multiplicity and dimension formulae therefore apply there.
\end{proof}

The following corollary expresses the character of $H^{0}(M_{\nu}^{\leq d} ; \mcL_{\nu}^{\leq d})$ as a weighted sum over the lattice points of $|\Delta_{\nu}^{\leq d}|$.
In the outer cut chamber, Corollary~\ref{cor:outer-chamber-vanishing} identifies this character with the equivariant index computed by the localisation formula in Section~\ref{section:section-6-equivariant-quantisation-hypertoric-algebraic-cutting}.

\begin{corollary}[Weighted Lattice-Point Character]
	\label{cor:weighted-lattice-point-character}
	Let $d > 0$ be a regular integral cut level in the outer cut chamber.
	Then
	\begin{equation}
		\label{eqn:weighted-lattice-point-character}
		\operatorname{char}_{T^{n}}
		H^{0}(M_{\nu}^{\leq d} ; \mcL_{\nu}^{\leq d})(e^{\xi}) =
		\sum_{x \in |\Delta_{\nu}^{\leq d}| \cap (\mft_{\ZZ}^{n})^{\ast}}
		\binom{d - \mcS(x) + n}{n} e^{\langle x,\, \xi \rangle}.
	\end{equation}
	Moreover, setting $\xi = 0$, this gives
	\[
		\dim_{\CC} H^{0}(M_{\nu}^{\leq d} ; \mcL_{\nu}^{\leq d}) =
		\sum_{ x \in |\Delta_{\nu}^{\leq d}| \cap ( \mft_{\ZZ}^{n} )^{\ast} }
		\binom{d - \mcS(x) + n}{n},
	\]
	and hence, by Corollary~\ref{cor:outer-chamber-vanishing}, the character in \eqref{eqn:weighted-lattice-point-character} agrees with the Kawasaki localisation formula in
	Theorem~\ref{theorem:character-hypertoric-orbifold}.
\end{corollary}

\begin{proof}
	By \eqref{eqn:filtration-cut-sections}, the representation $H^{0}( M_{\nu}^{\leq d} ; \mcL_{\nu}^{\leq d} )$ has the weight decomposition given by Theorem~\ref{theorem:combinatorial-multiplicity}.
	Summing the corresponding characters gives \eqref{eqn:weighted-lattice-point-character}, and setting $\xi=0$ gives the dimension formula.

	Corollary~\ref{cor:outer-chamber-vanishing} identifies this section-space character with the equivariant index, which Theorem~\ref{theorem:character-hypertoric-orbifold} computes by orbifold localisation.
\end{proof}

\begin{remark}[Enlarged Moment Image]
	\label{remark:enlarged-cut-moment-image}
	Suppose $d > 0$ is a regular integral cut level in the outer cut chamber, and let $n = 1$.
	In this case, the multiplicity $\mult_{d}(x)$ counts the lattice points above $x$ in a two-dimensional moment polygon, obtained by using the auxiliary-circle moment map as a second coordinate.
	Rotation of the auxiliary coordinate $\zeta$ descends to a Hamiltonian circle action on $M_{\nu}^{\leq d}$, with moment map $h$ induced by $\tfrac{1}{2}|\zeta|^{2} = d - \Phi$ on $\rho^{-1}(d)$.
	Together with the residual $T^{1}$-action, this gives an effective Hamiltonian $T^{1} \times U_{1}$-action on the four-real-dimensional cut, making it a symplectic toric orbifold.

	After identifying the residual weight lattice with $\ZZ$, the joint moment map $(\mu_{\RR},h)$ has image
	\begin{equation}
		\label{eqn:enlarged-cut-moment-image}
		P_{\nu}^{\leq d} = \{ ( x , h ) \in \RR^{2} ~|~ 0 \leq h \leq d - \mcS(x) \}.
	\end{equation}
	Indeed, Lemma~\ref{lemma:slack-function} gives $h \leq d - \mcS(x)$, and Proposition~\ref{prop:cut-component-image} shows that the graph $h = d - \mcS(x)$ is attained on the cut extended core.
	Its endpoints lie on $h = 0$, so orbifold convexity \cite[Theorem~5.2]{LermanTolman1997} fills the region between this graph and the horizontal axis.
	For each integral $x$ with $\mcS(x)\leq d$, the lattice points in the vertical slice are
	\[
		\{ (x , j) ~|~ j \in \ZZ ,~ 0 \leq j \leq d - \mcS(x) \},
	\]
	whose number is $d - \mcS(x) + 1 = \mult_{d}(x)$.
	Thus, the multiplicity at $x$ counts the lattice points that project to $x$, and the weighted count becomes $\dim_{\CC}H^{0}(M_{\nu}^{\leq d} ; \mcL_{\nu}^{\leq d}) = \#( P_{\nu}^{\leq d} \cap \ZZ^{2} )$.
	We illustrate this interpretation in Figure~\ref{figure:rank-one-enlarged-moment-polytope}, using the $M_{\nu} \cong T^{\ast}\CC\PP^{1}$ case as an example.

	For $n \geq 2$, the same moment-image description holds in $( \mft^{n} )^{\ast} \times \RR$, but the enlarged torus has dimension $n + 1 < 2n$ and so this action is not toric.
	Hence the $d - \mcS(x) + 1$ lattice heights do not in general account for the full multiplicity $\binom{d - \mcS(x) + n}{n}$ when $n \geq 2$.

	\begin{figure}[htbp]
		\centering
		\begin{tikzpicture}[
			x=0.72cm,
			y=0.72cm,
			font=\small
			]

			\filldraw[
			fill=gray!12,
			draw=gray!65,
			line width=1.2pt
			]
			(-4 , 0) -- (0 , 4) -- (6 , 4) -- (10,0) -- cycle;

			\foreach \X/\H in {
				-2/2,
				2/4,
				4/4,
				6/4,
				8/2
			}{
				\draw[
				gray!75!black,
				dashed,
				line width=0.85pt
				] (\X , 0) -- (\X,\H);
			}

			\draw[->, line width=0.85pt]
			(-5.7 , 0) -- (11.1 , 0);
			\draw[->, line width=0.85pt]
			(0 , -0.10) -- (0 , 5.1);

			\node[left] at (-0.12 , 5.35) {$h$};
			\node[below] at (11.35 , -0.12) {$x$};

			\foreach \X/\Y in {
				-4/0,
				-2/0, -2/2,
				0/0,  0/2,  0/4,
				2/0,  2/2,  2/4,
				4/0,  4/2,  4/4,
				6/0,  6/2,  6/4,
				8/0,  8/2,
				10/0
			}{
				\filldraw[
				fill=black!78,
				draw=black!90,
				line width=0.3pt
				] (\X,\Y) circle[radius=1.65pt];
			}

			\node[text=gray!75!black]
			at (-2 , -0.55) {$\Delta_{1}^{\leq d}$};
			\node[text=gray!75!black]
			at (3 , -0.55) {$\Delta_{\emptyset}^{\leq d}$};
			\node[text=gray!75!black]
			at (8 , -0.55) {$\Delta_{2}^{\leq d}$};

		\end{tikzpicture}
		\caption{
			The enlarged moment polygon $P_{\nu}^{\leq d}$ for $M_{\nu} \cong T^{\ast}\CC\PP^{1}$, whose upper boundary is the graph $h = d - \mcS(x)$.
			For each integral residual weight $x \in |\Delta_{\nu}^{\leq d}|$, the vertical slice contains $\mult_{d}(x) = d - \mcS(x) + 1$ lattice points, and summing these multiplicities gives the dimension of the cut section space.
		}
		\label{figure:rank-one-enlarged-moment-polytope}
	\end{figure}
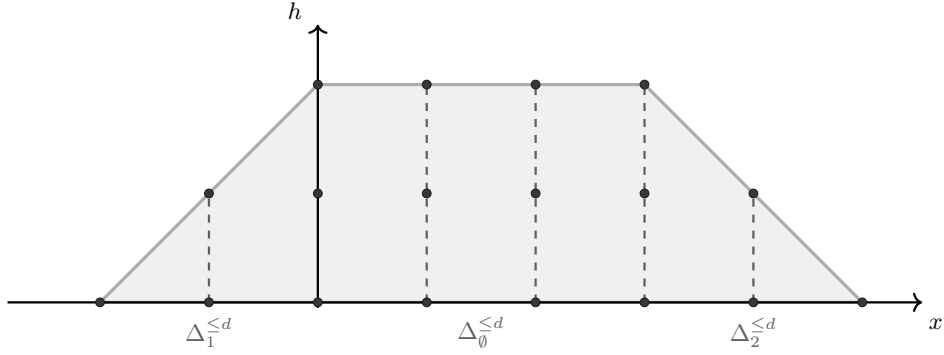
\end{remark}


\subsection{Examples}
\label{subsection:subsection-7.4-combinatorial-quantisation-examples}

We finish this section by returning to the examples of Subsections~\ref{subsection:subsection-4.4-worked-cut-spaces-moment-polyptychs}~and~\ref{subsection:subsection-6.2-worked-equivariant-character}.
We compute their quantisations from the combinatorial multiplicity formula of Theorem~\ref{theorem:combinatorial-multiplicity}, before comparing the resulting weighted lattice-point sums with the characters obtained earlier by the orbifold localisation formula.
In the cotangent-bundle examples, we also show how the same dimensions can be checked independently using the projective-bundle description of the cut space.

Throughout this subsection, $d$ is a positive regular integral level in the outer cut chamber, so that Theorem~\ref{theorem:algebraic-cut-weight-space} identifies $F_{d}H^{0}( M_{\nu} ; \mcL_{\nu} )$ with $H^{0}( M_{\nu}^{\leq d} ; \mcL_{\nu}^{\leq d} )$.
In each example, we determine the slack function $\mcS$ and identify the $n$ independent generators $q_{1} , \ldots , q_{n}$ of cut degree one obtained after imposing the complex moment-map equations.
Together with the auxiliary coordinate $y$, these give the basis \eqref{eqn:weight-x-basis} used to calculate the multiplicities.

\begin{example}[$T^{\ast}\CC\PP^{1}$, continued]
	\label{example:TCP1-example-3}
	We continue Example~\ref{example:TCP1-example-2}, identifying $( \mft^{1} )^{\ast} \cong \RR$, and writing $x$ for the residual moment coordinate.
	From Example~\ref{example:TCP1-example-1}, the arrangement functionals are $\ell_{1}(x) = x$ and $\ell_{2}(x) = m - x$, so that $\mcS(x) = \max( 0 , -x ) + \max( 0 , x - m )$ and $|\Delta_{\nu}^{\leq d}| = \{ \mcS \leq d \} = [ -d , m + d ]$, which agrees with the moment-map image computed there.
	For $x \in \ZZ$, the minimal monomial \eqref{eqn:minimal-monomial} is $p_{x} = z_{1}^{\max( 0 , x )} z_{2}^{\max( 0 , m - x )} w_{1}^{\max( 0 , -x )} w_{2}^{\max( 0 , x - m )}$, which has cut degree $\mcS(x)$, so the remaining cut degree at the residual weight $x$ is $d - \mcS(x)$.

	The complex moment-map relation $z_{1}w_{1} + z_{2}w_{2} = 0$ leaves one independent generator $q$ of cut degree one.
	By \eqref{eqn:weight-x-basis}, the monomials $p_{x}q^{a}y^{c}$ with $a , c \in \ZZ_{\geq 0}$ and $a + c = d - \mcS(x)$ form a basis of the residual-weight-$x$, cut-degree-$d$ subspace, and counting these monomials gives $\mult_{d}(x) = d - \mcS(x) + 1$ when $\mcS(x) \leq d$, and zero otherwise, as in \eqref{eqn:multiplicity} with $n = 1$.
	Substituting the piecewise formula for $\mcS(x)$ gives
	\[
		\mult_{d}(x) = \begin{cases}
			d + x + 1, \qquad     & -d \leq x < 0,     \\
			d + 1, \qquad         & 0 \leq x \leq m,   \\
			m + d - x + 1, \qquad & m < x \leq m + d,
		\end{cases}
	\]
	and $\mult_{d}(x) = 0$ otherwise.
	So the multiplicity is constant on the compact base region $\Delta_{\emptyset} = [0 , m]$ and decreases linearly from its endpoints towards the cut boundaries at $-d$ and $m+d$.

	Summing over the integral weights in the three ranges gives
	\[
		\dim_{\CC} H^{0}( M_{\nu}^{\leq d} ; \mcL_{\nu}^{\leq d} ) = \sum_{x \in \ZZ} \mult_{d}(x) = 2 \cdot \frac{d( d + 1 )}{2} + ( m + 1 )( d + 1 ) = ( d + 1 )( m + d + 1 ),
	\]
	and Corollary~\ref{cor:weighted-lattice-point-character} gives the full residual character $\chi_{d}(t) = \sum_{x = -d}^{m + d}(d - \mcS(x) + 1)t^{x}$, which is the Laurent polynomial obtained by the orbifold localisation formula in Example~\ref{example:TCP1-example-2}.
\end{example}


For the independent dimension checks in the cotangent-bundle examples, let us first fix the projective-bundle convention and identify the cut polarisation.
Write $\pi:\PP(E) \ra X_{\nu}$ for the projection, with $\PP(E)$ parametrising the one-dimensional subspaces of the fibres of $E$.
Then $\mcO_{\PP(E)}(-1)$ is the tautological subbundle, and $\pi_{\ast}\mcO_{\PP(E)}(d) = \Sym^{d}(E^{\vee})$ for $d \geq 0$.
The cut divisor is the relative hyperplane divisor: the section $s_{\zeta}$ from Lemma~\ref{lemma:variation-polarisation-lemma} identifies $\mcO(D_{d}) \cong \mcO_{\PP(E)}(1)$.

The restriction of $\mcL_{\nu}^{\leq d}$ to the distinguished zero section determines the line bundle pulled back from $X_{\nu}$.
To determine the relative factor, consider a projective line in a fibre of $\pi$.
It is the level-$d$ symplectic cut of a complex line and has reduced symplectic area $d$ in our period-one normalisation.
Since $c_{1}(\mcL_{\nu}^{\leq d})=[\w_{d}]$, the cut line bundle has degree $d$ on that projective line, so the relative factor is $\mcO_{\PP(E)}(d)$~--~these observations determine the polarisations used in Remarks~\ref{remark:projective-bundle-TCP1-computation},~\ref{remark:projective-bundle-TCP2-computation},~and~\ref{remark:projective-bundle-TCP1xCP1-computation} below.

\begin{remark}
	\label{remark:projective-bundle-TCP1-computation}
	With this established, let us show how the dimension obtained in Example~\ref{example:TCP1-example-3} can also be recovered via the projective-bundle description of the cut space.
	Write $\pi : \PP(E) \ra X_{\nu} = \CC\PP^{1}$ with $E = \Omega_{\CC\PP^{1}}^{1} \oplus \mcO_{\CC\PP^{1}} \cong \mcO_{\CC\PP^{1}}(-2) \oplus \mcO_{\CC\PP^{1}}$, so that we have $M_{\nu}^{\leq d} \cong \PP(E)$ and $\mcL_{\nu}^{\leq d} \cong \pi^{\ast}\mcO_{\CC\PP^{1}}(m) \otimes \mcO_{\PP(E)}(d)$.
	Since $\pi_{\ast} \mcO_{\PP(E)}(d) \cong \Sym^{d}( \mcO_{\CC\PP^{1}}(2) \oplus \mcO_{\CC\PP^{1}} ) \cong \bigoplus_{a = 0}^{d} \mcO_{\CC\PP^{1}}(2a)$, this gives
	\[
		H^{0}( M_{\nu}^{\leq d} ; \mcL_{\nu}^{\leq d} ) \cong \bigoplus\nolimits_{a = 0}^{d} H^{0}( \CC\PP^{1} ; \mcO_{\CC\PP^{1}}(m + 2a) ),
	\]
	and taking dimensions gives $\sum_{a = 0}^{d} ( m + 2a + 1 ) = ( d + 1 )( m + d + 1 )$, as above.
\end{remark}

\begin{remark}
	\label{remark:lattice-point-count-manifold}
	Remark~\ref{remark:enlarged-cut-moment-image} gives a geometric interpretation of the multiplicities in Example~\ref{example:TCP1-example-3}.
	The enlarged moment polygon is $P_{\nu}^{\leq d} = \Conv \{ (-d , 0) , (0 , d) , (m , d) , (m + d , 0) \}$.
	For each integral $x \in [-d , m + d]$, its vertical slice runs from the height $0$ to $d - \mcS(x)$, which contains $d - \mcS(x) + 1$ lattice points, equalling the multiplicity already computed.

	Thus, keeping the auxiliary-circle moment coordinate $h$ turns the weighted count on the interval into the ordinary lattice-point count $\dim_{\CC}H^{0}(M_{\nu}^{\leq d} ; \mcL_{\nu}^{\leq d}) = \#( P_{\nu}^{\leq d} \cap \ZZ^{2} )$.
\end{remark}


\begin{example}[$A_{2}$-Resolution, continued]
	\label{example:A2-resolution-example-3}
	We return to Example~\ref{example:A2-resolution-example-2} with $m_{1}, m_{2}, d \in \ZZ_{> 0}$, $0 < m_{1} < m_{2}$, and $d - (m_{2} - m_{1}) = 2s$ with $s \in \ZZ_{> 0}$.
	Note that the condition $d - (m_{2} - m_{1}) = 2s > 0$ places $d$ in the outer cut chamber and makes the right boundary vertex $m_{2} + s$ integral, so the pullback fibre character at that vertex is trivial.
	Furthermore, we continue using the parity assumption so that the result we obtain here can be compared directly with the localisation calculation in Example~\ref{example:A2-resolution-example-2}.
	From Example~\ref{example:A2-resolution-example-1}, the arrangement functionals are $\ell_{1}(x) = x$, $\ell_{2}(x) = m_{1} - x$, and $\ell_{3}(x) = m_{2} - x$. The bounded regions are $\Delta_{\emptyset} = [0 , m_{1}]$ and $\Delta_{\{2\}} = [m_{1} , m_{2}]$, the latter being the moment polytope of the additional compact core component $\mcE_{\{2\}}$, and the rightmost unbounded region is $\Delta_{\{2 , 3\}} = [m_{2} , +\infty)$.
	Hence, $\mcS(x) = \max( 0 , -x ) + \max( 0 , x - m_{1} ) + \max( 0 , x - m_{2} )$ and $|\Delta_{\nu}^{\leq d}| = \{ \mcS \leq d \} = [ -d , m_{2} + s ]$, where the rightmost endpoint is $m_{2} + s = (m_{1} + m_{2} + d) / 2$.

	The complex moment-map relations $z_{1}w_{1} + z_{2}w_{2} = 0$ and $z_{1}w_{1} + z_{3}w_{3} = 0$ leave one independent generator $q$ of cut degree one.
	With $p_{x}$ as in \eqref{eqn:minimal-monomial}, the basis \eqref{eqn:weight-x-basis} consists of $p_{x}q^{a}y^{c}$ with $a , c \in \ZZ_{\geq 0}$ and $a + c = d - \mcS(x)$.
	Thus, as in the previous rank-one example, $\mult_{d}(x) = d - \mcS(x) + 1$ for $\mcS(x) \leq d$, and substituting the piecewise expression for $\mcS(x)$ gives
	\[
		\mult_{d}(x) = \begin{cases}
			d + x + 1, \qquad                  & -d \leq x < 0,             \\
			d + 1, \qquad                      & 0 \leq x \leq m_{1},       \\
			m_{1} + d - x + 1, \qquad          & m_{1} < x \leq m_{2},      \\
			m_{1} + m_{2} + d - 2x + 1, \qquad & m_{2} < x \leq m_{2} + s ,
		\end{cases}
	\]
	and $\mult_{d}(x) = 0$ otherwise.
	The multiplicity is constant on $\Delta_{\emptyset}$, decreases with slope $-1$ across $\Delta_{\{2\}}$, and with slope $-2$ on $\Delta_{\{2 , 3\}}^{\leq d}$.
	On the rightmost region, $\mcS(x) = 2x - m_{1} - m_{2}$, so solving for $\mcS(x) = d$ gives the right endpoint $(m_{1} + m_{2} + d) / 2$.

	The dimension is obtained by summing the multiplicities over the four ranges; the range $m_{2} < x \leq m_{2} + s$ contributes $\sum_{x = m_{2} + 1}^{m_{2} + s} (m_{1} + m_{2} + d - 2x + 1) = s^{2}$, where we have used $d - (m_{2} - m_{1}) = 2s$.
	Adding the sums over the other three ranges gives
	\[
		\begin{aligned}
			\dim_{\CC}H^{0}(M_{\nu}^{\leq d} ; \mcL_{\nu}^{\leq d})
			 & =
			\sum\nolimits_{x \in \ZZ}\mult_{d}(x) \\
			 & =
			(m_{1} + d + 1)(d + 1) - s(s + 1),
		\end{aligned}
	\]
	which agrees with the dimension obtained in Example~\ref{example:A2-resolution-example-2}.
	At the level of characters, Corollary~\ref{cor:weighted-lattice-point-character} identifies the character $\chi_{d}(t) = \sum_{x = -d}^{m_{2} + s} \mult_{d}(x) t^{x}$ with the one computed there by the orbifold localisation formula.
\end{example}


\begin{remark} \label{remark:lattice-point-count-orbifold}
	By Remark~\ref{remark:enlarged-cut-moment-image}, the enlarged moment polygon in the $A_{2}$ example is
	\[
		P_{\nu}^{\leq d} = \Conv \{(-d , 0) , (0 , d) , (m_{1} , d) , (m_{2} , d - (m_{2} - m_{1})) , (m_{2} + s , 0)\}.
	\]
	The interpretation of Remark~\ref{remark:lattice-point-count-manifold} that considers the vertical slices again gives the quantisation dimension as its ordinary lattice-point count~--~what changes here is the local lattice geometry at the rightmost vertex.

	At $(m_{2} + s , 0)$, the adjacent primitive edge directions $(-1 , 0)$ and $(-1 , 2)$ have determinant of absolute value $2$, so the polygon is not Delzant there.
	Rather, the corresponding point lying above the rightmost vertex has the local orbifold group $\ZZ/2\ZZ$ by \cite[\S6]{LermanTolman1997}, agreeing with $\Gamma_{m_{2} + s}$ in Example~\ref{example:A2-resolution-example-2}.
\end{remark}

The weighted lattice-point description allows us to compare successive members of the $A_{r}$-family explicitly: adding the arrangement point $m_{r}$ shortens the support and reduces the multiplicities to its right.
For $r \geq 2$, let us compare the $A_{r}$- and $A_{r - 1}$-resolutions at the same positive integral cut level $d$.
The $A_{r}$-arrangement is obtained from the $A_{r - 1}$ one by adding the point $m_{r}$, so the slack functions satisfy
\begin{equation}
	\label{eqn:Ar-slack-recursion}
	\mcS^{(r)}(x) = \mcS^{(r - 1)}(x) + \max\{ 0 ,\, x - m_{r} \}.
\end{equation}
Set $c_{\max}^{(r)} := (r - 1)m_{r} - \sum_{i < r} m_{i}$ and assume that $d > c_{\max}^{(r)}$.
Since $c_{\max}^{(r)} - c_{\max}^{(r - 1)} = (r - 1)(m_{r} - m_{r - 1}) > 0$, the same level lies in the outer cut chambers of both resolutions.

For $s \in \{ r , r - 1 \}$, write $x_{s}^{+} := (d + \sum_{i = 1}^{s} m_{i}) / s$ for the rightmost endpoint of the $A_{s}$-support, as in Example~\ref{example:A2-resolution-example-1}.
The outer-chamber condition gives $m_{r} < x_{r}^{+} < x_{r - 1}^{+}$, so the $A_{r}$-support is contained in the $A_{r - 1}$-support, with its rightmost endpoint still lying beyond $m_{r}$.
Let $a_{r} := \lfloor x_{r}^{+} \rfloor$ be the largest integral weight in the $A_{r}$-support.

By the rank-one multiplicity formula \eqref{eqn:multiplicity}, equation \eqref{eqn:Ar-slack-recursion} gives $\mult_{d}^{(r)}(x) = \mult_{d}^{(r - 1)}(x) - \max \{ 0 ,\, x - m_{r} \}$ for every integral weight $x$ in the $A_{r}$-support.
Therefore, the $A_{r}$-character is obtained from the $A_{r - 1}$ one by truncating at $a_{r}$ and reducing the remaining multiplicities for weights $x > m_{r}$:
\begin{equation}
	\label{eqn:Ar-character-recursion}
	\chi^{(r)}_{d}(t) = \big[ \chi^{(r - 1)}_{d}(t) \big]_{\leq a_{r}} - \sum\nolimits_{x = m_{r} + 1}^{a_{r}} ( x - m_{r} ) \, t^{x}.
\end{equation}

Starting with the $T^{\ast}\CC\PP^{1}$ character at $r = 1$ from Example~\ref{example:TCP1-example-2}, applying \eqref{eqn:Ar-character-recursion} repeatedly computes every member of the family by these two operations.
Thus, the changes in the Kawasaki boundary contributions in
Subsection~\ref{subsection:subsection-6.2-worked-equivariant-character} are reflected here by truncating the weight support and reducing the retained multiplicities.


\begin{example}[$T^{\ast}\CC\PP^{2}$, continued]
	\label{example:TCP2-example-3}
	We continue Example~\ref{example:TCP2-example-2}, identifying $( \mft^{2} )^{\ast} \cong \RR^{2}$ and writing $x = ( x_{1} , x_{2} )$.
	From Example~\ref{example:TCP2-example-1}, the arrangement functionals are $\ell_{1}(x) = x_{1}$, $\ell_{2}(x) = x_{2}$, and $\ell_{3}(x) = m - x_{1} - x_{2}$, so that $\mcS(x) = \max( 0 , -x_{1} ) + \max( 0 , -x_{2} ) + \max( 0 , x_{1} + x_{2} - m )$, and $|\Delta_{\nu}^{\leq d}| = \{ \mcS \leq d \}$ is the hexagonal support.
	The complex moment-map relation $z_{1}w_{1} + z_{2}w_{2} + z_{3}w_{3} = 0$ leaves two independent generators $q_{1} , q_{2}$ of cut degree one, hence the basis in \eqref{eqn:weight-x-basis} consists of the monomials $p_{x}q_{1}^{a}q_{2}^{b}y^{c}$ where $a, b, c \in \ZZ_{\geq 0}$ and $a + b + c = d - \mcS(x)$.
	Distributing the remaining degree among $q_{1},q_{2},y$ gives $\mult_{d}(x) = \tbinom{d-\mcS(x)+2}{2}$ when $\mcS(x) \leq d$, and zero otherwise, in agreement with \eqref{eqn:multiplicity}.
	Furthermore, the character is
	\[
		\chi_{d}(t_{1} , t_{2}) = \sum_{x \in |\Delta_{\nu}^{\le d}| \cap \ZZ^{2}} \binom{d - \mcS(x) + 2}{2} \, t_{1}^{x_{1}}t_{2}^{x_{2}},
	\]
	which agrees with the character computed by the orbifold localisation formula in Example~\ref{example:TCP2-example-2}, using Corollary~\ref{cor:weighted-lattice-point-character}.
	Unlike the two previous rank-one examples, the count here is now over a two-dimensional region, and the multiplicity at $x$ is quadratic rather than linear in the remaining cut degree $d - \mcS(x)$.
\end{example}


\begin{remark}
	\label{remark:projective-bundle-TCP2-computation}
	The projective-bundle model also gives an independent calculation of $\dim_{\CC} H^{0}(M_{\nu}^{\leq d} ; \mcL_{\nu}^{\leq d})$.
	Set $E:=\Omega_{\CC\PP^{2}}^{1}\oplus\mcO_{\CC\PP^{2}}$, and let $\pi:\PP(E)\ra\CC\PP^{2}$ be the projection, so that the cut space and its line bundle are $M_{\nu}^{\leq d} \cong \PP(E)$ and $\mcL_{\nu}^{\leq d}\cong \pi^{\ast}\mcO_{\CC\PP^{2}}(m) \otimes\mcO_{\PP(E)}(d)$.
	Using $\pi_{\ast}\mcO_{\PP(E)}(d) = \Sym^{d}(E^{\vee})$, we get
	\[
		H^{0}( M_{\nu}^{\leq d} ; \mcL_{\nu}^{\leq d} )
		\cong
		\bigoplus\nolimits_{a = 0}^{d} H^{0}\!\left( \CC\PP^{2} ; \mcO_{\CC\PP^{2}}(m) \otimes \Sym^{a}(T\CC\PP^{2}) \right).
	\]
	The Euler sequence $0 \to \mcO \to \mcO(1)^{\oplus 3} \to T\CC\PP^{2} \to 0$ gives, for $a \ge 0$,
	\[
		\dim_{\CC} H^{0}\!\left( \CC\PP^{2} ; \mcO_{\CC\PP^{2}}(m) \otimes \Sym^{a} (T\CC\PP^{2}) \right) = \binom{a + 2}{2}\binom{m + a + 2}{2} - \binom{a + 1}{2}\binom{m + a + 1}{2},
	\]
	and the sum over $a = 0, \ldots, d$ telescopes, giving $\dim_{\CC} H^{0}( M_{\nu}^{\leq d} ; \mcL_{\nu}^{\leq d} ) = \binom{d + 2}{2} \binom{m + d + 2}{2}$.
	This agrees with the dimension obtained by the localisation formula in Example~\ref{example:TCP2-example-2}, and hence with the weighted lattice-point count over the hexagonal support.
\end{remark}


\begin{example}[$T^{\ast}(\CC\PP^{1} \times \CC\PP^{1})$, continued]
	\label{example:TCP1xCP1-example-3}
	Lastly, we continue Example~\ref{example:TCP1xCP1-example-2}, using the same three-variable monomial count as in the previous $T^{\ast}\CC\PP^{2}$ example.
	From Example~\ref{example:TCP1xCP1-example-1}, the arrangement functionals are $\ell_{1}(x) = x_{1}$, $\ell_{2}(x) = x_{2}$, $\ell_{3}(x) = m_{1} - x_{1}$, $\ell_{4}(x) = m_{2} - x_{2}$, so that $\mcS(x) = \sum_{i = 1}^{4}\max( 0 , -\ell_{i}(x) )$ and $|\Delta_{\nu}^{\leq d}| = \{ \mcS \leq d \}$ is the octagonal support.
	The relations $z_{1}w_{1} + z_{3}w_{3} = 0$ and $z_{2}w_{2} + z_{4}w_{4} = 0$ leave two independent generators $q_{1} , q_{2}$ of cut degree one.
	Together with $y$, they give $\mult_{d}(x) = \binom{d - \mcS(x) + 2}{2}$ for $\mcS(x) \leq d$, and zero otherwise.
	Corollary~\ref{cor:weighted-lattice-point-character} therefore gives the character as the weighted lattice-point sum over this support, and also agrees with the character computed in Example~\ref{example:TCP1xCP1-example-2}.
\end{example}


\begin{remark}
	\label{remark:projective-bundle-TCP1xCP1-computation}
	Set $E := \mcO(-2 , 0) \oplus \mcO(0 , -2) \oplus \mcO$, and let $\pi : \PP(E) \ra \CC\PP^{1} \times \CC\PP^{1}$ be the projection.
	The projective-bundle model gives $M_{\nu}^{\leq d} \cong \PP(E)$ and $\mcL_{\nu}^{\leq d} \cong \pi^{\ast}\mcO(m_{1} , m_{2}) \otimes \mcO_{\PP(E)}(d)$.
	Here, $\pi_{\ast} \mcO_{\PP(E)}(d) = \Sym^{d}(E^{\vee})$ decomposes into the line bundles $\mcO(2a , 2b)$ with $a , b \geq 0$ and $a + b \leq d$, so
	\[
		\dim_{\CC} H^{0}(M_{\nu}^{\leq d} ; \mcL_{\nu}^{\leq d}) = \sum\nolimits_{\substack{a , b \geq 0 \\ a + b \leq d}} (m_{1} + 2a + 1)(m_{2} + 2b + 1).
	\]
	Evaluating this sum once again recovers the dimension formula in Example~\ref{example:TCP1xCP1-example-2}, and hence the projective-bundle model provides an independent check of the weighted lattice-point count over the octagonal support.
\end{remark}


The projective-bundle computations in Remarks~\ref{remark:projective-bundle-TCP1-computation},~\ref{remark:projective-bundle-TCP2-computation}, and~\ref{remark:projective-bundle-TCP1xCP1-computation} provide independent dimension checks in the cotangent-bundle examples.
The two descriptions organise the sections differently: the monomial basis \eqref{eqn:weight-x-basis} distributes the remaining cut degree among $q_{1} , \ldots , q_{n} , y$, while the projective-bundle calculation uses the symmetric powers in $\pi_{\ast}\mcO_{\PP(E)}(d) = \Sym^{d}(E^{\vee})$.
Note, however, that the weighted character formula \eqref{eqn:weighted-lattice-point-character} also applies when the compact core is reducible, as in Example~\ref{example:A2-resolution-example-3}, whereas the projective-bundle calculations can only be done in a few select cases.
It therefore computes the cut quantisations in both settings without requiring a projective-bundle description.

%% file: sections/8-conclusion.tex
\section{Conclusion}
\label{section:section-8-conclusion}

At positive regular integral levels in the outer cut chamber, the compact cut spaces provide exact finite stages of the hypertoric quantisation with the fibre-rotation grading.
Theorem~\ref{theorem:algebraic-cut-weight-space} identifies $H^{0}(M_{\nu}^{\leq d} ; \mcL_{\nu}^{\leq d})$ with the finite-dimensional filtration piece $F_{d}H^{0}(M_{\nu} ; \mcL_{\nu})$, while Corollary~\ref{cor:outer-chamber-vanishing} identifies the equivariant index with the class of this representation in $R(T^{n})$.
The successive filtration quotients recover the individual fibre-rotation weight spaces.
The construction therefore relates the non-compact quantisation to compact geometric quantisation without requiring an asymptotic limit of the fixed-point expressions.

The moment polyptych makes these finite stages computable in two complementary ways.
Its vertices, incident edges, and local lattice data, together with the pre-quantum fibre data, determine the contributions to the Kawasaki localisation formula in Theorem~\ref{theorem:character-hypertoric-orbifold}, while the slack function gives the residual-weight multiplicities in Theorem~\ref{theorem:combinatorial-multiplicity}.
By Corollary~\ref{cor:weighted-lattice-point-character}, the weighted lattice-point formula \eqref{eqn:weighted-lattice-point-character} computes the same character, and the worked examples exhibit this agreement.
For a smooth hypertoric variety $M_{\nu}$, Corollary~\ref{corollary:smooth-cut-space-characterisation} shows that an outer-chamber cut is smooth only if the distinguished open subset identified with $T^{\ast}X_{\nu}$ coincides with the whole of $M_{\nu}$.

Three directions for further study arise directly from these results.
First, the multiplicity formula counts monomials in the generators $q_{1} , \ldots , q_{n}$ and the auxiliary coordinate $y$, all of cut degree one, as in \eqref{eqn:weight-x-basis}.
The question is whether this count has a geometric realisation in higher rank.
At a positive regular integral level $d$ in the outer cut chamber, Remark~\ref{remark:enlarged-cut-moment-image} provides such an interpretation when $n = 1$: rotating the auxiliary cut coordinate enlarges the residual action to a toric action, and the lattice points in vertical slices of $P_{\nu}^{\leq d}$ count the multiplicities, as in Remarks~\ref{remark:lattice-point-count-manifold}~and~\ref{remark:lattice-point-count-orbifold}.
For $n \geq 2$, the enlarged torus has dimension $n + 1 < 2n$, so its action on the cut space is no longer toric.
For a residual lattice weight $x$ with $\mcS(x) \leq d$, the $d - \mcS(x) + 1$ lattice heights do not in general account for the full multiplicity $\binom{d - \mcS(x) + n}{n}$.
This binomial coefficient counts the lattice points in the $(d - \mcS(x))$-fold dilate of the standard $n$-simplex.
It remains to investigate whether a completely integrable system extending the enlarged torus action, or a toric degeneration of the cut space, can realise this simplex count geometrically.
Such a construction could provide a closer comparison with the moment-polytope geometry in the Jeffrey--Weitsman setting.

Second, what further geometric information is determined by the edge and lattice data of the moment polyptych, beyond the local representations and quantisation characters computed here?
Subsection~\ref{subsection:subsection-7.2-higher-cohomology-vanishing} describes variation within a cut chamber; a further problem is to describe how the polyptych and the cut quantisation change when the cut level crosses a wall.
One may also ask how the polyptych and quantisation change under reversal of the arrangement coorientations.
As Remark~\ref{remark:hyperkahler-analogues} explains, such a reversal can change the distinguished fibre-rotation action even when the underlying hypertoric variety remains unchanged up to $T^{n}$-equivariant diffeomorphism.

Third, the section-ring construction of the algebraic cut and the associated fibre-rotation filtration apply more generally to a polarised semiprojective variety equipped with a compatible algebraic $\CC^{\ast}$-action and a finitely generated section ring.
Our section-space and combinatorial results, however, use the additional properties of hypertoric varieties.
At outer-chamber levels, the unstable locus is $\mcC \times \{0\}$, of complex codimension $n + 1 \geq 2$.
Together with normality and reflexivity, this allows sections to extend uniquely across the unstable locus, thus allowing us to identify the filtration pieces with cut section spaces in Theorem~\ref{theorem:algebraic-cut-weight-space}.
Separately, the hypertoric coordinate ring provides an explicit monomial description of the filtration, yielding non-negative fibre-rotation weights and the residual $T^{n}$-weight multiplicity of Theorem~\ref{theorem:combinatorial-multiplicity}.
A natural question then is which other polarised semiprojective varieties admit analogous extension and combinatorial results.

The Higgs-bundle motivation mentioned in Section~\ref{section:section-1-introduction} raises a broader question: can analogous relations between the compact cut spaces, a circle grading, and weighted lattice-point counts be established for Higgs-bundle moduli?
Our results concern only hypertoric varieties and do not establish such a construction or an identification with the equivariant Verlinde formula.
A separate comparison concerns the brane approach to quantisation of Gukov and Witten and its development in mirror symmetry \cite{GukovWitten2009, Gukov2011}.
The discussion of the equivariant index and regularised traces in \cite[\S3.4]{Gukov2011} suggests a parallel with the quantisation and its fibre-rotation grading considered here; no identification of the brane Hilbert space with the virtual Dirac index is established in this paper.

In such a comparison, $M_{\nu}$ would play the role of the complex holomorphic-symplectic space rather than the real phase space to be complexified.
One would need to choose a suitable real form $M_{\RR}\subset M_{\nu}$, specify compatible symplectic and brane data, and determine how the residual $T^{n}$-action and fibre rotation would then interact with these choices~--~the real loci studied by Harada and Proudfoot \cite[\S5]{HaradaProudfoot2004} provide a possible starting point.
With those choices specified, one may ask whether the K\"ahler quantisation with its fibre-rotation grading, its cut filtration, and its weighted multiplicities admit compatible brane- or mirror-theoretic interpretations.